\documentclass{amsart}[12pt]
\usepackage{amsmath}
\usepackage{amsfonts}
\usepackage{amssymb}
\usepackage{mathrsfs}

\usepackage{color}
\usepackage{graphicx}
\usepackage{blkarray}
\usepackage{subfigure}
\usepackage[margin=1.2in]{geometry}

\usepackage[numbers]{natbib} \setcitestyle{open={},close={}}

\usepackage{tikz}
\usetikzlibrary{matrix,arrows,decorations.pathmorphing}

\allowdisplaybreaks[4]
\usepackage[all]{xy}

\newtheorem{theorem}{Theorem}[section]
\newtheorem{proposition}[theorem]{Proposition}
\newtheorem{lemma}[theorem]{Lemma}
\newtheorem{remark}[theorem]{Remark}

\newtheorem{definition}[theorem]{Definition}
\newtheorem{corollary}[theorem]{Corollary}

\newtheorem{conjecture}[theorem]{Conjecture}

\newcommand{\be}{\begin{equation}}
\newcommand{\ee}{\end{equation}}
\newcommand{\bea}{\begin{eqnarray}}
\newcommand{\eea}{\end{eqnarray}}
\newcommand{\ben}{\begin{eqnarray*}}
	\newcommand{\een}{\end{eqnarray*}}

\begin{document}

\title{The orbifold DT/PT  vertex correspondence}

\author[Yijie Lin]{Yijie Lin}

\address{School of Mathematics and Statistics\\Key Laboratory of Analytical Mathematics and Applications (Ministry of Education)\\Fujian Normal University\\350117 Fuzhou, P.R. China}
\email{yjlin@fjnu.edu.cn; yjlin12@163.com}

\maketitle

\begin{abstract}
We present an orbifold topological vertex formalism for PT invariants of toric Calabi-Yau 3-orbifolds with transverse $A_{n-1}$ singularities.
We give a proof of  the  orbifold DT/PT Calabi-Yau topological vertex correspondence. As an application, we derive  an explicit formula for the PT $\mathbb{Z}_{n}$-vertex in terms of loop Schur functions and prove the multi-regular orbifold DT/PT correspondence.
\end{abstract}




\tableofcontents

\section{Introduction}
The topological vertex has been proposed to be an effective tool for computing enumerative invariants such as Gromov-Witten invariants, Donaldson-Thomas invariants and Pandharipande-Thomas invariants of toric Calabi-Yau 3-folds since it made an appearance in physics for calculating the topological string amplitudes [\cite{AKMV}]. The GW topological vertex is  developed in [\cite{AKMV,LLZ1,LLZ2,LLLZ,OP,MOOP}] as a generating function of Hodge integrals while the DT topological vertex is a generating function of 3D partitions [\cite{ORV,MNOP1}]. Applying the technique of virtual localization [\cite{GP}], the GW/DT correspondence [\cite{MNOP1}] can be reduced to verify first the GW/DT vertex correspondence and then the compatibility with the edge terms in their gluing laws. As a geometric approach to the reduced DT theory, the PT theory [\cite{PT1}]  equips with  its topological vertex as a generating function of labelled box configurations [\cite{PT2}], where the conjectured DT/PT vertex correspondence  is resolved recently in [\cite{JWY}] for the Calabi-Yau case. There are  alternative approaches to the DT/PT correspondence  in [\cite{Bri1,Toda,ST}]. With the rationality property of PT theory [\cite{PP3,PP5}], the GW/PT correspondence is treated in [\cite{PP1,PP2,OOP1}].

For toric CY 3-orbifolds, the orbifold DT and GW topological  vertices are developed in [\cite{BCY}] and [\cite{Ross1}] as generating functions of 3D partitions colored by irreducible representations of a finite Abelian group and of Hurwitz-Hodge integrals respectively. The orbifold GW/DT vertex correspondence for toric CY 3-orbifolds with transverse $A_{n-1}$ quotient singularities is studied in [\cite{Ross1,Zong1,RZ1,RZ2,Ross2}].
The foundation of the (relative) orbifold PT theory in the general case is investigated in [\cite{Lyj1,Lyj2}]. When I had obtained  PT invariants of toric CY 3-orbifolds with transverse $A_{n-1}$ singularities (see Theorem \ref{PT-partition function}), I was informed by Patrick Lei last October  that Zhang [\cite{Zhang}] had also studied the orbifold PT topological vertex and obtained the 1-leg orbifold PT vertex formula in terms of Schur functions. Thanks to the  proof of [\cite{PT2}, Conjecture 2] by Jenne, Webb and Young in [\cite{JWY}],  I had independently obtained the same formula of the PT partition function as in Zhang's work [\cite{Zhang}, Theorem/Conjecture 4.19] and settled his unsolved  3-leg case. However, there are several  differences between ours as follows.  We have slightly different arguments on the description of  the $T$-fixed orbifold PT stable pairs, see [\cite{Zhang}, Section 4.6] and Section 2.2 respectively. Compared with [\cite{Zhang}, Proposition 4.25], the more explicit proof of $T$-equivariant K-theory decomposition of the finite cokernel $\mathtt{Q}$ as a sum of vertex terms is presented in Proposition \ref{K-decomposition}, which is essential for deriving the sign formula and hence the PT partition function later. Next, the PT partition function here is defined by virtual localization as in [\cite{PT2}] together with   perfect obstruction theories and virtual fundamental classes developed in [\cite{Lyj1}] especially for the 3-leg case with non-isolated fixed points, while  the PT invariant defined as Behrend's weighted Euler characteristic in [\cite{Zhang}, Definition 4.4]  fails to deal with the full 3-leg case by [\cite{BF}, Theorem 3.4], see Section 3.3 and [\cite{Zhang}, Section 4] respectively. Explicitly, we have the following PT version of [\cite{BCY}, Theorem 10].
\begin{theorem}(see Theorem \ref{PT-partition function} and [\cite{Zhang}, Theorem/Conjecture 4.19])
		Let $\mathcal{X}$ be a  toric CY 3-orbifold with transverse $A_{n-1}$ singularities. Then the PT partition function $PT(\mathcal{X})$ is
	\ben
	PT(\mathcal{X})=\underline{PT}(\mathcal{X})\bigg|_{q_{e,0}\to -q_{e,0},\;q\to-q}
	\een	
where
		\ben
		\underline{PT}(\mathcal{X}):=\sum_{\{\lambda_{e}\}\in\Lambda_{\mathcal{X}}}\prod_{e\in E_{\mathcal{X}}}E_{\lambda_{e}}^e\prod_{v\in V_{\mathcal{X}}}(-1)^{\varXi_{\widetilde{\pi}_{v}}}W_{\lambda_{1,v}\lambda_{2,v}\lambda_{3,v}}^{n_{e_{3,v}}}((-1)^{\widetilde{s}(\lambda_{3,v})}\mathbf{q}_{v})
		\een
	and
	\ben
	E_{\lambda_{e}}^e=(-1)^{\mathbf{S}_{\lambda_{e}}^e}\cdot\left(\prod_{k=0}^{n_{e}-1}v_{e,k}^{|\lambda_{e}|_{k,n_{e}}}\right)\cdot q_{e}^{C_{m_{e},m_{e}^\prime}^{\lambda_{e}}}\cdot\left(\overline{q_{f_{v_{0}}}^{A_{\lambda_{e}}}}\right)^{\delta_{0,e}}\cdot\left(q_{f^\prime_{v_{0}}}^{A_{\lambda^\prime_{e}}}\right)^{\delta_{0,e}^\prime}\cdot\left(q_{g_{v_{\infty}}}^{A_{\lambda_{e}}}\right)^{\delta_{\infty,e}}\cdot\left(\overline{q_{g^\prime_{v_{\infty}}}^{A_{\lambda^\prime_{e}}}}\right)^{\delta_{\infty,e}^\prime}
	\een
	with edges $(e,f_{v_{0}},f^\prime_{v_{0}},g_{v_{\infty}},g^\prime_{v_{\infty}})$  oriented   as in Figure \ref{figure1} (see other notations in Section 3).
\end{theorem}

In [\cite{Zhang}], the author has proved the 1-leg orbifold DT/PT vertex correspondence by obtaining the explicit formula for the 1-leg orbifold PT topological vertex. In this paper, we will extend this result to the full 3-leg case, i.e., prove the 3-leg orbifold DT/PT vertex correspondence by generalizing the DT/PT Calabi-Yau topological vertex correspondence in [\cite{JWY}] to the orbifold case.
 
In [\cite{JWY}], the DT topological vertex $V_{\lambda\mu\nu}$ (see Definition \ref{DT-vertex}) is described as some dimer model via the well-known correspondence between 3D partitions and  dimer configurations on the honeycomb graph while the PT topological vertex $W_{\lambda\mu\nu}$ (see Definition \ref{PT-vertex}) has certain double-dimer model description with the established correspondence between labelled box configurations and tripartite double-dimer configurations on the honeycomb graph. With these descriptions, the so-called graphical condensation discovered by Kuo [\cite{Kuo}, Theorem 5.1] is applied  to derive a graphical condensation recurrence for the DT vertex while the analogue technique developed in [\cite{Jenne}, Theorem 1.0.2] is employed to obtain a graphical condensation recurrence for the PT vertex. Then the DT/PT vertex correspondence [\cite{JWY}, Theorem 1.0.1] (see also Theorem \ref{DT-PT-correspondence}) is a consequence of the  same  recurrence obtained in both sides together with the base 2-leg case proved in [\cite{PT2}]. This strategy can be extended and carried out for  toric CY 3-orbifolds with cyclic orbifold structure explicitly as follows. The orbifold DT vertex $V^n_{\lambda\mu\nu}$ (see Definition \ref{orbifoldDT-vertex}) and the orbifold PT vertex $W^n_{\lambda\mu\nu}$ (see Definition \ref{orbifoldPT-vertex}) as generating functions of colored 3D partitions and colored labelled box configurations can be also described as some dimer model and double-dimer model respectively but with the same generalized weighted rule in Definition \ref{weight-rule}, see Section 4.1 and Section 5.2. Next, the graphical condensation recurrence for orbifold DT theory is obtained in Section 4.2 as follows
\ben
\frac{V^n_{\lambda\mu\nu}}{V^n_{\emptyset\emptyset\emptyset}} =\frac{V^n_{\lambda^{rc}\mu\nu}}{V^n_{\emptyset\emptyset\emptyset}}\cdot \frac{V^n_{\lambda\mu^{rc}\nu}}{V^n_{\emptyset\emptyset\emptyset}}\cdot\left(\frac{V^n_{\lambda^{rc}\mu^{rc}\nu}}{V^n_{\emptyset\emptyset\emptyset}}\right)^{-1}+q^{K_{1}(\lambda,\mu,\nu)}\cdot\frac{V^n_{\lambda^r\mu^c\nu}}{V^n_{\emptyset\emptyset\emptyset}}\cdot \frac{V^n_{\lambda^c\mu^r\nu}}{V^n_{\emptyset\emptyset\emptyset}}\cdot\left(\frac{V^n_{\lambda^{rc}\mu^{rc}\nu}}{V^n_{\emptyset\emptyset\emptyset}}\right)^{-1}
\een
which  coincides with the graphical condensation recurrence for orbifold PT theory in Section 5.3
\ben
W^{n}_{\lambda\mu\nu}
=W^{n}_{\lambda^{rc}\mu\nu}\cdot W^{n}_{\lambda\mu^{rc}\nu}\cdot\left(W^{n}_{\lambda^{rc}\mu^{rc}\nu}\right)^{-1}
+q^{K_{1}(\lambda,\mu,\nu)}\cdot W^{n}_{\lambda^r\mu^c\nu}\cdot W^{n}_{\lambda^c\mu^r\nu}\cdot\left(W^{n}_{\lambda^{rc}\mu^{rc}\nu}\right)^{-1}
\een
where 
\ben
q^{K_{1}(\lambda,\mu,\nu)}:=q_{-\widetilde{d}(\lambda^\prime)}^{d(\lambda^\prime)}\cdot q_{\widetilde{d}(\mu)}^{d(\mu)}\cdot \prod\limits_{i=-\widetilde{d}(\lambda^\prime)}^{\widetilde{d}(\mu)}q_{i}^{-1}\cdot\prod\limits_{i=\widetilde{d}(\lambda^\prime)+1}^{\ell(\lambda^\prime)}\left(\frac{q_{-i}}{q_{-i+1}}\right)^{\lambda^\prime_{i}}\cdot\prod\limits_{i=\widetilde{d}(\mu)+1}^{\ell(\mu)}\left(\frac{q_{i}}{q_{i-1}}\right)^{\mu_{i}}.
\een 
That is, both $Y_{1}(\lambda,\mu,\nu):=\frac{V^n_{\lambda\mu\nu}}{V^n_{\emptyset\emptyset\emptyset}}$ and $Y_{2}(\lambda,\mu,\nu):=W^{n}_{\lambda\mu\nu}$ are solutions $Y$ of the following equation:
\bea\label{recurrence1}
Y(\lambda,\mu,\nu)=\frac{Y(\lambda^{rc},\mu,\nu)\cdot Y(\lambda,\mu^{rc},\nu)}{Y(\lambda^{rc},\mu^{rc},\nu)}+q^{K_{1}(\lambda,\mu,\nu)}\cdot \frac{Y(\lambda^r,\mu^c,\nu)\cdot Y(\lambda^c,\mu^r,\nu)}{Y(\lambda^{rc},\mu^{rc},\nu)}.
\eea

The above  graphical condensation recurrence \eqref{recurrence1}  generalizes [\cite{JWY}, Equation (2)] if we set $q_{l}:=q$ for all $l$ (or equivalently $n=1$). However, it is not enough to determine the 3-leg orbifold DT/PT vertex correspondence since we don't have the 2-leg correspondence $Y_{1}(\lambda,\emptyset,\nu)=Y_{2}(\lambda,\emptyset,\nu)$ and $Y_{1}(\emptyset,\mu,\nu)=Y_{2}(\emptyset,\mu,\nu)$ as the base case via Remark \ref{size-comparision}.  Actually, it is observed that  the 2-leg correspondence $Y_{1}(\lambda,\mu,\emptyset)=Y_{2}(\lambda,\mu,\emptyset)$ can be obtained by Equation \eqref{recurrence1}, [\cite{Zhang}, Theorem 5.22] and Remark \ref{size-comparision}. This observation leads us to explore more graphical condensation recurrences from orbifold DT and PT theories in order to obtain the desired 2-leg  correspondence and hence the full 3-leg case. Specifically, with different choices of graphs for the dimer model and double-dimer model, both $Y_{1}(\lambda,\mu,\nu)$ and $Y_{2}(\lambda,\mu,\nu)$ are also the solutions $Y$ of the following two more recurrences: 
\bea\label{recurrence2}
Y(\lambda,\mu,\nu)=\frac{Y(\lambda^{rc},\mu,\nu)\cdot Y(\lambda,\mu,\nu^{rc})}{Y(\lambda^{rc},\mu,\nu^{rc})}+q^{K_{2}(\lambda,\mu,\nu)}\cdot \frac{Y(\lambda^r,\mu,\nu^c)\cdot Y(\lambda^c,\mu,\nu^r)}{Y(\lambda^{rc},\mu,\nu^{rc})}
\eea
and 
\bea\label{recurrence3}
Y(\lambda,\mu,\nu)=\frac{Y(\lambda,\mu^{rc},\nu)\cdot Y(\lambda,\mu,\nu^{rc})}{Y(\lambda,\mu^{rc},\nu^{rc})}+q^{K_{3}(\lambda,\mu,\nu)}\cdot \frac{Y(\lambda,\mu^r,\nu^c)\cdot Y(\lambda,\mu^c,\nu^r)}{Y(\lambda,\mu^{rc},\nu^{rc})}
\eea
where 
\ben
&&q^{K_{2}(\lambda,\mu,\nu)}:=q_{1-d(\lambda^\prime)}^{-\lambda^\prime_{d(\lambda^\prime)}}\cdot\prod\limits_{i=1}^{d(\lambda^\prime)-1}q_{-i}\cdot\prod\limits_{i=1}^{d(\lambda^\prime)-1}\left(\frac{q_{-i}}{q_{-i+1}}\right)^{\lambda^\prime_{i}}\cdot\prod\limits_{i=1}^{\ell(\mu)}\left(\frac{q_{i}}{q_{i-1}}\right)^{\mu_{i}}\cdot\prod\limits_{i=1}^{\nu_{d(\nu)}-d(\nu)}q_{i}^{-1},\\
&&q^{K_{3}(\lambda,\mu,\nu)}:=q_{d(\mu)-1}^{-\mu_{d(\mu)}}\cdot \prod\limits_{i=1}^{d(\mu)-1}q_{i}\cdot\prod\limits_{i=1}^{d(\mu)-1}\left(\frac{q_{i}}{q_{i-1}}\right)^{\mu_{i}}\cdot\prod\limits_{i=1}^{\ell(\lambda^\prime)}\left(\frac{q_{-i}}{q_{-i+1}}\right)^{\lambda^\prime_{i}}\cdot \prod_{i=1}^{\widetilde{d}(\nu)-d(\nu)}q_{-i}^{-1},
\een
see  Section 4.2 and Section 5.3 for more details.

When $Y$ is chosen to be $Y_{1}$ or $Y_{2}$, there exists certain symmetry between Equation \eqref{recurrence2} and Equation \eqref{recurrence3} due to  symmetries of $Y_{1}(\lambda,\mu,\nu)$ and $Y_{2}(\lambda,\mu,\nu)$ in Remarks \ref{symmetry1} and \ref{symmetry2} together with Lemma \ref{transpose}, but Equation \eqref{recurrence2} or \eqref{recurrence3} can not be derived from Equation \eqref{recurrence1} via symmetries unless $n=1$, see also Section 4.2 and 5.3. With the 1-leg correspondence from [\cite{Zhang}, Theorem 5.22] as the base case, it follows from graphical condensation recurrences \eqref{recurrence1}, \eqref{recurrence2} and \eqref{recurrence3} that 
\begin{theorem}(see Theorem \ref{3-leg correspondence})\label{mr-DT-PT}
If $\nu$ is multi-regular, then 	
\ben
\frac{V^n_{\lambda\mu\nu}}{V_{\emptyset\emptyset\emptyset}^n}=W_{\lambda\mu\nu}^n.
\een
\end{theorem}
This completes the  proof of the full 3-leg orbifold DT/PT vertex correspondence, which generalizes DT/PT vertex correspondence in [\cite{JWY}, Theorem 1.0.1]. 
Although we don't impose the multi-regular condition on $\nu$ in recurrences \eqref{recurrence1}, \eqref{recurrence2} and \eqref{recurrence3}, the 3-leg correspondence in Theorem \ref{mr-DT-PT} doesn't hold in general when  $\nu$ is not multi-regular and  $n>1$, see Remark \ref{DT-PT-correspondence1}.
Now we have the explicit formula for $W^n_{\lambda\mu\nu}$ in terms of loop Schur functions.
\begin{corollary}(see Corollary \ref{application1} and Remark \ref{loop-Schur-function})\label{explicit-formula}
	If $\nu$ is multi-regular, then 	
	\ben
	W_{\lambda\mu\nu}^n=q^{-A_{\lambda}}\cdot \overline{q^{-A_{\mu^\prime}}}\cdot \left(\prod_{(i,j)\in\nu}q_{i-j}^{-j}\right)\cdot\mathfrak{s}_{\nu}\cdot \sum_{\eta}q_{0}^{-|\eta|}\cdot \overline{s_{\lambda^\prime/\eta}(\mathbf{q}_{\bullet-\nu^\prime})}\cdot s_{\mu/\eta}(\mathbf{q}_{\bullet-\nu}).
	\een
\end{corollary}
Notice that if $\lambda,\mu=\emptyset$ and $\nu\neq\emptyset$, the  formula in Corollary \ref{explicit-formula} is true without multi-regularity of $\nu$ due to [\cite{Zhang}, Theorem 5.22].
The second application is 
\begin{corollary}(see Corollary \ref{application2})\label{mr DT/PT correspondence}
	Let $\mathcal{X}$ be a  toric CY 3-orbifold with transverse $A_{n-1}$ singularities. Then we have 
	\ben
	DT^\prime_{mr}(\mathcal{X})=PT_{mr}(\mathcal{X}).
	\een	
where $DT^\prime_{mr}(\mathcal{X})$	and $PT_{mr}(\mathcal{X})$ are defined in Section 2.3 and Section 3.3 respectively.
\end{corollary}	
This gives a proof of orbifold DT/PT correspondence [\cite{Zhang}, Conjecture 4.7] for toric CY 3-orbifolds with transverse $A_{n-1}$ singularities.  Restricted on any given multi-regular edge assignment (see Section 2.1), the equality in Corollary \ref{mr DT/PT correspondence} still  holds at the level of formal series  without using the  equality of rational functions proposed in the alternative proof [\cite{BCR}, Theorem A]. We may investigate orbifold GW/DT/PT correspondence and DT  crepant resolution conjecture for other  CY 3-orbifolds satisfying the hard Lefschetz condition such as the ones with singularities of type $D$ and $E$.  

This paper is organized as follows. In Section 2, we recollect some definitions  and properties of toric CY 3-orbifolds and definitions of DT partition functions, and characterize $T$-fixed locus of orbifold PT stable pairs. We devote Section 3 to present the PT orbifold topological vertex formalism, which resolves conjectural parts in [\cite{Zhang}, Conjectures 4.16, 4.19, 4.23]. In Section 4, we generalize the dimer description of the DT topological vertex in [\cite{JWY}, Section 3.2] to the orbifold case and derive three graphical condensation recurrences for orbifold DT theory, one of which generalizes the one in [\cite{JWY}, Section 3.3]. In Section 5, we show that AB configurations and double-dimer configurations as two combinatorial descriptions of  the PT topological vertex in [\cite{JWY}, Section 4.2, 4.3 and 4.4] can be extended to the ones for the  PT $\mathbb{Z}_{n}$-vertex, and obtain graphical condensation recurrences for orbifold PT theory containing the one in [\cite{JWY}, Section 4.5] as a special case. In Section 6, we present a proof of orbifold DT/PT vertex correspondence, which leads to an explicit formula for the PT $\mathbb{Z}_{n}$-vertex in terms of loop Schur functions  and  a proof of multi-regular orbifold DT/PT correspondence for toric CY 3-orbifolds with transverse $A_{n-1}$ singularities.

{\bf Acknowledgements.}
The author would like to thank Patrick Lei, a PhD student of Professor Chiu-Chu Melissa Liu, for bringing to my attention the work of Zhang [\cite{Zhang}],
and thank Xiaowen Hu for his helpful suggestions. The author also would like to thank an anonymous referee for his helpful  suggestions and comments.

\section{Preliminaries}

We first recollect some relevant ingredients of the structure of toric CY 3-orbifolds (with transverse $A_{n-1}$ singularities) in [\cite{BCY}] for formulating PT partition functions in Section 3, and then discuss the relation between $T$-fixed PT stable pairs and $T$-fixed orbifold PT stable pairs, which give the description of $T$-fixed orbifold PT stable pairs in   Section 2.2, see also [\cite{Zhang}, Section 4.6] for comparison as mentioned in Introduction. We also recall definitions of reduced (multi-regular) DT partition functions in [\cite{BCY}] for describing orbifold DT/PT correspondence later.

\subsection{Geometry of toric Calabi-Yau 3-orbifolds}
In this subsection, we briefly recall some relevant geometry of toric Calabi-Yau (CY) 3-orbifolds in [\cite{BCY}]. A CY 3-orbifold  is a smooth quasi-projective Deligne-Mumford stack $\mathcal{X}$ of dimension 3 over $\mathbb{C}$ with generically trivial  stabilizers and its canonical bundle is trivial, i.e., $K_{\mathcal{X}}\cong\mathcal{O}_{\mathcal{X}}$. The local model of  $\mathcal{X}$ at any point $p$ is given by $[\mathbb{C}^3/H_{p}]$, where $H_{p}$ is a finite subgroup of $SL(3,\mathbb{C})$ as the automorphism group of $p$. It is shown in [\cite{BCY}, Lemma 40] that a toric CY 3-orbifold $\mathcal{X}$ is uniquely determined by its coarse moduli space $X$, which is a simplicial toric variety. The Deligne-Mumford torus of $\mathcal{X}$ defined in [\cite{FMN}, Definition 3.1] is isomorphic to $T=(\mathbb{C}^*)^3\subset \mathcal{X}$ with the naturally extended action on $\mathcal{X}$. More generally, one may refer to [\cite{BCS,FMN}] for the theory of toric Deligne-Mumford stacks.

 A toric CY 3-orbifolds $\mathcal{X}$ is determined by its combinatorial data called a web diagram $\Gamma_{\mathcal{X}}$, which is a finite  trivalent planar graph with edges and vertices satisfying certain conditions, see [\cite{BCY}, Appendix B] for more details. Denote by  $V_{\mathcal{X}}$ all vertices of $\Gamma_{\mathcal{X}}$ and by $E_{\mathcal{X}}$ all edges of $\Gamma_{\mathcal{X}}$. Set $\Gamma_{\mathcal{X}}:=\{V_{\mathcal{X}},E_{\mathcal{X}}\}$, where vertices and edges correspond to torus fixed points  and torus invariant curves of $\mathcal{X}$ respectively, and regions divided by edges in the plane correspond to torus invariant divisors. Let $p_{v}$ be the torus fixed point corresponding to $v\in V_{\mathcal{X}}$. Here, $E_{\mathcal{X}}$ may contain some non-compact edges. If $e\in E_{\mathcal{X}}$ is a compact edge corresponding to a torus invariant curve $\mathcal{C}\subset\mathcal{X}$, then  $\mathcal{C}$ is an abelian gerbe [\cite{FMN}]  over a football $\mathbb{P}^1_{l_{0},l_{\infty}}$ constructed from a  $\mathbb{P}^1$ with root construction of order $l_{0}$ and $l_{\infty}$ at $0$ and $\infty$.
 
 Next, we assume that $\mathcal{X}$ is a toric CY 3-orbifold with transverse $A_{n-1}$ singularities, which means that nontrivial orbifold structure is supported on a union of disjoint smooth curves and the stabilizer along each curve is $\mathbb{Z}_{n}$ where $n$ is variable for different curves. Then the local model of $\mathcal{X}$ at a torus fixed point is of the form $[\mathbb{C}^3/\mathbb{Z}_{n}]$ with the generator of $\mathbb{Z}_{n}$ acting on the coordinates with weights $(1,-1,0)$. Denote by $\mathcal{C}_{e}$ the torus invariant curve (line) corresponding to $e$ and by $n_{e}$ the order of the stabilizer on $\mathcal{C}_{e}$ for each $e\in E_{\mathcal{X}}$. 
 
 An orientation on $\Gamma_{\mathcal{X}}$ is defined by taking a choice of direction for each edge and labeling three edges incident to each vertex $v$ by $(e_{1,v}, e_{2,v}, e_{3,v})$ in counterclockwise ordering where if an edge $e$ is incident to a vertex $v$ with $n_{e}>1$, define  $e_{3,v}=e$. On $\Gamma_{\mathcal{X}}$, we define an edge assignment by choosing a partition $\lambda_{e}$ for each edge $e$ with the convention that $\lambda_{e}=\emptyset$ for any non-compact edge.  Denoted by $\Lambda_{\mathcal{X}}$ the set of all edge assignments $\{\lambda_{e}\;|\;e\in E_{\mathcal{X}}\}$. With one edge assignment $\{\lambda_{e}\}$, we define a triple of partitions $(\lambda_{1,v},\lambda_{2,v},\lambda_{3,v})$ for each vertex $v$ by taking $\lambda_{i,v}=\lambda_{e_{i,v}}$ for  $e_{i,v}$ oriented outward from $v$ and $\lambda_{i,v}=\lambda_{e_{i,v}}^\prime$ for $e_{i,v}$ with the opposite orientation. Here, $\lambda^\prime=\{(j,i)\;|\;(i,j)\in\lambda\}$ is the conjugate of $\lambda$. This orientation convention is used to guarantee the compatibility between edge partitions and triple vertex outgoing partitions  in Section 3. We call an edge assignment  $\{\lambda_{e}\}$ multi-regular if for each $e\in E_{\mathcal{X}}$, we have $|\lambda_{e}|_{l}=\frac{1}{n}|\lambda|$ for all $l\in\{0,1,\cdots,n-1\}$ where $|\lambda_{e}|_{l}$ is defined in Section 3.4, and each partition $\lambda_{e}$ is called multi-regular.
 
 According to the orientation of $e$, let $\mathcal{D}_{e}$ and $\mathcal{D}_{e}^\prime$ be torus invariant divisors corresponding to the right and left regions incident to  $e$ respectively, see Figure \ref{figure1}. Denote by $p_{0,e}$ and $p_{\infty,e}$ the torus fixed points corresponding to the initial  and final vertices $v_{0}$, $v_{\infty}$ adjacent to $e$. Let $\mathcal{D}_{0,e}$ and $\mathcal{D}_{\infty,e}$ be the torus invariant divisors intersecting $\mathcal{C}_{e}$ transversely at $p_{0,e}$ and $p_{\infty,e}$. 
 Let $\mathcal{D}_{1,v}$, $\mathcal{D}_{2,v}$ and $\mathcal{D}_{3,v}$ be torus invariant divisors corresponding to the regions opposite to the edges $e_{1,v}$, $e_{2,v}$ and  $e_{3,v}$. Let $p_{e}$ be a generic point on $\mathcal{C}_{e}$. As in [\cite{BCY}], we assume that for each edge $e$ that
 \ben
 &&\mathcal{O}_{\mathcal{C}_{e}}(\mathcal{D}_{e})=\mathcal{O}_{\mathcal{C}_{e}}(m_{e} p_{e}-\delta_{0,e}p_{0,e}-\delta_{\infty,e}p_{\infty,e}),\\
 &&\mathcal{O}_{\mathcal{C}_{e}}(\mathcal{D}_{e}^\prime)=\mathcal{O}_{\mathcal{C}_{e}}(m^\prime_{e} p_{e}-\delta_{0,e}^\prime p_{0,e}-\delta_{\infty,e}^\prime p_{\infty,e})
 \een
where $\delta_{0,e}=1$ if the edge $f_{v_{0}}:=\mathcal{D}_{e}\cap\mathcal{D}_{0,e}$ is labelled by $e_{3,v_{0}}$ and $\delta_{0,e}=0$ otherwise, and $\delta^\prime_{0,e}=1$ if the edge $f^\prime_{v_{0}}:=\mathcal{D}^\prime_{e}\cap\mathcal{D}_{0,e}$ is labelled by $e_{3,v_{0}}$ and $\delta^\prime_{0,e}=0$ otherwise. Let $g_{v_{\infty}}:=\mathcal{D}_{e}\cap\mathcal{D}_{\infty,e}$  and $g^\prime_{v_{\infty}}:=\mathcal{D}_{e}^\prime\cap\mathcal{D}_{\infty,e}$. One can define $\delta_{\infty,e}, \delta_{\infty,e}^\prime$ corresponding to $g_{v_{\infty}}, g^\prime_{v_{\infty}}$ similarly. Then we have
\ben
&&\deg\mathcal{O}_{\mathcal{C}_{e}}(\mathcal{D}_{e})=m_{e}-\frac{\delta_{0,e}}{n_{f_{v_{0}}}}-\frac{\delta_{\infty,e}}{n_{g_{v_{\infty}}}},\\
&&\deg\mathcal{O}_{\mathcal{C}_{e}}(\mathcal{D}^\prime_{e})=m_{e}^\prime-\frac{\delta^\prime_{0,e}}{n_{f^\prime_{v_{0}}}}-\frac{\delta^\prime_{\infty,e}}{n_{g^\prime_{v_{\infty}}}}.
\een
The Calabi-Yau condition implies that
\ben
m_{e}+m_{e}^\prime-\delta_{0,e}-\delta_{0,e}^\prime-\delta_{\infty,e}-\delta^\prime_{\infty,e}+2=0.
\een

Let $K(\mathcal{X})$ be the Grothendieck group of compactly supported coherent sheaves on $\mathcal{X}$ modulo the numerical equivalence $\sim$, where $\mathcal{G}_{1}\sim \mathcal{G}_{2}$
if $\chi(\mathcal{G}_{1}\otimes\mathcal{H})=\chi(\mathcal{G}_{2}\otimes\mathcal{H})$ for any locally free sheave $\mathcal{H}$ on $\mathcal{X}$.
For $0\leq d\leq 3$, let $F_{d}K(\mathcal{X})$ be the subgroup of those   elements in $FK(\mathcal{X})$ with support of dimension at most $d$. It is shown in  [\cite{BCY}, Section 3.3] that  classes in $F_{1}K(\mathcal{X})$ are generated by $[\mathcal{O}_{p}]$ for a generic point $p$ in $\mathcal{X}$, $[\mathcal{O}_{p_{e}}\otimes\rho_{k}]$ for any (compact or not) edge $e\in E_{\mathcal{X}}$, and $[\mathcal{O}_{\mathcal{C}_{e}}(-1)\otimes\rho_{k}]$ for any compact edge $e\in E_{\mathcal{X}}$.
Here, $\rho_{k}$, $0\leq k< n_{e}$, are the irreducible representations of $\mathbb{Z}_{n_{e}}$ with the convention that $\mathcal{O}_{p_{e}}(-k\mathcal{D}_{e})\cong\mathcal{O}_{p_{e}}\otimes\rho_{k}$, and the classes $[\mathcal{O}_{p}], [\mathcal{O}_{p_{e}}\otimes\rho_{k}]$ and $[\mathcal{O}_{\mathcal{C}_{e}}(-1)\otimes\rho_{k}]$ are associated to variables $q$, $q_{e,k}$ and $v_{e,k}$ respectively. Among all these generators for $F_{1}K(\mathcal{X})$, we have the relation 
\ben
[\mathcal{O}_{p}]=[\mathcal{O}_{p_{e}}\otimes R_{reg}]=\sum_{k=0}^{n_{e}-1}[\mathcal{O}_{p_{e}}\otimes\rho_{k}]
\een
with the associated relation in variables $q=\prod\limits_{k=0}^{n_{e}-1}q_{e,k}$ for each edge $e\in E_{\mathcal{X}}$, where $R_{reg}=\sum\limits_{k=0}^{n_{e}-1}\rho_{k}$ is the regular representation of $\mathbb{Z}_{n_{e}}$.

\subsection{Orbifold PT stable pairs}
In this subsection, we will recall moduli spaces of orbifold PT stable pairs in [\cite{Lyj1,Lyj2}] and show that every orbifold PT stable pair has the similar characterization as in [\cite{PT2}].
Let  $\mathcal{X}$ be a 3-dimensional smooth projective Delinge-Mumford stack over $\mathbb{C}$.  The definition of every group in the natural filtration $F_{0}K(\mathcal{X})\subset F_{1}K(\mathcal{X})\subset F_{2}K(\mathcal{X})\subset F_{3}K(\mathcal{X})=K(\mathcal{X})$ is the same as the one in Section 2.1. Given $\beta\in F_{1}K(\mathcal{X})$, let $\mathrm{PT}(\mathcal{X},\beta)$ be the moduli space of orbifold PT stable pairs $\varphi: \mathcal{O}_{\mathcal{X}}\to\mathcal{F}$ with $[\mathcal{F}]=\beta$, where $\mathcal{F}$ is pure of dimension one and $\varphi$ has finite 0-dimensional  cokernel. Here,  $\mathrm{PT}(\mathcal{X},\beta)$ is denoted by $\mathrm{PT}^\beta_{\mathcal{X}/\mathbb{C}}$ in [\cite{Lyj2}, Remark 5.6], which  is a projective scheme over $\mathbb{C}$. And $\mathrm{PT}(\mathcal{X},\beta)$ is a fine moduli space, which carries a perfect obstruction theory and has a virtual fundamental class  $[\mathrm{PT}(\mathcal{X},\beta)]^{vir}$, see [\cite{Lyj1}, Theorem 5.20 and 5.21] for more details. Since $\dim\mathrm{Supp}\mathcal{F}=1$, we have  $\beta\notin F_{0}K(\mathcal{X})$ implicitly throughout this paper.

For an orbifold PT stable pair $\varphi: \mathcal{O}_{\mathcal{X}}\to\mathcal{F}$, as in [\cite{PT1}, Section 1.3], we have the following exact sequence
\ben
0\to\mathcal{I}_{\mathcal{C}_{\mathcal{F}}}\to\mathcal{O}_{\mathcal{X}}\to\mathcal{F}\to\mathrm{coker}\,\varphi\to0,
\een
where $\mathcal{C}_{\mathcal{F}}=\mathrm{Supp}(\mathcal{F})\subset\mathcal{X}$ is a stacky curve, and $\mathcal{I}_{\mathcal{C}_{\mathcal{F}}}$ is the ideal sheaf of $\mathcal{C}_{\mathcal{F}}$ in $\mathcal{X}$. Here,  $\mathrm{Im}\,\varphi\subset\mathcal{F}$  is equal to the structure sheaf $\mathcal{O}_{\mathcal{C}_{\mathcal{F}}}$ as a quotient of $\mathcal{O}_{\mathcal{X}}$, and is also a pure
sheaf which implies that locally in the $\acute{e}$tale topology $\mathcal{O}_{\mathcal{C}_{\mathcal{F}}}$ is a Cohen-Macaulay module by [\cite{BS}, Appendix C] and [\cite{Serre}, IV-B]. Then $\mathcal{C}_{\mathcal{F}}$ is Cohen-Macaulay. Set $\mathtt{Q}=\mathrm{coker}\,\varphi$. Denote by $\mathrm{Supp}^{red}(\mathtt{Q})$ the finite $0$-dimensional reduced closed substack support of  $\mathtt{Q}$. Then $\mathrm{Supp}^{red}(\mathtt{Q})\subset\mathcal{C}_{\mathcal{F}}$.
Let $\mathfrak{m}\subset \mathcal{O}_{\mathcal{C}_{\mathcal{F}}}$ be the ideal sheaf of a finite $0$-dimensional closed substack of $\mathcal{C}_{\mathcal{F}}$. It can be shown  as in [\cite{PT1}, Proposition 1.8] that an orbifold PT stable pair $\varphi: \mathcal{O}_{\mathcal{X}}\to\mathcal{F}$ (or denoted by $(\mathcal{F},\varphi)$) with support $\mathcal{C}_{\mathcal{F}}$ satisfying $\mathrm{Supp}^{red}(\mathtt{Q})\subset\mathrm{Supp}(\mathcal{O}_{\mathcal{C}_{\mathcal{F}}}/\mathfrak{m})$ is equivalent to a subsheaf of $\mathcal{H}om(\mathfrak{m}^r,\mathcal{O}_{\mathcal{C}_{\mathcal{F}}})/\mathcal{O}_{\mathcal{C}_{\mathcal{F}}}$ for $r\gg0$.
As in the proof of  [\cite{PT1}, Proposition 1.8], $\mathtt{Q}$ can be viewed as a coherent subsheaf of 
\ben
\lim\limits_{\longrightarrow}\mathcal{H}om(\mathfrak{m}^r,\mathcal{O}_{\mathcal{C}_{\mathcal{F}}})/\mathcal{O}_{\mathcal{C}_{\mathcal{F}}}
\een
and hence $\mathcal{F}$ can be viewed as a coherent subsheaf of $\lim\limits_{\longrightarrow}\mathcal{H}om(\mathfrak{m}^r,\mathcal{O}_{\mathcal{C}_{\mathcal{F}}})$ such that we have  the following short exact sequence 
\ben
0\to\mathcal{O}_{\mathcal{C}_{\mathcal{F}}}\to\mathcal{F}\to\mathtt{Q}\to0.
\een

Let $\mathcal{X}$ be a toric CY 3-orbifold.  We can compactify $\mathcal{X}$ to  the projective one $\overline{\mathcal{X}}$
and use the $T$-equivariant residue to define the PT invariants of $\mathcal{X}$ similarly as in  [\cite{OP1,MPT,Zhou2}]  due to the compactness of $T$-fixed locus $\mathrm{PT}(\mathcal{X},\beta)^\mathrm{T}$. One may view 
$\mathrm{PT}(\mathcal{X},\beta)$ as a quasi-projective moduli scheme as the restriction of $\mathrm{PT}(\overline{\mathcal{X}},\beta)$ by the equivalence of descriptions of (quasi)-projective Deligne-Mumford stacks [\cite{Kre}, Theorem 5.3 and Corollary 5.4].  Now we follow [\cite{PT2}, Section 2] to give a description of $T$-fixed orbifold PT stable pair. Let $[\mathcal{O}_{\mathcal{X}}\to\mathcal{F}]\in \mathrm{PT}(\mathcal{X},\beta)^\mathrm{T}$ be a $T$-fixed orbifold PT stable pair. Assume $\mathtt{Q}$ is the cokernel of $\varphi:\mathcal{O}_{\mathcal{X}}\to\mathcal{F}$. Then the sheaf $\mathcal{F}$ is supported on torus invariant curves and $\mathtt{Q}$ must be supported on torus fixed points.  Let $\mathcal{X}^T=\{p_{\alpha}\}$ be all the torus fixed points in $\mathcal{X}$. By [\cite{BCY}, Lemma 46], for each torus fixed point $p_{\alpha}$, there is a $T$-invariant open neighborhood $\mathcal{U}_{\alpha}=[\mathbb{C}^3/G_{\alpha}]$. Here, the finite subgroup $G_{\alpha}\subset T$ acts on $\mathbb{C}^3$ by $(z_{1}, z_{2}, z_{3})\mapsto (t_{1}z_{1},t_{2}z_{2},t_{3}z_{3})$ where $(t_{1},t_{2},t_{3})\in G_{\alpha}$ and $(z_{1}, z_{2}, z_{3})\in\mathbb{C}^3$.
Denote by $\varphi_{\alpha}: \mathcal{O}_{\mathcal{U}_{\alpha}}\to\mathcal{F}_{\alpha}$  the restriction of $\varphi: \mathcal{O}_{\mathcal{X}}\to\mathcal{F}$ on $\mathcal{U}_{\alpha}$,  then  $\varphi_{\alpha}$ is a $T$-invariant section of $T$-invariant sheaf $\mathcal{F}_{\alpha}$. Let $\mathcal{C}_{\alpha}$ be  the support of $\mathcal{F}_{\alpha}$. Assume that $\mathcal{C}_{\alpha}$ is nonempty, then $\mathcal{C}_{\alpha}$  is also Cohen-Macaulay and  supported on  at least one of three abelian gerbes over their corresponding footballs [\cite{FMN}] restricted in $\mathcal{U}_{\alpha}$. These three abelian gerbes along coordinate directions of $\mathcal{U}_{\alpha}$ intersect at the $T$-fixed point $[(0,0,0)/G_{\alpha}]\in\mathcal{C}_{\alpha}$.
Let $\mathtt{Q}_{\alpha}=\mathrm{coker}\,\varphi_{\alpha}$, we have the following exact sequence
\ben
0\to\mathcal{I}_{\mathcal{C}_{\alpha}}\to\mathcal{O}_{\mathcal{U}_{\alpha}}\to\mathcal{F}_{\alpha}\to\mathtt{Q}_{\alpha}\to0.
\een

Since $\mathcal{U}_{\alpha}=[\mathbb{C}^3/G_{\alpha}]$, coherent sheaves $\mathcal{I}_{\mathcal{C}_{\alpha}}$, $\mathcal{O}_{\mathcal{C}_{\alpha}}$, $\mathcal{O}_{\mathcal{U}_{\alpha}}$, $\mathcal{F}_{\alpha}$ and $\mathtt{Q}_{\alpha}$ are equivalent to the corresponding $G_{\alpha}$-equivariant sheaves  $\mathcal{I}_{C_{\alpha}}$, $\mathcal{O}_{C_{\alpha}}$, $\mathcal{O}_{U_{\alpha}}$, $F_{\alpha}$ and $\mathcal{Q}_{\alpha}$ on $U_{\alpha}=\mathbb{C}^3$  respectively by [44, Example 7.21]. Here $F_{\alpha}$  is  supported on  $C_{\alpha}\subset U_{\alpha}=\mathbb{C}^3$, which is corresponding to  $\mathcal{C}_{\alpha}\subset\mathcal{U}_{\alpha}$. And $C_{\alpha}$ is  $T$-fixed subscheme of pure dimension one and hence is Cohen-Macaulay. By the definition of (orbifold) PT stable pairs, an orbifold PT stable pair $\mathcal{O}_{\mathcal{U}_{\alpha}}\to\mathcal{F}_{\alpha}$ is equivalent to $G_{\alpha}$-equivariant PT stable pair  $\mathcal{O}_{U_{\alpha}}\to F_{\alpha}$ fitting into the short exact sequence $0\to\mathcal{I}_{C_{\alpha}}\to\mathcal{O}_{U_{\alpha}}\to F_{\alpha}\to\mathcal{Q}_{\alpha}\to0$ on $U_{\alpha}$  corresponding to [\cite{PT2}, (2-3)].
Since $\mathtt{Q}_{\alpha}$ has finite zero-dimensional support  at the origin of  $\mathcal{U}_{\alpha}=[\mathbb{C}^3/G_{\alpha}]$, then $\mathtt{Q}_{\alpha}$ can be viewed as a finite dimensional representation $\mathcal{Q}_{\alpha}$ of $G_{\alpha}$ by [\cite{CG}, Lemma 5.1.23 or Section 5.2.1].

By [\cite{PT1}, Proposition 1.8], a $T$-fixed orbifold PT stable pair $\varphi_{\alpha}: \mathcal{O}_{\mathcal{U}_{\alpha}}\to\mathcal{F}_{\alpha}$   corresponds to the $G_{\alpha}$-equivariant subsheaf of 
\ben
\lim\limits_{\longrightarrow}\mathcal{H}om(\mathbf{m}_{\alpha}^r,\mathcal{O}_{C_{\alpha}})/\mathcal{O}_{C_{\alpha}}
\een
on $U_{\alpha}=\mathbb{C}^3$, where $\mathbf{m}_{\alpha}$ is the ideal sheaf of $(0,0,0)$ in $C_{\alpha}\subset\mathbb{C}^3$.

To characterize  orbifold PT stable pairs, we now breifly recall some notation for the description of PT stable pairs $\mathcal{O}_{U_{\alpha}}\to F_{\alpha}$ in [\cite{PT2}]. Since $C_{\alpha}$ is $T$-fixed and Cohen-Macaulay, then $\mathcal{I}_{C_{\alpha}}$ is a monomial ideal in $\mathbb{C}[x_{1},x_{2},x_{3}]$ which determines a 3D partition $\pi$ given by the union of the infinite cylinders on the three nonnegative axes with cross-sections $\lambda, \mu, \nu$. As $\dim C_{\alpha}=1$, at least one of  $\lambda, \mu, \nu$ is nonempty. Let $\mathbf{M}=\bigoplus_{i=1}^3 \mathbf{M}_{i}$
 with 
 \ben
&&\mathbf{M}_{1}=\mathbb{C}[x_{1},x_{1}^{-1}]\otimes\frac{\mathbb{C}[x_{2},x_{3}]}{\lambda[x_{2},x_{3}]},\\
&&\mathbf{M}_{2}=\mathbb{C}[x_{2},x_{2}^{-1}]\otimes\frac{\mathbb{C}[x_{3},x_{1}]}{\mu[x_{3},x_{1}]},\\
&&\mathbf{M}_{3}=\mathbb{C}[x_{3},x_{3}^{-1}]\otimes\frac{\mathbb{C}[x_{1},x_{2}]}{\nu[x_{1},x_{2}]},
\een
where $\lambda[x_{2},x_{3}]$, $\mu[x_{3},x_{1}]$ and $\nu[x_{1},x_{2}]$ are  monomial ideals determined by partitions $\lambda$, $\mu$ and $\nu$ respectively as in  [\cite{PT2}, Section 2.2] .

By the similar argument in [\cite{PT2}, Section 2],  the $T$-fixed orbifold PT stable pair $(\mathcal{F}_{\alpha},\varphi_{\alpha})$ corresponds to a finitely generated $T$-invariant $\mathbb{C}[x_{1},x_{2},x_{3}]$-submodule 
\ben
\mathcal{Q}_{\alpha}\subset \mathbf{M}/\langle(1,1,1)\rangle
\een
 as a $G_{\alpha}$-module (or a representation of $G_{\alpha}$), and every finitely generated $T$-invariant $\mathbb{C}[x_{1},x_{2},x_{3}]$-submodule $\mathcal{Q}_{\alpha}\subset \mathbf{M}/\langle(1,1,1)\rangle$ as  a $G_{\alpha}$-module   can be viewed as the restriction of a $T$-fixed orbifold PT stable pair on $\mathcal{X}$ to $[\mathbb{C}^3/G_{\alpha}]$. Here the $i$-th component of $(1,1,1)$ is the canonical $T$-invariant element in $\mathbf{M}_{i}$, and it should be intepreted as 0 if $\mathbf{M}_{i}$ is 0. In [\cite{PT2}],   $\mathbf{M}/\langle(1,1,1)\rangle$ is also denoted by  $\mathbf{M}/\mathcal{O}_{C_{\alpha}}$. It is shown in [\cite{PT2}] that a connected component of the moduli space of $T$-invariant $\mathbb{C}[x_{1},x_{2},x_{3}]$-submodules of $\mathbf{M}/\langle(1,1,1)\rangle$ (or $\mathbf{M}/\mathcal{O}_{C_{\alpha}}$) is a product of $\mathbb{P}^1$'s. This implies that a connected component $\mathsf{Q}$ of $T$-fixed locus  of $\mathrm{PT}(\mathcal{X},\beta)$ is also a product of $\mathbb{P}^1$'s, which is projective and nonsingular. 
As in the proof of [\cite{PT2}, Proposition 3], the component $\mathsf{Q}$ is corresponding to some connected component of moduli space of the labelling of a collection of labelled box configurations, see Lemma \ref{PT-fix} for more details. See also [\cite{Zhang}, Lemma 4.21] for comparison.

\subsection{The DT partition function}
We briefly recall several DT partition functions in [\cite{BCY}]. Let $\mathcal{X}$ be a toric CY 3-orbifold with transverse $A_{n-1}$ singularities. Let $F_{mr}K(\mathcal{X})$ be the multi-regular part of $F_{1}K(\mathcal{X})$, each of whose elements is represented by some coherent sheaf such that  the associated representation of the stabilizer group at the generic point is a multiple of the regular representation, see [\cite{BCY}, Section 4.1] for more details. As $F_{0}K(\mathcal{X})\subset F_{mr}K(\mathcal{X})\subset F_{1}K(\mathcal{X})$, we define 
\ben
DT(\mathcal{X})=\sum_{\alpha\in F_{1}K(\mathcal{X})}DT_{\alpha}(\mathcal{X})q^\alpha;\;\;
DT_{mr}(\mathcal{X})=\sum_{\alpha\in F_{mr}K(\mathcal{X})}DT_{\alpha}(\mathcal{X})q^\alpha;\;\;
DT_{0}(\mathcal{X})=\sum_{\alpha\in F_{0}K(\mathcal{X})}DT_{\alpha}(\mathcal{X})q^\alpha
\een
and 
\ben
DT^\prime(\mathcal{X})=\frac{DT(\mathcal{X})}{DT_{0}(\mathcal{X})};\;\;\;\;DT_{mr}^\prime(\mathcal{X})=\frac{DT_{mr}(\mathcal{X})}{DT_{0}(\mathcal{X})}
\een
where $DT_{\alpha}(\mathcal{X})$ is defined by the weighted Euler characteristic of $\mathrm{Hilb}^\alpha(\mathcal{X})$ in [\cite{BCY}, Definition 1], and for some choice of a basis $\mathfrak{b}_{1},\cdots,\mathfrak{b}_{l}$ of $F_{1}K(\mathcal{X})$ and $\alpha=\sum_{i=1}^l c_{i} \mathfrak{b}_{i}$, we have $q^\alpha=q_{1}^{c_{1}}\cdots q_{l}^{c_{l}}$ and the similar definition of $q^\alpha$ for $\alpha\in F_{mr}K(\mathcal{X})$ and $\alpha\in F_{0}K(\mathcal{X})$. The explicit formula for $DT(\mathcal{X})$ is obtained as a sum over all edge assignments in [\cite{BCY}, Theorem 10] while $DT_{mr}(\mathcal{X})$ is the sum over all multi-regular edge assignments.

\section{The orbifold PT topological vertex}
The goal of this section is to derive the PT topological vertex formalism (see Theorem \ref{PT-partition function}) for toric CY 3-orbifolds with transverse $A_{n-1}$ singularities, which is also obtained by Zhang [\cite{Zhang}, Theorem/Conjecture 4.19] mentioned in Introduction where we  point out some differences between ours. In [\cite{JWY}], the authors prove the nonsingularity of $T_{0}$-fixed PT stable pairs, which is useful  for defining the PT partition function (see Formula \eqref{PT-formula}) in Section 3.3 via virtual localization. This  completes the proof of [\cite{Zhang}, Conjecture 4.23] (and [\cite{Zhang}, Proposition/Conjecture 4.16]) and hence the proof of [\cite{Zhang}, Theorem/Conjecture 4.19] for the full 3-leg case. For completeness and comparison, we present our argument here and suggest the reader also refer to [\cite{Zhang}, Section 4] for his argument.  
We first recall  the (orbifold)  DT and PT topological vertices and  the (orbifold) DT/PT vertex correspondence in Section 3.1. We explicitly present the proof of $K$-theory class of the finite cokernel $\mathtt{Q}$ as a sum over vertex terms in Proposition \ref{K-decomposition} which is important for deriving the sign formula  in Theorems \ref{sign-formula} and \ref{sign-formula2} and hence PT partition functions in Section 3.4 following the method in [\cite{BCY}, Section 6].

\subsection{Partitions, labelled box configurations and  topological vertices}
In this subsection, we recall  definitions of 3D partitions and the (orbifold) DT  topological vertex in [\cite{BCY}, Section 3.1], and then recollect labelled box configurations and the PT topological vertex defined in [\cite{PT2}], and finally give the definition of orbifold  PT topological vertex that will be used later.
\begin{definition}([\cite{BCY}, Definition 4])
Let $\lambda, \mu, \nu$ be three ordinary 2D partitions.  Let  $\pi$ be  the subset of $(\mathbb{Z}_{\geq0})^3$ satisfy the following conditions:\\
$(i)$ for $i, j, k\geq0$, if any of $(i+1,j,k)$, $(i,j+1,k)$, $(i,j,k+1)$ is in $\pi$, then $(i,j,k)$ is also in $\pi$;\\
$(ii)$ three 2D partitions $\lambda, \mu, \nu$ are associated to $\pi$ such that $(a)$ $(j,k)\in\lambda$ if and only if $(i,j,k)\in\pi$ for all $i\gg0$, $(b)$ $(k,i)\in\mu$ if and only if $(i,j,k)\in\pi$ for all $j\gg0$, $(c)$ $(i,j)\in\nu$ if and only if $(i,j,k)\in\pi$ for all $k\gg0$;\\
Then we call $\pi$ a 3D partition asymptotic to $(\lambda,\mu,\nu)$. We say  $\{(i,j,k)| (j,k)\in\lambda\}$, $\{(i,j,k)| (k,i)\in\mu\}$, and $\{(i,j,k)| (i,j)\in\nu\}$ the leg of $\pi$ in the first, second and third direction respectively.
\end{definition}
For the triple of ordinary 2D partitions $(\lambda,\mu,\nu)$, we employ some notation in  [\cite{JWY}, Section 2] and [\cite{PT2}, Section 2.5]. Let 
\ben
&&\mathrm{Cyl}_{1}=\{(i,j,k)\in\mathbb{Z}^3| (j,k)\in\lambda\},\\
&&\mathrm{Cyl}_{2}=\{(i,j,k)\in\mathbb{Z}^3| (k,i)\in\mu\},\\
&&\mathrm{Cyl}_{3}=\{(i,j,k)\in\mathbb{Z}^3| (i,j)\in\nu\}.
\een
For $1\leq l\leq3$, let $\mathrm{Cyl}_{l}^+=(\mathbb{Z}_{\geq0})^3\cap\mathrm{Cyl}_{l}$
 and $\mathrm{Cyl}_{l}^{-}=\mathrm{Cyl}_{l}\backslash \mathrm{Cyl}_{l}^+$. Set
 \ben
&&\mathrm{I}^{-}=\mathrm{Cyl}_{1}^{-}\cup\mathrm{Cyl}_{2}^{-}\cup\mathrm{Cyl}_{3}^{-},\\
&&\mathrm{I}^{+}=\left(\mathrm{Cyl}_{1}^{+}\cup\mathrm{Cyl}_{2}^{+}\cup\mathrm{Cyl}_{3}^{+}\right)\backslash(\mathrm{II}\cup\mathrm{III}),\\
&&\mathrm{II}_{\hat{1}}=\left(\mathrm{Cyl}_{2}\cap\mathrm{Cyl}_{3}\right)\backslash\mathrm{Cyl}_{1},\\ 
&&\mathrm{II}_{\hat{2}}=\left(\mathrm{Cyl}_{1}\cap\mathrm{Cyl}_{3}\right)\backslash\mathrm{Cyl}_{2},\\ 
&&\mathrm{II}_{\hat{3}}=\left(\mathrm{Cyl}_{1}\cap\mathrm{Cyl}_{2}\right)\backslash\mathrm{Cyl}_{3},\\ 
&&\mathrm{II}=\mathrm{II}_{\hat{1}}\cup\mathrm{II}_{\hat{2}}\cup\mathrm{II}_{\hat{3}},\\
&&\mathrm{III}=\mathrm{Cyl}_{1}\cap\mathrm{Cyl}_{2}\cap\mathrm{Cyl}_{3}.
 \een
Notice that $\mathrm{Cyl}_{l}$,  $\mathrm{Cyl}_{l}^{\pm}$, $\mathrm{I}^{\pm}$, $\mathrm{II}_{\hat{l}}$, $\mathrm{II}$, and  $\mathrm{III}$ are defined associated to $(\lambda,\mu,\nu)$, we also write them as $\mathrm{Cyl}_{l}(\lambda,\mu,\nu)$, $\mathrm{Cyl}_{l}^{\pm}(\lambda,\mu,\nu)$, $\mathrm{I}^{\pm}(\lambda,\mu,\nu)$, $\mathrm{II}_{\hat{l}}(\lambda,\mu,\nu)$, $\mathrm{II}(\lambda,\mu,\nu)$, and  $\mathrm{III}(\lambda,\mu,\nu)$ respectively.
Let $\pi$ be a 3D partition asymptotic to $(\lambda,\mu,\nu)$. Then $\mathrm{Cyl}_{l}^+$ is the leg of $\pi$ in the $l$-th direction defined above. Define the renormalized volume of $\pi$ as
\ben
\Vert\pi\Vert=\sum_{(i,j,k)\in\pi}\xi_{\pi}(i,j,k)
\een
where
\ben
\xi_{\pi}(i,j,k)=1-|\{l\in\{1,2,3\}|(i,j,k)\in\mathrm{Cyl}_{l}^+\}|.
\een
Then $\Vert\pi\Vert=|\pi\backslash(\mathrm{I}^{+}\cup\mathrm{II}\cup\mathrm{III})|-|\mathrm{II}|-2|\mathrm{III}|$ as in [\cite{JWY}, Section 3.1].
Let $\mathcal{P}(\lambda,\mu,\nu)$ be the set of all 3D partitions $\pi$ asymptotic to $(\lambda,\mu,\nu)$. Now we have the topological vertex defined by Okounkov, Reshetikhin and Vafa as follows.
\begin{definition}([\cite{BCY}, Definition 5]) \label{DT-vertex}
The $\mathrm{DT}$ topological vertex is defined by
\ben
V_{\lambda\mu\nu}=\sum_{\pi\in \mathcal{P}(\lambda,\mu,\nu)}q^{\Vert\pi\Vert}.
\een
\end{definition}
\begin{remark}
To avoid the confusion, we employ the notation $\Vert\cdot\Vert$ to denote the renormalized volume rather than $|\cdot|$ in [\cite{BCY}], which has been used to denote the number of elements in some set.
\end{remark}	
Assume that a finite abelian group $G$ acts on $\mathbb{C}^3$
with  characters $r_{1}, r_{2}, r_{3}$ for three coordinate lines. Let $\widehat{G}$ be the set of all the characters of irreducible representations of $G$, i.e, the set of all nonzero complex-valued functions on $G$. The DT $G$-vertex is defined in [\cite{BCY}] by
\ben
V_{\lambda\mu\nu}^G=\sum_{\pi\in\mathcal{P}(\lambda,\mu,\nu)}\prod_{r\in\widehat{G}}q_{r}^{\Vert\pi\Vert_{r}}
\een
where 
\ben
\Vert\pi\Vert_{r}=\sum_{\substack{(i,j,k)\in\pi\\r_{1}^i r_{2}^j r_{3}^k=r}}\xi_{\pi}(i,j,k),
\een
is the renormalized  number of boxes $(i,j,k)\in\pi$ with color $r\in\widehat{G}$. Roughly speaking, the orbifold topolocial vertex $V_{\lambda\mu\nu}^G$ counts the 3D partitions colored with elements of $\widehat{G}$.

If $G=\mathbb{Z}_{n}$, we have 
\begin{definition}([\cite{BCY}, Definition 6])\label{orbifoldDT-vertex}
The $\mathrm{DT}$ $\mathbb{Z}_{n}$-vertex is defined by
\ben
V_{\lambda\mu\nu}^n=\sum_{\pi\in\mathcal{P}(\lambda,\mu,\nu)}q_{0}^{\Vert\pi\Vert_{0}}\cdots q_{n-1}^{\Vert\pi\Vert_{n-1}}
\een
where for $0\leq l< n$,
\ben
\Vert\pi\Vert_{l}=\sum_{\substack{(i,j,k)\in\pi\\ i-j=l \; \mathrm{mod} \; n}}\xi_{\pi}(i,j,k)
\een
is the renormalized number of boxes $(i,j,k)\in\pi$ with color $i-j=l\mbox{ mod }n$.
\end{definition}
\begin{remark}\label{symmetry1}
There are  symmetries for the (orbifold) DT toplogical vertex [\cite{BCY}, Section 3.1] as follows
\ben
&&V_{\lambda\mu\nu}=V_{\mu^\prime\lambda^\prime\nu^\prime};\;\;V_{\lambda\mu\nu}=V_{\lambda^\prime\nu^\prime\mu^\prime};\;\;V_{\lambda\mu\nu}=V_{\nu^\prime\mu^\prime\lambda^\prime};\\
&&V_{\lambda\mu\nu}^n(q_{0},q_{1},\cdots,q_{n-1})=V^n_{\mu^\prime\lambda^\prime\nu^\prime}(q_{0},q_{n-1},\cdots,q_{1}), \;\;\;\mbox{i.e.,} \;\;\;V_{\lambda\mu\nu}^n=\overline{V}^n_{\mu^\prime\lambda^\prime\nu^\prime}.
\een
where the overline denotes the exchange of variables $q_{k}\leftrightarrow q_{-k}$ with the subscripts in $\mathbb{Z}_{n}$.
\end{remark}

The labelled box configurations are introduced in [\cite{PT2}] in order to describe the PT topological vertex. It is shown in [\cite{PT2}, Section 2.5] that a finitely generated $T$-invariant $\mathbb{C}[x_{1},x_{2},x_{3}]$-submodule of $\mathbf{M}/\langle(1,1,1)\rangle$ associated to $(\lambda,\mu,\nu)$ defined in Section 2.2  corresponds to a labelled box configuration with outgoing partitions $(\lambda,\mu,\nu)$ defined as follows. 
\begin{definition}([\cite{PT2}, Section 2.5])
A labelled box configuration with outgoing partitions $(\lambda,\mu,\nu)$ is defined as a finite subset of $\mathrm{I}^{-}\cup\mathrm{II}\cup\mathrm{III}$ whose elements are called boxes, where a type $\mathrm{III}$ box $w$ may be labelled by an element of 
\ben
\mathbb{P}^1=\mathbb{P}\left(\frac{\mathbb{C}\cdot\mathbf{1}_{w}\oplus\mathbb{C}\cdot\mathbf{2}_{w}\oplus\mathbb{C}\cdot\mathbf{3}_{w}}{\mathbb{C}\cdot(1,1,1)_{w}}\right).
\een
These boxes satisfy the following rule:\\
$(i)$ If $w=(w_{1}, w_{2}, w_{3})\in\mathrm{I}^{-}$ and any of 
\ben
(w_{1}-1,w_{2},w_{3}),\;\;(w_{1},w_{2}-1,w_{3}),\;\;(w_{1},w_{2},w_{3}-1)
\een
is a box, then $w$ must be a box.\\
$(ii)$ If $w\in\mathrm{II}_{\hat{i}}$ and any of 
\ben
(w_{1}-1,w_{2},w_{3}),\;\;(w_{1},w_{2}-1,w_{3}),\;\;(w_{1},w_{2},w_{3}-1)
\een
is a box other than a type $\mathrm{III}$ box labelled by the 1-dimensional subspace $\mathbb{C}\cdot \mathbf{i}$, then  $w$ must be a box.\\
$(iii)$ If $w\in\mathrm{III}$ and the span of the subspaces of 
\ben
\frac{\mathbb{C}\cdot\mathbf{1}_{w}\oplus\mathbb{C}\cdot\mathbf{2}_{w}\oplus\mathbb{C}\cdot\mathbf{3}_{w}}{\mathbb{C}\cdot(1,1,1)_{w}}
\een
induced by boxes 
\ben
(w_{1}-1,w_{2},w_{3}),\;\;(w_{1},w_{2}-1,w_{3}),\;\;(w_{1},w_{2},w_{3}-1)
\een
is not zero, then $w$ must be  a box. And if the span has dimension 2, then $w$ must be an ublabelled box. If the span has dimension 1, then $w$ may be either a box labelled by the span or an unlabelled box.
\end{definition}

Let  $\widetilde{\pi}$ be  a labelled box configuration with outgoing partitions $(\lambda,\mu,\nu)$.
We define the length of a labelled box configuration $\widetilde{\pi}$ by the sum of contributions from all boxes in $\widetilde{\pi}$ as follows: $(i)$ each  of  type $\mathrm{I}^{-}$ boxes, type $\mathrm{II}$ boxes, and  type $\mathrm{III}$ labelled boxes contributes 1; $(ii)$  each of  type $\mathrm{III}$ unlabelled boxes contributes 2.  Let $\mathrm{I}^{-}(\widetilde{\pi})$ be the set of  type $\mathrm{I}^{-}$ boxes in  $\widetilde{\pi}$ and $\mathrm{II}(\widetilde{\pi})$ the set of type $\mathrm{II}$ boxes in  $\widetilde{\pi}$.
Denote by $\mathrm{III}_{lb}(\widetilde{\pi})$ the set of type $\mathrm{III}$ labelled boxes in $\widetilde{\pi}$ and by $\mathrm{III}_{ulb}(\widetilde{\pi})$ the set of type $\mathrm{III}$ unlabelled boxes in $\widetilde{\pi}$.  Set $\mathrm{III}(\widetilde{\pi})=\mathrm{III}_{lb}(\widetilde{\pi})\cup\mathrm{III}_{ulb}(\widetilde{\pi})$. Now we
define the length of $\widetilde{\pi}$ by
\ben
\ell(\widetilde{\pi})=\sum_{(i,j,k)\in\widetilde{\pi}}\eta_{\widetilde{\pi}}(i,j,k)
\een
where 
\ben
\eta_{\widetilde{\pi}}(i,j,k)=\left\{
\begin{aligned}
	&1, \;\;\mbox{if} \;\; (i,j,k)\in\mathrm{I}^{-}(\widetilde{\pi})\cup\mathrm{II}(\widetilde{\pi})\cup\mathrm{III}_{lb}(\widetilde{\pi}); \\
	& 2, \;\;\mbox{if} \;\;(i,j,k)\in\mathrm{III}_{ulb}(\widetilde{\pi}).
\end{aligned}
\right.
\een
Notice that $|\widetilde{\pi}|=|\mathrm{I}^{-}(\widetilde{\pi})|+|\mathrm{II}(\widetilde{\pi})|+|\mathrm{III}_{lb}(\widetilde{\pi})|+|\mathrm{III}_{ulb}(\widetilde{\pi})|$ and $\ell(\widetilde{\pi})=|\widetilde{\pi}|+|\mathrm{III}_{ulb}(\widetilde{\pi})|$.
Let $\underline{\widetilde{\pi}}$ be the set of underlying boxes of $\widetilde{\pi}$ by forgetting the labellings of $\mathrm{III}_{lb}(\widetilde{\pi})$. And we call $\widetilde{\pi}$ a labelling of $\underline{\widetilde{\pi}}$.
Let $\widetilde{\mathcal{P}}(\lambda,\mu,\nu)$ be the set of components of the moduli space of  labelled box configurations with outgoing partitions $(\lambda,\mu,\nu)$. Let $[\widetilde{\pi}]\in\widetilde{\mathcal{P}}(\lambda,\mu,\nu)$ be the component of the moduli space of labellings of $\underline{\widetilde{\pi}}$ containing $\widetilde{\pi}$. It is obvious that all elements in one component $[\widetilde{\pi}]$ have the same length, which we denoted  by $\ell(\widetilde{\pi})$. Let $\pi_{\mathrm{min}}(\lambda,\mu,\nu)$ be the minimal 3D partition asymptotic to $(\lambda,\mu,\nu)$, i.e., the union of infinite cylinders on three nonegative axes with cross-sections $\lambda, \mu, \nu$.
It is proved in [\cite{PT2}, Proposition 3] that  each component $[\widetilde{\pi}]\in\widetilde{\mathcal{P}}(\lambda,\mu,\nu)$ corresponding to some component of moduli space of $T$-fixed $\mathbb{C}[x_{1},x_{2},x_{3}]$-submodule of $\mathbf{M}/\langle(1,1,1)\rangle$ is a product of $\mathbb{P}^{1}$'s.

\begin{definition}([\cite{PT2}, Section 5.3])\label{PT-vertex}
The $\mathrm{PT}$ topological vertex is defined by 
\ben
W_{\lambda\mu\nu}=\sum_{[\widetilde{\pi}]\in\widetilde{\mathcal{P}}(\lambda,\mu,\nu)}\chi_{\mathrm{top}}([\widetilde{\pi}])q^{\ell(\widetilde{\pi})+\Vert\pi_{\mathrm{min}}(\lambda,\mu,\nu)\Vert}
\een
where $\chi_{\mathrm{top}}([\widetilde{\pi}])$ is topological Euler characteristic of   $[\widetilde{\pi}]$.
\end{definition}
\begin{remark}\label{symmetry}
The PT topological vertex defined here is slightly different from that in [\cite{PT2}, (5-3)] by replacing the variable $-q$ by $q$, but coincides with the definition  in [\cite{JWY}, Section 4.1]
due to the equality $\ell(\widetilde{\pi})+\Vert\pi_{\mathrm{min}}(\lambda,\mu,\nu)\Vert=|\widetilde{\pi}|+|\mathrm{III}_{ulb}(\widetilde{\pi})|-|\mathrm{II}(\pi_{\mathrm{min}}(\lambda,\mu,\nu))|-2|\mathrm{III}	(\pi_{\mathrm{min}}(\lambda,\mu,\nu))|$. The symmetry of $W_{\lambda\mu\nu}$ is the same as that of $V_{\lambda\mu\nu}$:
\ben
W_{\lambda\mu\nu}=W_{\mu^\prime\lambda^\prime\nu^\prime};\;\;W_{\lambda\mu\nu}=W_{\lambda^\prime\nu^\prime\mu^\prime};\;\;W_{\lambda\mu\nu}=W_{\nu^\prime\mu^\prime\lambda^\prime}.
\een 
\end{remark}
The DT/PT equivariant topological vertex correspondence  is conjectured by Pandharipande and Thomas in [\cite{PT2}, Conjecture 4]  and proved  in the following Calabi-Yau case.
\begin{theorem}([\cite{JWY}, Theorem 1.0.1])\label{DT-PT-correspondence}
	\ben
	V_{\lambda\mu\nu}=M(q)W_{\lambda\mu\nu}
	\een
	where $M(q)$ is the MacMahon function defined by 
	\ben
	M(q)=\prod_{n\geq1}\frac{1}{(1-q^n)^n}.
	\een	
\end{theorem}

As in the DT side, we define the following PT $G$-vertex
\ben
W^{G}_{\lambda\mu\nu}=\sum_{[\widetilde{\pi}]\in\widetilde{\mathcal{P}}(\lambda,\mu,\nu)}\chi_{\mathrm{top}}([\widetilde{\pi}])\prod_{r\in\widehat{G}} q_{r}^{\ell(\widetilde{\pi})_{r}+\Vert\pi_{\mathrm{min}}(\lambda,\mu,\nu)\Vert_{r}}
\een
where 
\ben
\ell(\widetilde{\pi})_{r}=\sum_{\substack{(i,j,k)\in\widetilde{\pi}\\r_{1}^i r_{2}^j r_{3}^k=r}}\eta_{\widetilde{\pi}}(i,j,k)
\een
is the length of boxes  $(i,j,k)\in\widetilde{\pi}$ colored by $r\in\widehat{G}$. Again, when $G=\mathbb{Z}_{n}$, one expects the following definition of PT $\mathbb{Z}_{n}$-vertex used later, see also [\cite{Zhang}, Definition 4.27].
\begin{definition}\label{orbifoldPT-vertex}
The $\mathrm{PT}$ $\mathbb{Z}_{n}$-vertex is defined by 
\ben
W^{n}_{\lambda\mu\nu}=\sum_{[\widetilde{\pi}]\in\widetilde{\mathcal{P}}(\lambda,\mu,\nu)}\chi_{\mathrm{top}}([\widetilde{\pi}]) q_{0}^{\ell(\widetilde{\pi})_{0}+\Vert\pi_{\mathrm{min}}(\lambda,\mu,\nu)\Vert_{0}}\cdots q_{n-1}^{\ell(\widetilde{\pi})_{n-1}+\Vert\pi_{\mathrm{min}}(\lambda,\mu,\nu)\Vert_{n-1}}
\een
where for $0\leq l\leq n-1$,
\ben
\ell(\widetilde{\pi})_{l}=\sum_{\substack{(i,j,k)\in\widetilde{\pi}\\ i-j=l \; \mathrm{mod} \; n}}\eta_{\widetilde{\pi}}(i,j,k)
\een
is the length of boxes  $(i,j,k)\in\widetilde{\pi}$ colored by $i-j=l\mbox{ mod }n$.
\end{definition}
\begin{remark}\label{symmetry2}
As in Remark \ref{symmetry1}, we have the following symmetries for the $\mathrm{PT}$ $\mathbb{Z}_{n}$-vertex
\ben
W_{\lambda\mu\nu}^n(q_{0},q_{1},\cdots,q_{n-1})=W^n_{\mu^\prime\lambda^\prime\nu^\prime}(q_{0},q_{n-1},\cdots,q_{1})\;\;\;\mbox{i.e.,} \;\;\;W_{\lambda\mu\nu}^n=\overline{W}^n_{\mu^\prime\lambda^\prime\nu^\prime}.
\een
which has fewer symmetries than that of $W_{\lambda\mu\nu}$ in Remark \ref{symmetry}.
\end{remark}	
As in the manifold case, we expect the following conjecture  of the DT/PT $\mathbb{Z}_{n}$-orbifold topological vertex correspondence.
\begin{conjecture} ([\cite{Zhang}, Conjecture 5.20]) \label{orbifold-DT/PT-correspondence}
If $\nu$ is multi-regular, then
\ben
V^n_{\lambda\mu\nu}=V_{\emptyset\emptyset\emptyset}^nW_{\lambda\mu\nu}^n
\een
where
\ben
V_{\emptyset\emptyset\emptyset}^n=M(1,q)^n\prod_{0<a\leq b<n}M(q_{a}\cdots q_{b},q)M(q_{a}^{-1}\cdots q_{b}^{-1},q)
\een
and 
\ben
M(v,q)=\prod_{m=1}^\infty\frac{1}{(1-vq^m)^m}.
\een

\end{conjecture}
This conjecture will be proved in Section 6.

\subsection{$K$-theory decomposition of orbifold PT stable pairs}

As mentioned in Introduction, the more explicit proof of $T$-equivaraint $K$-theory decomposition of the finite cokernel $\mathtt{Q}$ is present in Proposition \ref{K-decomposition},  which is essential for obtaining the sign formula and PT partition functions in Section 3.3 and 3.4 respectively. One can refer to [\cite{Zhang}, Proposition 4.25] for comparison.  As a generalization of the description of $T$-fixed PT stable pairs in [\cite{PT2}, Section 2.6], we have the following  orbifold PT version of [\cite{BCY}, Lemma 13].
\begin{lemma}\label{PT-fix}
The $T$-fixed points in 
\ben
\coprod_{\beta\in F_{1}K(\mathcal{X})}\mathrm{PT}(\mathcal{X},\beta)
\een
are in bijective correspondence with the set $\{\widetilde{\pi}_{v}, \lambda_{e}\}$ where $\{\widetilde{\pi}_{v}\; |\; v\in V_{\mathcal{X}}\}$
is a collection of labelled box configurations $\widetilde{\pi}_{v}$ with outgoing partitions $(\lambda_{1,v}, \lambda_{2,v}, \lambda_{3,v})$ and $\{\lambda_{e} \;| \;e\in E_{\mathcal{X}}\}$ is an edge assignment which is compatible with triple outgoing partitions at all vertices as in Section 2.1.
\end{lemma}
\begin{proof}
It follows from the description of $T$-fixed orbifold PT stable pairs in Section 2.2. See also the proof of [\cite{Zhang}, Lemma 4.21].
\end{proof}

\begin{remark}\label{not-isolated}
In Lemma \ref{PT-fix},  $T$-fixed points are not isolated when the moduli space of the labelling of $\widetilde{\pi}_{v}$ for some $v\in V_{\mathcal{X}}$ has positive dimension, i.e., the moduli space is a $m$-th  product of $\mathbb{P}^1$'s for some integer $m\geq1$. And every labelled box configuration for each vertice has the color scheme  uniquely  determined by the coordinate positions of their boxes and representations of the local group.
\end{remark}
In order to formulate the $K$-theory classes of $T$-fixed orbifold PT stable pairs, we make the same convention as in DT side [\cite{PT2}, Remark 14]  to identify a box  in a labelled box configuration, or a 3D partition, or an edge partition with its corresponding divisor as follows. For a labelled box configuration $\widetilde{\pi}_{v}$, the notation $\mathcal{D}\in\widetilde{\pi}_{v}$ corresponding to a box $(i,j,k)\in\widetilde{\pi}_{v}$ means that $\mathcal{D}=i\mathcal{D}_{1,v}+j\mathcal{D}_{2,v}+k\mathcal{D}_{3,v}$. This convention is the same for  a 3D partition  $\pi_{v}$. For an edge partition $\lambda_{e}$, the notation $\mathcal{D}\in\lambda_{e}$ corresponding to a box $(i,j)\in\lambda_{e}$ means  $\mathcal{D}=i\mathcal{D}_{e}+j\mathcal{D}_{e}^\prime$. These divisors are compatible due to the orientation in Section 2.1. Now we have 

\begin{proposition}\label{K-decomposition}
Given $\beta\in F_{1}K(\mathcal{X})$, let $[\mathcal{O}_{\mathcal{X}}\xrightarrow{\varphi}\mathcal{F}]\in \mathrm{PT}(\mathcal{X},\beta)^\mathrm{T}$ be a $T$-fixed orbifold PT stable pair corresponding to the set $\{\widetilde{\pi}_{v}, \lambda_{e}\}$ of labelled box configurations $\widetilde{\pi}_{v}$ with outgoing partitions $(\lambda_{1,v},\lambda_{2,v},\lambda_{3,v})$   and edge partitions $\lambda_{e}$. Let $\mathtt{Q}=\mathrm{coker}\,\varphi$. Then we have the following $T$-equivariant $K$-theory decomposition:
\ben
\mathtt{Q}=\sum_{v\in V_{\mathcal{X}}}\sum_{\mathcal{D}\in\widetilde{\pi}_{v}}\eta_{\widetilde{\pi}_{v}}(\mathcal{D})\mathcal{O}_{p_{v}}(-\mathcal{D})
\een
and 
\ben
\mathcal{F}=\sum_{v\in V_{\mathcal{X}}}\left(\sum_{\mathcal{D}\in\widetilde{\pi}_{v}}\eta_{\widetilde{\pi}_{v}}(\mathcal{D})\mathcal{O}_{p_{v}}(-\mathcal{D})+\sum_{\mathcal{D}\in\overline{\pi}_{v}}\xi_{\overline{\pi}_{v}}(\mathcal{D})\mathcal{O}_{p_{v}}(-\mathcal{D})\right)+\sum_{e\in E_{\mathcal{X}}}\sum_{\mathcal{D}\in\lambda_{e}}\mathcal{O}_{\mathcal{C}_{e}}(-\mathcal{D})
\een
where  $\overline{\pi}_{v}:=\pi_{\mathrm{min}}(\lambda_{1,v},\lambda_{2,v},\lambda_{3,v})$  and the set $\{\overline{\pi}_{v},\lambda_{e} \}$ is corresponding to a $T$-invariant Cohen-Macaulay substack $\mathcal{C}_{\mathcal{F}}$ of dimension one, or equivalently a torus fixed point in $ \mathrm{Hilb}^{\alpha}(\mathcal{X})$ for some $\alpha\in F_{1}K(\mathcal{X})$ as in [\cite{BCY}, Lemma 13].
\end{proposition}

\begin{proof}
If $Y\subset\mathcal{X}$ is a $T$-invariant substack of dimension zero, then $\mathcal{O}_{Y}$ must be supported on the torus fixed points of $\mathcal{X}$. By [\cite{BCY}, Lemma 13] it corresponds to the set $\{\widehat{\pi}_{v}\;|\;v\in V_{\mathcal{X}}\}$ which is a collection of 3D partitions $\widehat{\pi}_{v}$ asymptotic to $(\emptyset,\emptyset,\emptyset)$. By [\cite{BCY}, Proposition 4] we have 
\ben
\mathcal{O}_{Y}=\sum_{v\in V_{\mathcal{X}}}\sum_{\mathcal{D}\in\widehat{\pi}_{v}}\mathcal{O}_{p_{v}}(-\mathcal{D}).
\een	
It is shown in the proof of [\cite{BCY}, Lemma 13] that for each $v$ the   ideal sheaf $\mathcal{I}_{\widehat{\pi}_{v}}\subset\mathcal{O}_{\mathcal{X}}$  corresponding to a $T$-invariant substack of dimension zero $Y\subset\mathcal{X}$ is generated by the image of the maps 
\ben
\mathcal{O}_{\mathcal{X}}(-i\mathcal{D}_{1,v}-j\mathcal{D}_{2,v}-k\mathcal{D}_{3,v})\to\mathcal{O}_{\mathcal{X}},\;\;\;(i,j,k)\notin\widehat{\pi}_{v}.
\een
For a $T$-invariant open neighborhood $\mathcal{U}_{v}:=[\mathbb{C}^3/G_{v}]$ of $p_{v}$, a  finite $T$-invariant $0$-dimensional sheaf supported on $[(0,0,0)/G_{v}]$ is equivalent to a finitely generated $T$-invariant $\mathbb{C}[x_{1},x_{2},x_{3}]$-module as a $G_{v}$-module.
Therefore  the $K$-class decomposition $\sum_{\mathcal{D}\in\widehat{\pi}_{v}}\mathcal{O}_{p_{v}}(-\mathcal{D})$  corresponds to   the finitely generated $T$-invariant $\mathbb{C}[x_{1},x_{2},x_{3}]$-module (as a $G_{v}$-module) whose $\mathbb{C}$-basis is given  by monomials $x_{1}^ix_{2}^jx_{3}^k$ for all $(i,j,k)$ belonging to the 3D partition $\widehat{\pi}_{v}$. 
This correspondence holds for any 3D partition $\widehat{\pi}_{v}$ asymptotic to $(\emptyset,\emptyset,\emptyset)$. And it has the  correspondence between tensor product of elements $\mathcal{O}_{p_{v}}(-\mathcal{D})$ for $\mathcal{D}=i\mathcal{D}_{1,v}+j\mathcal{D}_{2,v}+k\mathcal{D}_{3,v}$ in $K$-theory and multiplication of elements $x_{1}^ix_{2}^jx_{3}^k$ in the $T$-invariant $\mathbb{C}[x_{1},x_{2},x_{3}]$-module as a $G_{v}$-module. This implies that tensoring with $\mathcal{O}_{p_{v}}(\pm D_{i,v})$ corresponds to multiplication by $x_{i}^{\mp1}$ for $1\leq i\leq3$.

Since $\mathcal{C}_{\mathcal{F}}$ is a $T$-invariant Cohen-Macaulay substack of dimension one, it follows from  [\cite{BCY}, Proposition 4] that 
\ben
\mathcal{O}_{\mathcal{C}_{\mathcal{F}}}=\sum_{v\in V_{\mathcal{X}}}\sum_{\mathcal{D}\in\overline{\pi}_{v}}\xi_{\overline{\pi}_{v}}(\mathcal{D})\mathcal{O}_{p_{v}}(-\mathcal{D})+\sum_{e\in E_{\mathcal{X}}}\sum_{\mathcal{D}\in\lambda_{e}}\mathcal{O}_{\mathcal{C}_{e}}(-\mathcal{D})
\een
where the 3D partition $\overline{\pi}_{v}$ is equal to  $\pi_{\mathrm{min}}(\lambda_{1,v},\lambda_{2,v},\lambda_{3,v})$ as $\mathcal{C}_{\mathcal{F}}$ is Cohen-Macaulay  and the set $\{\overline{\pi}_{v},\lambda_{e} \}$ is corresponding to $\mathcal{C}_{\mathcal{F}}$ by [\cite{BCY}, Lemma 13]. Here, $\overline{\pi}_{v}$ is determined by $\widetilde{\pi}_{v}$ via its outgoing partitions.
Let $\mathfrak{m}$ be the ideal sheaf of torus fixed points in $\mathcal{C}_{\mathcal{F}}$. Then  on the local model at $p_{v}$ for each $v$, $\mathfrak{m}\subset\mathcal{O}_{\mathcal{C}_{\mathcal{F}}}$ is  generated by the image of the maps 
\ben
\mathcal{O}_{\mathcal{C}_{\mathcal{F}}}(-i\mathcal{D}_{1,v}-j\mathcal{D}_{2,v}-k\mathcal{D}_{3,v})\to\mathcal{O}_{\mathcal{C}_{\mathcal{F}}}
\een
where $(i,j,k)\in\overline{\pi}_{v}$ and $i+j+k\geq1$. Hence $\mathfrak{m}^r\subset\mathcal{O}_{\mathcal{C}_{\mathcal{F}}}$ is generated by the image of the maps
\ben
\mathcal{O}_{\mathcal{C}_{\mathcal{F}}}(-i\mathcal{D}_{1,v}-j\mathcal{D}_{2,v}-k\mathcal{D}_{3,v})\to\mathcal{O}_{\mathcal{C}_{\mathcal{F}}}
\een
where $(i,j,k)\in\overline{\pi}_{v}$ and $i+j+k\geq r$ on the local model at $p_{v}$ for each $v$. We define $\mathcal{I}_{r}\subset \mathcal{O}_{\mathcal{C}_{\mathcal{F}}}$ to be the ideal such that for each $v$ it is locally generated by the image of the maps
\ben
\mathcal{O}_{\mathcal{C}_{\mathcal{F}}}(-i\mathcal{D}_{1,v}-j\mathcal{D}_{2,v}-k\mathcal{D}_{3,v})\to\mathcal{O}_{\mathcal{C}_{\mathcal{F}}}
\een
where $(i,j,k)\in\overline{\pi}_{v}\setminus\{(i,j,k)\,|\,0\leq i,j,k< r\}$. Now it is easy to see that 
\ben
\lim\limits_{\longrightarrow}\mathcal{H}om(\mathfrak{m}^r,\mathcal{O}_{\mathcal{C}_{\mathcal{F}}})=\lim\limits_{\longrightarrow}\mathcal{H}om(\mathcal{I}_{r},\mathcal{O}_{\mathcal{C}_{\mathcal{F}}}).
\een
As in the direct sum decomposition of $\lim\limits_{\longrightarrow}\mathcal{H}om(\mathfrak{m}^r,\mathcal{O}_{\mathcal{C}_{\mathcal{F}}})$ as modules on the local model  [\cite{PT2}, Section 2.4], we have the following decomposition in $K$-theory 
\ben
\lim\limits_{\longrightarrow}\mathcal{H}om(\mathcal{I}_{r},\mathcal{O}_{\mathcal{C}_{\mathcal{F}}})=\lim\limits_{\longrightarrow}\sum_{e\in E_{\mathcal{X}}}\mathcal{H}om(\mathcal{I}_{r}|_{\mathcal{C}_{\lambda_{e}}},\mathcal{O}_{\mathcal{C}_{\lambda_{e}}})
\een
where $\mathcal{C}_{\lambda_{e}}$ is $T$-invariant substack of $\mathcal{C}_{\mathcal{F}}$ supported only on $\mathcal{C}_{e}$  with $\mathcal{O}_{\mathcal{C}_{\lambda_{e}}}=\sum_{\mathcal{D}\in\lambda_{e}}\mathcal{O}_{\mathcal{C}_{e}}(-\mathcal{D})$ by [\cite{BCY}, Proposition 4]. Notice that $\{\lambda_{1,v},\lambda_{2,v},\lambda_{3,v}\}$ and $\{\lambda_{e}\}$ are compatible.
Since  for $r\gg0$
\ben
\mathcal{H}om(\mathcal{I}_{r}|_{\mathcal{C}_{\lambda_{e}}},\mathcal{O}_{\mathcal{C}_{\lambda_{e}}})\otimes\mathcal{O}_{\mathcal{X}}(-r\mathcal{D}_{0,e}-r\mathcal{D}_{\infty,e})\cong\mathcal{H}om(\mathcal{I}_{r}|_{\mathcal{C}_{\lambda_{e}}}\otimes\mathcal{O}_{\mathcal{X}}(r\mathcal{D}_{0,e}+r\mathcal{D}_{\infty,e}),\mathcal{O}_{\mathcal{C}_{\lambda_{e}}})
\cong\mathcal{O}_{\mathcal{C}_{\lambda_{e}}},
\een
we have $\mathcal{H}om(\mathcal{I}_{r}|_{\mathcal{C}_{\lambda_{e}}},\mathcal{O}_{\mathcal{C}_{\lambda_{e}}})=\sum_{\mathcal{D}\in\lambda_{e}}\mathcal{O}_{\mathcal{C}_{e}}(-\mathcal{D}+r\mathcal{D}_{0,e}+r\mathcal{D}_{\infty,e})$. Notice that for each $\mathcal{D}\in\lambda_{e}$
\ben
\mathcal{O}_{\mathcal{C}_{e}}(-\mathcal{D}+r\mathcal{D}_{0,e}+r\mathcal{D}_{\infty,e})=\mathcal{O}_{\mathcal{C}_{e}}(-\mathcal{D})+\sum_{k=1}^r\mathcal{O}_{p_{0,e}}(-\mathcal{D}+k\mathcal{D}_{0,e})+\sum_{k=1}^r\mathcal{O}_{p_{\infty,e}}(-\mathcal{D}+k\mathcal{D}_{\infty,e}).
\een
Now we have
\ben
\sum_{e\in E_{\mathcal{X}}}\mathcal{H}om(\mathcal{I}_{r}|_{\mathcal{C}_{\lambda_{e}}},\mathcal{O}_{\mathcal{C}_{\lambda_{e}}})=\sum_{v\in V_{\mathcal{X}}}\sum_{i=1}^3\sum_{\mathcal{D}\in\lambda_{e_{i,v}}}\sum_{k=1}^{r}\mathcal{O}_{p_{v}}(-\mathcal{D}+k\mathcal{D}_{i,v})+\sum_{e\in E_{\mathcal{X}}}\sum_{\mathcal{D}\in\lambda_{e}}\mathcal{O}_{\mathcal{C}_{e}}(-\mathcal{D}).
\een
Then 
\ben
&&\lim\limits_{\longrightarrow}\mathcal{H}om(\mathfrak{m}^r,\mathcal{O}_{\mathcal{C}_{\mathcal{F}}})/\mathcal{O}_{\mathcal{C}_{\mathcal{F}}}\\
&=&\lim\limits_{\longrightarrow}\mathcal{H}om(\mathcal{I}_{r},\mathcal{O}_{\mathcal{C}_{\mathcal{F}}})/\mathcal{O}_{\mathcal{C}_{\mathcal{F}}}\\
&=&\lim\limits_{\longrightarrow}\left(\sum_{v\in V_{\mathcal{X}}}\left(\sum_{i=1}^3\sum_{\mathcal{D}\in\lambda_{e_{i,v}}}\sum_{k=1}^{r}\mathcal{O}_{p_{v}}(-\mathcal{D}+k\mathcal{D}_{i,v})+\sum_{\mathcal{D}\in\overline{\pi}_{v}}\left(-\xi_{\overline{\pi}_{v}}(\mathcal{D})\right)\mathcal{O}_{p_{v}}(-\mathcal{D})\right)\right).
\een
As before, for each $v$, one representative of the equivalence class $[x_{1}^ix_{2}^jx_{3}^k]$ in the $\mathbb{C}$-basis described  in [\cite{PT2}, Page 1848] of the $T$-module $\mathbf{M}/\langle(1,1,1)\rangle$ (as a $G_{v}$-module)   corresponds to a unique element $\mathcal{O}_{p_{v}}(-i\mathcal{D}_{1,v}-j\mathcal{D}_{2,v}-k\mathcal{D}_{3,v})$ in the $v$-term part of $K$-class decomposition of $\lim\limits_{\longrightarrow}\mathcal{H}om(\mathfrak{m}^r,\mathcal{O}_{\mathcal{C}_{\mathcal{F}}})/\mathcal{O}_{\mathcal{C}_{\mathcal{F}}}$ above. Here, the $\mathbb{C}$-dimension of  $[x_{1}^ix_{2}^jx_{3}^k]$ is equal to the coefficient of  $\mathcal{O}_{p_{v}}(-i\mathcal{D}_{1,v}-j\mathcal{D}_{2,v}-k\mathcal{D}_{3,v})$. 

As the $0$-dimensional sheaf $\mathtt{Q}$ is a coherent subsheaf of $\lim\limits_{\longrightarrow}\mathcal{H}om(\mathfrak{m}^r,\mathcal{O}_{\mathcal{C}_{\mathcal{F}}})/\mathcal{O}_{\mathcal{C}_{\mathcal{F}}}$, then the $K$-theory decomposition of   $\mathtt{Q}$  follows from  the  description of $\mathtt{Q}$ locally as a finitely generated $T$-invariant $\mathbb{C}[x_{1},x_{2},x_{3}]$-submodule of $\mathbf{M}/\langle(1,1,1)\rangle$ (as a $G_{v}$-module)  corresponding to  a labelled box  configuration $\widetilde{\pi}_{v}$ with outgoing partitions $(\lambda_{1,v},\lambda_{2,v},\lambda_{3,v})$ as shown in [\cite{PT2}, Section 2.5].

Now the proof is completed by  the following short exact sequence
\ben
0\to\mathcal{O}_{\mathcal{C}_{\mathcal{F}}}\to\mathcal{F}\to\mathtt{Q}\to0.
\een
\end{proof}
\begin{remark}\label{Simplification1}
By the definition of $\widetilde{\pi}_{v}$ and $\overline{\pi}_{v}$, the $K$-theory decomposition of $\mathcal{F}$ can be rewritten as
\ben
\mathcal{F}&=&\sum_{v\in V_{\mathcal{X}}}\sum_{\mathcal{D}\in \breve{\pi}_{v}}\zeta_{\breve{\pi}_{v}}(\mathcal{D})\mathcal{O}_{p_{v}}(-\mathcal{D})+\sum_{e\in E_{\mathcal{X}}}\sum_{\mathcal{D}\in\lambda_{e}}\mathcal{O}_{\mathcal{C}_{e}}(-\mathcal{D})
\een
where
\ben
\zeta_{\breve{\pi}_{v}}(i,j,k)=\left\{
\begin{aligned}
	&  1,\; \;\;\;\;\;\;\mbox{if} \;\;(i,j,k)\in\mathrm{I}^{-}(\widetilde{\pi}_{v}); \\
	& -1, \;\;\;\mbox{if}\;\; (i,j,k)\in(\mathrm{II}(\overline{\pi}_{v})\setminus \mathrm{II}(\widetilde{\pi}_{v}))\cup\mathrm{III}_{lb}(\widetilde{\pi}_{v});\\
	&-2,\;\;\;\mbox{if}\;\; (i,j,k)\in \mathrm{III}(\overline{\pi}_{v})\setminus\mathrm{III}(\widetilde{\pi}_{v})
\end{aligned}
\right.
\een
with
\ben
&&\breve{\pi}_{v}:=\mathrm{I}^{-}(\widetilde{\pi}_{v})\cup(\mathrm{II}(\overline{\pi}_{v})\setminus \mathrm{II}(\widetilde{\pi}_{v}))\cup\mathrm{III}_{lb}(\widetilde{\pi}_{v})\cup(\mathrm{III}(\overline{\pi}_{v})\setminus\mathrm{III}(\widetilde{\pi}_{v}))\\
&&\mathrm{II}(\overline{\pi}_{v})=\mathrm{II}(\lambda_{1,v},\lambda_{2,v},\lambda_{3,v}),\;\; \mathrm{III}(\overline{\pi}_{v})=\mathrm{III}(\lambda_{1,v},\lambda_{2,v},\lambda_{3,v}).
\een

\end{remark}

\subsection{The PT partition function}
As mentioned in Introduction,  based on the perfect obstruction theory and virtual fundamental classes developed in [\cite{Lyj1}, Section 5] for orbifold PT theory, we define the PT partition function via the virtual localization following [\cite{PT2}] as follows. 

Let $\mathcal{X}$ be a toric  CY 3-orbifold. 
The DT invariant of $\mathcal{X}$ is defined in [\cite{BCY}, Section 5.1] by Behrend's weighted Euler characteristic [\cite{Beh}]  of the moduli space $\mathrm{Hilb}^{\alpha}(\mathcal{X})$ for  $\alpha\in F_{1}K(\mathcal{X})$, which is equal to a signed count over the $T$-fixed points  by [\cite{BF}, Theorem 3.4 and Corollary 3.5].  This definition is also employed  for  PT invariants of $\mathcal{X}$ in  [\cite{Zhang}, Section 4].  However,  it is not suitable to  take the formula in [\cite{BF}, Theorem 3.4 and Corollary 3.5] directly for the definition of  PT invariant of $\mathcal{X}$ since the $T$-fixed points of the moduli space $\mathrm{PT}(\mathcal{X},\beta)$ for $\beta\in F_{1}K(\mathcal{X})$ are not necessary isolated as stated in  Remark \ref{not-isolated}.  More generally, as the integration over the virtual fundamental class  $[\mathrm{PT}(\mathcal{X},\beta)]^{vir}$,  PT invariants of $\mathcal{X}$ can be defined by $T$-equivariant residue, which resolves  [\cite{Zhang}, Conjecture 4.23, 4.16, 4.19].

The virtual fundamental class of $\mathrm{PT}(\mathcal{X},\beta)$  has been defined for a $3$-dimensional smooth projective Deligne-Mumford stack $\mathcal{X}$ in [\cite{Lyj1}, Theorem 5.21].  However, a toric CY 3-orbifold $\mathcal{X}$ here  is quasi-projective, we will employ the strategy as in  [\cite{BP,OP1}] by restricting to the $T$-fixed locus of $\mathrm{PT}(\mathcal{X},\beta)$ each of whose  connected component is a product of $\mathbb{P}^1$'s, and  considering the $T$-equivariant perfect obstruction theories on $T$-fixed locus and the associated virtual fundamental classes in $T$-equivariant Chow ring, and then applying the virtual localization formula [\cite{GP}] to define $T$-equivariant integral for PT invariant of $\mathcal{X}$.
Actually, we will follow the argument in  [\cite{PT2}, Section 4.2 and Section 5.2].

Let  $\mathsf{Q}$ be a connected component of $T$-fixed locus of $\mathrm{PT}(\mathcal{X},\beta)$. By Lemma \ref{PT-fix},  we have $\mathsf{Q}=\prod_{v\in V_{\mathcal{X}}}\mathsf{Q}_{v}$ where $\mathsf{Q}_{v}$ is a connected component of moduli space of labelled box configurations containing $\widetilde{\pi}_{v}$ for each $v\in V_{\mathcal{X}}$. Each $\mathsf{Q}_{v}$ can be viewed as a connected component $[\widetilde{\pi}_{v}]$ of moduli space of labellings of $\underline{\widetilde{\pi}_{v}}$ due to [\cite{PT2}, Proposition 3]. Since $\mathsf{Q}_{v}$ is a product of $\mathbb{P}^1$'s, then $\mathsf{Q}$ is also
a product of $\mathbb{P}^1$'s which is obviously projective and nonsingular. By the similar argument  in [\cite{PT2}, Section 4.1 and 4.2], we have the following formula in localized $T$-equvariant Chow ring of $\mathrm{PT}(\mathcal{X},\beta)$:
\be\label{vir-loc-1}
\sum_{\mathsf{Q}\subset\mathrm{PT}(\mathcal{X},\beta)^T}\varsigma_{*}\left(e(T_{\mathsf{Q}})\cdot\frac{e(E_{\mathsf{Q}}^1)}{e(E_{\mathsf{Q}}^{0})}\cap[\mathsf{Q}]\right)=[\mathrm{PT}(\mathcal{X},\beta)]^{\mathrm{vir}},
\ee
where $\varsigma: \mathsf{Q}\hookrightarrow\mathrm{PT}(\mathcal{X},\beta)^T$ is the inclusion map  and $E_{\mathsf{Q}}^0\to E_{\mathsf{Q}}^1$ is the dual of $T$-equivariant perfect obstruction theory restricted on $\mathsf{Q}$. One may take integration over the left hand side of the equation \eqref{vir-loc-1} to define PT invariant of $\mathcal{X}$ as in [\cite{BP,OP1}] although it seems difficult to follow [\cite{PT2}, Section 4]  to compute $T$-equivariant  class  $\frac{e(E_{\mathsf{Q}}^1)}{e(E_{\mathsf{Q}}^{0})}$. Alternatively, we will adopt another virtual localization with a different action, which is easy to deal with $T$-equivariant Euler classes in  our CY case.

Let $T_{0}=\{(t_{1},t_{2},t_{3})\in T\;|\;t_{1}t_{2}t_{3}=1\}$. Then $T_{0}$ preserves the canonical form of $\mathcal{X}$ since it acts trivially on the form $dz_{1}\wedge dz_{2}\wedge dz_{3}$ on the local model $[\mathbb{C}^3/G] $ of each $T$-fixed point due to [\cite{BCY}, Lemma 46]. Let $\mathsf{Q}_{0}$ be a connected component of  $T_{0}$-fixed locus of $\mathrm{PT}(\mathcal{X},\beta)$.  The moduli space of $T_{0}$-invariant submodules of  $\mathbf{M}/\langle(1,1,1)\rangle$ as in Section 2.2 is conjectured to be nonsigular [\cite{PT2}, Conjecture 2], and fortunately it is  proved recently in [\cite{JWY}]. Then $\mathsf{Q}_{0}$ is also nonsingular. And as above we have
\be\label{vir-loc-2}
\sum_{\mathsf{Q}_{0}\subset\mathrm{PT}(\mathcal{X},\beta)^{T_{0}}}\varsigma_{0*}\left(e(T_{\mathsf{Q}_{0}})\cdot\frac{e(E_{\mathsf{Q}_{0}}^1)}{e(E_{\mathsf{Q}_{0}}^0)}\cap[\mathsf{Q}_{0}]\right)=[\mathrm{PT}(\mathcal{X},\beta)]^{\mathrm{vir}}
\ee
in $A^{T_{0}}_{*}(\mathrm{PT}(\mathcal{X},\beta))_{\mathrm{loc}}$,
where $\varsigma_{0}: \mathsf{Q}_{0}\hookrightarrow\mathrm{PT}(\mathcal{X},\beta)^{T_{0}}$ is the inclusion map and $E_{\mathsf{Q}_{0}}^0\to E_{\mathsf{Q}_{0}}^1$ is the dual of $T_{0}$-equivariant symmetric perfect obstruction theory $E^{-1}\to E^0$ (see [\cite{Lyj1}, Section 5.3]) restricted on $\mathsf{Q}_{0}$. By the self-duality of perfect obstruction theory, we define the following PT invariant of $\mathcal{X}$ by integration against the left hand side classes of the formula $\eqref{vir-loc-2}$:
\ben
PT_{\beta}(\mathcal{X})&=&\sum_{\mathsf{Q}_{0}\subset\mathrm{PT}(\mathcal{X},\beta)^{T_{0}}}\int_{\mathsf{Q}_{0}}e(T_{\mathsf{Q}_{0}})(-1)^{\mathrm{rk}\left( \mathrm{Ext}^1(\mathbb{I}_{\mathsf{Q}_{0}}^\bullet,\mathbb{I}_{\mathsf{Q}_{0}}^\bullet)_{0}\right)}\\
&=&\sum_{\mathsf{Q}_{0}\subset\mathrm{PT}(\mathcal{X},\beta)^{T_{0}}}\chi_{\mathrm{top}}(\mathsf{Q}_{0})(-1)^{\mathrm{rk}\left( \mathrm{Ext}^1(\mathbb{I}_{\mathsf{Q}_{0}}^\bullet,\mathbb{I}_{\mathsf{Q}_{0}}^\bullet)_{0}\right)}
\een
where $\chi_{\mathrm{top}}(\mathsf{Q}_{0})$ is the topological Euler chracteristic of $\mathsf{Q}_{0}$, and  $\mathbb{I}^\bullet_{\mathsf{Q}_{0}}=\{\mathcal{O}_{\mathcal{X}\times\mathsf{Q}_{0}}\to\mathbb{F}\}\in \mathrm{D}^b(\mathcal{X}\times\mathsf{Q}_{0})$ is the universal complex with  $\mathbb{F}$  flat over $\mathsf{Q}_{0}$. Here, we have identified each orbifold PT stable pair $\mathcal{O}_{\mathcal{X}}\to\mathcal{F}$ in $\mathsf{Q}_{0}$ with the complex $\{\mathcal{O}_{\mathcal{X}}\to\mathcal{F}\}$ in $\mathrm{D}^b(\mathcal{X})$ as in [\cite{PT2}, Proposition 1.21]. As in [\cite{PT2}, Section 5.2], we have the $\mathbb{C}^*=T/T_{0}$ action on every $T_{0}$-fixed component $\mathsf{Q}_{0}$, and every element of $\mathsf{Q}$ is the $\mathbb{C}^*$-fixed locus of a unique component $\mathsf{Q}_{0}\subset\mathrm{PT}(\mathcal{X},\beta)^{T_{0}}$ while any $\mathbb{C}^*$-fixed lous of $\mathsf{Q}_{0}$ is an element of some $\mathsf{Q}\subset\mathrm{PT}(\mathcal{X},\beta)^{T}$. Then we have 
\ben
\chi_{\mathrm{top}}(\mathsf{Q}_{0})=\sum_{\substack{\mathsf{Q}\subset\mathsf{Q}_{0}\\ \mathsf{Q}\subset\mathrm{PT}(\mathcal{X},\beta)^T}}\chi_{\mathrm{top}}(\mathsf{Q})
\een
and
\ben
PT_{\beta}(\mathcal{X})=\sum_{\mathsf{Q}\subset\mathrm{PT}(\mathcal{X},\beta)^{T}}\chi_{\mathrm{top}}(\mathsf{Q})(-1)^{\mathrm{rk}\left( \mathrm{Ext}^1(\mathbb{I}_{\mathsf{Q}}^\bullet,\mathbb{I}_{\mathsf{Q}}^\bullet)_{0}\right)}
\een
where $\mathbb{I}^\bullet_{\mathsf{Q}}=\{\mathcal{O}_{\mathcal{X}\times\mathsf{Q}}\to\mathbb{F}\}\in \mathrm{D}^b(\mathcal{X}\times\mathsf{Q})$. 

\begin{definition}\label{PT-def}
The PT partition function of $\mathcal{X}$ is defined by 
\ben
PT(\mathcal{X})=\sum_{\beta\in F_{1}K(\mathcal{X})}PT_{\beta}(\mathcal{X})q^\beta
\een
and the multi-regular PT partition function is defined by
\ben
PT_{mr}(\mathcal{X})=\sum_{\beta\in F_{mr}K(\mathcal{X})}PT_{\beta}(\mathcal{X})q^\beta
\een
where $q^\beta$ for $\beta\in F_{1}K(\mathcal{X})$ is defined in Section 2.3.
\end{definition}

Assume that $[\mathbf{I}^\bullet]\in\mathsf{Q}\subset\mathrm{PT}(\mathcal{X},\beta)^{T}$ where the complex $\mathbf{I}^\bullet=\{\mathcal{O}_{\mathcal{X}}\to\mathcal{F}\}$ in $\mathrm{D}^b(\mathcal{X})$ concentrated in degree $0$ and 1 corresponds to the $T$-invariant orbifold PT stable pair $\mathcal{O}_{\mathcal{X}}\to\mathcal{F}$, which is the restriction of the $T$-invariant universal complex   $\mathbb{I}^\bullet_{\mathsf{Q}}$ on $[\mathbf{I}^\bullet]$. Let $\{\widetilde{\pi}_{v},\lambda_{e}\}$ be the set of labelled box configurations and edge partitions corresponding to $\mathcal{O}_{\mathcal{X}}\to\mathcal{F}$ (or  $\mathbf{I}^\bullet$) by Lemma \ref{PT-fix}. Since $\mathsf{Q}=\prod_{v\in V_{\mathcal{X}}}\mathsf{Q}_{v}$ and $\mathsf{Q}_{v}$ corresponds to a   component $[\widetilde{\pi}_{v}]$ of moduli space of labellings of $\underline{\widetilde{\pi}_{v}}$ for each $v\in V_{\mathcal{X}}$.  Then 
\ben
\chi(\mathsf{Q})=\prod_{v\in V_{\mathcal{X}}}\chi_{\mathrm{top}}(\mathsf{Q}_{v})=\prod_{v\in V_{\mathcal{X}}}\chi_{\mathrm{top}}([\widetilde{\pi}_{v}]).
\een
 Let $\widetilde{\mathcal{P}}(\beta)$ be the set of all collections $\{[\widetilde{\pi}_{v}]\}$  corresponding to all components $\mathsf{Q}\subset\mathrm{PT}(\mathcal{X},\beta)^{T}$. Then we have 
\bea\label{PT-formula}
PT_{\beta}(\mathcal{X})&=&\sum_{\mathsf{Q}=\prod\limits_{v\in V_{\mathcal{X}}}\mathsf{Q}_{v}\subset\mathrm{PT}(\mathcal{X},\beta)^{T}}(-1)^{\dim \mathrm{Ext}^1(\mathbf{I}^\bullet,\mathbf{I}^\bullet)_{0}}\prod_{v\in V_{\mathcal{X}}}\chi_{\mathrm{top}}(\mathsf{Q}_{v})\nonumber\\
&=&\sum_{\{[\widetilde{\pi}_{v}]\}\in\widetilde{\mathcal{P}}(\beta)}(-1)^{\dim \mathrm{Ext}^1(\mathbf{I}^\bullet,\mathbf{I}^\bullet)_{0}}\prod_{v\in V_{\mathcal{X}}}\chi_{\mathrm{top}}([\widetilde{\pi}_{v}]).
\eea

Next, we will compute the parity of $\dim\mathrm{Ext}^1(\mathbf{I}^\bullet,\mathbf{I}^\bullet)_{0}$ to obtain the explicit formula of $PT(\mathcal{X})$.  We first employ some notation from  the DT side  in [\cite{BCY}, Section 6] as follows. Define 
\ben
&&s_{1}(W)=\mathrm{vdim}\, W\mbox{ mod } 2,\;\;\;\mbox{ for a virtual $T$-representation W};\\
&&s_{2}(M)= M \mbox{ mod }2,\;\;\;\mbox{ for  $M\in\mathbb{Z}$}.
\een
For a  $T$-representation $W$, the shifted dual of $W$ is defined by $W^*=W^\vee\otimes\mathbb{C}[-\mu]$ where $W^\vee$ is the dual representation of $W$ and $\mu\in \mathrm{Hom}(T,\mathbb{C}^*)$ is the primitive weight defined by $\omega_{\mathcal{X}}\cong\mathcal{O}_{\mathcal{X}}\otimes_{\mathbb{C}}\mathbb{C}[\mu]$ in [\cite{BCY}, Lemma 18]. Here $\mathbb{C}[\mu]$ is a one-dimensional $T$-representation with weight $\mu$. And one can define $\sigma(W-W^*)=s_{1}(W)$ for a virtual $T$-representation $W$ by [\cite{BCY}, Proposition 6] where $\sigma$ is defined on anti-self shifted dual virtual representations. For $\mathcal{H}_{1},\mathcal{H}_{2}\in \mathrm{D}^b(\mathcal{X})$, we define $\widetilde{\chi}(\mathcal{H}_{1},\mathcal{H}_{2})=\sum_{i}(-1)^i \mathrm{Ext}^i(\mathcal{H}_{1},\mathcal{H}_{2})$ in $K$-theory, and $\chi(\mathcal{H}_{1},\mathcal{H}_{2})=\sum_{i}(-1)^i\dim
 \mathrm{Ext}^i(\mathcal{H}_{1},\mathcal{H}_{2})$ with the special case  $\chi(\mathcal{H}_{1})=\chi(\mathcal{O}_{\mathcal{X}},\mathcal{H}_{1})$. 
Define $\lambda_{-1}W=\sum_{i}(-1)^i\Lambda^iW$ for a vector bundle $V$.
Now it follows from the argument in [\cite{BCY}, Section 6] that $s_{2}(\dim\mathrm{Ext}^1(\mathbf{I}^\bullet,\mathbf{I}^\bullet)_{0})=s_{1}(\mathrm{Ext}^1(\mathbf{I}^\bullet,\mathbf{I}^\bullet)_{0})$ can be  computed as follows, see also the proof in [\cite{Zhang}, Section 4.7].

\begin{theorem}\label{sign-formula}
Let $\varphi:\mathcal{O}_{\mathcal{X}}\to\mathcal{F}$ be a $T$-invariant orbifold PT stable pair and let $\{\widetilde{\pi}_{v},\lambda_{e}\}$ be the corresponding set of labelled box configurations and edge partitions. Let $\overline{\pi}_{v}:=\pi_{\mathrm{min}}(\lambda_{1,v},\lambda_{2,v},\lambda_{3,v})$ be the 3D partitions determined by outgoing partitions $(\lambda_{1,v},\lambda_{2,v},\lambda_{3,v})$ of the labelled box configuration $\widetilde{\pi}_{v}$ for each $v\in V_{\mathcal{X}}$.
Let $\mathbf{I}^\bullet=\{\mathcal{O}_{\mathcal{X}}\to\mathcal{F}\}\in\mathrm{D}^b(\mathcal{X})$ be the complex  concentrated in degree $0$ and 1. Then we have
\ben
s_{1}(\mathrm{Ext}^1(\mathbf{I}^\bullet,\mathbf{I}^\bullet)_{0})=s_{2}(\chi(\mathcal{O}_{\mathcal{C}_{\mathcal{F}}}))+\sum_{v\in V_{\mathcal{X}}}s_{2}(\mathbf{V}(v))+\sum_{e\in E(\mathcal{X})}s_{2}(\mathbf{E}(e))
\een
where 
\ben
\mathbf{V}(v)&=&\ell(\widetilde{\pi}_{v})+\sum_{\mathcal{A}\in\widetilde{\pi}_{v}}\eta_{\widetilde{\pi}_{v}}(\mathcal{A})h^0\left(\mathcal{O}_{p_{v}}(-\mathcal{A})\right)\\
&&+\sum_{\mathcal{A},\mathcal{B}\in\widetilde{\pi}_{v}}\eta_{\widetilde{\pi}_{v}}(\mathcal{A})\eta_{\widetilde{\pi}_{v}}(\mathcal{B})h^0\left(\sum_{i=1}^3\mathcal{O}_{p_{v}}(\mathcal{A}-\mathcal{B}+\mathcal{D}_{i,v})\right)\\
&&+\sum_{\mathcal{A}\in\widetilde{\pi}_{v}}\sum_{\mathcal{B}\in\overline{\pi}_{v}}\eta_{\widetilde{\pi}_{v}}(\mathcal{A})\xi_{\overline{\pi}_{v}}(\mathcal{B})h^0\left(\sum_{i=1}^3\mathcal{O}_{p_{v}}(\mathcal{A}-\mathcal{B}+\mathcal{D}_{i,v})\right)\\
&&+\sum_{\mathcal{A}\in\overline{\pi}_{v}}\sum_{\mathcal{B}\in\widetilde{\pi}_{v}}\xi_{\overline{\pi}_{v}}(\mathcal{A})\eta_{\widetilde{\pi}_{v}}(\mathcal{B})h^0\left(\sum_{i=1}^3\mathcal{O}_{p_{v}}(\mathcal{A}-\mathcal{B}+\mathcal{D}_{i,v})\right)\\
&&+\sum_{i=1}^3\sum_{\mathcal{A}\in\lambda_{i,v}}\sum_{\mathcal{B}\in\widetilde{\pi}_{v}}\eta_{\widetilde{\pi}_{v}}(\mathcal{B})h^0\left(\mathcal{O}_{p_{v}}(\mathcal{A}-\mathcal{B})\otimes\lambda_{-1}(N_{\mathcal{C}_{e_{i,v}}/\mathcal{X}})\right)\\
&&+\Vert\overline{\pi}_{v}\Vert+\sum_{\mathcal{A},\mathcal{B}\in\overline{\pi}_{v}}\xi_{\overline{\pi}_{v}}(\mathcal{A})\xi_{\overline{\pi}_{v}}(\mathcal{B})h^0\left(\sum_{i=1}^3\mathcal{O}_{p_{v}}(\mathcal{A}-\mathcal{B}+\mathcal{D}_{i,v})\right)\\
&&+\sum_{i=1}^3\sum_{\mathcal{A}\in\lambda_{i,v}}\sum_{\mathcal{B}\in\overline{\pi}_{v}}\xi_{\overline{\pi}_{v}}(\mathcal{B})h^0\left(\mathcal{O}_{p_{v}}(\mathcal{A}-\mathcal{B})\otimes\lambda_{-1}N_{\mathcal{C}_{e_{i,v}}/\mathcal{X}}\right)\\
&&+\sum_{i\neq j}\sum_{\mathcal{A}\in\lambda_{i,v}}\sum_{\mathcal{B}\in\lambda_{j,v}}h^0(\mathcal{O}_{p_{v}}(\mathcal{A}-\mathcal{B}+\mathcal{D}_{j,v}))
\een
and 
\ben
\mathbf{E}(e)=\sum_{\mathcal{A},\mathcal{B}\in\lambda_{e}}\chi(\mathcal{O}_{\mathcal{C}_{e}}(\mathcal{A}-\mathcal{B})+\mathcal{O}_{\mathcal{C}_{e}}(\mathcal{A}-\mathcal{B}+\mathcal{D}_{e})).
\een
\end{theorem}
\begin{proof}
By [\cite{Lyj1}, Lemma 5.16], we have the following equality 
\ben
\mathrm{Ext}^1(\mathbf{I}^\bullet,\mathbf{I}^\bullet)_{0}-\mathrm{Ext}^2(\mathbf{I}^\bullet,\mathbf{I}^\bullet)_{0}=\widetilde{\chi}(\mathcal{O}_{\mathcal{X}},\mathcal{O}_{\mathcal{X}})-\widetilde{\chi}(\mathbf{I}^\bullet,\mathbf{I}^\bullet).
\een
As there exist an exact triangle in $\mathrm{D}^b(\mathcal{X})$:
\ben
\mathcal{F}[-1]\to\mathbf{I}^\bullet\to\mathcal{O}_{\mathcal{X}}\to\mathcal{F},
\een 
and a short exact sequence 
\ben
0\to\mathcal{O}_{\mathcal{C}_{\mathcal{F}}}\to\mathcal{F}\to\mathtt{Q}\to0,
\een
then by [\cite{BCY}, Proposition 6-(1)], we have
\ben
&&\mathrm{Ext}^1(\mathbf{I}^\bullet,\mathbf{I}^\bullet)_{0}-\left(\mathrm{Ext}^1(\mathbf{I}^\bullet,\mathbf{I}^\bullet)_{0}\right)^*\\
&=&\widetilde{\chi}(\mathcal{O}_{\mathcal{X}},\mathcal{O}_{\mathcal{X}})-\widetilde{\chi}(\mathbf{I}^\bullet,\mathbf{I}^\bullet)\\
&=&\widetilde{\chi}(\mathcal{O}_{\mathcal{X}},\mathcal{O}_{\mathcal{C}_{\mathcal{F}}})-\left(\widetilde{\chi}(\mathcal{O}_{\mathcal{X}},\mathcal{O}_{\mathcal{C}_{\mathcal{F}}})\right)^*+\widetilde{\chi}(\mathcal{O}_{\mathcal{X}},\mathtt{Q})-\left(\widetilde{\chi}(\mathcal{O}_{\mathcal{X}},\mathtt{Q})\right)^*\\
&&-\left(\widetilde{\chi}(\mathcal{O}_{\mathcal{C}_{\mathcal{F}}},\mathtt{Q})-\left(\widetilde{\chi}(\mathcal{O}_{\mathcal{C}_{\mathcal{F}}},\mathtt{Q})\right)^*\right)-\widetilde{\chi}(\mathcal{O}_{\mathcal{C}_{\mathcal{F}}},\mathcal{O}_{\mathcal{C}_{\mathcal{F}}})-\widetilde{\chi}(\mathtt{Q},\mathtt{Q})
\een
and
\ben
&&s_{1}(\mathrm{Ext}^1(\mathbf{I}^\bullet,\mathbf{I}^\bullet)_{0})\\
&=&s_{1}(\widetilde{\chi}(\mathcal{O}_{\mathcal{X}},\mathcal{O}_{\mathcal{C}_{\mathcal{F}}}))+s_{1}(\widetilde{\chi}(\mathcal{O}_{\mathcal{X}},\mathtt{Q}))+s_{1}(\widetilde{\chi}(\mathcal{O}_{\mathcal{C}_{\mathcal{F}}},\mathtt{Q}))+\sigma(\widetilde{\chi}(\mathcal{O}_{\mathcal{C}_{\mathcal{F}}},\mathcal{O}_{\mathcal{C}_{\mathcal{F}}}))+\sigma(\widetilde{\chi}(\mathtt{Q},\mathtt{Q}))\\
&=&s_{2}(\chi(\mathcal{O}_{\mathcal{C}_{\mathcal{F}}}))+s_{2}(\chi(\mathtt{Q}))+s_{1}(\widetilde{\chi}(\mathcal{O}_{\mathcal{C}_{\mathcal{F}}},\mathtt{Q}))+\sigma(\widetilde{\chi}(\mathcal{O}_{\mathcal{C}_{\mathcal{F}}},\mathcal{O}_{\mathcal{C}_{\mathcal{F}}}))+\sigma(\widetilde{\chi}(\mathtt{Q},\mathtt{Q})).
\een
As in the proof of Proposition \ref{K-decomposition}, we have
\ben
\mathcal{O}_{\mathcal{C}_{\mathcal{F}}}=\sum_{v\in V_{\mathcal{X}}}\sum_{\mathcal{D}\in\overline{\pi}_{v}}\xi_{\overline{\pi}_{v}}(\mathcal{D})\mathcal{O}_{p_{v}}(-\mathcal{D})+\sum_{e\in E_{\mathcal{X}}}\sum_{\mathcal{D}\in\lambda_{e}}\mathcal{O}_{\mathcal{C}_{e}}(-\mathcal{D}).
\een
It is proved in [\cite{BCY}, Section 6.2] that 
\ben
\sigma(\widetilde{\chi}(\mathcal{O}_{\mathcal{C}_{\mathcal{F}}},\mathcal{O}_{\mathcal{C}_{\mathcal{F}}}))&=&\sum_{e\in E_{\mathcal{X}}}s_{2}\bigg(\sum_{\mathcal{A},\mathcal{B}\in\lambda_{e}}\chi(\mathcal{O}_{\mathcal{C}_{e}}(\mathcal{A}-\mathcal{B})+\mathcal{O}_{\mathcal{C}_{e}}(\mathcal{A}-\mathcal{B}+\mathcal{D}_{e}))\bigg)\\
&&+\sum_{v\in V_{\mathcal{X}}}s_{2}\Bigg(\Vert\overline{\pi}_{v}\Vert+\sum_{\mathcal{A},\mathcal{B}\in\overline{\pi}_{v}}\xi_{\overline{\pi}_{v}}(\mathcal{A})\xi_{\overline{\pi}_{v}}(\mathcal{B})h^0\left(\sum_{i=1}^3\mathcal{O}_{p_{v}}(\mathcal{A}-\mathcal{B}+\mathcal{D}_{i,v})\right)\\
&&+\sum_{i=1}^3\sum_{\mathcal{A}\in\lambda_{i,v}}\sum_{\mathcal{B}\in\overline{\pi}_{v}}\xi_{\overline{\pi}_{v}}(\mathcal{B})h^0\left(\mathcal{O}_{p_{v}}(\mathcal{A}-\mathcal{B})\otimes\lambda_{-1}N_{\mathcal{C}_{e_{i,v}}/\mathcal{X}}\right)\\
&&+\sum_{i\neq j}\sum_{\mathcal{A}\in\lambda_{i,v}}\sum_{\mathcal{B}\in\lambda_{j,v}}h^0(\mathcal{O}_{p_{v}}(\mathcal{A}-\mathcal{B}+\mathcal{D}_{j,v}))\Bigg).
\een
By Proposition \ref{K-decomposition} and using the similar argument as in [\cite{BCY}, Section 6.2], we have
\ben
\chi(\mathtt{Q})&=&\sum_{v\in V_{\mathcal{X}}}\sum_{\mathcal{A}\in\widetilde{\pi}_{v}}\eta_{\widetilde{\pi}_{v}}(\mathcal{A})h^0\left(\mathcal{O}_{p_{v}}(-\mathcal{A})\right),\\
\sigma\left(\widetilde{\chi}(\mathtt{Q},\mathtt{Q})\right)
&=&\sum_{v\in V_{\mathcal{X}}}s_{2}\left(\ell(\widetilde{\pi}_{v})+\sum_{\mathcal{A},\mathcal{B}\in\widetilde{\pi}_{v}}\eta_{\widetilde{\pi}_{v}}(\mathcal{A})\eta_{\widetilde{\pi}_{v}}(\mathcal{B})h^0\left(\sum_{i=1}^3\mathcal{O}_{p_{v}}(\mathcal{A}-\mathcal{B}+\mathcal{D}_{i,v})\right)\right)
\een
and
\ben
s_{1}(\widetilde{\chi}(\mathcal{O}_{\mathcal{C}_{\mathcal{F}}},\mathtt{Q}))
&=&\sum_{v\in V_{\mathcal{X}}}s_{2}\left(\sum_{\mathcal{A}\in\widetilde{\pi}_{v}}\sum_{\mathcal{B}\in\overline{\pi}_{v}}\eta_{\widetilde{\pi}_{v}}(\mathcal{A})\xi_{\overline{\pi}_{v}}(\mathcal{B})h^0\left(\sum_{i=1}^3\mathcal{O}_{p_{v}}(\mathcal{A}-\mathcal{B}+\mathcal{D}_{i,v})\right)\right)\\
&&+\sum_{v\in V_{\mathcal{X}}}s_{2}\left(\sum_{\mathcal{A}\in\overline{\pi}_{v}}\sum_{\mathcal{B}\in\widetilde{\pi}_{v}}\xi_{\overline{\pi}_{v}}(\mathcal{A})\eta_{\widetilde{\pi}_{v}}(\mathcal{B})h^0\left(\sum_{i=1}^3\mathcal{O}_{p_{v}}(\mathcal{A}-\mathcal{B}+\mathcal{D}_{i,v})\right)\right)\\
&&+\sum_{v\in V_{\mathcal{X}}}s_{2}\left(\sum_{i=1}^3\sum_{\mathcal{A}\in\lambda_{i,v}}\sum_{\mathcal{B}\in\widetilde{\pi}_{v}}\eta_{\widetilde{\pi}_{v}}(\mathcal{B})h^0\left(\mathcal{O}_{p_{v}}(\mathcal{A}-\mathcal{B})\otimes\lambda_{-1}(N_{\mathcal{C}_{e_{i,v}}/\mathcal{X}})\right)\right).
\een
Now the proof is completed by combining all the above results.		
\end{proof}	
\begin{remark}\label{DT-PT-comparision}
The pairity of $\dim\mathrm{Ext}^1(\mathbf{I}^\bullet,\mathbf{I}^\bullet)_{0}$ in Theorem \ref{sign-formula} have the contribution from the edge part $\mathbf{E}(e)$ which is  the same as $\mathbf{SE}_{\lambda(e)}(e)$ in  DT case [\cite{BCY}, Theorem 21], while its vertex part $\mathbf{V}(v)$ is more complicated involving labelled box configurations $\widetilde{\pi}_{v}$ due to the first 6 terms. The sum of the last 4 terms in $\mathbf{V}(v)$ is exactly the vertex part of DT side $\mathbf{SV}_{\pi(v)}(v)$ with $\pi(v)=\overline{\pi}_{v}$. The formula for $\mathbf{V}(v)$ is arranged for comparing with the DT case and noticing the  computation needed only for the additional terms later.
Actually, as in Remark \ref{Simplification1}, the formula $\mathbf{V}(v)$ can be rewritten as
\ben
\mathbf{V}(v)&=&\ell(\widetilde{\pi}_{v})+\Vert\overline{\pi}_{v}\Vert+\sum_{\mathcal{A}\in\widetilde{\pi}_{v}}\eta_{\widetilde{\pi}_{v}}(\mathcal{A})h^0\left(\mathcal{O}_{p_{v}}(-\mathcal{A})\right)\\
&&+\sum_{\mathcal{A},\mathcal{B}\in\breve{\pi}_{v}}\zeta_{\breve{\pi}_{v}}(\mathcal{A})\zeta_{\breve{\pi}_{v}}(\mathcal{B})h^0\left(\sum_{i=1}^3\mathcal{O}_{p_{v}}(\mathcal{A}-\mathcal{B}+\mathcal{D}_{i,v})\right)\\
&&+\sum_{i=1}^3\sum_{\mathcal{A}\in\lambda_{i,v}}\sum_{\mathcal{B}\in\breve{\pi}_{v}}\zeta_{\breve{\pi}_{v}}(\mathcal{B})h^0\left(\mathcal{O}_{p_{v}}(\mathcal{A}-\mathcal{B})\otimes\lambda_{-1}(N_{\mathcal{C}_{e_{i,v}}/\mathcal{X}})\right)\\
&&+\sum_{i\neq j}\sum_{\mathcal{A}\in\lambda_{i,v}}\sum_{\mathcal{B}\in\lambda_{j,v}}h^0(\mathcal{O}_{p_{v}}(\mathcal{A}-\mathcal{B}+\mathcal{D}_{j,v})),
\een
where $\ell(\widetilde{\pi}_{v})+\Vert\overline{\pi}_{v}\Vert=\sum_{\mathcal{A}\in\breve{\pi}_{v}}\zeta_{\breve{\pi}_{v}}(\mathcal{A})$. See also  [\cite{Zhang}, Page 94].
\end{remark}

\begin{remark}
When $\mathcal{X}$ is a scheme $X$, then $\mathcal{C}_{e}\cong\mathbb{P}^1$ for all $e\in E_{\mathcal{X}}$. By Remark \ref{DT-PT-comparision} and   [\cite{BCY}, Example 22], we have 
\ben
s_{2}(\mathbf{E}(e))=m_{e}|\lambda_{e}|\;\mathrm{ mod }\;2
\een
and 
\ben
s_{2}(\mathbf{V}(v))&=&s_{2}\left(\ell(\widetilde{\pi}_{v})+\sum_{\mathcal{D}\in\widetilde{\pi}_{v}}\eta_{\widetilde{\pi}_{v}}(\mathcal{D})+3\sum_{\mathcal{A},\mathcal{B}\in\widetilde{\pi}_{v}}\eta_{\widetilde{\pi}_{v}}(\mathcal{A})\eta_{\widetilde{\pi}_{v}}(\mathcal{B})+6\sum_{\mathcal{A}\in\widetilde{\pi}_{v}}\sum_{\mathcal{B}\in\overline{\pi}_{v}}\eta_{\widetilde{\pi}_{v}}(\mathcal{A})\xi_{\overline{\pi}_{v}}(\mathcal{B})\right)\\
&=&s_{2}\left(2\ell(\widetilde{\pi}_{v})+3\ell(\widetilde{\pi}_{v})^2\right)\\
&=&\ell(\widetilde{\pi}_{v})\;\mathrm{ mod }\;2.
\een
Then 
\ben
s_{1}(\mathrm{Ext}^1(\mathbf{I}^\bullet,\mathbf{I}^\bullet)_{0})=s_{2}\left(\chi(\mathcal{O}_{\mathcal{C}_{\mathcal{F}}})+\sum_{v\in V_{\mathcal{X}}}\ell(\widetilde{\pi}_{v})+\sum_{e\in E_{\mathcal{X}}}m_{e}|\lambda_{e}|\right)
\een
exactly coincides with the sign in [\cite{PT2}, Theorem/Conjecture 2] proved in [\cite{JWY}] due to  \ben
\chi(\mathcal{F})=\chi(\mathcal{O}_{\mathcal{C}_{\mathcal{F}}})+\chi(\mathtt{Q})=\chi(\mathcal{O}_{\mathcal{C}_{\mathcal{F}}})+\sum_{v\in V_{\mathcal{X}}}\ell(\widetilde{\pi}_{v}).
\een

\end{remark}

Assume that the local model of a torus fixed point $p_{v}$ is of the form $[\mathbb{C}^3/G_{v}]$  with  a finite abelian group $G_{v}$. Let $r_{i}\in\widehat{G_{v}}$ 
be the character obtained from the action of $G_{v}$ on $\mathcal{O}_{p_{v}}(-D_{i})$ for $1\leq i\leq3$,  different from the convention in [\cite{BCY}, Example 23]. Let $0\in\widehat{G_{v}}$ be the trivial character where we have identified the group $\widehat{G_{v}}$ of characters  of $G_{v}\subset(\mathbb{C}^*)^3$ with the corresponding additive subgroup of $\mathbb{Z}^3$. Then 
we have
\ben
s_{2}(\mathbf{V}(v))&=&s_{2}(\ell(\widetilde{\pi}_{v})_{0}+\ell(\widetilde{\pi}_{v})+\Vert\overline{\pi}_{v}\Vert)+\sum_{i\neq j}\sum_{r\in\widehat{G_{v}}}s_{2}(|\lambda_{i,v}|_{r}\cdot|\lambda_{j,v}|_{r+r_{j}})\\
&&+\sum_{r\in\widehat{G_{v}}}s_{2}\bigg((\ell(\widetilde{\pi}_{v})_{r}+\Vert\overline{\pi}_{v}\Vert_{r})\sum_{i=1}^3(\ell(\widetilde{\pi}_{v})_{r+r_{i}}+\Vert\overline{\pi}_{v}\Vert_{r+r_{i}})\bigg)\\
&&+\sum_{i=1}^3\sum_{r\in\widehat{G_{v}}}s_{2}\bigg(|\lambda_{i,v}|_{r}\bigg(\ell(\widetilde{\pi}_{v})_{r}+\Vert\overline{\pi}_{v}\Vert_{r}+\sum_{j\in\{1,2,3\}\setminus\{i\}}(\ell(\widetilde{\pi}_{v})_{r+r_{j}}+\Vert\overline{\pi}_{v}\Vert_{r+r_{j}})\\
&&+\ell(\widetilde{\pi}_{v})_{r-r_{i}+\sum_{j=1}^3 r_{j}}+\Vert\overline{\pi}_{v}\Vert_{r-r_{i}+\sum_{j=1}^3 r_{j}}\bigg)\bigg)
\een
where $|\lambda_{i,v}|_{r}$ is the number of elements in 2D partition $\lambda_{i,v}\subset\mathbb{Z}^2$ colored by $r\in\widehat{G_{v}}$, and $\ell(\widetilde{\pi}_{v})_{r}$, $\Vert\overline{\pi}_{v}\Vert_{r}$ are defined as in Section 3.1.

As in the DT case [\cite{BCY}, Remark 24], one can obtain the general formula for the PT partition function $PT(\mathcal{X})$ in Definition \ref{PT-def} by combining Lemma \ref{PT-fix}, Proposition \ref{K-decomposition}, and Theorem \ref{sign-formula} with the formula \eqref{PT-formula}. The explicit formula is presented in the following transverse $A_{n-1}$ case.

\subsection{The  PT topological vertex formalism} Let $\mathcal{X}$ be a toric CY 3-orbifold with transverse $A_{n-1}$ singularities. We present the explicit formula for PT partition function of $\mathcal{X}$ for full 3-leg case as follows.
We first recall some notations  in [\cite{BCY}, Section 3.3] that are used later.
For a 2D ordinary partition $\lambda$, define
\ben
\lambda[k,n]:=\{(i,j)\in\lambda\;|\; i-j\equiv k \mbox{ mod }n\}\;\;\;\mbox{and} \;\;\;|\lambda|_{k,n}:=|\lambda[k,n]|
\een
where we also write $|\lambda|_{k,n}$ as $|\lambda|_{k}$ for simplicity. We call a partition $\eta$ multi-regular if $|\eta|_{k}=\frac{|\eta|}{n}$ for all $k\in\{0, 1, \cdots ,n-1\}$. With some abuse of notation, we use the notation $|\cdot|_{k}$ not only for  edge partitions $\lambda_{e}$ with the above definition, but also for outgoing partitions  $(\lambda_{1,v},\lambda_{2,v},\lambda_{3,v})$ of  labelled box configurations $\widetilde{\pi}_{v}$ with the slightly different convention: we define  the color of $(j,k)\in\lambda_{1,v}$, $(k,i)\in\lambda_{2,v}$, and $(i,j)\in\lambda_{3,v}$  by $-j\mbox{ mod }n$, $i\mbox{ mod }n$, and $i-j\mbox{ mod }n$ respectively due to transverse $A_{n-1}$ singularities and hence $|\lambda_{i,v}|_{l}$ is defined to be the number of boxes in $\lambda_{i,v}$ colored by $l\mbox{ mod }n$ for $1\leq i\leq 3$ where $n=n_{e_{3,v}}$. Define 
\ben
\varXi_{\widetilde{\pi}_{v}}=\sum_{k=0}^{n_{e_{3,v}}-1}|\lambda_{3,v}|_{k}(|\lambda_{1,v}|_{k}+|\lambda_{2,v}|_{k}+|\lambda_{1,v}|_{k+1}+|\lambda_{2,v}|_{k-1}).
\een
Let
\ben
&&C_{m,m^\prime}^\lambda:=\sum_{(i,j)\in\lambda}\left(-mi-m^\prime j+1\right),\\
&&C_{m,m^\prime}^\lambda[k,n]:=\sum_{(i,j)\in\lambda[k,n]}\left(-mi-m^\prime j+1\right),\\
&&A_{\lambda}(k,n):=\sum_{(i,j)\in\lambda}\left\lfloor\frac{i+k}{n}\right\rfloor,
\een
and 
\ben
q_{e}^{C_{m_{e},m_{e}^\prime}^{\lambda_{e}}}=:\prod_{k=0}^{n_{e}-1}q_{e,k}^{C_{m_{e},m_{e}^\prime}^{\lambda_{e}}[k,n_{e}]},\;\;\;\;\;q_{e}^{A_{\lambda}}:=\prod_{k=0}^{n_{e}-1}q_{e,k}^{A_{\lambda}(k,n_{e})}.
\een
We define for each edge $e\in E_{\mathcal{X}}$
\ben
\mathbf{q}_{e}=(q_{e,0},q_{e,1},\cdots,q_{e,n_{e}-1}),\;\;\;\overline{\mathbf{q}_{e}}=(q_{e,0},q_{e,n_{e}-1},\cdots,q_{e,1})
\een
and 
\ben
\mathbf{S}_{\lambda_{e}}^e=\sum_{k=0}^{n_{e}-1}\bigg(C_{m_{e},m^\prime_{e}}^{\lambda_{e}}[k,n_{e}](|\lambda_{e}|_{k-1}-|\lambda_{e}|_{k+1})+|\lambda_{e}|_{k}\big(1+(1+m_{e}+\delta_{0,e}+\delta_{\infty,e})|\lambda_{e}|_{k-1}\big)\bigg).
\een
Let 
\ben
\mathbf{q}_{v}=\left\{
\begin{aligned}
	& \mathbf{q}_{e_{3,v}} ,\; \;\;\mbox{if $e_{3,v}$ is oriented outward from $v$}; \\
	& \overline{\mathbf{q}_{e_{3,v}}}, \;\;\;\mbox{if $e_{3,v}$ is oriented inward toward $v$}.
\end{aligned}
\right.
\een
and $(-1)^{\widetilde{s}_{k}(\lambda_{3,v})}\mathbf{q}_{v}$ be the vector $\mathbf{q}_{v}$ with its each component $q_{e,k}$ multiplied by the sign $(-1)^{\widetilde{s}_{k}(\lambda_{3,v})}$ where $\widetilde{s}_{k}(\lambda_{3,v})=|\lambda_{3,v}|_{k-1}+|\lambda_{3,v}|_{k+1}$ and $e=e_{3,v}$.

With these notations, we have
\begin{theorem}\label{sign-formula2}
Let $\mathcal{X}$ be a  toric CY 3-orbifold with transverse $A_{n-1}$-singularities. Then the pairity of $\dim\mathrm{Ext}(\mathbf{I}^\bullet,\mathbf{I}^\bullet)_{0}$ in Theorem \ref{sign-formula} is
\ben
s_{1}(\mathrm{Ext}^1(\mathbf{I}^\bullet,\mathbf{I}^\bullet)_{0})=s_{2}(\chi(\mathcal{O}_{\mathcal{C}_{\mathcal{F}}}))+\sum_{v\in V_{\mathcal{X}}}s_{2}(\mathbf{V}(v))+\sum_{e\in E_{\mathcal{X}}}s_{2}(\mathbf{E}(e))
\een
where
\ben
\mathbf{E}(e)=\left\{
\begin{aligned}
	& |\lambda_{e}|(m_{e}+\delta_{0}+\delta_{\infty}),\; \;\;\mbox{if} \;\;n_{e}=1; \\
	&\sum_{k=0}^{n-1}C_{m_{e},m^\prime_{e}}^{\lambda_{e}}[k,n](|\lambda_{e}|_{k-1}-|\lambda_{e}|_{k+1})+|\lambda_{e}|_{k}(1+(1+m_{e})|\lambda_{e}|_{k-1}), \;\;\mbox{if}\;\; n_{e}>1.
\end{aligned}
\right.
\een
and 
\ben
\mathbf{V}(v)&=&\ell(\widetilde{\pi}_{v})_{0}+\sum_{k=0}^{n_{e_{3,v}}-1}(\ell(\widetilde{\pi}_{v})_{k}+\Vert\overline{\pi}_{v}\Vert_{k})(|\lambda_{3,v}|_{k-1}+|\lambda_{3,v}|_{k+1})\\
&&+\sum_{k=0}^{n_{e_{{3,v}}}-1}|\lambda_{3,v}|_{k}(|\lambda_{1,v}|_{k}+|\lambda_{2,v}|_{k}+|\lambda_{1,v}|_{k+1}+|\lambda_{2,v}|_{k-1}).
\een
\end{theorem}
\begin{proof}
It follows from Theorem \ref{sign-formula} and the similar argument in the proof of [\cite{BCY}, Theorem 25]. See also the proof of [\cite{Zhang}, Proposition 4.29].
\end{proof}
Let $\Gamma_{\mathcal{X}}$ be the web diagram of $\mathcal{X}$. Given an edge assignment $\{\lambda_{e}\;|\;e\in E_{\mathcal{X}}\}\in\Lambda_{\mathcal{X}}$, there should be many collections of labelled box configurations where each collection $\{\widetilde{\pi}_{v}\;|\;v\in V_{\mathcal{X}}\}$ of labelled box configurations $\widetilde{\pi}_{v}$ with  outgoing partitions $(\lambda_{1,v},\lambda_{2,v},\lambda_{3,v})$ determined by $\{\lambda_{e}\}$ via the compatibility in Section 2.1. And each set $\{\widetilde{\pi}_{v}, \lambda_{e}\}$ corresponds to one $T$-fixed point in $\coprod_{\beta\in F_{1}K(\mathcal{X})}\mathrm{PT}(\mathcal{X},\beta)$ by Lemma \ref{PT-fix}. Therefore, for a given edge assignment $\{\lambda_{e}\}\in\Lambda_{\mathcal{X}}$, there are corresponding several connected components of $T$-fixed locus of  $\coprod\limits_{\beta\in F_{1}K(\mathcal{X})}\mathrm{PT}(\mathcal{X},\beta)$ where each component is of the form $\prod\limits_{v\in V_{\mathcal{X}}}[\widetilde{\pi}_{v}]$ satisfying some compatible conditions.

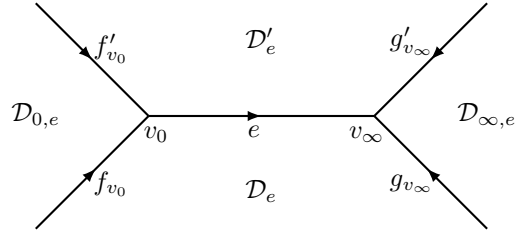
\begin{figure}[htp]
	\centering
	\begin{tikzpicture}[>=latex]
\draw[->,thick]  (-1.5,0)--(0,0); 
\draw[thick] (-0.1,0)node[below]{$e$}--(1.5,0); 
\node at (1.4,-0.25){$v_{\infty}$};	
\node at (-1.4,-0.25){$v_{0}$};	
\node at  (0,-1){$\mathcal{D}_{e}$};
\node at  (0,1){$\mathcal{D}_{e}^\prime$};

\draw[->,thick] (-3,1.5)--(-2.25,0.75);
\draw[thick] (-2.3,0.8)--(-1.5,0);
\node at (-2,0.9){$f^\prime_{v_{0}}$};

\draw[,thick] (-1.5,0)--(-2.25,-0.75);
\draw[<-,thick] (-2.2,-0.7)--(-3,-1.5);
\node at (-2,-0.9){$f_{v_{0}}$};

\node at (-3,0){$\mathcal{D}_{0,e}$};

\draw[->,thick] (3,1.5)--(2.25,0.75);
\draw[thick] (2.3,0.8)--(1.5,0);
\node at (2,1){$g^\prime_{v_{\infty}}$};

\draw[thick] (1.5,0)--(2.25,-0.75);
\draw[<-,thick] (2.2,-0.7)--(3,-1.5);
\node at (2,-0.9){$g_{v_{\infty}}$};

\node at (3,0){$\mathcal{D}_{\infty,e}$};
	\end{tikzpicture} 
	\caption{\label{figure1} Orientations of the edge $e$ and its adjacent edges.}
\end{figure}

Now we have the following PT version of [\cite{BCY}, Theorem 10], see also [\cite{Zhang}, Theorem/Conjecture 4.19].
\begin{theorem}\label{PT-partition function}
Let $\mathcal{X}$ be a  toric CY 3-orbifold with transverse $A_{n-1}$ singularities. Define
\ben
\underline{PT}(\mathcal{X}):=\sum_{\{\lambda_{e}\}\in\Lambda_{\mathcal{X}}}\prod_{e\in E_{\mathcal{X}}}E_{\lambda_{e}}^e\prod_{v\in V_{\mathcal{X}}}(-1)^{\varXi_{\widetilde{\pi}_{v}}}W_{\lambda_{1,v}\lambda_{2,v}\lambda_{3,v}}^{n_{e_{3,v}}}((-1)^{\widetilde{s}(\lambda_{3,v})}\mathbf{q}_{v})
\een
where 
\ben
E_{\lambda_{e}}^e=(-1)^{\mathbf{S}_{\lambda_{e}}^e}\cdot\left(\prod_{k=0}^{n_{e}-1}v_{e,k}^{|\lambda_{e}|_{k,n_{e}}}\right)\cdot q_{e}^{C_{m_{e},m_{e}^\prime}^{\lambda_{e}}}\cdot\left(\overline{q_{f_{v_{0}}}^{A_{\lambda_{e}}}}\right)^{\delta_{0,e}}\cdot\left(q_{f^\prime_{v_{0}}}^{A_{\lambda^\prime_{e}}}\right)^{\delta_{0,e}^\prime}\cdot\left(q_{g_{v_{\infty}}}^{A_{\lambda_{e}}}\right)^{\delta_{\infty,e}}\cdot\left(\overline{q_{g^\prime_{v_{\infty}}}^{A_{\lambda^\prime_{e}}}}\right)^{\delta_{\infty,e}^\prime}
\een
with edges $(e,f_{v_{0}},f^\prime_{v_{0}},g_{v_{\infty}},g^\prime_{v_{\infty}})$ are the oriented edges presented as in Figure \ref{figure1}, and the overline denotes the exchange of variables $q_{k}\leftrightarrow q_{-k}$.
Then  we have
\ben
PT(\mathcal{X})=\underline{PT}(\mathcal{X})\bigg|_{q_{e,0}\to -q_{e,0},\;q\to-q}
\een 
that is, $PT(\mathcal{X})$ is obtained from $\underline{PT}(\mathcal{X})$ by replacing  varibles $q_{e,0}$ by $-q_{e,0}$ for all $e\in E_{\mathcal{X}}$ and hence replacing $q$ by $-q$.
\end{theorem}
\begin{proof}
Since 
\ben
\mathcal{O}_{\mathcal{C}_{\mathcal{F}}}=\sum_{v\in V_{\mathcal{X}}}\sum_{\mathcal{D}\in\overline{\pi}_{v}}\xi_{\overline{\pi}_{v}}(\mathcal{D})\mathcal{O}_{p_{v}}(-\mathcal{D})+\sum_{e\in E_{\mathcal{X}}}\sum_{\mathcal{D}\in\lambda_{e}}\mathcal{O}_{\mathcal{C}_{e}}(-\mathcal{D})
\een
then
\ben
\chi(\mathcal{O}_{\mathcal{C}_{\mathcal{F}}})&=&\sum_{v\in V_{\mathcal{X}}}\Vert\overline{\pi}_{v}\Vert_{0}+\sum_{e\in E_{\mathcal{X}}}C_{m_{e},m_{e}^\prime}^{\lambda_{e}}[0,n_{e}]+\sum_{e\in E_{\mathcal{X}}}\bigg(\delta_{0,e}\cdot A_{\lambda_{e}}(0,n_{f_{v_{0}}})\\
&&+\delta_{0,e}^\prime\cdot A_{\lambda_{e}^\prime}(0,n_{f^\prime_{v_{0}}})+\delta_{\infty,e}\cdot A_{\lambda_{e}}(0,n_{g_{v_{\infty}}})+\delta^\prime_{\infty,e}\cdot A_{\lambda^\prime_{e}}(0,n_{g^\prime_{v_{\infty}}})\bigg).
\een
The proof is completed by using Definitions \ref{orbifoldPT-vertex}, \ref{PT-def}, the formula \eqref{PT-formula},  Proposition \ref{K-decomposition} for $\beta=\mathcal{F}$, Theorem \ref{sign-formula2}, and [\cite{BCY}, Lemma 15 and Proposition 5] via the similar argument in [\cite{BCY}, Section 6.3].	
\end{proof}
\begin{remark}
The terms involving $f_{v_{0}}$ and $g^\prime_{v_{\infty}}$ in the expression of $E_{\lambda_{e}}^e$ are added by the overline   since the choice of orientations of $f_{v_{0}}$ and $g^\prime_{v_{\infty}}$ opposite to those of $f$ and $g^\prime$ in [\cite{BCY}]  exchanges generators $[\mathcal{O}_{p_{f_{v_{0}}}}\otimes\rho_{k}]\leftrightarrow [\mathcal{O}_{p_{f_{v_{0}}}}\otimes\rho_{-k}]$, $[\mathcal{O}_{p_{g^\prime_{v_{\infty}}}}\otimes\rho_{k}]\leftrightarrow [\mathcal{O}_{p_{g^\prime_{v_{\infty}}}}\otimes\rho_{-k}]$ in [\cite{BCY}, Lemma 17 and Proposition 5].
\end{remark}

\section{Orbifold DT theory and the graphical  condensation recurrence}
In this section, we show the dimer model for DT theory in [\cite{JWY}, Section 3.2] can be extended to our orbifold case in Section 4.1 and generalize the graphical condensation recurrence for DT theory in [\cite{JWY}, Section 3.3] in Section 4.2, where we also obtain two more recurrences for orbifold DT theory and exhibit the equivalence between them by the symmetry of DT $\mathbb{Z}_{n}$-vertex. We compute the weights appearing in these three recurrences in Section 4.3, which generalize the ones in [\cite{JWY}, Section 5.2].

\subsection{The dimer model for orbifold DT theory}
In   order to describe the orbifold DT topological vertex as some dimer model, we recall some relevant notation  and definitions in [\cite{JWY}, Section 2 and 3] as follows. Notice that any ordinary 2D partition $\eta$ in this section has its parts indexed starting from 1 other than 0 in Section 3. 

\begin{definition}([\cite{JWY}, Definition 2.0.1 and 2.0.2])
For any partition $\eta=(\eta_{1}, \eta_{2},\cdots,\eta_{l})$, define the Maya diagram of $\eta$ by the set
\ben
S(\eta)=\left\{\eta_{i}-i+\frac{1}{2}\; \bigg|\; i\geq1\right\}\subseteq\mathbb{Z}+\frac{1}{2}.
\een
For a subset $S\subseteq\mathbb{Z}+\frac{1}{2}$, define 
\ben
S^+=\left\{t\in S \;\bigg| \; t>0\right\};\;\;\;\;\; S^-=\left\{t\in\mathbb{Z}+\frac{1}{2}\setminus S \;\bigg|\;t<0 \right\}.
\een
If $|S^+|, |S^-|<+\infty$, define the charge of $S$ as 
\ben
c(S)=|S^+|-|S^-|.
\een
Then the set $\{t-c(S)\;|\;t\in S\}$ is the Maya diagram of some partition $\widehat{\eta}$, i.e., 
\ben
\{t-c(S)\;|\;t\in S\}=S(\widehat{\eta})
\een
and we call $S$ the charge $c(S)$ Maya diagram of $\widehat{\eta}$.
\end{definition}

\begin{definition}([\cite{JWY}, Definition 2.0.3])
For any partition $\eta\neq\emptyset$ with the Maya diagram $S(\eta)$,  the partitions  $\eta^r$, $\eta^c$ and $\eta^{rc}$ are defined by the following equations respectively
\ben
&&S(\eta^r)=\{t+1\;|\;t\in S(\eta)\setminus\{\min S(\eta)^+\}\},\\
&&S(\eta^c)=\{t-1\;|\;t\in S(\eta)\cup\{\max S(\eta)^-\}\},\\
&&S(\eta^{rc})=\{t \;|\; t\in(S(\eta)\setminus\{\min S(\eta)^+\})\cup\{\max S(\eta)^-\}\}.
\een
That is,  the set $S(\eta)\setminus\{\min S(\eta)^+\}$ is the charge $-1$ Maya diagram of $\eta^r$, the set $S(\eta)\cup\{\max S(\eta)^-\}$ is the charge $1$ Maya diagram of $\eta^c$, and the set $(S(\eta)\setminus\{\min S(\eta)^+\})\cup\{\max S(\eta)^-\}$ is the (charge $0$) Maya diagram of $\eta^{rc}$.
\end{definition}

As in [\cite{JWY}], let $\mathbf{H}(N)$ be the $N\times N\times N$ honeycomb graph divided into three sectors with some choice of  labellings of vertices on the outer face, which is presented in different ways for the DT and PT case respectively, see  Figure \ref{figure2} for the graph $\mathbf{H}(4)$. With the convention in [\cite{JWY}],  the labelled vertice $v$ in sector $i$ is said to be in sector $i^+$ or $i^-$ if the labelled number is positive or negative.

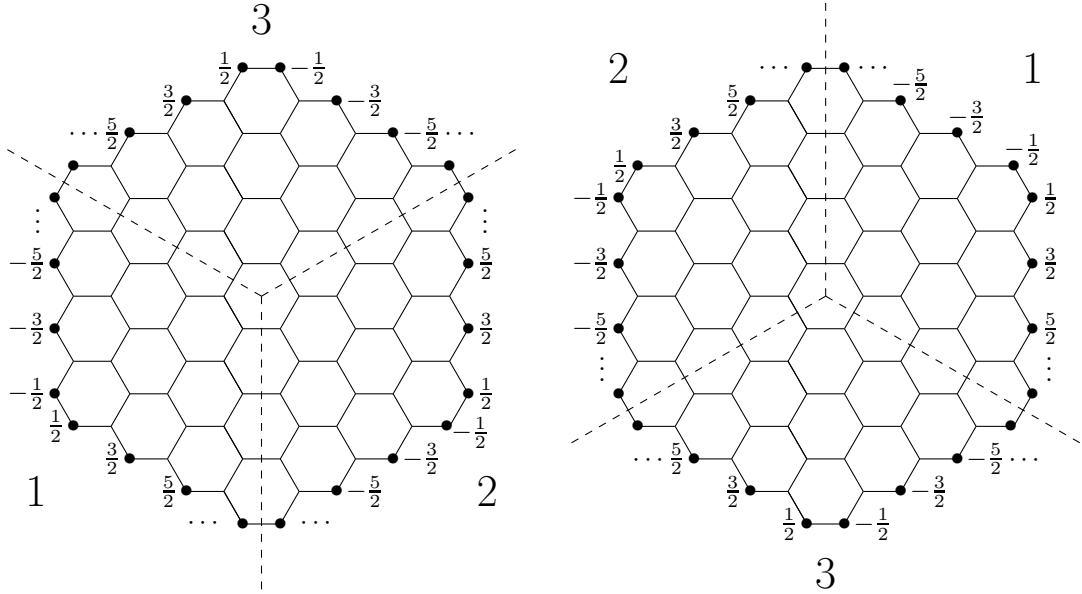
\begin{figure}[htp]
	\centering
	\begin{tikzpicture}[>=latex]
	\begin{scope}
	\foreach \x in
	{0,...,3,4}
	{
		\draw (2,2*0.433*\x-4*0.433)--(2.5,2*0.433*\x-4*0.433);
	}
	\foreach \x in
	{0,...,4,5}
	{
		\draw (1.25,2*0.433*\x-5*0.433)--(1.75,2*0.433*\x-5*0.433);	
	}
	\foreach \x in
	{0,...,5,6}
	{	
		\draw (0.5,2*0.433*\x-6*0.433)--(1,2*0.433*\x-6*0.433);	
	}
	\foreach \x in
	{0,1,...,7}
	{	
		\draw (-0.25,2*0.433*\x-7*0.433)--(0.25,2*0.433*\x-7*0.433);	
	}
	
	\foreach \x in
	{1,...,6,7}
	{	
		\draw (-1,2*0.433*\x-8*0.433)--(-0.5,2*0.433*\x-8*0.433);	
	}
	
	\foreach \x in
	{2,...,6,7}
	{	
		\draw (-1.75,2*0.433*\x-9*0.433)--(-1.25,2*0.433*\x-9*0.433);	
	}
	
	\foreach \x in
	{3,...,6,7}
	{	
		\draw (-2.5,2*0.433*\x-10*0.433)--(-2,2*0.433*\x-10*0.433);	
	}

	\foreach \x in
	{1,...,4}
	{	
		\draw (-2.75,2*0.433*\x-5*0.433)--(-2.5,2*0.433*\x-4*0.433);	
	}
	
	\foreach \x in
	{1,...,5}
	{	
		\draw (-2,2*0.433*\x-6*0.433)--(-1.75,2*0.433*\x-5*0.433);	
	}
	
	\foreach \x in
	{1,...,6}
	{	
		\draw (-1.25,2*0.433*\x-7*0.433)--(-1,2*0.433*\x-6*0.433);	
	}
	
	\foreach \x in
	{1,...,6,7}
	{	
		\draw (-0.5,2*0.433*\x-8*0.433)--(-0.25,2*0.433*\x-7*0.433);	
	}
	
	\foreach \x in
	{1,...,6,7}
	{	
		\draw (0.25,2*0.433*\x-9*0.433)--(0.5,2*0.433*\x-8*0.433);
	}

	\foreach \x in
	{1,...,6}
	{
		\draw (1,2*0.433*\x-8*0.433)--(1.25,2*0.433*\x-7*0.433);	
	}
	
	\foreach \x in
	{1,...,5}
	{	
		\draw (1.75,2*0.433*\x-7*0.433)--(2,2*0.433*\x-6*0.433);	
	}
	
	\foreach \x in
	{1,...,4}
	{	
		\draw (2.5,2*0.433*\x-6*0.433)--(2.75,2*0.433*\x-5*0.433);	
	}
	
	\foreach \x in
	{1,...,6,7}
	{
		\draw (-0.5,2*0.433*\x-8*0.433)--(-0.25,2*0.433*\x-9*0.433);	
	}
	
	\foreach \x in
	{1,...,4}
	{	
		\draw (-2.75,2*0.433*\x-5*0.433)--(-2.5,2*0.433*\x-6*0.433);	
	}
	
	\foreach \x in
	{1,...,5}
	{
		\draw (-2,2*0.433*\x-6*0.433)--(-1.75,2*0.433*\x-7*0.433);	
	}
	
	\foreach \x in
	{1,...,6}
	{	
		\draw (-1.25,2*0.433*\x-7*0.433)--(-1,2*0.433*\x-8*0.433);	
	}

	\foreach \x in
	{1,...,6,7}
	{	
		\draw (-0.5,2*0.433*\x-8*0.433)--(-0.25,2*0.433*\x-9*0.433);
	}
	
	\foreach \x in
	{1,...,6,7}
	{	
		\draw (0.25,2*0.433*\x-7*0.433)--(0.5,2*0.433*\x-8*0.433);
	}
	
	\foreach \x in
	{1,...,6}
	{	
		\draw (1,2*0.433*\x-6*0.433)--(1.25,2*0.433*\x-7*0.433);	
	}
	
	\foreach \x in
	{1,...,5}
	{
		\draw (1.75,2*0.433*\x-5*0.433)--(2,2*0.433*\x-6*0.433);	
	}
	
	\foreach \x in
	{1,...,4}
	{	
		\draw (2.5,2*0.433*\x-4*0.433)--(2.75,2*0.433*\x-5*0.433);	
	}
	\draw[dashed] (0,0)--(1.732*2,2); 
	\draw[dashed] (0,0)--(-1.732*2,2); 
	\draw[dashed] (0,0)--(0,-4); 
	\node[left] at (-0.25,0.433*7){$\frac{1}{2}$};
	\node at (-0.25,0.433*7){$\bullet$};
	\node[left] at (-1,0.433*6){$\frac{3}{2}$};
		\node at (-1,0.433*6){$\bullet$};
	\node[left] at (-1.75,0.433*5){$\cdots\frac{5}{2}$};
	\node at (-1.75,0.433*5){$\bullet$};
	\node at (-2.5,0.433*4){$\bullet$};
	\node[right] at (0.25,0.433*7){$-\frac{1}{2}$};
	\node at (0.25,0.433*7){$\bullet$};
	\node[right] at (1,0.433*6){$-\frac{3}{2}$};
	\node at (1,0.433*6){$\bullet$};
	\node[right] at (1.75,0.433*5){$-\frac{5}{2}\cdots$};
	\node at (1.75,0.433*5){$\bullet$};
	\node at (2.5,0.433*4){$\bullet$};
	\node at (0,0.433*8.5){$\huge \mbox{3}$};
	
	\node[left] at (-2.75,-0.433*3){$-\frac{1}{2}$};
	\node at (-2.75,-0.433*3){$\bullet$};
	\node[left] at (-2.75,-0.433){$-\frac{3}{2}$};
	\node at (-2.75,-0.433){$\bullet$};
	\node[left] at (-2.75,0.433){$-\frac{5}{2}$};
	\node at (-2.75,0.433){$\bullet$};
	\node at (-2.75,0.433*3){$\bullet$};
	\node[left] at (-2.8,1.1){$\vdots$};
	\node[left] at (-2.5,-0.433*4){$\frac{1}{2}$};
	\node at (-2.5,-0.433*4){$\bullet$};
	\node[left] at (-1.75,-0.433*5){$\frac{3}{2}$};
	\node at (-1.75,-0.433*5){$\bullet$};
	\node[left] at (-1,-0.433*6){$\frac{5}{2}$};
	\node at (-1,-0.433*6){$\bullet$};
	\node at (-0.25,-0.433*7){$\bullet$};
	\node[below] at (-0.75,-0.433*6.5){$\cdots$};
	\node at (-3,-0.433*6){$\huge \mbox{1}$};
	
	\node[right] at (2.75,-0.433*3){$\frac{1}{2}$};
	\node at (2.75,-0.433*3){$\bullet$};
	\node[right] at (2.75,-0.433){$\frac{3}{2}$};
	\node at (2.75,-0.433){$\bullet$};
	\node[right] at (2.75,0.433){$\frac{5}{2}$};
	\node at (2.75,0.433){$\bullet$};
	\node at (2.75,0.433*3){$\bullet$};
	\node[right] at (2.8,1.1){$\vdots$};
	\node[right] at (2.4,-0.433*4.2){$-\frac{1}{2}$};
	\node[right] at (2.25,-0.433*4){$\bullet$};
	\node[right] at (1.75,-0.433*5){$-\frac{3}{2}$};
	\node at (1.75,-0.433*5){$\bullet$};
	\node[right] at (1,-0.433*6){$-\frac{5}{2}$};
	\node at (1,-0.433*6){$\bullet$};
	\node at (0.25,-0.433*7){$\bullet$};
	\node[below] at (0.75,-0.433*6.5){$\cdots$};
	\node at (3,-0.433*6){$\huge \mbox{2}$};
	
	\end{scope}

	\begin{scope}[xshift=7.5cm]
		\foreach \x in
		{0,...,3,4}
		{
			\draw (2,2*0.433*\x-4*0.433)--(2.5,2*0.433*\x-4*0.433);
		}
		\foreach \x in
		{0,...,4,5}
		{
			\draw (1.25,2*0.433*\x-5*0.433)--(1.75,2*0.433*\x-5*0.433);	
		}
		\foreach \x in
		{0,...,5,6}
		{	
			\draw (0.5,2*0.433*\x-6*0.433)--(1,2*0.433*\x-6*0.433);	
		}
		\foreach \x in
		{0,1,...,7}
		{	
			\draw (-0.25,2*0.433*\x-7*0.433)--(0.25,2*0.433*\x-7*0.433);	
		}
		
		\foreach \x in
		{1,...,6,7}
		{	
			\draw (-1,2*0.433*\x-8*0.433)--(-0.5,2*0.433*\x-8*0.433);	
		}
		
		\foreach \x in
		{2,...,6,7}
		{	
			\draw (-1.75,2*0.433*\x-9*0.433)--(-1.25,2*0.433*\x-9*0.433);	
		}
		
		\foreach \x in
		{3,...,6,7}
		{	
			\draw (-2.5,2*0.433*\x-10*0.433)--(-2,2*0.433*\x-10*0.433);	
		}

		\foreach \x in
		{1,...,4}
		{	
			\draw (-2.75,2*0.433*\x-5*0.433)--(-2.5,2*0.433*\x-4*0.433);	
		}
		
		\foreach \x in
		{1,...,5}
		{	
			\draw (-2,2*0.433*\x-6*0.433)--(-1.75,2*0.433*\x-5*0.433);	
		}
		
		\foreach \x in
		{1,...,6}
		{	
			\draw (-1.25,2*0.433*\x-7*0.433)--(-1,2*0.433*\x-6*0.433);	
		}
		
		\foreach \x in
		{1,...,6,7}
		{	
			\draw (-0.5,2*0.433*\x-8*0.433)--(-0.25,2*0.433*\x-7*0.433);	
		}
		
		\foreach \x in
		{1,...,6,7}
		{	
			\draw (0.25,2*0.433*\x-9*0.433)--(0.5,2*0.433*\x-8*0.433);
		}

		\foreach \x in
		{1,...,6}
		{
			\draw (1,2*0.433*\x-8*0.433)--(1.25,2*0.433*\x-7*0.433);	
		}
		
		\foreach \x in
		{1,...,5}
		{	
			\draw (1.75,2*0.433*\x-7*0.433)--(2,2*0.433*\x-6*0.433);	
		}
		
		\foreach \x in
		{1,...,4}
		{	
			\draw (2.5,2*0.433*\x-6*0.433)--(2.75,2*0.433*\x-5*0.433);	
		}
		
		\foreach \x in
		{1,...,6,7}
		{
			\draw (-0.5,2*0.433*\x-8*0.433)--(-0.25,2*0.433*\x-9*0.433);	
		}
		
		\foreach \x in
		{1,...,4}
		{	
			\draw (-2.75,2*0.433*\x-5*0.433)--(-2.5,2*0.433*\x-6*0.433);	
		}
		
		\foreach \x in
		{1,...,5}
		{
			\draw (-2,2*0.433*\x-6*0.433)--(-1.75,2*0.433*\x-7*0.433);	
		}
		
		\foreach \x in
		{1,...,6}
		{	
			\draw (-1.25,2*0.433*\x-7*0.433)--(-1,2*0.433*\x-8*0.433);	
		}

		\foreach \x in
		{1,...,6,7}
		{	
			\draw (-0.5,2*0.433*\x-8*0.433)--(-0.25,2*0.433*\x-9*0.433);
		}
		
		\foreach \x in
		{1,...,6,7}
		{	
			\draw (0.25,2*0.433*\x-7*0.433)--(0.5,2*0.433*\x-8*0.433);
		}
		
		\foreach \x in
		{1,...,6}
		{	
			\draw (1,2*0.433*\x-6*0.433)--(1.25,2*0.433*\x-7*0.433);	
		}
		
		\foreach \x in
		{1,...,5}
		{
			\draw (1.75,2*0.433*\x-5*0.433)--(2,2*0.433*\x-6*0.433);	
		}
		
		\foreach \x in
		{1,...,4}
		{	
			\draw (2.5,2*0.433*\x-4*0.433)--(2.75,2*0.433*\x-5*0.433);	
		}
		
		\draw[dashed] (0,0)--(1.732*2,-2); 
		\draw[dashed] (0,0)--(-1.732*2,-2); 
		\draw[dashed] (0,0)--(0,4);

		\node at (-0.25,0.433*7){$\bullet$};
		
		\node at (-1,0.433*6){$\bullet$};
		\node[left] at (-1,0.433*6){$\frac{5}{2}$};
		\node[left] at (-.3,0.433*7){$\cdots$};
		\node at (-1.75,0.433*5){$\bullet$};
		\node[left] at (-1.75,0.433*5){$\frac{3}{2}$};
		\node at (-2.5,0.433*4){$\bullet$};
		\node[left] at (-2.5,0.433*4){$\frac{1}{2}$};
		\node at (0.25,0.433*7){$\bullet$};
		\node[right] at (.3,0.433*7){$\cdots$};
		\node at (1,0.433*6){$\bullet$};
		\node[right] at (0.75,0.433*6.5){$-\frac{5}{2}$};
		\node at (1.75,0.433*5){$\bullet$};
		\node[right] at (1.5,0.433*5.5){$-\frac{3}{2}$};
		\node at (2.5,0.433*4){$\bullet$};
		\node[right] at (2.25,0.433*4.5){$-\frac{1}{2}$};
		
		\node at (0,-0.433*8.5){$\huge \mbox{3}$};

		\node at (-2.75,-0.433*3){$\bullet$};
		\node[left] at (-2.8,-0.433*2){$\vdots$};
		\node at (-2.75,-0.433){$\bullet$};
		\node[left] at (-2.75,-0.433){$-\frac{5}{2}$};
		\node at (-2.75,0.433){$\bullet$};
		\node[left] at (-2.75,0.433){$-\frac{3}{2}$};
		\node at (-2.75,0.433*3){$\bullet$};
		\node[left] at (-2.75,0.433*3){$-\frac{1}{2}$};
		
		\node at (-2.5,-0.433*4){$\bullet$};
	
	    \node at (-2.75,0.433*7){$\huge \mbox{2}$};
		
		\node at (2.75,0.433*7){$\huge \mbox{1}$};
		
		\node at (2.75,-0.433*3){$\bullet$};
		
		\node at (2.75,-0.433){$\bullet$};
		\node[right] at (2.75,-0.433){$\frac{5}{2}$};
		\node[right] at (2.8,-0.433*2){$\vdots$};
		\node at (2.75,0.433){$\bullet$};
		\node[right] at (2.75,0.433){$\frac{3}{2}$};
		\node at (2.75,0.433*3){$\bullet$};
		\node[right] at (2.75,0.433*3){$\frac{1}{2}$};
		
		\node[right] at (2.25,-0.433*4){$\bullet$};
	
		\node at (1.75,-0.433*5){$\bullet$};
		
		\node at (1,-0.433*6){$\bullet$};
		\node at (0.25,-0.433*7){$\bullet$};
		\node[right] at (0.25,-0.433*7){$-\frac{1}{2}$};
		\node[left] at (-0.25,-0.433*7){$\frac{1}{2}$};
		\node[left] at (-1,-0.433*6){$\frac{3}{2}$};
		\node[right] at (1,-0.433*6){$-\frac{3}{2}$};
		\node at (-1,-0.433*6){$\bullet$};
		\node at (-0.25,-0.433*7){$\bullet$};
		\node at (-1.75,-0.433*5){$\bullet$};
		\node[left] at (-1.75,-0.433*5){$\cdots\frac{5}{2}$};
		\node[right] at (1.75,-0.433*5){$-\frac{5}{2}\cdots$};
		
		\end{scope}
	
	\end{tikzpicture} 
	\caption{\label{figure2}The left graph $\mathbf{H}(4)$ is divided into 3 sectors with the labellings of outer vertices for DT case. The right graph is again $\mathbf{H}(4)$ but divided into  another 3 sectors with the different labellings of outer vertices for PT case.}
\end{figure}

As is shown in [\cite{Kuo,JWY}], a 3D partition (or  a plane partition) visualized as some collection of boxes can be viewed from a suitable  direction as certain rhombus tiling within a hexagonal region consisting of triangles, denoted by faces of a finite planar graph. And its dual graph is exactly the beautiful honeycomb graph, which is the dimer configuration (or the so-called perfect matching) of this dual graph, see [\cite{Kuo}, Fig. 17] or [\cite{JWY}, Figure 3]  for examples. Labelled box configurations are related to AB configurations which also have this visualized description in Section 5.2. In order to describe the orbifold DT or PT topological vertex formula via  some weighted configurations on the modified honeycomb graph as in [\cite{JWY}], we  assign the weights of edges of the graph $\mathbf{H}(N)$  as follows.

\begin{definition}\label{weight-rule}
Given $n\in\mathbb{Z}_{>0}$, the honeycomb graph $\mathbf{H}(N)$ is weighted by the following rule:\\
$(i)$ all the non-horizontal edges are weighted by 1;	\\
$(ii)$ the $2N$ horizontal edges in the central column are weighted from bottom to top by
\ben
 q_{0}^0, q_{0}^1, q_{0}^2, \cdots, q_{0}^{2N-1};
\een 
$(iii)$ for $1\leq i\leq N-1$, the $(2N-i)$ horizontal edges in the $i$-th column on the right hand side of the  central column (corresponding to $0$-th column) are weighted from bottom to top by
\ben
q_{-i}^{0}, q_{-i}^1, q_{-i}^2, \cdots, q_{-i}^{2N-i-1};
\een
$(iv)$ for $1\leq i\leq N-1$, the $(2N-i)$ horizontal edges in the $i$-th column on the left hand side of  the  central column (corresponding to $0$-th column) are weighted from bottom to top by
\ben
q_{i}^{i}, q_{i}^{i+1}, q_{i}^{i+2}, \cdots, q_{i}^{2N-1},
\een
where for any $k\in\mathbb{Z}$, the equality $q_{k}=q_{k^\prime}$ is used implicitly  through this paper if $k^\prime\in\{0,1,\cdots,n-1\}$ and $k^\prime\equiv k \mbox{ mod } n$.
\end{definition}
The weighted rule in Definition \ref{weight-rule} is shown in Figure \ref{figure3} for  $\mathbf{H}(4)$ as an example.

\begin{figure}[htp]
	\centering
\begin{tikzpicture}[>=latex]
\foreach \x in
{0,...,3,4}
{
	\draw (2,2*0.433*\x-4*0.433)--node[below]{$q^{\x}_{-3}$}(2.5,2*0.433*\x-4*0.433);
}
\foreach \x in
{0,...,4,5}
{
	\draw (1.25,2*0.433*\x-5*0.433)--node[below]{$q^{\x}_{-2}$}(1.75,2*0.433*\x-5*0.433);	
}
\foreach \x in
{0,...,5,6}
{	
	\draw (0.5,2*0.433*\x-6*0.433)--node[below]{$q^{\x}_{-1}$}(1,2*0.433*\x-6*0.433);	
}
\foreach \x in
{0,1,...,7}
{	
	\draw (-0.25,2*0.433*\x-7*0.433)--node[below]{$q^{\x}_{0}$}(0.25,2*0.433*\x-7*0.433);	
}

\foreach \x in
{1,...,6,7}
{	
	\draw (-1,2*0.433*\x-8*0.433)--node[below]{$q^{\x}_{1}$}(-0.5,2*0.433*\x-8*0.433);	
}

\foreach \x in
{2,...,6,7}
{	
	\draw (-1.75,2*0.433*\x-9*0.433)--node[below]{$q^{\x}_{2}$}(-1.25,2*0.433*\x-9*0.433);	
}

\foreach \x in
{3,...,6,7}
{	
	\draw (-2.5,2*0.433*\x-10*0.433)--node[below]{$q^{\x}_{3}$}(-2,2*0.433*\x-10*0.433);	
}

\foreach \x in
{1,...,4}
{	
	\draw (-2.75,2*0.433*\x-5*0.433)--(-2.5,2*0.433*\x-4*0.433);	
}

\foreach \x in
{1,...,5}
{	
	\draw (-2,2*0.433*\x-6*0.433)--(-1.75,2*0.433*\x-5*0.433);	
}

\foreach \x in
{1,...,6}
{	
	\draw (-1.25,2*0.433*\x-7*0.433)--(-1,2*0.433*\x-6*0.433);	
}

\foreach \x in
{1,...,6,7}
{	
	\draw (-0.5,2*0.433*\x-8*0.433)--(-0.25,2*0.433*\x-7*0.433);	
}

\foreach \x in
{1,...,6,7}
{	
	\draw (0.25,2*0.433*\x-9*0.433)--(0.5,2*0.433*\x-8*0.433);
}

\foreach \x in
{1,...,6}
{
	\draw (1,2*0.433*\x-8*0.433)--(1.25,2*0.433*\x-7*0.433);	
}

\foreach \x in
{1,...,5}
{	
	\draw (1.75,2*0.433*\x-7*0.433)--(2,2*0.433*\x-6*0.433);	
}

\foreach \x in
{1,...,4}
{	
	\draw (2.5,2*0.433*\x-6*0.433)--(2.75,2*0.433*\x-5*0.433);	
}

\foreach \x in
{1,...,6,7}
{
	\draw (-0.5,2*0.433*\x-8*0.433)--(-0.25,2*0.433*\x-9*0.433);	
}

\foreach \x in
{1,...,4}
{	
	\draw (-2.75,2*0.433*\x-5*0.433)--(-2.5,2*0.433*\x-6*0.433);	
}

\foreach \x in
{1,...,5}
{
	\draw (-2,2*0.433*\x-6*0.433)--(-1.75,2*0.433*\x-7*0.433);	
}

\foreach \x in
{1,...,6}
{	
	\draw (-1.25,2*0.433*\x-7*0.433)--(-1,2*0.433*\x-8*0.433);	
}

\foreach \x in
{1,...,6,7}
{	
	\draw (-0.5,2*0.433*\x-8*0.433)--(-0.25,2*0.433*\x-9*0.433);
}

\foreach \x in
{1,...,6,7}
{	
	\draw (0.25,2*0.433*\x-7*0.433)--(0.5,2*0.433*\x-8*0.433);
}

\foreach \x in
{1,...,6}
{	
	\draw (1,2*0.433*\x-6*0.433)--(1.25,2*0.433*\x-7*0.433);	
}

\foreach \x in
{1,...,5}
{
	\draw (1.75,2*0.433*\x-5*0.433)--(2,2*0.433*\x-6*0.433);	
}

\foreach \x in
{1,...,4}
{	
	\draw (2.5,2*0.433*\x-4*0.433)--(2.75,2*0.433*\x-5*0.433);	
}

\end{tikzpicture} 
\caption{\label{figure3}The graph $\mathbf{H}(4)$ with weights as assigned in Definition \ref{weight-rule}.}
\end{figure}
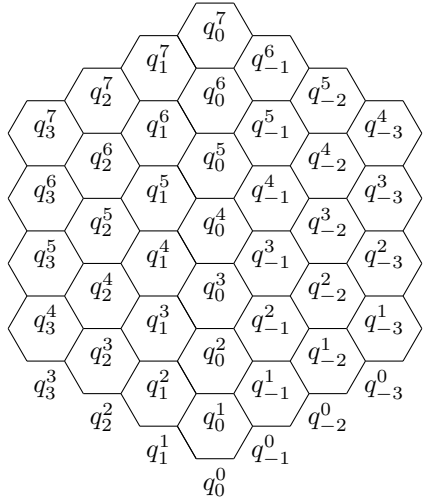

\begin{remark}
The weighted rule is defined such that if one add  a box $(i,j,k)$ to  a 3D partition  asymptotic to $(\lambda,\mu,\nu)$, the weight of the corresponding dimer configuration should be multiplied by a factor of $q_{i-j}$, and removing a box $(i,j,k)$  will multiply the weight by a factor $q_{i-j}^{-1}$. If  all $q_{l}$ are  equal to $q$, the weighted rule here generalizes the one in [\cite{Kuo}, Section 6] (see also [\cite{JWY}, Definition 2.0.4]).
\end{remark}
Let $(\lambda,\mu,\nu)$ be a triple of partitions with their Maya diagrams 
\ben
S_{1}:=S(\lambda), \;\;\;\; S_{2}:=S(\mu), \;\;\;\; S_{3}:=S(\nu).
\een
The modified honeycomb graph $\mathbf{H}(N;\lambda,\mu,\nu)$ is defined by removing vertices labelled by elements in $S_{i}^+\cup S_{i}^-$ from the $i$-th sector of $\mathbf{H}(N)$.  
If we view any 3D partition asymptotic to $(\lambda,\mu,\nu)$ as a collection of boxes, it is  shown in [\cite{JWY}, Section 3.2] that there is a bijective correspondence between all these box collections and all the dimer configurations on (or perfect matchings of)  $\mathbf{H}(N;\lambda,\mu,\nu)$. One may refer to [\cite{JWY}, Figure 3] as an example for this correspondence.

Let $\mathbf{Z}^{D}(\mathbf{G})$ be the sum of the weights of all dimer configurations on the graph $\mathbf{G}$. There is a unique minimal dimer configuration on $\mathbf{H}(N;\lambda,\mu,\nu)$ corresponding to the 3D partition $\pi_{\min}(\lambda,\mu,\nu)$ defined in Section 3. By the weighted rule,  the dimer configuration being  minimal here means that if we set all $q_{l}$ to be equal to $q$, the weight of this dimer configuration with respect to $q$ is minimal as in [\cite{JWY}], or that this dimer configuration has minimal weight with respect to each $q_{l}$. Denote this minimal dimer configuration by $\mathbf{D}_{\min}(\lambda,\mu,\nu;N)$ and its weight by $q^{\omega_{\min}(\lambda,\mu,\nu;N)}$, which is a polynomial of variables $q_{1-N}, q_{2-N},\cdots,q_{0},\cdots,q_{N-2},q_{N-1}$.  For any subset $\mathbf{S}\subseteq\mathbb{Z}^3$, we
define 
\ben
|\mathbf{S}|_{l}=\{(i,j,k)\in \mathbf{S}\;|\; i-j\equiv l \mbox{ mod } n\}.
\een
Then we have 
\ben
\lim_{N\to\infty}\left(q^{\omega_{\min}(\lambda,\mu,\nu;N)}\right)^{-1}\cdot\mathbf{Z}^{D}(\mathbf{H}(N;\lambda,\mu,\nu))=V^n_{\lambda\mu\nu}\cdot\prod_{l=0}^{n-1}q_{l}^{|\mathrm{II}(\lambda,\mu,\nu)|_{l}}\cdot q_{l}^{2|\mathrm{III}(\lambda,\mu,\nu)|_{l}}.
\een
This implies the following orbifold version of [\cite{JWY}, Theorem 3.2.1].
\begin{theorem}
We have the following limit in the sense of formal Laurent series
\ben
\lim_{N\to\infty}\widetilde{\mathbf{Z}}^{D}(\mathbf{H}(N;\lambda,\mu,\nu))=V^n_{\lambda\mu\nu}(q_{0}, q_{1}, \cdots, q_{n-1})
\een	
where 
\ben
\widetilde{\mathbf{Z}}^{D}(\mathbf{H}(N;\lambda,\mu,\nu)):=\left(q^{\widetilde{\omega}_{\min}(\lambda,\mu,\nu;N)}\right)^{-1}\cdot\mathbf{Z}^{D}(\mathbf{H}(N;\lambda,\mu,\nu))
\een
and
\ben
q^{\widetilde{\omega}_{\min}(\lambda,\mu,\nu;N)}:=q^{\omega_{\min}(\lambda,\mu,\nu;N)}\cdot\prod\limits_{l=0}^{n-1}q_{l}^{|\mathrm{II}(\lambda,\mu,\nu)|_{l}}\cdot q_{l}^{2|\mathrm{III}(\lambda,\mu,\nu)|_{l}}.
\een
\end{theorem}

\subsection{The graphical  condensation recurrence for orbifold DT theory}

The graphical condensation is discovered by Kuo in [\cite{Kuo}], and is successfully applied  in [\cite{JWY}] to derive the recurrence for DT topological vertex. In this section, we will obtain  graphical condensation recurrences for  DT $\mathbb{Z}_{n}$-vertex with the weighted rule in Definition \ref{weight-rule} by the following technique. 
\begin{theorem}([\cite{Kuo}, Theorem 5.1])\label{graphical-condensation1}
Let $\mathbf{G}=(\mathbf{V}_{1},\mathbf{V}_{2},\mathbf{E})$ be a given edge-weighted plane bipartite graph with $|\mathbf{V}_{1}|=|\mathbf{V}_{2}|$. Let vertices $a,b,c$ and $d$ appear on a face of $\mathbf{G}$ in a cyclic order. If $a,c\in \mathbf{V}_{1}$ and $b,d\in\mathbf{V}_{2}$, then 
\ben
\mathbf{Z}^{D}(\mathbf{G})\mathbf{Z}^{D}(\mathbf{G}-\{a,b,c,d\})=\mathbf{Z}^D(\mathbf{G}-\{a,b\})\mathbf{Z}^D(\mathbf{G}-\{c,d\})+\mathbf{Z}^D(\mathbf{G}-\{a,d\})\mathbf{Z}^D(\mathbf{G}-\{b,c\}).
\een
\end{theorem}

As in [\cite{JWY}],  we choose $\mathbf{G}$ to be $\mathbf{H}(N;\lambda^{rc},\mu^{rc},\nu)$ and let 
\ben
a:=\max S_{1}^-,\;\;\;b:=\min S_{1}^+,\;\;\;c=\max S_{2}^-,\;\;\;d=\min S_{2}^+.
\een
Then we have 
\ben
&&\mathbf{G}-\{a,b,c,d\}=\mathbf{H}(N;\lambda,\mu,\nu),\\
&&\mathbf{G}-\{a,b\}=\mathbf{H}(N;\lambda,\mu^{rc},\nu),\\
&&\mathbf{G}-\{c,d\}=\mathbf{H}(N;\lambda^{rc},\mu,\nu).
\een
It is an important observation [\cite{JWY}] that the graph $\mathbf{G}-\{a,d\}$ is derived  from the charge $-1$ Maya diagram of $\lambda^r$, the charge 1 Maya diagram of $\mu^c$, and (the charge 0) Maya diagram of $\nu$, and that 3D partitions  asymptotic to $(\lambda^r,\mu^c,\nu)$ are corresponding to dimer configurations on $\mathbf{G}-\{a,d\}$ but with the orgin in $\mathbb{Z}^3$  shifted to the face directly above the central face of $\mathbf{H}(N)$. While the graph  $\mathbf{G}-\{b,c\}$ is obtained from the charge 1 Maya diagram of $\lambda^c$, the charge $-1$ Maya diagram of $\mu^r$, and  Maya diagram of $\nu$, together with 3D partitions  asymptotic to $(\lambda^c,\mu^r,\nu)$  corresponding to dimer configurations on $\mathbf{G}-\{b,c\}$ but with the orgin in $\mathbb{Z}^3$  shifted to the face directly below the central face of $\mathbf{H}(N)$. One can refer to [\cite{JWY}, Figure 4 and 5] for example. Let $\mathbf{D}_{\min}^{\uparrow}(\lambda^r,\mu^c,\nu;N)$ and  $\mathbf{D}_{\min}^{\downarrow}(\lambda^c,\mu^r,\nu;N)$ be the minimal dimer configurations on $\mathbf{G}-\{a,d\}$ and $\mathbf{G}-\{b,c\}$ with their weights denoted by  $q^{\omega_{\min}^{U}(\lambda^r,\mu^c,\nu;N)}$ and $q^{\omega_{\min}^{D}(\lambda^c,\mu^r,\nu;N)}$ respectively.
Let 
\ben
&&q^{\widetilde{\omega}_{\min}^{U}(\lambda^r,\mu^c,\nu;N)}=q^{\omega_{\min}^{U}(\lambda^r,\mu^c,\nu;N)}\cdot\prod\limits_{l=0}^{n-1}q_{l}^{|\mathrm{II}(\lambda^r,\mu^c,\nu)|_{l}}\cdot q_{l}^{2|\mathrm{III}(\lambda^r,\mu^c,\nu)|_{l}},\\
&&q^{\widetilde{\omega}_{\min}^{D}}(\lambda^c,\mu^r,\nu;N)=q^{\omega_{\min}^{D}(\lambda^c,\mu^r,\nu;N)}\cdot\prod\limits_{l=0}^{n-1}q_{l}^{|\mathrm{II}(\lambda^c,\mu^r,\nu)|_{l}}\cdot q_{l}^{2|\mathrm{III}(\lambda^c,\mu^r,\nu)|_{l}}.
\een
Set 
\ben
&&\widetilde{\mathbf{Z}}^{D}(\mathbf{H}(N;\lambda^{rc},\mu^{rc},\nu)-\{a,d\}):=\left(q^{\widetilde{\omega}_{\min}^{U}(\lambda^{r},\mu^{c},\nu;N)}\right)^{-1}\cdot\mathbf{Z}^{D}(\mathbf{H}(N;\lambda^{rc},\mu^{rc},\nu)-\{a,d\}),\\
&&\widetilde{\mathbf{Z}}^{D}(\mathbf{H}(N;\lambda^{rc},\mu^{rc},\nu)-\{b,c\}):=\left(q^{\widetilde{\omega}_{\min}^{D}(\lambda^{c},\mu^{r},\nu;N)}\right)^{-1}\cdot\mathbf{Z}^{D}(\mathbf{H}(N;\lambda^{rc},\mu^{rc},\nu)-\{b,c\}).
\een
As in Section 4.1, we have 
\ben
&&\lim_{N\to\infty}\widetilde{\mathbf{Z}}^{D}(\mathbf{H}(N;\lambda^{rc},\mu^{rc},\nu)-\{a,d\})=V^n_{\lambda^r\mu^c\nu},\\
&&\lim_{N\to\infty}\widetilde{\mathbf{Z}}^{D}(\mathbf{H}(N;\lambda^{rc},\mu^{rc},\nu)-\{b,c\})=V^n_{\lambda^c\mu^r\nu}.
\een 
Now it follows from Theorem \ref{graphical-condensation1}
that
\ben
&&q^{\widetilde{\omega}_{\min}(\lambda,\mu,\nu;N)}\cdot q^{\widetilde{\omega}_{\min}(\lambda^{rc},\mu^{rc},\nu;N)}\widetilde{\mathbf{Z}}^{D}(\mathbf{H}(N;\lambda,\mu,\nu))\widetilde{\mathbf{Z}}^{D}(\mathbf{H}(N;\lambda^{rc},\mu^{rc},\nu))\\
&=&q^{\widetilde{\omega}_{\min}(\lambda^{rc},\mu,\nu;N)}\cdot q^{\widetilde{\omega}_{\min}(\lambda,\mu^{rc},\nu;N)}\widetilde{\mathbf{Z}}^{D}(\mathbf{H}(N;\lambda^{rc},\mu,\nu))\widetilde{\mathbf{Z}}^{D}(\mathbf{H}(N;\lambda,\mu^{rc},\nu))\\
&+&q^{\widetilde{\omega}_{\min}^{U}(\lambda^r,\mu^c,\nu;N)}\cdot q^{\widetilde{\omega}_{\min}^{D}(\lambda^c,\mu^r,\nu;N)}\widetilde{\mathbf{Z}}^{D}(\mathbf{H}(N;\lambda^{rc},\mu^{rc},\nu)-\{a,d\})\widetilde{\mathbf{Z}}^{D}(\mathbf{H}(N;\lambda^{rc},\mu^{rc},\nu)-\{b,c\}).
\een
Then
\ben
&&\widetilde{\mathbf{Z}}^{D}(\mathbf{H}(N;\lambda,\mu,\nu))\widetilde{\mathbf{Z}}^{D}(\mathbf{H}(N;\lambda^{rc},\mu^{rc},\nu))\\
&=&\frac{q^{\widetilde{\omega}_{\min}(\lambda^{rc},\mu,\nu;N)}\cdot q^{\widetilde{\omega}_{\min}(\lambda,\mu^{rc},\nu;N)}}{q^{\widetilde{\omega}_{\min}(\lambda,\mu,\nu;N)}\cdot q^{\widetilde{\omega}_{\min}(\lambda^{rc},\mu^{rc},\nu;N)}}\widetilde{\mathbf{Z}}^{D}(\mathbf{H}(N;\lambda^{rc},\mu,\nu))\widetilde{\mathbf{Z}}^{D}(\mathbf{H}(N;\lambda,\mu^{rc},\nu))\\
&+&\frac{q^{\widetilde{\omega}_{\min}^{U}(\lambda^r,\mu^c,\nu;N)}\cdot q^{\widetilde{\omega}_{\min}^{D}(\lambda^c,\mu^r,\nu;N)}}{q^{\widetilde{\omega}_{\min}(\lambda,\mu,\nu;N)}\cdot q^{\widetilde{\omega}_{\min}(\lambda^{rc},\mu^{rc},\nu;N)}}\widetilde{\mathbf{Z}}^{D}(\mathbf{H}(N;\lambda^{rc},\mu^{rc},\nu)-\{a,d\})\widetilde{\mathbf{Z}}^{D}(\mathbf{H}(N;\lambda^{rc},\mu^{rc},\nu)-\{b,c\}).
\een
It is proved in Section 4.3.1 that 
\ben
\frac{q^{\widetilde{\omega}_{\min}(\lambda^{rc},\mu,\nu;N)}\cdot q^{\widetilde{\omega}_{\min}(\lambda,\mu^{rc},\nu;N)}}{q^{\widetilde{\omega}_{\min}(\lambda,\mu,\nu;N)}\cdot q^{\widetilde{\omega}_{\min}(\lambda^{rc},\mu^{rc},\nu;N)}}=1\;\;\;\;\;\;\mbox{and}\;\;\;\;\;\; \frac{q^{\widetilde{\omega}_{\min}^{U}(\lambda^r,\mu^c,\nu;N)}\cdot q^{\widetilde{\omega}_{\min}^{D}(\lambda^c,\mu^r,\nu;N)}}{q^{\widetilde{\omega}_{\min}(\lambda,\mu,\nu;N)}\cdot q^{\widetilde{\omega}_{\min}(\lambda^{rc},\mu^{rc},\nu;N)}}=q^{K_{1}(\lambda,\mu,\nu)}
\een
where $q^{K_{1}(\lambda,\mu,\nu)}$
is defined in Lemma \ref{recurrence1-weight9} and independent of $N$. Now take $N\to\infty$, we have the following graphical condensation recurrence
\ben
V^n_{\lambda\mu\nu}\cdot V^n_{\lambda^{rc}\mu^{rc}\nu}=V^n_{\lambda^{rc}\mu\nu}\cdot V^n_{\lambda\mu^{rc}\nu}+q^{K_{1}(\lambda,\mu,\nu)}\cdot V^n_{\lambda^r\mu^c\nu}\cdot V^n_{\lambda^c\mu^r\nu}.
\een
Then we have 
\bea\label{DT-recurrence1}
\frac{V^n_{\lambda\mu\nu}}{V^n_{\emptyset\emptyset\emptyset}} =\frac{V^n_{\lambda^{rc}\mu\nu}}{V^n_{\emptyset\emptyset\emptyset}}\cdot \frac{V^n_{\lambda\mu^{rc}\nu}}{V^n_{\emptyset\emptyset\emptyset}}\cdot\left(\frac{V^n_{\lambda^{rc}\mu^{rc}\nu}}{V^n_{\emptyset\emptyset\emptyset}}\right)^{-1}+q^{K_{1}(\lambda,\mu,\nu)}\cdot\frac{V^n_{\lambda^r\mu^c\nu}}{V^n_{\emptyset\emptyset\emptyset}}\cdot \frac{V^n_{\lambda^c\mu^r\nu}}{V^n_{\emptyset\emptyset\emptyset}}\cdot\left(\frac{V^n_{\lambda^{rc}\mu^{rc}\nu}}{V^n_{\emptyset\emptyset\emptyset}}\right)^{-1}
\eea
which generalizes the one in [\cite{JWY}, Section 3.3] by Remark \ref{weight-generalization1}.

However, this recurrence is still not enough to determine orbifold DT/PT vertex correspondence for the full 3-leg case unless one has  2-leg correspondence as the base case via Remark \ref{size-comparision}. As explained in Introduction, due to the fewer symmetries of $V^n_{\lambda\mu\nu}$ (resp. $W^n_{\lambda\mu\nu}$) than $V_{\lambda\mu\nu}$ (resp. $W_{\lambda\mu\nu}$), one need to explore more graphical condensation recurrences for our orbifold case. 

To derive the next  graphical condensation recurrence, we take $\mathbf{G}$ to be $\mathbf{H}(N;\lambda^{rc},\mu,\nu^{rc})$ and 
\ben
a:=\max S_{1}^-,\;\;\;b:=\min S_{1}^+,\;\;\;c:=\max S_{3}^-,\;\;\;d:=\min S_{3}^+.
\een
Then as above we have
\ben
&&\mathbf{G}-\{a,b,c,d\}=\mathbf{H}(N;\lambda,\mu,\nu),\\
&&\mathbf{G}-\{a,b\}=\mathbf{H}(N;\lambda,\mu,\nu^{rc}),\\
&&\mathbf{G}-\{c,d\}=\mathbf{H}(N;\lambda^{rc},\mu,\nu).
\een

Similarly, the graph $\mathbf{G}-\{a,d\}$ is obtained  from the charge $-1$ Maya diagram of $\lambda^r$, the Maya diagram of $\mu$, and the charge 1 Maya diagram of $\nu^c$, while the graph $\mathbf{G}-\{b,c\}$ is obtained from the charge 1 Maya diagram of $\lambda^c$, the Maya diagram of $\mu$, and the charge $-1$ Maya diagram of $\nu^r$. However, slightly different from the first case,  3D partitions  asymptotic to $(\lambda^r,\mu,\nu^c)$  correspond to dimer configurations on $\mathbf{G}-\{a,d\}$ but with the orgin in $\mathbb{Z}^3$  shifted directly to the face on the left-up side of the central face of $\mathbf{H}(N)$, and 3D partitions  asymptotic to $(\lambda^c,\mu,\nu^r)$  correspond to dimer configurations on $\mathbf{G}-\{b,c\}$ but with the orgin in $\mathbb{Z}^3$  shifted directly to the face on the right-down side of  the central face of $\mathbf{H}(N)$. We omit the figures and  one may take an exmple to draw (or modify) the pictures as in [\cite{JWY}, Figure 5] to witness the required shift.

Let $\mathbf{D}_{\min}^{\nwarrow}(\lambda^r,\mu,\nu^c;N)$ and  $\mathbf{D}_{\min}^{\searrow}(\lambda^c,\mu,\nu^r;N)$ be the minimal dimer configurations on $\mathbf{G}-\{a,d\}$ and $\mathbf{G}-\{b,c\}$ with their weights denoted by  $q^{\omega_{\min}^{LU}(\lambda^r,\mu,\nu^c;N)}$ and $q^{\omega_{\min}^{RD}(\lambda^c,\mu,\nu^r;N)}$ respectively. 
Let
\ben
&&\widetilde{\mathbf{Z}}^{D}(\mathbf{H}(N;\lambda^{rc},\mu,\nu^{rc})-\{a,d\}):=\left(q^{\widetilde{\omega}_{\min}^{LU}(\lambda^{r},\mu,\nu^{c};N)}\right)^{-1}\cdot\mathbf{Z}^{D}(\mathbf{H}(N;\lambda^{rc},\mu,\nu^{rc})-\{a,d\}),\\
&&\widetilde{\mathbf{Z}}^{D}(\mathbf{H}(N;\lambda^{rc},\mu,\nu^{rc})-\{b,c\}):=\left(q^{\widetilde{\omega}_{\min}^{RD}(\lambda^{c},\mu,\nu^{r};N)}\right)^{-1}\cdot\mathbf{Z}^{D}(\mathbf{H}(N;\lambda^{rc},\mu,\nu^{rc})-\{b,c\}),
\een
where
\ben
&&q^{\widetilde{\omega}_{\min}^{LU}(\lambda^r,\mu,\nu^c;N)}:=q^{\omega_{\min}^{LU}(\lambda^r,\mu,\nu^c;N)}\cdot\prod\limits_{l=0}^{n-1}q_{l}^{|\mathrm{II}(\lambda^r,\mu,\nu^c)|_{l}}\cdot q_{l}^{2|\mathrm{III}(\lambda^r,\mu,\nu^c)|_{l}},\\
&&q^{\widetilde{\omega}_{\min}^{RD}}(\lambda^c,\mu,\nu^r;N):=q^{\omega_{\min}^{RD}(\lambda^c,\mu,\nu^r;N)}\cdot\prod\limits_{l=0}^{n-1}q_{l}^{|\mathrm{II}(\lambda^c,\mu,\nu^r)|_{l}}\cdot q_{l}^{2|\mathrm{III}(\lambda^c,\mu,\nu^r)|_{l}}.
\een
As in the first case, we have 
\ben
&&\lim_{N\to\infty}\widetilde{\mathbf{Z}}^{D}(\mathbf{H}(N;\lambda^{rc},\mu,\nu^{rc})-\{a,d\})=V^n_{\lambda^r\mu\nu^c},\\
&&\lim_{N\to\infty}\widetilde{\mathbf{Z}}^{D}(\mathbf{H}(N;\lambda^{rc},\mu,\nu^{rc})-\{b,c\})=V^n_{\lambda^c\mu\nu^r}.
\een 
and 
\ben
&&\widetilde{\mathbf{Z}}^{D}(\mathbf{H}(N;\lambda,\mu,\nu))\widetilde{\mathbf{Z}}^{D}(\mathbf{H}(N;\lambda^{rc},\mu,\nu^{rc}))\\
&=&\frac{q^{\widetilde{\omega}_{\min}(\lambda^{rc},\mu,\nu;N)}\cdot q^{\widetilde{\omega}_{\min}(\lambda,\mu,\nu^{rc};N)}}{q^{\widetilde{\omega}_{\min}(\lambda,\mu,\nu;N)}\cdot q^{\widetilde{\omega}_{\min}(\lambda^{rc},\mu,\nu^{rc};N)}}\widetilde{\mathbf{Z}}^{D}(\mathbf{H}(N;\lambda^{rc},\mu,\nu))\widetilde{\mathbf{Z}}^{D}(\mathbf{H}(N;\lambda,\mu,\nu^{rc}))\\
&+&\frac{q^{\widetilde{\omega}_{\min}^{LU}(\lambda^r,\mu,\nu^c;N)}\cdot q^{\widetilde{\omega}_{\min}^{RD}(\lambda^c,\mu,\nu^r;N)}}{q^{\widetilde{\omega}_{\min}(\lambda,\mu,\nu;N)}\cdot q^{\widetilde{\omega}_{\min}(\lambda^{rc},\mu,\nu^{rc};N)}}\widetilde{\mathbf{Z}}^{D}(\mathbf{H}(N;\lambda^{rc},\mu,\nu^{rc})-\{a,d\})\widetilde{\mathbf{Z}}^{D}(\mathbf{H}(N;\lambda^{rc},\mu,\nu^{rc})-\{b,c\})
\een
It is shown in Section 4.3.2 that 
\ben
\frac{q^{\widetilde{\omega}_{\min}(\lambda^{rc},\mu,\nu;N)}\cdot q^{\widetilde{\omega}_{\min}(\lambda,\mu,\nu^{rc};N)}}{q^{\widetilde{\omega}_{\min}(\lambda,\mu,\nu;N)}\cdot q^{\widetilde{\omega}_{\min}(\lambda^{rc},\mu,\nu^{rc};N)}}=1\;\;\;\;\;\;\mbox{and}\;\;\;\;\;\;\frac{q^{\widetilde{\omega}_{\min}^{LU}(\lambda^r,\mu,\nu^c;N)}\cdot q^{\widetilde{\omega}_{\min}^{RD}(\lambda^c,\mu,\nu^r;N)}}{q^{\widetilde{\omega}_{\min}(\lambda,\mu,\nu;N)}\cdot q^{\widetilde{\omega}_{\min}(\lambda^{rc},\mu,\nu^{rc};N)}}=q^{K_{2}(\lambda,\mu,\nu)}
\een
where 
$q^{K_{2}(\lambda,\mu,\nu)}$ defined in Lemma \ref{recurrence2-weight7} 
is independent of $N$. Now let $N\to\infty$, we have 
\bea\label{DT-recurrence2}
\frac{V^n_{\lambda\mu\nu}}{V^n_{\emptyset\emptyset\emptyset}}=\frac{V^n_{\lambda^{rc}\mu\nu}}{V^n_{\emptyset\emptyset\emptyset}}\cdot \frac{V^n_{\lambda\mu\nu^{rc}}}{V^n_{\emptyset\emptyset\emptyset}}\cdot \left(\frac{V^n_{\lambda^{rc}\mu\nu^{rc}}}{V^n_{\emptyset\emptyset\emptyset}}\right)^{-1}+q^{K_{2}(\lambda,\mu,\nu)}\cdot \frac{V^n_{\lambda^r\mu\nu^c}}{V^n_{\emptyset\emptyset\emptyset}}\cdot \frac{V^n_{\lambda^c\mu\nu^r}}{V^n_{\emptyset\emptyset\emptyset}}\cdot \left(\frac{V^n_{\lambda^{rc}\mu\nu^{rc}}}{V^n_{\emptyset\emptyset\emptyset}}\right)^{-1}.
\eea

As explained in Introduction, it is enough to determine the orbifold DT/PT vertex correspondence from the above two equations and [\cite{Zhang}, Theorem 5.22] together with  the symmetries of $V^n_{\lambda\mu\nu}$ and $W^n_{\lambda\mu\nu}$,  see Section 5.3 for the orbifold PT case. 
Since $d(\nu)=d(\nu^\prime)$ and $\widetilde{d}(\nu^\prime)=\nu_{d(\nu)}$ shown in Section 4.3, one can verify that (refer to Lemma \ref{recurrence2-weight7} and Lemma \ref{recurrence3-weight5})
\ben
q^{K_{2}(\lambda,\mu,\nu)}=\overline{q^{K_{3}(\mu^\prime,\lambda^\prime,\nu^\prime)}}.
\een
In addition, by Remark \ref{symmetry1} and Lemma \ref{transpose},  we have 
\ben
\frac{\overline{V}^n_{\mu^\prime\lambda^\prime\nu^{\prime}}}{\overline{V}^n_{\emptyset\emptyset\emptyset}}\cdot \frac{\overline{V}^n_{\mu^\prime(\lambda^\prime)^{rc}(\nu^\prime)^{rc}}}{\overline{V}^n_{\emptyset\emptyset\emptyset}}=\frac{\overline{V}^n_{\mu^\prime(\lambda^\prime)^{rc}\nu^\prime}}{\overline{V}^n_{\emptyset\emptyset\emptyset}}\cdot \frac{\overline{V}^n_{\mu^\prime\lambda^\prime(\nu^{\prime})^{rc}}}{\overline{V}^n_{\emptyset\emptyset\emptyset}}+\overline{q^{K_{3}(\mu^\prime,\lambda^\prime,\nu^\prime)}}\cdot \frac{\overline{V}^n_{\mu^\prime(\lambda^\prime)^c(\nu^\prime)^r}}{\overline{V}^n_{\emptyset\emptyset\emptyset}}\cdot \frac{\overline{V}^n_{\mu^\prime(\lambda^\prime)^r(\nu^\prime)^c}}{\overline{V}^n_{\emptyset\emptyset\emptyset}}
\een
which is equivalent to Equation \eqref{DT-recurrence3}. Therefore, Equation \eqref{DT-recurrence3} follows from Equation  \eqref{DT-recurrence2} and the symmetry of  $V^n_{\lambda\mu\nu}$. Conversely,  Equation \eqref{DT-recurrence2} can again follow from Equation  \eqref{DT-recurrence3} and the symmetry of  $V^n_{\lambda\mu\nu}$. However, Equation \eqref{DT-recurrence1} doesn't follow from Equation \eqref{DT-recurrence2} or \eqref{DT-recurrence3} due to the fewer symmetries of $V^n_{\lambda\mu\nu}$ than $V_{\lambda\mu\nu}$  in Remark \ref{symmetry1} unless $n=1$.
For the completeness, we  independently obtain the third graphical condensation recurrence \eqref{DT-recurrence3}  as follows.
Let 
\ben
&&\mathbf{G}=\mathbf{H}(N;\lambda,\mu^{rc},\nu^{rc}),\\
&&\mathbf{G}-\{a,b,c,d\}=\mathbf{H}(N;\lambda,\mu,\nu),\\
&&\mathbf{G}-\{a,b\}=\mathbf{H}(N;\lambda,\mu,\nu^{rc}),\\
&&\mathbf{G}-\{c,d\}=\mathbf{H}(N;\lambda,\mu^{rc},\nu)
\een
where
\ben
a:=\max S_{2}^-,\;\;\;b:=\min S_{2}^+,\;\;\;c:=\max S_{3}^-,\;\;\;d:=\min S_{3}^+.
\een

Now the graph $\mathbf{G}-\{a,d\}$ (resp. $\mathbf{G}-\{b,c\}$) is obtained  from the Maya diagram of $\lambda$, the charge $-1$ (resp. $1$) Maya diagram of $\mu^r$ (resp. $\mu^c$),  and the charge 1 (resp. $-1$) Maya diagram of $\nu^c$ (resp. $\nu^r$). And 3D partitions  asymptotic to $(\lambda,\mu^r,\nu^c)$ (resp. $(\lambda,\mu^c,\nu^r)$) correspond to dimer configurations on $\mathbf{G}-\{a,d\}$ (resp. $\mathbf{G}-\{b,c\}$) but with the orgin in $\mathbb{Z}^3$  shifted directly to the face on the left-down (resp. right-up) side of the central face of $\mathbf{H}(N)$. Let $\mathbf{D}_{\min}^{\swarrow}(\lambda^r,\mu,\nu^c;N)$ and  $\mathbf{D}_{\min}^{\nearrow}(\lambda^c,\mu,\nu^r;N)$ be the minimal dimer configurations on $\mathbf{G}-\{a,d\}$ and $\mathbf{G}-\{b,c\}$ with their weights denoted by  $q^{\omega_{\min}^{LD}(\lambda,\mu^r,\nu^c;N)}$ and $q^{\omega_{\min}^{RU}(\lambda,\mu^c,\nu^r;N)}$ respectively. 
Let
\ben
&&\widetilde{\mathbf{Z}}^{D}(\mathbf{H}(N;\lambda,\mu^{rc},\nu^{rc})-\{a,d\}):=\left(q^{\widetilde{\omega}_{\min}^{LD}(\lambda,\mu^{r},\nu^{c};N)}\right)^{-1}\cdot\mathbf{Z}^{D}(\mathbf{H}(N;\lambda,\mu^{rc},\nu^{rc})-\{a,d\}),\\
&&\widetilde{\mathbf{Z}}^{D}(\mathbf{H}(N;\lambda,\mu^{rc},\nu^{rc})-\{b,c\}):=\left(q^{\widetilde{\omega}_{\min}^{RU}(\lambda,\mu^{c},\nu^{r};N)}\right)^{-1}\cdot\mathbf{Z}^{D}(\mathbf{H}(N;\lambda,\mu^{rc},\nu^{rc})-\{b,c\}),
\een
where
\ben
&&q^{\widetilde{\omega}_{\min}^{LD}(\lambda,\mu^r,\nu^c;N)}:=q^{\omega_{\min}^{LD}(\lambda,\mu^r,\nu^c;N)}\cdot\prod\limits_{l=0}^{n-1}q_{l}^{|\mathrm{II}(\lambda,\mu^r,\nu^c)|_{l}}\cdot q_{l}^{2|\mathrm{III}(\lambda,\mu^r,\nu^c)|_{l}},\\
&&q^{\widetilde{\omega}_{\min}^{RU}}(\lambda,\mu^c,\nu^r;N):=q^{\omega_{\min}^{RU}(\lambda,\mu^c,\nu^r;N)}\cdot\prod\limits_{l=0}^{n-1}q_{l}^{|\mathrm{II}(\lambda,\mu^c,\nu^r)|_{l}}\cdot q_{l}^{2|\mathrm{III}(\lambda,\mu^c,\nu^r)|_{l}}.
\een
Then we have 
\ben
&&\widetilde{\mathbf{Z}}^{D}(\mathbf{H}(N;\lambda,\mu,\nu))\widetilde{\mathbf{Z}}^{D}(\mathbf{H}(N;\lambda,\mu^{rc},\nu^{rc}))\\
&=&\frac{q^{\widetilde{\omega}_{\min}(\lambda,\mu^{rc},\nu;N)}\cdot q^{\widetilde{\omega}_{\min}(\lambda,\mu,\nu^{rc};N)}}{q^{\widetilde{\omega}_{\min}(\lambda,\mu,\nu;N)}\cdot q^{\widetilde{\omega}_{\min}(\lambda,\mu^{rc},\nu^{rc};N)}}\widetilde{\mathbf{Z}}^{D}(\mathbf{H}(N;\lambda,\mu^{rc},\nu))\widetilde{\mathbf{Z}}^{D}(\mathbf{H}(N;\lambda,\mu,\nu^{rc}))\\
&+&\frac{q^{\widetilde{\omega}_{\min}^{LD}(\lambda,\mu^r,\nu^c;N)}\cdot q^{\widetilde{\omega}_{\min}^{RU}(\lambda,\mu^c,\nu^r;N)}}{q^{\widetilde{\omega}_{\min}(\lambda,\mu,\nu;N)}\cdot q^{\widetilde{\omega}_{\min}(\lambda,\mu^{rc},\nu^{rc};N)}}\widetilde{\mathbf{Z}}^{D}(\mathbf{H}(N;\lambda,\mu^{rc},\nu^{rc})-\{a,d\})\widetilde{\mathbf{Z}}^{D}(\mathbf{H}(N;\lambda,\mu^{rc},\nu^{rc})-\{b,c\}).
\een
It is computed in Section 4.3.3 that
\ben
\frac{q^{\widetilde{\omega}_{\min}(\lambda,\mu^{rc},\nu;N)}\cdot q^{\widetilde{\omega}_{\min}(\lambda,\mu,\nu^{rc};N)}}{q^{\widetilde{\omega}_{\min}(\lambda,\mu,\nu;N)}\cdot q^{\widetilde{\omega}_{\min}(\lambda,\mu^{rc},\nu^{rc};N)}}=1\;\;\;\;\;\;\mbox{and}\;\;\;\;\;\;\frac{q^{\widetilde{\omega}_{\min}^{LD}(\lambda,\mu^r,\nu^c;N)}\cdot q^{\widetilde{\omega}_{\min}^{RU}(\lambda,\mu^c,\nu^r;N)}}{q^{\widetilde{\omega}_{\min}(\lambda,\mu,\nu;N)}\cdot q^{\widetilde{\omega}_{\min}(\lambda,\mu^{rc},\nu^{rc};N)}}=q^{K_{3}(\lambda,\mu,\nu)}
\een
where 
$
q^{K_{3}(\lambda,\mu,\nu)}$ defined in Lemma \ref{recurrence3-weight5}
is independent of $N$.
Since 
\ben
&&\lim_{N\to\infty}\widetilde{\mathbf{Z}}^{D}(\mathbf{H}(N;\lambda,\mu^{rc},\nu^{rc})-\{a,d\})=V^n_{\lambda\mu^r\nu^c},\\
&&\lim_{N\to\infty}\widetilde{\mathbf{Z}}^{D}(\mathbf{H}(N;\lambda,\mu^{rc},\nu^{rc})-\{b,c\})=V^n_{\lambda\mu^c\nu^r},
\een 
then we have
\bea\label{DT-recurrence3}
\frac{V^n_{\lambda\mu\nu}}{V^n_{\emptyset\emptyset\emptyset}}=\frac{V^n_{\lambda\mu^{rc}\nu}}{V^n_{\emptyset\emptyset\emptyset}}\cdot \frac{V^n_{\lambda\mu\nu^{rc}}}{V^n_{\emptyset\emptyset\emptyset}}\cdot \left(\frac{V^n_{\lambda\mu^{rc}\nu^{rc}}}{V^n_{\emptyset\emptyset\emptyset}}\right)^{-1}+q^{K_{3}(\lambda,\mu,\nu)}\cdot \frac{V^n_{\lambda\mu^r\nu^c}}{V^n_{\emptyset\emptyset\emptyset}}\cdot \frac{V^n_{\lambda\mu^c\nu^r}}{V^n_{\emptyset\emptyset\emptyset}}\cdot \left(\frac{V^n_{\lambda\mu^{rc}\nu^{rc}}}{V^n_{\emptyset\emptyset\emptyset}}\right)^{-1}.
\eea

\subsection{Weights for orbifold DT theory}
In order to formulate and compute weights for both orbifold DT and PT theories, we will first introduce some  properties of partitions and their modified partitions in [\cite{JWY}, Section 5.1] as follows. 
For any partition $\eta\neq\emptyset$, let 
\ben
d(\eta)=\max\{i:\eta_{i}\geq i\},\;\;\;\;
\widetilde{d}(\eta)=\max\{i:\eta_{i}\geq d(\eta)\}.
\een
Then we have $\ell(\eta)\geq\widetilde{d}(\eta)\geq d(\eta)\geq1$. It is easy to verify that $d(\eta)=d(\eta^\prime)$ and $\widetilde{d}(\eta^\prime)=\eta_{d(\eta)}$.
\begin{lemma}([\cite{JWY}, Lemma 5.1.9, 5.1.17 and Remark 5.1.31])
\label{modified-partition}
Let $\eta\neq\emptyset$ be a partition. Then
\ben
&&\eta_{i}^r=\left\{
\begin{aligned}
	& \eta_{i}+1 ,\; \;\;\mbox{if  $i<d(\eta)$}, \\
	& \eta_{i+1}, \;\;\;\;\;\;\mbox{if  $i\geq d(\eta)$},
\end{aligned}
\right.\\
&&\eta_{i}^c=\left\{
\begin{aligned}
	& \eta_{i}-1,\; \;\;\; \;\;\; \mbox{if $i\leq \widetilde{d}(\eta)$}, \\
	& d(\eta)-1, \; \;\;\mbox{if $i=\widetilde{d}(\eta)+1$}, \\
	& \eta_{i-1}, \;\;\;\;\;\;\; \;\;\mbox{if $i> \widetilde{d}(\eta)+1$},
\end{aligned}
\right.\\
&&\eta_{i}^{rc}=\left\{
\begin{aligned}
	& \eta_{i},\; \;\;\; \;\;\; \;\;\;\;\;\;\mbox{if $i< d(\eta)$}, \\
	& d(\eta)-1, \; \;\;\mbox{if $d(\eta)\leq i\leq\widetilde{d}(\eta)$}, \\
	& \eta_{i}, \;\;\;\;\;\;\; \;\;\;\;\;\; \mbox{if $i> \widetilde{d}(\eta)$}.
\end{aligned}
\right.
\een

\end{lemma}
\begin{remark}\label{size-comparision}
It follows easily from [\cite{JWY}, Remark 5.1.11, 5.1.27 and 5.1.28]  that  for any partition $\eta\neq\emptyset$, we have $|\eta^r|, |\eta^c|, |\eta^{rc}|<|\eta|$.
\end{remark}
\begin{lemma}([\cite{JWY}, Remark 5.1.12, 5.1.19 and Lemma 5.1.34])\label{length-relation}
Let $\eta\neq\emptyset$ be a partition. Then\\
$(i)$ if $d(\eta)=1$ and $\eta_{1}=1$, then $\ell(\eta^r)=\ell(\eta)-1$,  $\ell(\eta^c)=\ell(\eta^{rc})=0$;\\
$(ii)$  if $d(\eta)=1$ and $\eta_{1}>1$,  then $\ell(\eta^r)=\ell(\eta)-1$, $\ell(\eta^c)=1$ and $\ell(\eta^{rc})=0$;\\
$(iii)$ if $d(\eta)>1$, then $\ell(\eta^r)=\ell(\eta)-1$, $\ell(\eta^c)=\ell(\eta)+1$ and  $\ell(\eta^{rc})=\ell(\eta)$.
\end{lemma}
\begin{lemma}([\cite{JWY}, Lemma 5.1.13 and Remark 5.1.22, 5.1.32])\label{diag-length}
Let $\eta\neq\emptyset$ be a partition. Then \\
$(i)$ $d(\eta^r)=d(\eta)\Longleftrightarrow\eta_{d(\eta)+1}=d(\eta)$,\;\;\;\;$d(\eta^r)=d(\eta)-1\Longleftrightarrow\eta_{d(\eta)+1}<d(\eta)$,\\
$(ii)$ $d(\eta^c)=d(\eta)\Longleftrightarrow\eta_{d(\eta)}>d(\eta)$,\;\;\;\;$d(\eta^c)=d(\eta)-1\Longleftrightarrow\eta_{d(\eta)}=d(\eta)$,\\
$(iii)$ $d(\eta^{rc})=d(\eta)-1$.
\end{lemma}

\begin{remark}\label{special-value}
By Lemma \ref{length-relation} and Lemma \ref{diag-length}, we have \\
$(i)$  if $d(\eta)>1$ and $d(\eta^r)=d(\eta)-1$, then $\widetilde{d}(\eta)= d(\eta)$;
\\
$(ii)$ if $d(\eta)>1$ and $d(\eta^r)=d(\eta)$,  then $\widetilde{d}(\eta)\geq d(\eta)+1$  and $\eta_{i}=d(\eta)$ for $d(\eta)+1\leq i\leq\widetilde{d}(\eta)$.\\
$(iii)$ if $d(\eta)=1$ and $\eta_{1}>1$ or $\eta_{1}=1$, then $\widetilde{d}(\eta)=\ell(\eta)$ and $\eta_{i}=1$ for $2\leq i\leq \ell(\eta)$.\\
$(iv)$ $\eta^c=\emptyset$ if and only if $d(\eta)=1$ and $\eta_{1}=1$, i.e., $\eta_{i}=1$ for $1\leq i\leq\ell(\eta)$.
\end{remark}

\begin{lemma}([\cite{JWY}, Lemma 5.1.24 and Remark 5.1.30])\label{transpose}
Let $\eta\neq\emptyset$ be a partition. Then
\ben
(\eta^c)^\prime=(\eta^\prime)^r;\;\;\;\;\;\;(\eta^r)^\prime=(\eta^\prime)^c;\;\;\;\;\;\;(\eta^{rc})^\prime=(\eta^\prime)^{rc}.
\een
\end{lemma}

\begin{lemma}([\cite{JWY}, Lemma 5.1.16 and Lemma 5.1.23])\label{value-set}
For any partition $\eta\neq\emptyset$, we have \\
$(i)$ $\{i\in\mathbb{Z}_{>0}: \eta^r_{i}>i+1\}=\{i\in\mathbb{Z}_{>0}: i\leq d(\eta)-1\}$,\\
$(ii)$ $\{i\in\mathbb{Z}_{>0}: \eta^c_{i}\geq i-1\}=\{i\in\mathbb{Z}_{>0}: i\leq d(\eta)\}$.
\end{lemma}
Next, we employ the  following notations for  simplifications later.  
\ben
&&q^{\varpi_{1}(m,l,k)}:=\prod_{i=0}^{m}\prod_{j=0}^{m-l}q_{j-i+k}^{m-i},\\
&&q^{\varpi_{2}(\eta,m,k)}:=\prod_{i=1}^{\ell(\eta)}\prod_{j=0}^{m}q_{j-i+k}^{\eta_{i}},\\
&&q^{\varpi_{3}(\eta,m,k)}:=\prod_{i=1}^{\ell(\eta)}\prod_{j=0}^{\eta_{i}-1}q_{j-i+k}^{-m+i},\\
&&q^{\varpi_{4}(\eta,m,k)}:=\prod_{i:1\leq i\leq\eta_{i}}\prod_{j=1}^{i-k}q^{m+\eta_{i}-i+k-1}_{m-i+j+k-1}\cdot\prod_{i:\eta_{i}<i\leq\ell(\eta)}\prod_{j=1}^{\eta_{i}}q_{m-i+j+k-1}^{m+\eta_{i}-i+k-1},\\
&&q^{\varpi_{5}(\eta,m,k)}:=\prod_{i:1\leq i\leq \eta_{i}}\prod_{j=1}^{i-k}q_{i-j-k-m+1}^{\eta_{i}-j}\cdot\prod_{i:\eta_{i}<i\leq\ell(\eta)}\prod_{j=1}^{\eta_{i}}q_{i-j-k-m+1}^{\eta_{i}-j},\\
&&q^{\varpi_{6}(\eta,m)}:=\left(q_{m}^{m+\eta_{1}}\right)^{-1}\cdot\prod_{i:2\leq i\leq \eta_{i}+1}\prod_{j=1}^{i-2}q_{m+j+1-i}^{m+1+\eta_{i}-i}\cdot\prod_{i:\eta_{i}+1<i\leq\ell(\eta)}\prod_{j=1}^{\eta_{i}}q_{m+j+1-i}^{m+1+\eta_{i}-i},\\
&&q^{\varpi_{7}(\eta,m)}:=\left(q^{\eta_{1}}_{-m}\right)^{-1}\cdot \prod_{i:2\leq i\leq\eta_{i}+1}\prod_{j=1}^{i-2}q_{i-j-m-1}^{\eta_{i}-j}\cdot\prod_{i:\eta_{i}+1<i\leq\ell(\eta)}\prod_{j=1}^{\eta_{i}}q_{i-j-m-1}^{\eta_{i}-j},
\een
where $m, l, k$ are all integers and $\eta$ is a partition. And  for  any function $F(\{q_{k}\})$ of variables $\{q_{k}\}$, the notation $\overline{F(\{q_{k}\})}$ denotes $F(\{q_{k}\})\big|_{q_{k}\leftrightarrow q_{-k}}$  throughout this paper.

\subsubsection{Weights for the graphical condensation recurrence with $\mathbf{G}=\mathbf{H}(N;\lambda^{rc},\mu^{rc},\nu)$}
In this subsection, we will  compute the weights of several dimer configurations as  in [\cite{JWY}, Section 5.2] with the given weighted rule in Definition \ref{weight-rule}. The minimal dimer configuration of $\mathbf{H}(N)$ corresponding to the empty 3D partition has the weight from the $N^2$ horizontal dimers (see the tilings on the floor of the rightmost picture in [\cite{JWY}, Figure 15] for $\mathbf{H}(5)$ as an example), which is 
\ben
\prod_{i=0}^{N-1}\prod_{j=0}^{N-1}q_{j-i}^{N-1-i}.
\een
The dimer configuration corresponding to the 3D partition $\pi_{\min}(\lambda,\mu,\nu)$ with $N(|\lambda|+|\mu|+|\nu|)-|\mathrm{II}(\lambda,\mu,\nu)|-2|\mathrm{III}(\lambda,\mu,\nu)|$ boxes has the weight computed as follows
\ben
\dfrac{\prod\limits_{i=0}^{N-1}\prod\limits_{j=0}^{N-1}q_{j-i}^{N-1-i}\cdot\prod\limits_{i=1}^{\ell(\lambda^\prime)}\prod\limits_{j=0}^{N-1}q_{j-i+1}^{\lambda^\prime_{i}}\cdot\prod\limits_{i=1}^{\ell(\mu)}\prod\limits_{j=0}^{N-1}q_{i-j-1}^{\mu_{i}}\cdot\prod\limits_{i=1}^{\ell(\nu)}\prod\limits_{j=0}^{\nu_{i}-1}q_{j-i+1}^N}{\prod\limits_{(i,j,k)\in\mathrm{II}(\lambda,\mu,\nu)}q_{i-j}\cdot\prod\limits_{(i,j,k)\in\mathrm{III}(\lambda,\mu,\nu)}q_{i-j}^2}.
\een
The minimal dimer configuration of $\mathbf{H}(N;\lambda,\mu,\nu)$ has more horizontal dimers in sector 1 and 2 and fewer horizontal dimers in sector 3 than the above dimer configuration. In fact, in sector 1, if $(\lambda^\prime)_{i}\geq i$, the $i$-the part of $\lambda^\prime$ 
contributes $(i-1)$ more horizontal dimers with weights: $q_{N+1-i}^{N-i+\lambda^\prime_{i}}, \cdots, q_{N-1}^{N-i+\lambda^\prime_{i}}$; if $\lambda^\prime_{i}<i$, the $i$-th part of $\lambda^\prime$ contributes $\lambda^\prime_{i}$ more horizontal dimers with weights: $q_{N+1-i}^{N-i+\lambda^\prime_{i}}, \cdots, q_{N-i+\lambda^\prime_{i}}^{N-i+\lambda^\prime_{i}}$.
In sector 2, if $\mu_{i}\geq i$, the $i$-th part of $\mu$ contributes $(i-1)$ more horizontal dimers with weights: $q_{i-1-N}^{\mu_{i}-1}, q_{i-2-N}^{\mu_{i}-2}, \cdots, q_{1-N}^{\mu_{i}-i+1}$; if $\mu_{i}<i$, the  $i$-th part of $\mu$ contributes $\mu_{i}$ more horizontal dimers with weights: $q_{i-1-N}^{\mu_{i}-1}, q_{i-2-N}^{\mu_{i}-2},\cdots,q_{i-\mu_{i}-N}^0$. In sector 3, the  horizontal dimers which are not in the minimal dimer configuration of $\mathbf{H}(N;\lambda,\mu,\nu)$  have the weight $\prod\limits
_{i=1}^{\ell(\nu)}\prod\limits_{j=0}^{\nu_{i}-1}q_{j-i+1}^{2N-i}$.
Then we have

\begin{lemma}\label{recurrence1-weight1} 
The minimal dimer configuration of $\mathbf{H}(N;\lambda,\mu,\nu)$ has the weight\\ $q^{\omega_{\min}(\lambda,\mu,\nu;N)}=q^{\widetilde{\omega}_{\min}(\lambda,\mu,\nu;N)}\cdot\prod\limits_{l=0}^{n-1}q_{l}^{-|\mathrm{II}(\lambda,\mu,\nu)|_{l}}\cdot q_{l}^{-2|\mathrm{III}(\lambda,\mu,\nu)|_{l}}$,	where
\ben
&&q^{\widetilde{\omega}_{\min}(\lambda,\mu,\nu;N)}\\
&=&\prod_{i=0}^{N-1}\prod_{j=0}^{N-1}q_{j-i}^{N-1-i}\cdot\prod_{i=1}^{\ell(\lambda^\prime)}\prod_{j=0}^{N-1}q_{j-i+1}^{\lambda^\prime_{i}}\cdot\prod_{i=1}^{\ell(\mu)}\prod_{j=0}^{N-1}q_{i-j-1}^{\mu_{i}}\cdot\prod_{i=1}^{\ell(\nu)}\prod_{j=0}^{\nu_{i}-1}q_{j-i+1}^N\cdot\prod_{i=1}^{\ell(\nu)}\prod_{j=0}^{\nu_{i}-1}q_{j-i+1}^{-2N+i}\\
&&\times\prod_{i:1\leq i\leq \lambda^\prime_{i}}\prod_{j=1}^{i-1}q_{N-i+j}^{N-i+\lambda^\prime_{i}}\cdot\prod_{i:\lambda^\prime_{i}< i\leq\ell(\lambda^\prime)}\prod_{j=1}^{\lambda^\prime_{i}}q_{N-i+j}^{N-i+\lambda^\prime_{i}}\cdot\prod_{i:1\leq i\leq\mu_{i}}\prod_{j=1}^{i-1}q_{i-j-N}^{\mu_{i}-j}\cdot\prod_{i:\mu_{i}< i\leq\ell(\mu)}\prod_{j=1}^{\mu_{i}}q_{i-j-N}^{\mu_{i}-j}\\
&=&q^{\varpi_{1}(N-1,0,0)}\cdot q^{\varpi_{2}(\lambda^\prime,N-1,1)}\cdot\overline{q^{\varpi_{2}(\mu,N-1,1)}}\cdot q^{\varpi_{3}(\nu,N,1)}\cdot q^{\varpi_{4}(\lambda^\prime,N,1)}\cdot q^{\varpi_{5}(\mu,N,1)}.
\een
\end{lemma}
Similarly, we have the following 
\begin{lemma}\label{recurrence1-weight2} 
The minimal dimer configuration of $\mathbf{H}(N;\lambda^{rc},\mu^{rc},\nu)-\{a,d\}$ has the weight \\
$q^{\omega_{\min}^{U}(\lambda^r,\mu^c,\nu;N)}=q^{\widetilde{\omega}_{\min}^{U}(\lambda^r,\mu^c,\nu;N)}\cdot\prod\limits_{l=0}^{n-1}q_{l}^{-|\mathrm{II}(\lambda^r,\mu^c,\nu)|_{l}}\cdot q_{l}^{-2|\mathrm{III}(\lambda^r,\mu^c,\nu)|_{l}}$, where
\ben
&&q^{\widetilde{\omega}^{U}_{\min}(\lambda^r,\mu^c,\nu;N)}\\
&=&\prod_{i=0}^{N}\prod_{j=0}^{N}q_{j-i}^{N-i}\cdot\prod_{i=1}^{\ell((\lambda^r)^\prime)}\prod_{j=0}^{N}q_{j-i+1}^{(\lambda^r)^\prime_{i}}\cdot\prod_{i=1}^{\ell(\mu^c)}\prod_{j=0}^{N}q_{i-j-1}^{\mu^c_{i}}\cdot\prod_{i=1}^{\ell(\nu)}\prod_{j=0}^{\nu_{i}-1}q_{j-i+1}^{N-1}\cdot\prod_{i=1}^{\ell(\nu)}\prod_{j=0}^{\nu_{i}-1}q_{j-i+1}^{-2N+i}\\
&&\times\left(q_{N}^{N+(\lambda^r)^\prime_{1}}\right)^{-1}\cdot\prod_{i:2\leq i\leq (\lambda^r)^\prime_{i}+1}\prod_{j=1}^{i-2}q_{N+j+1-i}^{N+1+(\lambda^r)^\prime_{i}-i}\cdot\prod_{i:(\lambda^r)^\prime_{i}+1<i\leq\ell((\lambda^r)^\prime)}\prod_{j=1}^{(\lambda^r)^\prime_{i}}q_{N+j+1-i}^{N+1+(\lambda^r)^\prime_{i}-i}\\
&&\times\left(q^{\mu^{c}_{1}}_{-N}\right)^{-1}\prod_{i:2\leq i\leq\mu^c_{i}+1}\prod_{j=1}^{i-2}q_{i-j-N-1}^{\mu^c_{i}-j}\cdot\prod_{i:\mu^c_{i}+1<i\leq\ell(\mu^{c})}\prod_{j=1}^{\mu^c_{i}}q_{i-j-N-1}^{\mu^c_{i}-j}\\
&=&q^{\varpi_{1}(N,0,0)}\cdot q^{\varpi_{2}((\lambda^\prime)^c,N,1)}\cdot\overline{q^{\varpi_{2}(\mu^c,N,1)}}\cdot q^{\varpi_{3}(\nu,N+1,1)}\cdot q^{\varpi_{6}((\lambda^\prime)^c,N)}\cdot q^{\varpi_{7}(\mu^c,N)}.
\een
\end{lemma}

\begin{lemma}\label{recurrence1-weight3} 
The minimal dimer configuration of $\mathbf{H}(N;\lambda^{rc},\mu^{rc},\nu)-\{b,c\}$ has the weight \\
$q^{\omega_{\min}^{D}(\lambda^c,\mu^r,\nu;N)}=q^{\widetilde{\omega}_{\min}^{D}}(\lambda^c,\mu^r,\nu;N)\cdot\prod\limits_{l=0}^{n-1}q_{l}^{-|\mathrm{II}(\lambda^c,\mu^r,\nu)|_{l}}\cdot q_{l}^{-2|\mathrm{III}(\lambda^c,\mu^r,\nu)|_{l}}$, where
\ben
&&q^{\widetilde{\omega}^{D}_{\min}(\lambda^c,\mu^r,\nu;N)}\\
&=&\prod_{i=0}^{N-2}\prod_{j=0}^{N-2}q_{j-i}^{N-i-2}\cdot\prod_{i=1}^{\ell((\lambda^c)^\prime)}\prod_{j=0}^{N-2}q_{j-i+1}^{(\lambda^c)^\prime_{i}}\cdot\prod_{i=1}^{\ell(\mu^r)}\prod_{j=0}^{N-2}q_{i-j-1}^{\mu^r_{i}}\cdot\prod_{i=1}^{\ell(\nu)}\prod_{j=0}^{\nu_{i}-1}q_{j-i+1}^{N+1}\cdot\prod_{i=1}^{\ell(\nu)}\prod_{j=0}^{\nu_{i}-1}q_{j-i+1}^{-2N+i}\\
&&\times\prod_{i:1\leq i\leq(\lambda^c)_{i}^\prime}\prod_{j=1}^iq^{N+(\lambda^c)_{i}^\prime-i-1}_{N-i+j-1}\cdot\prod_{i:(\lambda^c)_{i}^\prime<i\leq\ell((\lambda^c)^\prime)}\prod_{j=1}^{(\lambda^c)_{i}^\prime}q_{N-i+j-1}^{N+(\lambda^c)_{i}^\prime-i-1}\\
&&\times\prod_{i:1\leq i\leq \mu^r_{i}}\prod_{j=1}^iq_{i-j-N+1}^{\mu^r_{i}-j}\cdot\prod_{i:\mu^r_{i}<i\leq\ell(\mu^r)}\prod_{j=1}^{\mu^r_{i}}q_{i-j-N+1}^{\mu^r_{i}-j}\\
&=&q^{\varpi_{1}(N-2,0,0)}\cdot q^{\varpi_{2}((\lambda^\prime)^r,N-2,1)}\cdot\overline{q^{\varpi_{2}(\mu^r,N-2,1)}}\cdot q^{\varpi_{3}(\nu,N-1,1)}\cdot q^{\varpi_{4}((\lambda^\prime)^r,N,0)}\cdot q^{\varpi_{5}(\mu^r,N,0)}.
\een

\end{lemma}
Now it follows from Lemma \ref{recurrence1-weight1} that
\ben
\frac{q^{\widetilde{\omega}_{\min}(\lambda^{rc},\mu,\nu;N)}\cdot q^{\widetilde{\omega}_{\min}(\lambda,\mu^{rc},\nu;N)}}{q^{\widetilde{\omega}_{\min}(\lambda,\mu,\nu;N)}\cdot q^{\widetilde{\omega}_{\min}(\lambda^{rc},\mu^{rc},\nu;N)}}=1.
\een
We still need the following lemmas to obtain the graphical condensation recurrence \eqref{DT-recurrence1}.
 
\begin{lemma}\label{recurrence1-weight9}
	 For any partitions $\lambda,\mu\neq\emptyset$ and $\nu$, we have 
	\ben
	\frac{q^{\widetilde{\omega}_{\min}^{U}(\lambda^r,\mu^c,\nu;N)}\cdot q^{\widetilde{\omega}_{\min}^{D}(\lambda^{c},\mu^{r},\nu;N)}}{q^{\widetilde{\omega}_{\min}(\lambda,\mu,\nu;N)}\cdot q^{\widetilde{\omega}_{\min}(\lambda^{rc},\mu^{rc},\nu;N)}}
	=q^{K_{1}(\lambda,\mu,\nu)}
	\een
where 
\ben
q^{K_{1}(\lambda,\mu,\nu)}:=q_{-\widetilde{d}(\lambda^\prime)}^{d(\lambda^\prime)}\cdot q_{\widetilde{d}(\mu)}^{d(\mu)}\cdot \prod\limits_{i=-\widetilde{d}(\lambda^\prime)}^{\widetilde{d}(\mu)}q_{i}^{-1}\cdot\prod\limits_{i=\widetilde{d}(\lambda^\prime)+1}^{\ell(\lambda^\prime)}\left(\frac{q_{-i}}{q_{-i+1}}\right)^{\lambda^\prime_{i}}\cdot\prod\limits_{i=\widetilde{d}(\mu)+1}^{\ell(\mu)}\left(\frac{q_{i}}{q_{i-1}}\right)^{\mu_{i}}
\een 
is independent of $N$.
\end{lemma}
\begin{proof}
It follows from Lemmas \ref{recurrence1-weight1},  \ref{recurrence1-weight2},  \ref{recurrence1-weight3}, 	 \ref{recurrence1-weight4},  \ref{recurrence1-weight5}, 	 \ref{recurrence1-weight6}, 	 \ref{recurrence1-weight7},  \ref{recurrence1-weight8}. 					
\end{proof}
\begin{remark}\label{weight-generalization1}
When $n=1$, i.e., all $q_{l}$	are equal to $q$, then we have
\ben
q^{K_{1}(\lambda,\mu,\nu)}:=q^{d(\lambda^\prime)+d(\mu)-\widetilde{d}(\mu)-\widetilde{d}(\lambda^\prime)-1}=q^{d(\lambda)+d(\mu)-\mu^\prime_{d(\mu)}-\lambda_{d(\lambda)}-1}
\een
which is equal to $q^{-K}$ in [\cite{JWY}, Section 5.2].
\end{remark}	
\begin{lemma}\label{recurrence1-weight4} 
\ben
\frac{q^{\varpi_{1}(N,0,0)}\cdot q^{\varpi_{1}(N-2,0,0)}}{\left(q^{\varpi_{1}(N-1,0,0)}\right)^2}=q_{N}^N\cdot\prod_{i=1-N}^{N-1}q_{i}.
\een
\end{lemma}

\begin{lemma}\label{recurrence1-weight5} 
For any partition $\eta\neq\emptyset$, we have
\ben
\frac{q^{\varpi_{2}(\eta^c,N,1)}\cdot q^{\varpi_{2}(\eta^r,N-2,1)}}{q^{\varpi_{2}(\eta,N-1,1)}\cdot q^{\varpi_{2}(\eta^{rc},N-1,1)}}=q_{-\widetilde{d}(\eta)}^{d(\eta)-1}\cdot\prod\limits_{i=\widetilde{d}(\eta)+1}^{\ell(\eta)}\left(\frac{q_{-i}}{q_{-i+1}}\right)^{\eta_{i}}\cdot\frac{\prod\limits_{i=1}^{\widetilde{d}(\eta)}q_{N-i+1}^{\eta_{i}-1}\cdot\prod\limits_{i=0}^{N-1}q_{i-\widetilde{d}(\eta)+1}^{-1}}{\prod\limits_{i=1}^{d(\eta)-1}q_{N-i}^{\eta_{i}+1}\cdot\prod\limits_{i=N-\widetilde{d}(\eta)+1}^{N-d(\eta)}q_{i}^{d(\eta)}}.
\een

\end{lemma}
\begin{proof}
If $d(\eta)>1$, by Lemma \ref{modified-partition},  Lemma \ref{length-relation} and Remark \ref{special-value}, we have
\ben
\frac{q^{\varpi_{2}(\eta^c,N,1)}}{q^{\varpi_{2}(\eta,N-1,1)}}=\frac{\prod\limits_{i=1}^{\widetilde{d}(\eta)}\prod\limits_{j=0}^{N}q_{j-i+1}^{\eta_{i}-1}\cdot\prod\limits_{j=0}^{N}q_{j-\widetilde{d}(\eta)}^{d(\eta)-1}\cdot\prod\limits_{i=\widetilde{d}(\eta)+2}^{\ell(\eta)+1}\prod\limits_{j=0}^{N}q_{j-i+1}^{\eta_{i-1}}}{\prod\limits_{i=1}^{\widetilde{d}(\eta)}\prod\limits_{j=0}^{N-1}q_{j-i+1}^{\eta_{i}}\cdot\prod\limits_{i=\widetilde{d}(\eta)+1}^{\ell(\eta)}\prod\limits_{j=0}^{N-1}q_{j-i+1}^{\eta_{i}}}
=\frac{\prod\limits_{i=1}^{\widetilde{d}(\eta)}q_{N-i+1}^{\eta_{i}-1}\cdot\prod\limits_{j=0}^{N}q_{j-\widetilde{d}(\eta)}^{d(\eta)-1}\cdot\prod\limits_{i=\widetilde{d}(\eta)+1}^{\ell(\eta)}q_{-i}^{\eta_{i}}}{\prod\limits_{i=1}^{\widetilde{d}(\eta)}\prod\limits_{j=0}^{N-1}q_{j-i+1}}
\een
and
\ben
\frac{ q^{\varpi_{2}(\eta^r,N-2,1)}}{ q^{\varpi_{2}(\eta^{rc},N-1,1)}}
&=&\frac{\prod\limits_{i=1}^{d(\eta)-1}\prod\limits_{j=0}^{N-2}q_{j-i+1}^{\eta_{i}+1}\cdot\prod\limits_{i=d(\eta)}^{\ell(\eta)-1}\prod\limits_{j=0}^{N-2}q_{j-i+1}^{\eta_{i+1}}}{\prod\limits_{i=1}^{d(\eta)-1}\prod\limits_{j=0}^{N-1}q_{j-i+1}^{\eta_{i}}\cdot\prod\limits_{i=d(\eta)}^{\widetilde{d}(\eta)}\prod\limits_{j=0}^{N-1}q_{j-i+1}^{d(\eta)-1}\cdot\prod\limits_{i=\widetilde{d}(\eta)+1}^{\ell(\eta)}\prod\limits_{j=0}^{N-1}q_{j-i+1}^{\eta_{i}}}\\
&=&\frac{\prod\limits_{i=1}^{d(\eta)-1}\prod\limits_{j=0}^{N-2}q_{j-i+1}\cdot\prod\limits_{i=d(\eta)+1}^{\widetilde{d}(\eta)}\prod\limits_{j=1}^{N-1}q_{j-i+1}}{\prod\limits_{i=1}^{d(\eta)-1}q_{N-i}^{\eta_{i}}\cdot\prod\limits_{j=0}^{N-1}q_{j-d(\eta)+1}^{d(\eta)-1}\cdot\prod\limits_{i=d(\eta)+1}^{\widetilde{d}(\eta)}q_{-i+1}^{d(\eta)-1}\cdot\prod\limits_{i=\widetilde{d}(\eta)+1}^{\ell(\eta)}q_{-i+1}^{\eta_{i}}}.
\een
Then 
\ben
&&\frac{q^{\varpi_{2}(\eta^c,N,1)}\cdot q^{\varpi_{2}(\eta^r,N-2,1)}}{q^{\varpi_{2}(\eta,N-1,1)}\cdot q^{\varpi_{2}(\eta^{rc},N-1,1)}}\\
&=&\frac{\prod\limits_{i=1}^{\widetilde{d}(\eta)}q_{N-i+1}^{\eta_{i}-1}\cdot\prod\limits_{j=0}^{N}q_{j-\widetilde{d}(\eta)}^{d(\eta)-1}\cdot\prod\limits_{i=\widetilde{d}(\eta)+1}^{\ell(\eta)}q_{-i}^{\eta_{i}}}{\prod\limits_{i=1}^{d(\eta)-1}q_{N-i}^{\eta_{i}+1}\cdot\prod\limits_{j=1-\widetilde{d}(\eta)}^{N-d(\eta)}q_{j}^{d(\eta)-1}\cdot\prod\limits_{i=\widetilde{d}(\eta)+1}^{\ell(\eta)}q_{-i+1}^{\eta_{i}}\cdot\prod\limits_{i=d(\eta)+1}^{\widetilde{d}(\eta)}q_{-i+1}\cdot\prod\limits_{j=0}^{N-1}q_{j-d(\eta)+1}}\\
&=&q_{-\widetilde{d}(\eta)}^{d(\eta)-1}\cdot\prod\limits_{i=\widetilde{d}(\eta)+1}^{\ell(\eta)}\left(\frac{q_{-i}}{q_{-i+1}}\right)^{\eta_{i}}\cdot\frac{\prod\limits_{i=1}^{\widetilde{d}(\eta)}q_{N-i+1}^{\eta_{i}-1}\cdot\prod\limits_{j=0}^{N-1}q_{j-\widetilde{d}(\eta)+1}^{-1}}{\prod\limits_{i=1}^{d(\eta)-1}q_{N-i}^{\eta_{i}+1}\cdot\prod\limits_{j=N-\widetilde{d}(\eta)+1}^{N-d(\eta)}q_{j}^{d(\eta)}}.
\een

If $d(\eta)=1$ and $\eta_{1}>1$,  then by Remark  \ref{special-value}
\ben
&&\frac{q^{\varpi_{2}(\eta^c,N,1)}\cdot q^{\varpi_{2}(\eta^r,N-2,1)}}{q^{\varpi_{2}(\eta,N-1,1)}\cdot q^{\varpi_{2}(\eta^{rc},N-1,1)}}
=\frac{\prod\limits_{j=0}^{N}q_{j}^{\eta_{1}-1}\prod\limits_{i=1}^{\ell(\eta)-1}\prod\limits_{j=0}^{N-2}q_{j-i+1}^{\eta_{i+1}}}{\prod\limits_{i=1}^{\ell(\eta)}\prod\limits_{j=0}^{N-1}q_{j-i+1}^{\eta_{i}}}=\frac{q_{N}^{\eta_{1}}}{\prod\limits_{j=1-\ell(\eta)}^{N}q_{j}}.
\een

Similarly, if $d(\eta)=1$ and $\eta_{1}=1$,  we have
\ben
&&\frac{q^{\varpi_{2}(\eta^c,N,1)}\cdot q^{\varpi_{2}(\eta^r,N-2,1)}}{q^{\varpi_{2}(\eta,N-1,1)}\cdot q^{\varpi_{2}(\eta^{rc},N-1,1)}}=\frac{\prod\limits_{i=1}^{\ell(\eta)-1}\prod\limits_{j=0}^{N-2}q_{j-i+1}^{\eta_{i+1}}}{\prod\limits_{i=1}^{\ell(\eta)}\prod\limits_{j=0}^{N-1}q_{j-i+1}^{\eta_{i}}}=\frac{1}{\prod\limits_{j=1-\ell(\eta)}^{N-1}q_{j}}.
\een	
Now the proof is completed.
\end{proof}	

\begin{lemma}\label{recurrence1-weight6}  For any partition $\eta$, we have
	\ben
	\frac{q^{\varpi_{3}(\eta,N+1,1)}q^{\varpi_{3}(\eta,N-1,1)}}{\left(q^{\varpi_{3}(\eta,N,1)}\right)^2}=1.
	\een
\end{lemma}

\begin{lemma}\label{recurrence1-weight7}
For any partition $\eta\neq\emptyset$, we have
\ben
\frac{q^{\varpi_{6}(\eta^c,N)}\cdot q^{\varpi_{4}(\eta^r,N,0)}}{q^{\varpi_{4}(\eta,N,1)}\cdot q^{\varpi_{4}(\eta^{rc},N,1)}}=\frac{\prod\limits_{i=1}^{d(\eta)-1}q^{\eta_{i}-i}_{N-i}}{q_{N}^{N}\cdot\prod\limits_{i=1}^{d(\eta)}q_{N-i+1}^{\eta_{i}-i}}.
\een
\end{lemma}	
\begin{proof}
If $d(\eta)>1$ and $d(\eta^r)=d(\eta)$, then by Lemmas \ref{modified-partition}, \ref{length-relation}, \ref{diag-length}, \ref{value-set} and Remark \ref{special-value}
\ben
\frac{q^{\varpi_{6}(\eta^c,N)}}{q^{\varpi_{4}(\eta,N,1)}}
&=&\frac{\prod\limits_{i=2}^{d(\eta)}\prod\limits_{j=2}^{i-1}q_{N+j-i}^{N+\eta_{i}-i}\cdot\prod\limits_{i=d(\eta)+1}^{\widetilde{d}(\eta)}\prod\limits_{j=2}^{\eta_{i}}q_{N+j-i}^{N+\eta_{i}-i}\cdot\prod\limits_{j=1}^{d(\eta)-1}q_{N+j-\widetilde{d}(\eta)}^{N+d(\eta)-\widetilde{d}(\eta)-1}\cdot\prod\limits_{i=\widetilde{d}(\eta)+1}^{\ell(\eta)}\prod\limits_{j=1}^{\eta_{i}}q_{N+j-i}^{N+\eta_{i}-i}}{q_{N}^{N+\eta_{1}-1}\cdot\prod\limits_{i=2}^{d(\eta)}\prod\limits_{j=1}^{i-1}q_{N-i+j}^{N+\eta_{i}-i}\cdot\prod\limits_{i=d(\eta)+1}^{\widetilde{d}(\eta)}\prod\limits_{j=1}^{\eta_{i}}q_{N-i+j}^{N+\eta_{i}-i}\cdot\prod\limits_{i=\widetilde{d}(\eta)+1}^{\ell(\eta)}\prod\limits_{j=1}^{\eta_{i}}q_{N-i+j}^{N+\eta_{i}-i}}\\
&=&\frac{\prod\limits_{j=1}^{d(\eta)-1}q_{N+j-\widetilde{d}(\eta)}^{N+d(\eta)-\widetilde{d}(\eta)-1}}{\prod\limits_{i=1}^{d(\eta)}q_{N-i+1}^{N+\eta_{i}-i}\cdot\prod\limits_{i=d(\eta)+1}^{\widetilde{d}(\eta)}q_{N-i+1}^{N+d(\eta)-i}}
\een
and 

\ben
\frac{q^{\varpi_{4}(\eta^r,N,0)}}{q^{\varpi_{4}(\eta^{rc},N,1)}}
&=&\frac{\prod\limits_{i=1}^{d(\eta)-1}\prod\limits_{j=0}^{i-1}q^{N+\eta_{i}-i}_{N-i+j}\cdot\prod\limits_{j=1}^{d(\eta)}q^{N-1}_{N-d(\eta)+j-1}\cdot\prod\limits_{i=d(\eta)+2}^{\ell(\eta)}\prod\limits_{j=1}^{\eta_{i}}q_{N-i+j}^{N+\eta_{i}-i}}{\prod\limits_{i=1}^{d(\eta)-1}\prod\limits_{j=1}^{i-1}q_{N-i+j}^{N-i+\eta_{i}}\cdot\prod\limits_{i=d(\eta)}^{\widetilde{d}(\eta)}\prod\limits_{j=1}^{d(\eta)-1}q_{N-i+j}^{N-i+d(\eta)-1}\cdot\prod\limits_{i=\widetilde{d}(\eta)+1}^{\ell(\eta)}\prod\limits_{j=1}^{\eta_{i}}q_{N-i+j}^{N-i+\eta_{i}}}\\
&=&\frac{\prod\limits_{i=1}^{d(\eta)-1}q^{N+\eta_{i}-i}_{N-i}\cdot\prod\limits_{i=d(\eta)+1}^{\widetilde{d}(\eta)}q_{N+d(\eta)-i}^{N+d(\eta)-i}\cdot\prod\limits_{i=d(\eta)+1}^{\widetilde{d}(\eta)}\prod\limits_{j=1}^{d(\eta)-1}q_{N-i+j}}{\prod\limits_{j=1}^{d(\eta)-1}q_{N-d(\eta)+j}^{N-1}}.
\een
Then 
\ben
&&\frac{q^{\varpi_{6}(\eta^c,N)}\cdot q^{\varpi_{4}(\eta^r,N,0)}}{q^{\varpi_{4}(\eta,N,1)}\cdot q^{\varpi_{4}(\eta^{rc},N,1)}}\\
&=&\frac{\prod\limits_{j=1}^{d(\eta)-1}q_{N+j-\widetilde{d}(\eta)}^{N+d(\eta)-\widetilde{d}(\eta)-1}\cdot\prod\limits_{i=d(\eta)+1}^{\widetilde{d}(\eta)}q_{N+d(\eta)-i}^{N+d(\eta)-i}}{\prod\limits_{i=d(\eta)+1}^{\widetilde{d}(\eta)}q_{N-i+1}^{N+d(\eta)-i}\cdot\prod\limits_{j=1}^{d(\eta)-1}q_{N-d(\eta)+j}^{N-1}}\cdot\frac{\prod\limits_{i=1}^{d(\eta)-1}q^{N+\eta_{i}-i}_{N-i}\cdot\prod\limits_{i=d(\eta)+1}^{\widetilde{d}(\eta)}\prod\limits_{j=1}^{d(\eta)-1}q_{N-i+j}}{\prod\limits_{i=1}^{d(\eta)}q_{N-i+1}^{N+\eta_{i}-i}}\\
&=&\frac{\prod\limits_{j=1}^{d(\eta)-1}q_{N+j-\widetilde{d}(\eta)}^{d(\eta)-\widetilde{d}(\eta)-1}\cdot\prod\limits_{i=d(\eta)}^{\widetilde{d}(\eta)}\prod\limits_{j=1}^{d(\eta)-1}q_{N-i+j}\cdot\prod\limits_{i=d(\eta)+1}^{\widetilde{d}(\eta)}q_{N+d(\eta)-i}^{d(\eta)-i}}{\prod\limits_{i=d(\eta)+1}^{\widetilde{d}(\eta)}q_{N-i+1}^{d(\eta)-i}}\cdot\frac{\prod\limits_{i=1}^{d(\eta)-1}q^{\eta_{i}-i}_{N-i}}{q_{N}^{N}\cdot\prod\limits_{i=1}^{d(\eta)}q_{N-i+1}^{\eta_{i}-i}}\\
&=&\frac{\prod\limits_{i=1}^{d(\eta)-1}q^{\eta_{i}-i}_{N-i}}{q_{N}^{N}\cdot\prod\limits_{i=1}^{d(\eta)}q_{N-i+1}^{\eta_{i}-i}}.
\een
The proof is completed by using the above similar argument  for the other cases.
\end{proof}

\begin{lemma}\label{recurrence1-weight8}
For any partition $\eta\neq\emptyset$, we have
\ben
\frac{q^{\varpi_{7}(\eta^c,N)}\cdot q^{\varpi_{5}(\eta^r,N,0)}}{q^{\varpi_{5}(\eta,N,1)}\cdot q^{\varpi_{5}(\eta^{rc},N,1)}}=\frac{\prod\limits_{i=1}^{d(\eta)-1}q_{i-N}^{\eta_{i}}}{\prod\limits_{i=1}^{d(\eta)}q_{i-N-1}^{\eta_{i}-1}}.
\een
\end{lemma}
\begin{proof}
It follows from the similar argument as in the proof of Lemma \ref{recurrence1-weight7}.
\end{proof}

\subsubsection{Weights for the graphical condensation recurrence with $\mathbf{G}=\mathbf{H}(N;\lambda^{rc},\mu,\nu^{rc})$}
The similar argument as in Section 4.3.1 shows that 

\begin{lemma}\label{recurrence2-weight1}
The minimal dimer configuration of 	$\mathbf{H}(N;\lambda^{rc},\mu,\nu^{rc})-\{a,d\}$ has the weight \\
$q^{\omega_{\min}^{LU}(\lambda^r,\mu,\nu^c;N)}:=q^{\widetilde{\omega}_{\min}^{LU}(\lambda^r,\mu,\nu^c;N)}\cdot\prod\limits_{l=0}^{n-1}q_{l}^{-|\mathrm{II}(\lambda^r,\mu,\nu^c)|_{l}}\cdot q_{l}^{-2|\mathrm{III}(\lambda^r,\mu,\nu^c)|_{l}}$, where
\ben
&&q^{\widetilde{\omega}^{LU}_{\min}(\lambda^r,\mu,\nu^c;N)}\\
&=&\prod_{i=0}^{N}\prod_{j=0}^{N-1}q_{j-i+1}^{N-i}\cdot\prod_{i=1}^{\ell((\lambda^r)^\prime)}\prod_{j=0}^{N-1}q_{j-i+2}^{(\lambda^r)^\prime_{i}}\cdot\prod_{i=1}^{\ell(\mu)}\prod_{j=0}^{N}q_{i-j}^{\mu_{i}}\cdot\prod_{i=1}^{\ell(\nu^c)}\prod_{j=0}^{\nu^c_{i}-1}q_{j-i+2}^{N-1}\cdot\prod_{i=1}^{\ell(\nu^c)}\prod_{j=0}^{\nu^c_{i}-1}q_{j-i+2}^{-2N+i}\\
&&\times\left(q_{N}^{N+(\lambda^r)^\prime_{1}}\right)^{-1}\cdot\prod_{i:2\leq i\leq (\lambda^r)^\prime_{i}+1}\prod_{j=1}^{i-2}q_{N+j+1-i}^{N+1+(\lambda^r)^\prime_{i}-i}\cdot\prod_{i:(\lambda^r)^\prime_{i}+1<i\leq\ell((\lambda^r)^\prime)}\prod_{j=1}^{(\lambda^r)^\prime_{i}}q_{N+j+1-i}^{N+1+(\lambda^r)^\prime_{i}-i}\\
&&\times\prod_{i:1\leq i\leq\mu_{i}}\prod_{j=1}^{i-1}q_{i-j-N}^{\mu_{i}-j}\cdot\prod_{i:\mu_{i}<i\leq\ell(\mu)}\prod_{j=1}^{\mu_{i}}q_{i-j-N}^{\mu_{i}-j}\\
&=&q^{\varpi_{1}(N,1,1)}\cdot q^{\varpi_{2}((\lambda^\prime)^c,N-1,2)}\cdot\overline{q^{\varpi_{2}(\mu,N,0)}}\cdot q^{\varpi_{3}(\nu^c,N+1,2)}\cdot q^{\varpi_{6}((\lambda^\prime)^c,N)}\cdot q^{\varpi_{5}(\mu,N,1)}.
\een
\end{lemma}

\begin{lemma}\label{recurrence2-weight2}
The minimal dimer configuration of 	$\mathbf{H}(N;\lambda^{rc},\mu,\nu^{rc})-\{b,c\}$ has the weight \\
$q^{\omega_{\min}^{RD}(\lambda^c,\mu,\nu^r;N)}:=q^{\widetilde{\omega}_{\min}^{RD}}(\lambda^c,\mu,\nu^r;N)\cdot\prod\limits_{l=0}^{n-1}q_{l}^{|\mathrm{II}(\lambda^c,\mu,\nu^r)|_{l}}\cdot q_{l}^{2|\mathrm{III}(\lambda^c,\mu,\nu^r)|_{l}}$, where
\ben
&&q^{\widetilde{\omega}^{RD}_{\min}(\lambda^c,\mu,\nu^r;n)}\\
&=&\prod_{i=0}^{N-2}\prod_{j=0}^{N-1}q_{j-i-1}^{N-i-2}\cdot\prod_{i=1}^{\ell((\lambda^c)^\prime)}\prod_{j=0}^{N-1}q_{j-i}^{(\lambda^c)^\prime_{i}}\cdot\prod_{i=1}^{\ell(\mu)}\prod_{j=0}^{N-2}q_{i-j-2}^{\mu_{i}}\cdot\prod_{i=1}^{\ell(\nu^r)}\prod_{j=0}^{\nu^r_{i}-1}q_{j-i}^{N}\cdot\prod_{i=1}^{\ell(\nu^r)}\prod_{j=0}^{\nu^r_{i}-1}q_{j-i}^{-2N+i+1}\\
&&\times\prod_{i:1\leq i\leq(\lambda^c)_{i}^\prime}\prod_{j=1}^iq^{N+(\lambda^c)_{i}^\prime-i-1}_{N-i+j-1}\cdot\prod_{i:(\lambda^c)_{i}^\prime<i\leq\ell((\lambda^c)^\prime)}\prod_{j=1}^{(\lambda^c)_{i}^\prime}q_{N-i+j-1}^{N+(\lambda^c)_{i}^\prime-i-1}\\
&&\times\prod_{i:1\leq i\leq \mu_{i}}\prod_{j=1}^{i-1}q_{i-j-N}^{\mu_{i}-j}\cdot\prod_{i:\mu_{i}<i\leq\ell(\mu)}\prod_{j=1}^{\mu_{i}}q_{i-j-N}^{\mu_{i}-j}\\
&=&q^{\varpi_{1}(N-2,-1,-1)}\cdot q^{\varpi_{2}((\lambda^\prime)^r,N-1,0)}\cdot\overline{q^{\varpi_{2}(\mu,N-2,2)}}\cdot q^{\varpi_{3}(\nu^r,N-1,0)}\cdot q^{\varpi_{4}((\lambda^\prime)^r,N,0)}\cdot q^{\varpi_{5}(\mu,N,1)}.
\een

\end{lemma}
In this case, we have 
\ben
\frac{q^{\widetilde{\omega}_{\min}(\lambda^{rc},\mu,\nu;N)}\cdot q^{\widetilde{\omega}_{\min}(\lambda,\mu,\nu^{rc};N)}}{q^{\widetilde{\omega}_{\min}(\lambda,\mu,\nu;N)}\cdot q^{\widetilde{\omega}_{\min}(\lambda^{rc},\mu,\nu^{rc};N)}}=1
\een
and still require the following lemmas to derive the  graphical condensation recurrence \eqref{DT-recurrence2}.
\begin{lemma}\label{recurrence2-weight7}
For any partitions $\lambda,\nu\neq\emptyset$ and $\mu$, we have 
	\ben
	\frac{q^{\widetilde{\omega}_{\min}^{LU}(\lambda^{r},\mu,\nu^c;N)}\cdot q^{\widetilde{\omega}_{\min}^{RD}(\lambda^{c},\mu,\nu^{r};N)}}{q^{\widetilde{\omega}_{\min}(\lambda,\mu,\nu;N)}\cdot q^{\widetilde{\omega}_{\min}(\lambda^{rc},\mu,\nu^{rc};N)}}
	=q^{K_{2}(\lambda,\mu,\nu)}
	\een
where 
\ben
q^{K_{2}(\lambda,\mu,\nu)}:=q_{1-d(\lambda^\prime)}^{-\lambda^\prime_{d(\lambda^\prime)}}\cdot\prod\limits_{i=1}^{d(\lambda^\prime)-1}q_{-i}\cdot\prod\limits_{i=1}^{d(\lambda^\prime)-1}\left(\frac{q_{-i}}{q_{-i+1}}\right)^{\lambda^\prime_{i}}\cdot\prod\limits_{i=1}^{\ell(\mu)}\left(\frac{q_{i}}{q_{i-1}}\right)^{\mu_{i}}\cdot\prod\limits_{i=1}^{\nu_{d(\nu)}-d(\nu)}q_{i}^{-1}
\een
is independent of $N$.	
\end{lemma}
\begin{proof}
It follows from Lemmas \ref{recurrence1-weight1}, \ref{recurrence2-weight1},	\ref{recurrence2-weight2},		\ref{recurrence2-weight3},		\ref{recurrence2-weight4},		\ref{recurrence2-weight5},		\ref{recurrence2-weight6},	\ref{recurrence1-weight7}.
\end{proof}

The following lemmas follow from the similar argument in Section 4.3.1.

\begin{lemma}\label{recurrence2-weight3}
	\ben
	\frac{q^{\varpi_{1}(N,1,1)}\cdot q^{\varpi_{1}(N-2,-1,-1)}}{\left(q^{\varpi_{1}(N-1,0,0)}\right)^2}=q_{0}^{1-N}\cdot q_{N}^N\cdot\prod\limits_{i=1}^{N-1}q_{i}.
	\een
\end{lemma}

\begin{lemma}\label{recurrence2-weight4} 
For any partition $\eta\neq\emptyset$, we have
	\ben
	\frac{q^{\varpi_{2}(\eta^c,N-1,2)}q^{\varpi_{2}(\eta^r,N-1,0)}}{q^{\varpi_{2}(\eta,N-1,1)}q^{\varpi_{2}(\eta^{rc},N-1,1)}}
	=q_{1-d(\eta)}^{-\eta_{d(\eta)}}\cdot\prod\limits_{i=1}^{d(\eta)-1}\left(\frac{q_{-i}}{q_{-i+1}}\right)^{\eta_{i}}\cdot\frac{\prod\limits_{i=1}^{d(\eta)}q_{N-i+1}^{\eta_{i}-1}\cdot\prod\limits_{i=1}^{d(\eta)-1}q_{-i}}{\prod\limits_{i=1}^{d(\eta)-1}q_{N-i}^{\eta_{i}}\cdot\prod\limits_{i=1}^{N-1}q_{i}}.
	\een
\end{lemma}

\begin{lemma} \label{recurrence2-weight5}
For any partition $\eta$, we have
	\ben
	\frac{q^{\varpi_{2}(\eta,N,0)}q^{\varpi_{2}(\eta,N-2,2)}}{\left(q^{\varpi_{2}(\eta,N-1,1)}\right)^2}=\prod\limits_{i=1}^{\ell(\eta)}\left(\frac{q_{-i}}{q_{-i+1}}\right)^{\eta_{i}}.
	\een
\end{lemma}

\begin{lemma} \label{recurrence2-weight6}
	For any partition $\eta\neq\emptyset$, we have
	\ben
	\frac{q^{\varpi_{3}(\eta^c,N+1,2)}q^{\varpi_{3}(\eta^r,N-1,0)}}{q^{\varpi_{3}(\eta,N,1)}q^{\varpi_{3}(\eta^{rc},N,1)}}=q_{0}^{N-1}\cdot\prod\limits_{i=1}^{\eta_{d(\eta)}-d(\eta)}q_{i}^{-1}.
	\een
\end{lemma}

\subsubsection{Weights for the graphical condensation recurrence with $\mathbf{G}=\mathbf{H}(N;\lambda,\mu^{rc},\nu^{rc})$}
Again, by the similar argument as in Section 4.3.1 we have
\begin{lemma}\label{recurrence3-weight1}
The minimal dimer configuration of $\mathbf{H}(N;\lambda,\mu^{rc},\nu^{rc})-\{a,d\}$ has the weight\\
$q^{\omega_{\min}^{LD}(\lambda,\mu^r,\nu^c;N)}:=q^{\widetilde{\omega}_{\min}^{LD}(\lambda,\mu^r,\nu^c;N)}\cdot\prod\limits_{l=0}^{n-1}q_{l}^{-|\mathrm{II}(\lambda,\mu^r,\nu^c)|_{l}}\cdot q_{l}^{-2|\mathrm{III}(\lambda,\mu^r,\nu^c)|_{l}}$, where
\ben
&&q^{\widetilde{\omega}^{LD}_{\min}(\lambda,\mu^r,\nu^c;N)}\\
&=&\prod_{i=0}^{N-1}\prod_{j=0}^{N-2}q_{j-i+1}^{N-i-1}\cdot\prod_{i=1}^{\ell(\lambda^\prime)}\prod_{j=0}^{N-2}q_{j-i+2}^{\lambda^\prime_{i}}\cdot\prod_{i=1}^{\ell(\mu^r)}\prod_{j=0}^{N-1}q_{i-j}^{\mu^r_{i}}\cdot\prod_{i=1}^{\ell(\nu^c)}\prod_{j=0}^{\nu^c_{i}-1}q_{j-i+2}^{N}\cdot\prod_{i=1}^{\ell(\nu^c)}\prod_{j=0}^{\nu^c_{i}-1}q_{j-i+2}^{-2N+i}\\
&&\times\prod_{i:1\leq i\leq \lambda^\prime_{i}}\prod_{j=1}^{i-1}q_{N+j-i}^{N+\lambda^\prime_{i}-i}\cdot\prod_{i:\lambda^\prime_{i}<i\leq\ell(\lambda^\prime)}\prod_{j=1}^{\lambda^\prime_{i}}q_{N+j-i}^{N+\lambda^\prime_{i}-i}\\
&&\times\prod_{i:1\leq i\leq\mu^r_{i}}\prod_{j=1}^{i}q_{i-j-N+1}^{\mu^r_{i}-j}\cdot\prod_{i:\mu^r_{i}<i\leq\ell(\mu^r)}\prod_{j=1}^{\mu^r_{i}}q_{i-j-N+1}^{\mu^r_{i}-j}\\
&=&q^{\varpi_{1}(N-1,1,1)}\cdot q^{\varpi_{2}(\lambda^\prime,N-2,2)}\cdot\overline{q^{\varpi_{2}(\mu^r,N-1,0)}}\cdot q^{\varpi_{3}(\nu^c,N,2)}\cdot  q^{\varpi_{4}(\lambda^\prime,N,1)}\cdot q^{\varpi_{5}(\mu^r,N,0)}.
\een

\end{lemma}

\begin{lemma}\label{recurrence3-weight2}
The minimal dimer configuration of $\mathbf{H}(N;\lambda,\mu^{rc},\nu^{rc})-\{b,c\}$ has the weight\\
$q^{\omega_{\min}^{RU}(\lambda,\mu^c,\nu^r;N)}:=q^{\widetilde{\omega}_{\min}^{RU}}(\lambda,\mu^c,\nu^r;N)\cdot\prod\limits_{l=0}^{n-1}q_{l}^{-|\mathrm{II}(\lambda,\mu^c,\nu^r)|_{l}}\cdot q_{l}^{-2|\mathrm{III}(\lambda,\mu^c,\nu^r)|_{l}}$, where
\ben
&&q^{\widetilde{\omega}^{RU}_{\min}(\lambda,\mu^c,\nu^r;N)}\\
&=&\prod_{i=0}^{N-1}\prod_{j=0}^{N}q_{j-i-1}^{N-i-1}\cdot\prod_{i=1}^{\ell(\lambda^\prime)}\prod_{j=0}^{N}q_{j-i}^{\lambda^\prime_{i}}\cdot\prod_{i=1}^{\ell(\mu^c)}\prod_{j=0}^{N-1}q_{i-j-2}^{\mu^c_{i}}\cdot\prod_{i=1}^{\ell(\nu^r)}\prod_{j=0}^{\nu^r_{i}-1}q_{j-i}^{N-1}\cdot\prod_{i=1}^{\ell(\nu^r)}\prod_{j=0}^{\nu^r_{i}-1}q_{j-i}^{-2N+i+1}\\
&&\times\prod_{i:1\leq i\leq\lambda_{i}^\prime}\prod_{j=1}^{i-1}q^{N+\lambda_{i}^\prime-i}_{N-i+j}\cdot\prod_{i:\lambda_{i}^\prime<i\leq\ell(\lambda^\prime)}\prod_{j=1}^{\lambda_{i}^\prime}q_{N-i+j}^{N+\lambda_{i}^\prime-i}\\
&&\times\left(q_{-N}^{\mu^c_{1}}\right)^{-1}\cdot\prod_{i:2\leq i\leq \mu^c_{i}+1}\prod_{j=1}^{i-2}q_{i-j-N-1}^{\mu^c_{i}-j}\cdot\prod_{i:\mu^c_{i}+1<i\leq\ell(\mu^c)}\prod_{j=1}^{\mu^c_{i}}q_{i-j-N-1}^{\mu^c_{i}-j}\\
&=&q^{\varpi_{1}(N-1,-1,-1)}\cdot q^{\varpi_{2}(\lambda^\prime,N,0)}\cdot\overline{q^{\varpi_{2}(\mu^c,N-1,2)}}\cdot q^{\varpi_{3}(\nu^r,N,0)}\cdot q^{\varpi_{4}(\lambda^\prime,N,1)}\cdot q^{\varpi_{7}(\mu^c,N)}.
\een
\end{lemma}

Now we have
\ben
\frac{q^{\widetilde{\omega}_{\min}(\lambda,\mu^{rc},\nu;N)}\cdot q^{\widetilde{\omega}_{\min}(\lambda,\mu,\nu^{rc};N)}}{q^{\widetilde{\omega}_{\min}(\lambda,\mu,\nu;N)}\cdot q^{\widetilde{\omega}_{\min}(\lambda,\mu^{rc},\nu^{rc};N)}}=1
\een
and still demand the following lemmas to achieve the recurrence \eqref{DT-recurrence3}.
\begin{lemma}\label{recurrence3-weight5}
	For any partitions $\mu, \nu\neq\emptyset$ and $\lambda$, we have 
	\ben
	\frac{q^{\widetilde{\omega}_{\min}^{LD}(\lambda,\mu^{r},\nu^c;n)}\cdot q^{\widetilde{\omega}_{\min}^{RU}(\lambda,\mu^{c},\nu^{r};n)}}{q^{\widetilde{\omega}_{\min}(\lambda,\mu,\nu;n)}\cdot q^{\widetilde{\omega}_{\min}(\lambda,\mu^{rc},\nu^{rc};n)}}
	=q^{K_{3}(\lambda,\mu,\nu)}
	\een
where 
\ben
q^{K_{3}(\lambda,\mu,\nu)}:=q_{d(\mu)-1}^{-\mu_{d(\mu)}}\cdot \prod\limits_{i=1}^{d(\mu)-1}q_{i}\cdot\prod\limits_{i=1}^{d(\mu)-1}\left(\frac{q_{i}}{q_{i-1}}\right)^{\mu_{i}}\cdot\prod\limits_{i=1}^{\ell(\lambda^\prime)}\left(\frac{q_{-i}}{q_{-i+1}}\right)^{\lambda^\prime_{i}}\cdot \prod_{i=1}^{\widetilde{d}(\nu)-d(\nu)}q_{-i}^{-1}
\een
is independent of $N$.	
\end{lemma}

\begin{proof}
It follows from Lemmas \ref{recurrence1-weight1}, \ref{recurrence3-weight1},	\ref{recurrence3-weight2}, \ref{recurrence3-weight3},  \ref{recurrence2-weight5}, \ref{recurrence2-weight4}, \ref{recurrence3-weight4},  \ref{recurrence1-weight8}.	
\end{proof}

As in Section 4.3.2, we have
\begin{lemma}\label{recurrence3-weight3}
	\ben
	&&\frac{q^{\varpi_{1}(N-1,1,1)}\cdot q^{\varpi_{1}(N-1,-1,-1)}}{\left(q^{\varpi_{1}(N-1,0,0)}\right)^2}=q_{0}^{1-N}\cdot\prod_{i=1}^{N-1}q_{-i}.
	\een
\end{lemma}

\begin{lemma} \label{recurrence3-weight4}
	For any partition $\eta\neq\emptyset$, we have
	\ben
	\frac{q^{\varpi_{3}(\eta^c,N,2)}q^{\varpi_{3}(\eta^r,N,0)}}{q^{\varpi_{3}(\eta,N,1)}q^{\varpi_{3}(\eta^{rc},N,1)}}=q_{0}^N\cdot\prod_{i=0}^{\widetilde{d}(\eta)-d(\eta)}q_{-i}^{-1}.
	\een
\end{lemma}

\section{Orbifold PT theory and the graphical  condensation recurrence}
We recall the essential idea  in [\cite{JWY}, Section 4] that the PT vertex can be alternatively expressed as the sum over certain subset of AB configurations, which is available for deriving the graphical condensation recurrence in double-dimer model [\cite{Jenne}] for PT theory.  In this section, we extend this idea  to the PT $\mathbb{Z}_{n}$-vertex with the similar double-dimer model description, and obtain three   recurrences as in the orbifold DT case one of which generalizes the recurrence in [\cite{JWY}, Section 4.5]. 

\subsection{Labelled box configurations and AB configurations for orbifold PT theory}
We will briefly review the relation between labelled box configurations with AB configurations in [\cite{JWY}, Section 4.2], and then show the  PT $\mathbb{Z}_{n}$-vertex as a sum indexed by some AB configurations.  In the following we will employ notations $\mathrm{I}^-, \mathrm{II}, \mathrm{III}$   defined with respect to  $(\lambda,\mu,\nu)$ in Section 3.
\begin{definition}([\cite{JWY}, Definition 4.2.1 and Conditions 4.2.2, 4.2.3])
\\$(1)$ An AB configuration is a  pair $(A,B)$ of finite sets where $A\subseteq\mathrm{I}^-\cup\mathrm{III}$ and $B\subseteq\mathrm{II}\cup\mathrm{III}$ has elements called boxes, satisfying the following conditions \\
$(i)$ if $w\in\mathrm{I}^-\cup\mathrm{III}$ and any of 
\ben
(w_{1}-1,w_{2},w_{3}),\;\;(w_{1},w_{2}-1,w_{3}),\;\;(w_{1},w_{2},w_{3}-1)
\een
is a box in $A$, then $w\in A$.\\
$(ii)$ if $w\in\mathrm{II}\cup\mathrm{III}$ and any of 
\ben
(w_{1}-1,w_{2},w_{3}),\;\;(w_{1},w_{2}-1,w_{3}),\;\;(w_{1},w_{2},w_{3}-1)
\een
is a box in $B$, then $w\in B$.\\
We denote by $\mathcal{AB}_{all}$ the set of all AB configurations.\\
$(2)$ Given a labelled box configuration $\widetilde{\pi}$ with outgoing partitions $(\lambda,\mu,\nu)$. If  an AB  configuration $(A,B)$ satisfies 
$A\cup B=\widetilde{\pi}$  and $A\cap B=\mathrm{III}_{ulb}(\widetilde{\pi})$ as the same sets of boxes in $\widetilde{\pi}$, then we call $(A,B)$ an AB configuration on $\widetilde{\pi}$. 
We denote by $\mathcal{AB}(\widetilde{\pi})$ the set of all AB configurations on $\widetilde{\pi}$.
\end{definition}

Let $\widehat{\mathcal{P}}(\lambda,\mu,\nu)$ be the set of all labelled box configurations with outgoing partitions $(\lambda,\mu,\nu)$. Given  $\widetilde{\pi}\in\widehat{\mathcal{P}}(\lambda,\mu,\nu)$, one can construct the base AB configuration $AB_{base}(\widetilde{\pi})=(A,B)$ in [\cite{JWY}, Definition 4.2.4] where $A:=\mathrm{I}^-(\widetilde{\pi})\cup\mathrm{III}(\widetilde{\pi})$  and $B:=\mathrm{II}(\widetilde{\pi})\cup\mathrm{III}_{ulb}(\widetilde{\pi})$. It is shown in [\cite{JWY}, Section 4.2.1] that $AB_{base}(\widetilde{\pi})\in\mathcal{AB}(\widetilde{\pi})$. Conversely, $(A,B)\in\mathcal{AB}_{all}$ is called a labelled AB configuration if the algorithm [\cite{JWY}, Algorithm 4.2.13] succeeds in labelling connected components of $\mathcal{L}(A,B):=(\mathrm{I}^-\cap A)\cup(\mathrm{II}\setminus B)\cup(\mathrm{III}\cap(A\Delta B))$. For a labelled AB configuration $(A,B)$, the labelled boxes $\widetilde{\pi}(A,B)$ is defined in [\cite{JWY}, Definition 4.2.22].
Due to [\cite{JWY}, Lemma 4.2.24], we define 
\ben
\mathcal{AB}:=\bigcup_{\widetilde{\pi}\in\widehat{\mathcal{P}}(\lambda,\mu,\nu)}\mathcal{AB}(\widetilde{\pi}).
\een
Then $(A,B)\in\mathcal{AB}$ if and only if $(A,B)$ is a labelled AB configuration by [\cite{JWY}, Theorem 4.2.26] and   $\widetilde{\pi}(A,B)$ is a labelled box configuration if $(A,B)\in\mathcal{AB}$ by [\cite{JWY}, Corollary 4.2.27].
Recall that $\widetilde{\mathcal{P}}(\lambda,\mu,\nu)$ is the set of components of the moduli space of  labelled box configurations with outgoing partitions $(\lambda,\mu,\nu)$ in Section 3. 
By [\cite{JWY}, Corollary 4.2.31], if $[\widetilde{\pi}_{1}], [\widetilde{\pi}_{2}]\in \widetilde{\mathcal{P}}(\lambda,\mu,\nu)$ and $[\widetilde{\pi}_{1}]\neq[\widetilde{\pi}_{2}]$, then $\mathcal{AB}(\widetilde{\pi}_{1})\cap\mathcal{AB}(\widetilde{\pi}_{2})=\emptyset$. For any element  $\widetilde{\pi}\in\widehat{\mathcal{P}}(\lambda,\mu,\nu)$ in one given component in $\widetilde{\mathcal{P}}(\lambda,\mu,\nu)$, the sets $\mathcal{AB}(\widetilde{\pi})$ are the same by the definition.  Then 
\ben
\mathcal{AB}:=\bigcup_{[\widetilde{\pi}]\in\widetilde{\mathcal{P}}(\lambda,\mu,\nu)}\mathcal{AB}(\widetilde{\pi}).
\een
And it is shown in [\cite{JWY}, Corollary 4.2.34] that $|\mathcal{AB}(\widetilde{\pi})|=\chi_{\mathrm{top}}([\widetilde{\pi}])$
where $[\widetilde{\pi}]\in\widetilde{\mathcal{P}}(\lambda,\mu,\nu)$ is the component of the moduli space of labellings of $\underline{\widetilde{\pi}}$ containing $\widetilde{\pi}$.  The orbifold version of [\cite{JWY}, Definition 4.2.35] is presented as follows.
\begin{definition}
\ben
Z_{\mathcal{AB}}(q_{0},q_{1},\cdots,q_{n-1})=\sum_{(A,B)\in\mathcal{AB}}\prod_{l=0}^{n-1}q_{l}^{|A|_{l}+|B|_{l}-|\mathrm{II}|_{l}-2|\mathrm{III}|_{l}}.
\een
\end{definition}
This definition  is an alternative expression  of the orbifold PT topological vertex $W^n_{\lambda\mu\nu}$, i.e., we have the following orbifold version of [\cite{JWY}, Theorem 4.2.36].
\begin{theorem}\label{unlabel-formula}
\ben
W^n_{\lambda\mu\nu}(q_{0},q_{1},\cdots,q_{n-1})=Z_{\mathcal{AB}}(q_{0},q_{1},\cdots,q_{n-1}).
\een
\end{theorem}
\begin{proof}
With the discussion above, we have  	
\ben
Z_{\mathcal{AB}}(q_{0},q_{1},\cdots,q_{n-1})&=&\sum_{(A,B)\in\mathcal{AB}}\prod_{l=0}^{n-1}q_{l}^{|A|_{l}+|B|_{l}-|\mathrm{II}|_{l}-2|\mathrm{III}|_{l}}\\
&=&\sum_{[\widetilde{\pi}]\in\widetilde{\mathcal{P}}(\lambda,\mu,\nu)}\sum_{(A,B)\in\mathcal{AB}(\widetilde{\pi})}\prod_{l=0}^{n-1}q_{l}^{|A\cup B|_{l}+|A\cap B|_{l}-|\mathrm{II}|_{l}-2|\mathrm{III}|_{l}}\\
&=&\sum_{[\widetilde{\pi}]\in\widetilde{\mathcal{P}}(\lambda,\mu,\nu)}\sum_{(A,B)\in\mathcal{AB}(\widetilde{\pi})}\prod_{l=0}^{n-1}q_{l}^{\ell(\widetilde{\pi})_{l}+\Vert\pi_{\mathrm{min}}(\lambda,\mu,\nu)\Vert_{l}}\\
&=&\sum_{[\widetilde{\pi}]\in\widetilde{\mathcal{P}}(\lambda,\mu,\nu)}|\mathcal{AB}(\widetilde{\pi})|\prod_{l=0}^{n-1}q_{l}^{\ell(\widetilde{\pi})_{l}+\Vert\pi_{\mathrm{min}}(\lambda,\mu,\nu)\Vert_{l}}\\
&=&\sum_{[\widetilde{\pi}]\in\widetilde{\mathcal{P}}(\lambda,\mu,\nu)}\chi_{\mathrm{top}}([\widetilde{\pi}])\prod_{l=0}^{n-1}q_{l}^{\ell(\widetilde{\pi})_{l}+\Vert\pi_{\mathrm{min}}(\lambda,\mu,\nu)\Vert_{l}}\\
&=&W^{n}_{\lambda\mu\nu}(q_{0},q_{1},\cdots,q_{n-1}).
\een
\end{proof}

\subsection{The  double-dimer model for orbifold PT theory}
We will briefly review the (labelled) double-dimer model introduced in [\cite{JWY}, Section 4.3 and 4.4], which is also valid for describing the  PT $\mathbb{Z}_{n}$-vertex. 
Given an AB configuration $(A,B)$, define
\ben
\mathfrak{A}=R_{2}\cup(\mathrm{I}^-\cup\mathrm{III})\setminus A\;\;\;\;\mbox{and}\;\;\;\;  \mathfrak{B}=R_{1}\cup(\mathrm{II}\cup\mathrm{III})\setminus B
\een
where $R_{1}$ (resp. $R_{2}$) is the subset of $\mathbb{Z}^3$ whose elements have at least one negative coordinate (resp. two negative coordinates). One can view $\mathfrak{A}$ and $\mathfrak{B}$ as  rhombus tilings of the plane due to [\cite{JWY}, Lemma 4.3.3], which are corresponding to dimer configurations $\mathbf{D}_{A}$ and $\mathbf{D}_{B}$ of the infinite honeycomb graph $\mathbf{H}$. A double-dimer configuration is obtained from the union of two dimer configurations. Let $\mathbf{D}_{(A,B)}$ be the double-dimer configuration on $\mathbf{H}$ by superimposing $\mathbf{D}_{A}$ and $\mathbf{D}_{B}$ with their origins in $\mathbb{Z}^3$ corresponding to the same face of $\mathbf{H}$. Then $\mathbf{D}_{(A,B)}$ is composed of loops, double edges  and infinite paths.  One can refer to [\cite{JWY}, Figure 9 and 10] for examples. Given an AB configuration $(A,B)$, there is a labelling algorithm  [\cite{JWY}, Algorithm 4.3.11] for $\mathbf{D}_{(A,B)}$, which is shown to be equivalent to [\cite{JWY}, Algorithm 4.2.13] in the sense that $(A,B)\in\mathcal{AB}$ if and only if Algorithm 4.3.11 in  [\cite{JWY}] succeeds (see [\cite{JWY}, Theorem 4.4.10]), and for a given $(A,B)\in\mathcal{AB}$, the labelling in [\cite{JWY}, Algorithm 4.2.13] assigned for some cell is the same as the one in [\cite{JWY}, Algorithm 4.3.11] for the corresponding face (see [\cite{JWY}, Theorem 4.4.11]). And it is verified in [\cite{JWY}, Theorem 4.4.12] that $(A,B)\in\mathcal{AB}$ if and only if both ends of each path in $\mathbf{D}_{(A,B)}$ are contained in sector $i$ for some $i\in\{1,2,3\}$. To describe the orbifold PT vertex and  compare with the orbifold DT case, one need to truncate $\mathbf{D}_{(A,B)}$ on $\mathbf{H}$ to be the one on $\mathbf{H}(N)$. The following definition is useful for describing the truncation.

\begin{definition}([\cite{JWY}, Definition 4.4.13 and 4.4.15])
Let $\mathbf{G}=(\mathbf{V}_{1}, \mathbf{V}_{2}, \mathbf{E})$ be a finite edge-weighted bipartite planar graph with $|\mathbf{V}_{1}|=|\mathbf{V}_{2}|$. Let  $\mathbf{N}$ be a set of special vertices called nodes on the outer face of $\mathbf{G}$.
A double-dimer configuration on $(\mathbf{G},\mathbf{N})$ is a multiset of the edges of $\mathbf{G}$ such that each internal vertex is the endpoint of exactly two edges and each vertex in $\mathbf{N}$ is the endpoint of exactly one edge. A planar pairing $\sigma$ is tripartite if the nodes can be divided into three circularly contiguous sets $R$, $G$ and $B$ (the nodes in $R$, $G$ and $B$ are often colored by red, green and blue respectively) such that no node is paired with a node in the same set.
\end{definition}

For $N$ large enough, it is shown  that $\mathbf{D}_{B}$ should be truncated to be a perfect matching $\mathbf{D}_{B}(N)$ of $\mathbf{H}(N)$ by [\cite{JWY}, Lemma 4.4.14] and $\mathbf{D}_{A}$ could be truncated to be a partial matching $\mathbf{D}_{A}(N)$ of $\mathbf{H}(N)$, which implies that $\mathbf{D}_{(A,B)}$ can be truncated to be a  double-dimer configuration $\mathbf{D}_{(A,B)}(N)$ on $\mathbf{H}(N)$ (see [\cite{JWY}, Figure 11] for example) with nodes. By [\cite{JWY}, Theorem 4.4.17],  if $(A,B)\in\mathcal{AB}$ and $N$ is large enough, the double-dimer configuration $\mathbf{D}_{(A,B)}(N)$ is tripartite with nodes explicitly described  as follows.  Let $S_{1}=S(\lambda)$, $S_{2}=S(\mu)$ and $S_{3}=S(\nu)$ be the Maya diagrams of $\lambda, \mu, \nu$ respectively.  Define $\mathbf{N}_{i}^+(N)$ (resp. $\mathbf{N}_{i}^-(N)$) to be the set of vertices on the outer face of $\mathbf{H}(N)$ in sector $i^+$ (resp. in sector $i^-$) which are  not labelled by any of elements of $S_{i}^+$ (resp. $S_{i}^-$). Set $\mathbf{N}_{\lambda,\mu,\nu}(N)=\bigcup\limits_{i=1}^3\mathbf{N}_{i}^+(N)\cup\mathbf{N}_{i}^-(N)$. It is proved in [\cite{JWY}, Corollary 4.4.22] that the set of nodes of $\mathbf{D}_{(A,B)}(N)$ in sector $i$ is $\mathbf{N}_{i}=\mathbf{N}_{i}^+(N)\cup\mathbf{N}_{i}^-(N)$. We have the following color scheme  [\cite{JWY}] for the nodes of $\mathbf{D}_{(A,B)}(N)$. In sector 1, the nodes in $N_{1}^+$ are colored by blue and the nodes in $N_{1}^-$ are colored by red; in sector 2, the nodes in $N_{2}^+$ are colored by red and the nodes in  $N_{2}^-$ are colored by green; in sector 3, the nodes in  $N_{3}^+$ are colored by green and the nodes in  $N_{3}^-$ are colored by blue. Since $|\mathbf{N}_{i}^+(N)|=|\mathbf{N}_{i}^-(N)|$, one can equip nodes $\mathbf{N}_{\lambda,\mu,\nu}(N)$ with the following tripartite rainbow pairing.

\begin{definition}([\cite{JWY}, Section 4.4])
Suppose there are $2r$ nodes in sector $i$  numbered consecutively in clockwise order by $1, 2, \cdots, 2r$, which are vertices on the outer face of $\mathbf{H}(N)$. If $r>0$, we call the pairing $((1,2r),(2,2r-1),\cdots,(r,r+1))$ the rainbow pairing of the nodes in sector $i$. If $r=0$, we call the empty pairing as the rainbow pairing of the nodes in sector $i$. Denote the nodes by $\mathbf{N}$. If the nodes in sector $i$ are paired according to the rainbow pairing for every $1\leq i\leq3$, we call the resulting pairing of $\mathbf{N}$ the rainbow pairing of $\mathbf{N}$.
\end{definition}

Denote by $\mathbf{D}_{\sigma}(\mathbf{G},\mathbf{N})$ the set of all double-dimer configurations on $\mathbf{G}$ with nodes $\mathbf{N}$ subject to the rainbow pairing $\sigma$. Let $\mathbf{Z}^{DD}_{\sigma}(\mathbf{G},\mathbf{N})$ be the sum of all the weights of double-dimer configurations in $\mathbf{D}_{\sigma}(\mathbf{G},\mathbf{N})$, where the weight of a $k$-loop double-dimer configuration is $2^k$ times its edge-weight defined by the product of weights of all its edges.

It is shown before [\cite{JWY}, Definition 4.4.23] that  a given AB configuration $(A,B)$ can be consecutively changed into an AB configuration $(\mathrm{III},\mathrm{II}\cup\mathrm{III})\in\mathcal{AB}(\widetilde{\pi})$ by a sequence of  removing or adding a single box, where $\widetilde{\pi}$ is some labelled box configuration which is composed of type $\mathrm{II}$ boxes  and  type $\mathrm{III}$ boxes which are unlabelled. If an AB configuration $(A^\prime, B^\prime)$ is derived from an AB configuration $(A,B)$ by adding (resp.  removing) a box to (resp. from) $A$ or $B$ and all the boxes in $A^\prime\cup B^\prime$ and $A\cup B$ are corresponding to faces of $\mathbf{H}(N)$, then $\mathbf{D}_{A^\prime}(N)$ or $\mathbf{D}_{B^\prime}(N)$ is obtained from $\mathbf{D}_{A}(N)$ or $\mathbf{D}_{B}(N)$ by a local move shown in [\cite{JWY}, Figure 13] (resp. [\cite{JWY}, Figure 12]). According to the weighted rule in Definition \ref{weight-rule}, removing  a box $(i,j,k)$  will increase the edge-weight  by a factor of $q_{i-j}$ and adding  a box $(i,j,k)$ will decrease the edge-weight by a factor of $q_{i-j}$. Let $q^{\omega_{(A,B)}(N)}$ be the edge-weight of $\mathbf{D}_{(A,B)}(N)$. It is shown in [\cite{JWY}] that for $N$ large enough, we have $\mathbf{D}_{(\mathrm{III},\mathrm{II}\cup\mathrm{III})}(N)\in\mathbf{D}_{\sigma}(\mathbf{H}(N),\mathbf{N}_{\lambda,\mu,\nu}(N))$, which is called the base double-dimer configuration. Denote by $q^{\omega_{base}(\lambda,\mu,\nu;N)}$ the  edge-weight of $\mathbf{D}_{(\mathrm{III},\mathrm{II}\cup\mathrm{III})}(N)$.
Then we have
\ben
q^{w_{(A,B)}(N)}\cdot\prod_{l=0}^{n-1}q_{l}^{|A|_{l}+|B|_{l}}=q^{w_{base}(\lambda,\mu,\nu;N)}\cdot\prod_{l=0}^{n-1}q_{l}^{|\mathrm{II}|_{l}+2|\mathrm{III}|_{l}}
\een
if all of the boxes in $A\cup B$ correspond to faces of $\mathbf{H}(N)$.
For $N$ large enough, one can show as in [\cite{JWY}] that if $|A|+|B|\leq\frac{N}{2}$ and $(A,B)\in\mathcal{AB}$, we have $\mathbf{D}_{(A,B)}(N)\in\mathbf{D}_{\sigma}(\mathbf{H}(N),\mathbf{N}_{\lambda,\mu,\nu}(N))$. Then 
\ben
&&Z_{\mathcal{AB}}(q_{0}^{-1},q_{1}^{-1},\cdots,q_{n-1}^{-1})\\
&=&\sum_{(A,B)\in\mathcal{AB}}\prod_{l=0}^{n-1}q_{l}^{-|A|_{l}-|B|_{l}+|\mathrm{II}|_{l}+2|\mathrm{III}|_{l}}\\
&=&\sum_{\substack{(A,B)\in\mathcal{AB}\\|A|+|B|\leq\frac{N}{2}}}\prod_{l=0}^{n-1}q_{l}^{-|A|_{l}-|B|_{l}+|\mathrm{II}|_{l}+2|\mathrm{III}|_{l}}+\sum_{\substack{(A,B)\in\mathcal{AB}\\|A|+|B|>\frac{N}{2}}}\prod_{l=0}^{n-1}q_{l}^{-|A|_{l}-|B|_{l}+|\mathrm{II}|_{l}+2|\mathrm{III}|_{l}}\\
&=&\left(q^{\omega_{base}(\lambda,\mu,\nu;N)}\right)^{-1}\cdot\sum_{\substack{(A,B)\in\mathcal{AB}\\|A|+|B|\leq\frac{N}{2}}}q^{\omega_{(A,B)}(N)}+\sum_{\substack{(A,B)\in\mathcal{AB}\\|A|+|B|>\frac{N}{2}}}\prod_{l=0}^{n-1}q_{l}^{-|A|_{l}-|B|_{l}+|\mathrm{II}|_{l}+2|\mathrm{III}|_{l}}.
\een
The similar argument as in [\cite{JWY}] shows that
\ben
&&\left(q^{\omega_{base}(\lambda,\mu,\nu;N)}\right)^{-1}\cdot\mathbf{Z}^{DD}_{\sigma}(\mathbf{H}(N),\mathbf{N}_{\lambda,\mu,\nu}(N))\\
&=&\left(q^{\omega_{base}(\lambda,\mu,\nu;N)}\right)^{-1}\cdot\sum_{\substack{(A,B)\in\mathcal{AB}\\|A|+|B|\leq\frac{N}{2}}}q^{\omega_{(A,B)}(N)}+\sum_{\substack{(A,B)\in\tau^{-1}(\mathbf{D}_{\sigma}(\mathbf{H}(N),\mathbf{N}_{\lambda,\mu,\nu}(N)))\\|A|+|B|>\frac{N}{2}}}\prod_{l=0}^{n-1}q_{l}^{-|A|_{l}-|B|_{l}+|\mathrm{II}|_{l}+2|\mathrm{III}|_{l}}
\een
where $\tau:\mathcal{AB}_{all}\to\mathbf{D}(\mathbf{H}(N))$ maps any AB configuration $(A,B)$ to $\mathbf{D}_{(A,B)}(N)$. Here, $\mathbf{D}(\mathbf{H}(N))$ is the set of all double-dimer configurations with nodes on $\mathbf{H}(N)$ and $\tau^{-1}(\mathbf{D}_{\sigma}(\mathbf{H}(N),\mathbf{N}_{\lambda,\mu,\nu}(N)))\subseteq\mathcal{AB}$ as shown in [\cite{JWY}].
Now,  by Theorem \ref{unlabel-formula}, we have the following orbifold version of [\cite{JWY}, Theorem 4.4.24].
\begin{theorem}\label{orbifold-PT-vertex}
\ben
\lim_{N\to\infty}\widetilde{\mathbf{Z}}^{DD}_{\sigma}(\mathbf{H}(N),\mathbf{N}_{\lambda,\mu,\nu}(N))=W^{n}_{\lambda\mu\nu}(q_{0}^{-1},q_{1}^{-1},\cdots,q_{n-1}^{-1})
\een
where 
\ben
\widetilde{\mathbf{Z}}^{DD}_{\sigma}(\mathbf{H}(N),\mathbf{N}_{\lambda,\mu,\nu}(N)):=\left(q^{\omega_{base}(\lambda,\mu,\nu;N)}\right)^{-1}\cdot\mathbf{Z}^{DD}_{\sigma}(\mathbf{H}(N),\mathbf{N}_{\lambda,\mu,\nu}(N)).
\een
\end{theorem}

\subsection{The graphical condensation recurrence for orbifold PT theory}
In [\cite{Jenne}], the author has developed an analogue version of [\cite{Kuo}, Theorem 5.1] for the double-dimer case as follows.

\begin{theorem}(See [\cite{Jenne}, Theorem 1.0.2] or [\cite{JWY}, Theorem 4.5.1]) \label{graphical-condensation2} Let $\mathbf{G}=(\mathbf{V}_{1},\mathbf{V}_{2},\mathbf{E})$ be a finite edge-weighted planar bipartitie graph with a set of nodes $\mathbf{N}$. Divide the nodes into three circularly contiguous sets $R$, $G$, $B$ such that $|R|$, $|G|$ and $|B|$ satisfy the triangle inequality and let $\sigma$ be the corresponding tripartite pairing. Let $a, b, c, d$ be nodes appearing in a cyclic order such that the set $\{a,b,c,d\}$ contains at least one node of each RGB color. If $a,c\in \mathbf{V}_{1}$ and $b,d\in\mathbf{V}_{2}$. Then 
\ben
\mathbf{Z}_{\sigma}^{DD}(\mathbf{G},\mathbf{N})\mathbf{Z}_{\sigma_{abcd}}^{DD}(\mathbf{G},\mathbf{N}-\{a,b,c,d\})&=&\mathbf{Z}_{\sigma_{ab}}^{DD}(\mathbf{G},\mathbf{N}-\{a,b\})\mathbf{Z}^{DD}_{\sigma_{cd}}(\mathbf{G},\mathbf{N}-\{c,d\})\\
&&+\mathbf{Z}_{\sigma_{ad}}^{DD}(\mathbf{G},\mathbf{N}-\{a,d\})\mathbf{Z}^{DD}_{\sigma_{bc}}(\mathbf{G},\mathbf{N}-\{b,c\})
\een
where $\sigma_{abcd}$ is the unique planar pairing on $\mathbf{N}-\{a,b,c,d\}$ in which like RGB colors are not paired together, and for $i,j\in\{a,b,c,d\}$, $\sigma_{ij}$ is the unique planar pairing on $\mathbf{N}-\{i,j\}$ in which like RGB colors are not paired together.
\end{theorem}
As in [\cite{JWY}], we choose $\mathbf{G}=\mathbf{H}(N)$, $\mathbf{N}=\mathbf{N}_{\lambda^{rc},\mu^{rc},\nu}(N)$ and let 
\ben
a:=\max S_{1}^-,\;\;\;b:=\min S_{1}^+,\;\;\;c=\max S_{2}^-,\;\;\;d=\min S_{2}^+.
\een
Then
\ben
&&(\mathbf{H}(N),\mathbf{N}_{\lambda,\mu,\nu}(N))=(\mathbf{H}(N),\mathbf{N}_{\lambda^{rc},\mu^{rc},\nu}(N)-\{a,b,c,d\});\\
&&(\mathbf{H}(N),\mathbf{N}_{\lambda,\mu^{rc},\nu}(N))=(\mathbf{H}(N),\mathbf{N}_{\lambda^{rc},\mu^{rc},\nu}(N)-\{a,b\});\\
&&(\mathbf{H}(N),\mathbf{N}_{\lambda^{rc},\mu,\nu}(N))=(\mathbf{H}(N),\mathbf{N}_{\lambda^{rc},\mu^{rc},\nu}(N)-\{c,d\}).
\een

Now we recall the important description of two double-dimer configurations on $\mathbf{H}(N)$ with nodes $\mathbf{N}-\{a,d\}$ and $\mathbf{N}-\{b,c\}$ in [\cite{JWY}, Section 4.5] as follows. Since $\mathbf{N}-\{a,d\}=\mathbf{N}_{\lambda,\mu,\nu}(N)\cup\{b,c\}$ and $\mathbf{N}-\{b,c\}=\mathbf{N}_{\lambda,\mu,\nu}(N)\cup\{a,d\}$, the pairing $\sigma_{ad}$ on $\mathbf{N}-\{a,d\}$ has one more blue-green path connecting a blue node $b$ in sector 1 and a green node $c$ in sector 2 while the pairing $\sigma_{bc}$ has one fewer blue-green path but one more red-green path and one more red-blue path  than $\sigma_{abcd}$. The double-dimer configuration on $\mathbf{H}(N)$ with nodes $\mathbf{N}-\{a,d\}$ is obtained from an AB configuration $(A,B)$ with respect to $(\lambda^r,\mu^c,\nu)$, 
but is not equal to $\mathbf{D}_{(A,B)}(N)$. This dimer configuration and the corresponding tilings are shifted directly up by one unit prior to the  truncation. We call it  the base$_{up}$ double-dimer configuration and denote its edge-weight by $q^{\omega^{U}_{base}(\lambda^r,\mu^c,\nu;N)}$. Similarly, the double-dimer configuration on $\mathbf{H}(N)$ with nodes $\mathbf{N}-\{b,c\}$ is derived from an AB configuration  with respect to $(\lambda^c,\mu^r,\nu)$, and this  configuration and its corresponding tilings are shifed directly down by one unit prior to the truncation.  And we call it  the base$_{down}$ double-dimer configuration and denote its edge-weight by $q^{\omega^{D}_{base}(\lambda^c,\mu^r,\nu;N)}$. See  [\cite{JWY}, Example 4.5.2] for  illustrations. 
Let 
\ben
&&\widetilde{\mathbf{Z}}_{\sigma}^{DD}(\mathbf{H}(N),\mathbf{N}_{\lambda^{rc},\mu^{rc},\nu}(N)):=\left(q^{\omega_{base}(\lambda^{rc},\mu^{rc},\nu;N)}\right)^{-1}\mathbf{Z}_{\sigma}^{DD}(\mathbf{H}(N),\mathbf{N}_{\lambda^{rc},\mu^{rc},\nu}(N)),\\
&&\widetilde{\mathbf{Z}}_{\sigma_{abcd}}^{DD}(\mathbf{H}(N),\mathbf{N}_{\lambda,\mu,\nu}(N)):=\left(q^{\omega_{base}(\lambda,\mu,\nu;N)}\right)^{-1}\mathbf{Z}_{\sigma_{abcd}}^{DD}(\mathbf{H}(N),\mathbf{N}_{\lambda,\mu,\nu}(N)),\\
&&\widetilde{\mathbf{Z}}_{\sigma_{ab}}^{DD}(\mathbf{H}(N),\mathbf{N}_{\lambda,\mu^{rc},\nu}(N)):=\left(q^{\omega_{base}(\lambda,\mu^{rc},\nu;N)}\right)^{-1}\mathbf{Z}_{\sigma_{ab}}^{DD}(\mathbf{H}(N),\mathbf{N}_{\lambda,\mu^{rc},\nu}(N)),\\
&&\widetilde{\mathbf{Z}}_{\sigma_{cd}}^{DD}(\mathbf{H}(N),\mathbf{N}_{\lambda^{rc},\mu,\nu}(N)):=\left(q^{\omega_{base}(\lambda^{rc},\mu,\nu;N)}\right)^{-1}\mathbf{Z}_{\sigma_{cd}}^{DD}(\mathbf{H}(N),\mathbf{N}_{\lambda^{rc},\mu,\nu}(N)),\\
&&\widetilde{\mathbf{Z}}_{\sigma_{ad}}^{DD}(\mathbf{H}(N),\mathbf{N}_{\lambda^{rc},\mu^{rc},\nu}(N)-\{a,d\}):=\left(q^{\omega^{U}_{base}(\lambda^{r},\mu^{c},\nu;N)}\right)^{-1}\mathbf{Z}_{\sigma_{ad}}^{DD}(\mathbf{H}(N),\mathbf{N}_{\lambda^{rc},\mu^{rc},\nu}(N)-\{a,d\}),\\
&&\widetilde{\mathbf{Z}}_{\sigma_{bc}}^{DD}(\mathbf{H}(N),\mathbf{N}_{\lambda^{rc},\mu^{rc},\nu}(N)-\{b,c\}):=\left(q^{\omega^{D}_{base}(\lambda^{c},\mu^{r},\nu;N)}\right)^{-1}\mathbf{Z}_{\sigma_{bc}}^{DD}(\mathbf{H}(N),\mathbf{N}_{\lambda^{rc},\mu^{rc},\nu}(N)-\{b,c\}).
\een
Now, by Theorem \ref{graphical-condensation2}, we have 
\ben
&&\widetilde{\mathbf{Z}}_{\sigma}^{DD}(\mathbf{H}(N),\mathbf{N}_{\lambda^{rc},\mu^{rc},\nu}(N))\widetilde{\mathbf{Z}}_{\sigma_{abcd}}^{DD}(\mathbf{H}(N),\mathbf{N}_{\lambda,\mu,\nu}(N))\\
&=&\frac{q^{\omega_{base}(\lambda,\mu^{rc},\nu;N)}\cdot q^{\omega_{base}(\lambda^{rc},\mu,\nu;N)} }{q^{\omega_{base}(\lambda^{rc},\mu^{rc},\nu;N)}\cdot q^{\omega_{base}(\lambda,\mu,\nu;N)} }\cdot\widetilde{\mathbf{Z}}_{\sigma_{ab}}^{DD}(\mathbf{H}(N),\mathbf{N}_{\lambda,\mu^{rc},\nu}(N))\widetilde{\mathbf{Z}}_{\sigma_{cd}}^{DD}(\mathbf{H}(N),\mathbf{N}_{\lambda^{rc},\mu,\nu}(N))\\
&+&\frac{q^{\omega^{U}_{base}(\lambda^{r},\mu^{c},\nu;N)}\cdot q^{\omega^{D}_{base}(\lambda^{c},\mu^{r},\nu;N)} }{q^{\omega_{base}(\lambda^{rc},\mu^{rc},\nu;N)}\cdot q^{\omega_{base}(\lambda,\mu,\nu;N)}}\cdot\widetilde{\mathbf{Z}}_{\sigma_{ad}}^{DD}(\mathbf{H}(N),\mathbf{N}_{\lambda^{rc},\mu^{rc},\nu}(N)-\{a,d\})\widetilde{\mathbf{Z}}_{\sigma_{bc}}^{DD}(\mathbf{H}(N),\mathbf{N}_{\lambda^{rc},\mu^{rc},\nu}(N)-\{b,c\}).
\een
It is calculated in Section 5.4.1 that 
\ben
\frac{q^{\omega_{base}(\lambda,\mu^{rc},\nu;N)}\cdot q^{\omega_{base}(\lambda^{rc},\mu,\nu;N)} }{q^{\omega_{base}(\lambda^{rc},\mu^{rc},\nu;N)}\cdot q^{\omega_{base}(\lambda,\mu,\nu;N)}}=1\;\;\;\;\;\;\mbox{and}\;\;\;\;\;\;\frac{q^{\omega^{U}_{base}(\lambda^{r},\mu^{c},\nu;N)}\cdot q^{\omega^{D}_{base}(\lambda^{c},\mu^{r},\nu;N)} }{q^{\omega_{base}(\lambda^{rc},\mu^{rc},\nu;N)}\cdot q^{\omega_{base}(\lambda,\mu,\nu;N)}}=q^{-K_{1}(\lambda,\mu,\nu)}
\een
where $q^{-K_{1}(\lambda,\mu,\nu)}$ defined in Lemma \ref{recurrence4-weight8} is independent of $N$.
Since
\ben
&&\lim_{N\to\infty}\widetilde{\mathbf{Z}}_{\sigma_{ad}}^{DD}(\mathbf{H}(N),\mathbf{N}_{\lambda^{rc},\mu^{rc},\nu}(N)-\{a,d\})=W^{n}_{\lambda^r\mu^c\nu}(q_{0}^{-1},q_{1}^{-1},\cdots,q_{n-1}^{-1}),\\
&&\lim_{N\to\infty}\widetilde{\mathbf{Z}}_{\sigma_{bc}}^{DD}(\mathbf{H}(N),\mathbf{N}_{\lambda^{rc},\mu^{rc},\nu}(N)-\{b,c\})=W^{n}_{\lambda^c\mu^r\nu}(q_{0}^{-1},q_{1}^{-1},\cdots,q_{n-1}^{-1})
\een
by Theorem \ref{orbifold-PT-vertex}, let $N\to\infty$, we have
\ben
&&W^{n}_{\lambda\mu\nu}(q_{0}^{-1},q_{1}^{-1},\cdots,q_{n-1}^{-1})\cdot W^{n}_{\lambda^{rc}\mu^{rc}\nu}(q_{0}^{-1},q_{1}^{-1},\cdots,q_{n-1}^{-1})\\
&=&W^{n}_{\lambda\mu^{rc}\nu}(q_{0}^{-1},q_{1}^{-1},\cdots,q_{n-1}^{-1})\cdot W^{n}_{\lambda^{rc}\mu\nu}(q_{0}^{-1},q_{1}^{-1},\cdots,q_{n-1}^{-1})\\
&&+q^{-K_{1}(\lambda,\mu,\nu)}\cdot W^{n}_{\lambda^r\mu^c\nu}(q_{0}^{-1},q_{1}^{-1},\cdots,q_{n-1}^{-1})\cdot W^{n}_{\lambda^c\mu^r\nu}(q_{0}^{-1},q_{1}^{-1},\cdots,q_{n-1}^{-1}).
\een
It is equivalent to 
\bea\label{PT-recurrence1}
W^{n}_{\lambda\mu\nu}
=W^{n}_{\lambda^{rc}\mu\nu}\cdot W^{n}_{\lambda\mu^{rc}\nu}\cdot \left(W^{n}_{\lambda^{rc}\mu^{rc}\nu}\right)^{-1}
+q^{K_{1}(\lambda,\mu,\nu)}\cdot W^{n}_{\lambda^r\mu^c\nu}\cdot W^{n}_{\lambda^c\mu^r\nu}\cdot\left(W^{n}_{\lambda^{rc}\mu^{rc}\nu}\right)^{-1}.
\eea
which generalizes the one in [\cite{JWY}, Section 4.5] by Remark \ref{weight-generalization2}.

But we still need more recurrences as in the orbifold DT case. First, we take $\mathbf{G}=\mathbf{H}(N)$, $\mathbf{N}=\mathbf{N}_{\lambda^{rc},\mu,\nu^{rc}}(N)$ and let 
\ben
a:=\max S_{1}^-,\;\;\;b:=\min S_{1}^+,\;\;\;c=\max S_{3}^-,\;\;\;d=\min S_{3}^+.
\een
Then
\ben
&&(\mathbf{H}(N),\mathbf{N}_{\lambda,\mu,\nu}(N))=(\mathbf{H}(N),\mathbf{N}_{\lambda^{rc},\mu,\nu^{rc}}(N)-\{a,b,c,d\}),\\
&&(\mathbf{H}(N),\mathbf{N}_{\lambda,\mu,\nu^{rc}}(N))=(\mathbf{H}(N),\mathbf{N}_{\lambda^{rc},\mu,\nu^{rc}}(N)-\{a,b\}),\\
&&(\mathbf{H}(N),\mathbf{N}_{\lambda^{rc},\mu,\nu}(N))=(\mathbf{H}(N),\mathbf{N}_{\lambda^{rc},\mu,\nu^{rc}}(N)-\{c,d\}).
\een

 Since $\mathbf{N}-\{b,c\}=\mathbf{N}_{\lambda,\mu,\nu}(N)\cup\{a,d\}$ and $\mathbf{N}-\{a,d\}=\mathbf{N}_{\lambda,\mu,\nu}(N)\cup\{b,c\}$, the pairing $\sigma_{bc}$ on $\mathbf{N}-\{b,c\}$ has one more red-green path connecting a red node $a$ in sector 1 and a green node $d$ in sector 3 while the pairing $\sigma_{ad}$ has one fewer red-green path but one more blue-red path and one more blue-green path  than $\sigma_{abcd}$. The double-dimer configuration on $\mathbf{H}(N)$ with nodes $\mathbf{N}-\{a,d\}$ is obtained from an AB configuration  with respect to $(\lambda^r,\mu,\nu^c)$, 
 but is not equal to $\mathbf{D}_{(A,B)}(N)$. This dimer configuration and the corresponding tilings are shifted in the left-up  direction  by one unit prior to the truncation. We call it  the base$_{left-up}$ double-dimer configuration and denote its edge-weight by $q^{\omega^{LU}_{base}(\lambda^r,\mu,\nu^c;N)}$. Similarly, the double-dimer configuration on $\mathbf{H}(N)$ with nodes $\mathbf{N}-\{b,c\}$ is derived from an AB configuration  with respect to $(\lambda^c,\mu,\nu^r)$, and this  configuration and its corresponding tilings are shifed   by one unit in the right-down direction prior to the truncation.  We call it  the base$_{right-down}$ double-dimer configuration and denote its edge-weight by $q^{\omega^{RD}_{base}(\lambda^c,\mu,\nu^r;N)}$. The figures are omitted and the reader may take an example to draw the pictures (or modify the pictures in [\cite{JWY}, Example 4.5.2]) for the mentioned shift. Let 
 \ben
 &&\widetilde{\mathbf{Z}}_{\sigma}^{DD}(\mathbf{H}(N),\mathbf{N}_{\lambda^{rc},\mu,\nu^{rc}}(N)):=\left(q^{\omega_{base}(\lambda^{rc},\mu,\nu^{rc};N)}\right)^{-1}\mathbf{Z}_{\sigma}^{DD}(\mathbf{H}(N),\mathbf{N}_{\lambda^{rc},\mu,\nu^{rc}}(N)),\\
 &&\widetilde{\mathbf{Z}}_{\sigma_{ab}}^{DD}(\mathbf{H}(N),\mathbf{N}_{\lambda,\mu,\nu^{rc}}(N)):=\left(q^{\omega_{base}(\lambda,\mu,\nu^{rc};N)}\right)^{-1}\mathbf{Z}_{\sigma_{ab}}^{DD}(\mathbf{H}(N),\mathbf{N}_{\lambda,\mu,\nu^{rc}}(N)),\\
 &&\widetilde{\mathbf{Z}}_{\sigma_{ad}}^{DD}(\mathbf{H}(N),\mathbf{N}_{\lambda^{rc},\mu,\nu^{rc}}(N)-\{a,d\}):=\left(q^{\omega^{LU}_{base}(\lambda^{r},\mu,\nu^{c};N)}\right)^{-1}\mathbf{Z}_{\sigma_{ad}}^{DD}(\mathbf{H}(N),\mathbf{N}_{\lambda^{rc},\mu,\nu^{rc}}(N)-\{a,d\}),\\
 &&\widetilde{\mathbf{Z}}_{\sigma_{bc}}^{DD}(\mathbf{H}(N),\mathbf{N}_{\lambda^{rc},\mu,\nu^{rc}}(N)-\{b,c\}):=\left(q^{\omega^{RD}_{base}(\lambda^{c},\mu,\nu^{r};N)}\right)^{-1}\mathbf{Z}_{\sigma_{bc}}^{DD}(\mathbf{H}(N),\mathbf{N}_{\lambda^{rc},\mu,\nu^{rc}}(N)-\{b,c\}).
 \een
By Theorem \ref{graphical-condensation2}, we have 
\ben
&&\widetilde{\mathbf{Z}}_{\sigma}^{DD}(\mathbf{H}(N),\mathbf{N}_{\lambda^{rc},\mu,\nu^{rc}}(N))\widetilde{\mathbf{Z}}_{\sigma_{abcd}}^{DD}(\mathbf{H}(N),\mathbf{N}_{\lambda,\mu,\nu}(N))\\
&=&\frac{q^{\omega_{base}(\lambda,\mu,\nu^{rc};N)}\cdot q^{\omega_{base}(\lambda^{rc},\mu,\nu;N)} }{q^{\omega_{base}(\lambda^{rc},\mu,\nu^{rc};N)}\cdot q^{\omega_{base}(\lambda,\mu,\nu;N)} }\cdot\widetilde{\mathbf{Z}}_{\sigma_{ab}}^{DD}(\mathbf{H}(N),\mathbf{N}_{\lambda,\mu,\nu^{rc}}(N))\widetilde{\mathbf{Z}}_{\sigma_{cd}}^{DD}(\mathbf{H}(N),\mathbf{N}_{\lambda^{rc},\mu,\nu}(N))\\
&+&\frac{q^{\omega^{LU}_{base}(\lambda^{r},\mu,\nu^{c};N)}\cdot q^{\omega^{RD}_{base}(\lambda^{c},\mu,\nu^{r};N)} }{q^{\omega_{base}(\lambda^{rc},\mu,\nu^{rc};N)}\cdot q^{\omega_{base}(\lambda,\mu,\nu;N)}}\cdot\widetilde{\mathbf{Z}}_{\sigma_{ad}}^{DD}(\mathbf{H}(N),\mathbf{N}_{\lambda^{rc},\mu,\nu^{rc}}(N)-\{a,d\})\widetilde{\mathbf{Z}}_{\sigma_{bc}}^{DD}(\mathbf{H}(N),\mathbf{N}_{\lambda^{rc},\mu,\nu^{rc}}(N)-\{b,c\}).
\een
It is proved in Section 5.4.2 that 
\ben
&&\frac{q^{\omega_{base}(\lambda,\mu,\nu^{rc};N)}\cdot q^{\omega_{base}(\lambda^{rc},\mu,\nu;N)} }{q^{\omega_{base}(\lambda^{rc},\mu,\nu^{rc};N)}\cdot q^{\omega_{base}(\lambda,\mu,\nu;N)}}=1\;\;\;\;\;\;\mbox{and}\;\;\;\;\;\;\frac{q^{\omega^{LU}_{base}(\lambda^{r},\mu,\nu^{c};N)}\cdot q^{\omega^{RD}_{base}(\lambda^{c},\mu,\nu^{r};N)} }{q^{\omega_{base}(\lambda^{rc},\mu,\nu^{rc};N)}\cdot q^{\omega_{base}(\lambda,\mu,\nu;N)}}=q^{-K_{2}(\lambda,\mu,\nu)}
\een
where $q^{-K_{2}(\lambda,\mu,\nu)}$ defined in Lemma \ref{recurrence5-weight7} is independent of $N$.
Since
\ben
&&\lim_{N\to\infty}\widetilde{\mathbf{Z}}_{\sigma_{ad}}^{DD}(\mathbf{H}(N),\mathbf{N}_{\lambda^{rc},\mu,\nu^{rc}}(N)-\{a,d\})=W^{n}_{\lambda^r\mu\nu^c}(q_{0}^{-1},q_{1}^{-1},\cdots,q_{n-1}^{-1}),\\
&&\lim_{N\to\infty}\widetilde{\mathbf{Z}}_{\sigma_{bc}}^{DD}(\mathbf{H}(N),\mathbf{N}_{\lambda^{rc},\mu,\nu^{rc}}(N)-\{b,c\})=W^{n}_{\lambda^c\mu\nu^r}(q_{0}^{-1},q_{1}^{-1},\cdots,q_{n-1}^{-1})
\een
by Theorem \ref{orbifold-PT-vertex}, let $N\to\infty$, we have
\ben
&&W^{n}_{\lambda\mu\nu}(q_{0}^{-1},q_{1}^{-1},\cdots,q_{n-1}^{-1})\cdot W^{n}_{\lambda^{rc}\mu\nu^{rc}}(q_{0}^{-1},q_{1}^{-1},\cdots,q_{n-1}^{-1})\\
&=&W^{n}_{\lambda\mu\nu^{rc}}(q_{0}^{-1},q_{1}^{-1},\cdots,q_{n-1}^{-1})\cdot W^{n}_{\lambda^{rc}\mu\nu}(q_{0}^{-1},q_{1}^{-1},\cdots,q_{n-1}^{-1})\\
&&+q^{-K_{2}(\lambda,\mu,\nu)}\cdot W^{n}_{\lambda^r\mu\nu^c}(q_{0}^{-1},q_{1}^{-1},\cdots,q_{n-1}^{-1})\cdot W^{n}_{\lambda^c\mu\nu^r}(q_{0}^{-1},q_{1}^{-1},\cdots,q_{n-1}^{-1})
\een
which is equivalent to 
\bea\label{PT-recurrence2}
W^{n}_{\lambda\mu\nu}
=W^{n}_{\lambda^{rc}\mu\nu}\cdot W^{n}_{\lambda\mu\nu^{rc}}\cdot\left(W^{n}_{\lambda^{rc}\mu\nu^{rc}}\right)^{-1}
+q^{K_{2}(\lambda,\mu,\nu)}\cdot W^{n}_{\lambda^r\mu\nu^c}\cdot W^{n}_{\lambda^c\mu\nu^r}\left(W^{n}_{\lambda^{rc}\mu\nu^{rc}}\right)^{-1}.
\eea

As in DT case, one can derive another graphical condensation recurrence from the symmetry of $W^{n}_{\lambda\mu\nu}$ and Equation \eqref{PT-recurrence2} as follows
\ben
\overline{W}^{n}_{\mu^\prime\lambda^\prime\nu^\prime}\cdot\overline{W}^{n}_{\mu^\prime(\lambda^\prime)^{rc}(\nu^\prime)^{rc}}
=\overline{W}^{n}_{\mu^\prime\lambda^\prime(\nu^\prime)^{rc}}\cdot\overline{W}^{n}_{\mu^\prime(\lambda^\prime)^{rc}\nu^\prime}
+\overline{q^{K_{3}(\mu^\prime,\lambda^\prime,\nu^\prime)}}\cdot \overline{W}^{n}_{\mu^\prime(\lambda^\prime)^c(\nu^\prime)^r}\cdot\overline{W}^{n}_{\mu^\prime(\lambda^\prime)^r(\nu^\prime)^c}
\een
which is equivalent to Equation \eqref{PT-recurrence3}. Similarly, the relations among Equations \eqref{PT-recurrence1}, \eqref{PT-recurrence2}, \eqref{PT-recurrence3} are  similar as the orbifold DT case.  Independently, the  graphical condensation recurrence \eqref{PT-recurrence3} can be obtained by setting
$\mathbf{G}=\mathbf{H}(N)$, $\mathbf{N}=\mathbf{N}_{\lambda,\mu^{rc},\nu^{rc}}(N)$ and
\ben
a:=\max S_{2}^-,\;\;\;b:=\min S_{2}^+,\;\;\;c=\max S_{3}^-,\;\;\;d=\min S_{3}^+.
\een
Then
\ben
&&(\mathbf{H}(N),\mathbf{N}_{\lambda,\mu,\nu}(N))=(\mathbf{H}(N),\mathbf{N}_{\lambda,\mu^{rc},\nu^{rc}}(N)-\{a,b,c,d\});\\
&&(\mathbf{H}(N),\mathbf{N}_{\lambda,\mu,\nu^{rc}}(N))=(\mathbf{H}(N),\mathbf{N}_{\lambda,\mu^{rc},\nu^{rc}}(N)-\{a,b\});\\
&&(\mathbf{H}(N),\mathbf{N}_{\lambda,\mu^{rc},\nu}(N))=(\mathbf{H}(N),\mathbf{N}_{\lambda,\mu^{rc},\nu^{rc}}(N)-\{c,d\}).
\een

Since $\mathbf{N}-\{a,d\}=\mathbf{N}_{\lambda,\mu,\nu}(N)\cup\{b,c\}$ and $\mathbf{N}-\{b,c\}=\mathbf{N}_{\lambda,\mu,\nu}(N)\cup\{a,d\}$, the pairing $\sigma_{ad}$ on $\mathbf{N}-\{a,d\}$ has one more red-blue path connecting a red node $b$ in sector 2 and a blue node $c$ in sector 3 while the pairing $\sigma_{bc}$ has one fewer red-blue path but one more green-red path and one more green-blue path  than $\sigma_{abcd}$. The double-dimer configuration on $\mathbf{H}(N)$ with nodes $\mathbf{N}-\{a,d\}$ is obtained from an AB configuration $(A,B)$  with respect to $(\lambda,\mu^r,\nu^c)$, 
but is not equal to $\mathbf{D}_{(A,B)}(N)$. This dimer configuration and the corresponding tilings are shifted in the left-down direction  by one unit prior to truncation. We call it  the base$_{left-down}$ double-dimer configuration and denote its edge-weight by $q^{\omega^{LD}_{base}(\lambda,\mu^r,\nu^c;N)}$. Similarly, the double-dimer configuration on $\mathbf{H}(N)$ with nodes $\mathbf{N}-\{b,c\}$ is derived from an AB configuration  with respect to $(\lambda,\mu^c,\nu^r)$, and this  configuration and its corresponding tilings are shifed   by one unit in the right-up direction prior to truncation.  We call it  the base$_{right-up}$ double-dimer configuration and denote its edge-weight by $q^{\omega^{RU}_{base}(\lambda,\mu^c,\nu^r;N)}$. Let 
\ben
&&\widetilde{\mathbf{Z}}_{\sigma}^{DD}(\mathbf{H}(N),\mathbf{N}_{\lambda,\mu^{rc},\nu^{rc}}(N)):=\left(q^{\omega_{base}(\lambda,\mu^{rc},\nu^{rc};N)}\right)^{-1}\mathbf{Z}_{\sigma}^{DD}(\mathbf{H}(N),\mathbf{N}_{\lambda,\mu^{rc},\nu^{rc}}(N)),\\
&&\widetilde{\mathbf{Z}}_{\sigma_{cd}}^{DD}(\mathbf{H}(N),\mathbf{N}_{\lambda,\mu^{rc},\nu}(N)):=\left(q^{\omega_{base}(\lambda,\mu^{rc},\nu;N)}\right)^{-1}\mathbf{Z}_{\sigma_{cd}}^{DD}(\mathbf{H}(N),\mathbf{N}_{\lambda,\mu^{rc},\nu}(N)),\\
&&\widetilde{\mathbf{Z}}_{\sigma_{ad}}^{DD}(\mathbf{H}(N),\mathbf{N}_{\lambda,\mu^{rc},\nu^{rc}}(N)-\{a,d\}):=\left(q^{\omega^{LD}_{base}(\lambda,\mu^{r},\nu^{c};N)}\right)^{-1}\mathbf{Z}_{\sigma_{ad}}^{DD}(\mathbf{H}(N),\mathbf{N}_{\lambda,\mu^{rc},\nu^{rc}}(N)-\{a,d\}),\\
&&\widetilde{\mathbf{Z}}_{\sigma_{bc}}^{DD}(\mathbf{H}(N),\mathbf{N}_{\lambda,\mu^{rc},\nu^{rc}}(N)-\{b,c\}):=\left(q^{\omega^{RU}_{base}(\lambda,\mu^{c},\nu^{r};N)}\right)^{-1}\mathbf{Z}_{\sigma_{bc}}^{DD}(\mathbf{H}(N),\mathbf{N}_{\lambda,\mu^{rc},\nu^{rc}}(N)-\{b,c\}).
\een
By Theorem \ref{graphical-condensation2}, we have 
\ben
&&\widetilde{\mathbf{Z}}_{\sigma}^{DD}(\mathbf{H}(N),\mathbf{N}_{\lambda,\mu^{rc},\nu^{rc}}(N))\widetilde{\mathbf{Z}}_{\sigma_{abcd}}^{DD}(\mathbf{H}(N),\mathbf{N}_{\lambda,\mu,\nu}(N))\\
&=&\frac{q^{\omega_{base}(\lambda,\mu,\nu^{rc};N)}\cdot q^{\omega_{base}(\lambda,\mu^{rc},\nu;N)} }{q^{\omega_{base}(\lambda,\mu^{rc},\nu^{rc};N)}\cdot q^{\omega_{base}(\lambda,\mu,\nu;N)} }\cdot\widetilde{\mathbf{Z}}_{\sigma_{ab}}^{DD}(\mathbf{H}(N),\mathbf{N}_{\lambda,\mu,\nu^{rc}}(N))\widetilde{\mathbf{Z}}_{\sigma_{cd}}^{DD}(\mathbf{H}(N),\mathbf{N}_{\lambda,\mu^{rc},\nu}(N))\\
&+&\frac{q^{\omega^{LD}_{base}(\lambda,\mu^{r},\nu^{c};N)}\cdot q^{\omega^{RU}_{base}(\lambda,\mu^{c},\nu^{r};N)} }{q^{\omega_{base}(\lambda,\mu^{rc},\nu^{rc};N)}\cdot q^{\omega_{base}(\lambda,\mu,\nu;N)}}\cdot\widetilde{\mathbf{Z}}_{\sigma_{ad}}^{DD}(\mathbf{H}(N),\mathbf{N}_{\lambda,\mu^{rc},\nu^{rc}}(N)-\{a,d\})\widetilde{\mathbf{Z}}_{\sigma_{bc}}^{DD}(\mathbf{H}(N),\mathbf{N}_{\lambda,\mu^{rc},\nu^{rc}}(N)-\{b,c\})
\een
It is proved in Section 5.4.3 that 
\ben
&&\frac{q^{\omega_{base}(\lambda,\mu,\nu^{rc};N)}\cdot q^{\omega_{base}(\lambda,\mu^{rc},\nu;N)} }{q^{\omega_{base}(\lambda,\mu^{rc},\nu^{rc};N)}\cdot q^{\omega_{base}(\lambda,\mu,\nu;N)}}=1\;\;\;\;\;\;\mbox{and}\;\;\;\;\;\;\frac{q^{\omega^{LD}_{base}(\lambda,\mu^{r},\nu^{c};N)}\cdot q^{\omega^{RU}_{base}(\lambda,\mu^{c},\nu^{r};N)} }{q^{\omega_{base}(\lambda,\mu^{rc},\nu^{rc};N)}\cdot q^{\omega_{base}(\lambda,\mu,\nu;N)}}=q^{-K_{3}(\lambda,\mu,\nu)}
\een
where $q^{-K_{3}(\lambda,\mu,\nu)}$ defined in Lemma \ref{recurrence6-weight7} is independent of $N$.
Since
\ben
&&\lim_{N\to\infty}\widetilde{\mathbf{Z}}_{\sigma_{ad}}^{DD}(\mathbf{H}(N),\mathbf{N}_{\lambda,\mu^{rc},\nu^{rc}}(N)-\{a,d\})=W^{n}_{\lambda\mu^r\nu^c}(q_{0}^{-1},q_{1}^{-1},\cdots,q_{n-1}^{-1}),\\
&&\lim_{N\to\infty}\widetilde{\mathbf{Z}}_{\sigma_{bc}}^{DD}(\mathbf{H}(N),\mathbf{N}_{\lambda,\mu^{rc},\nu^{rc}}(N)-\{b,c\})=W^{n}_{\lambda\mu^c\nu^r}(q_{0}^{-1},q_{1}^{-1},\cdots,q_{n-1}^{-1})
\een
by Theorem \ref{orbifold-PT-vertex}, let $N\to\infty$, we have
\ben
&&W^{n}_{\lambda\mu\nu}(q_{0}^{-1},q_{1}^{-1},\cdots,q_{n-1}^{-1})\cdot W^{n}_{\lambda\mu^{rc}\nu^{rc}}(q_{0}^{-1},q_{1}^{-1},\cdots,q_{n-1}^{-1})\\
&=&W^{n}_{\lambda\mu\nu^{rc}}(q_{0}^{-1},q_{1}^{-1},\cdots,q_{n-1}^{-1})\cdot W^{n}_{\lambda\mu^{rc}\nu}(q_{0}^{-1},q_{1}^{-1},\cdots,q_{n-1}^{-1})\\
&&+q^{-K_{3}(\lambda,\mu,\nu)}\cdot W^{n}_{\lambda\mu^r\nu^c}(q_{0}^{-1},q_{1}^{-1},\cdots,q_{n-1}^{-1})\cdot W^{n}_{\lambda\mu^c\nu^r}(q_{0}^{-1},q_{1}^{-1},\cdots,q_{n-1}^{-1}).
\een
Now, we have
\bea\label{PT-recurrence3}
W^{n}_{\lambda\mu\nu}
=W^{n}_{\lambda\mu^{rc}\nu}\cdot W^{n}_{\lambda\mu\nu^{rc}}\cdot\left(W^{n}_{\lambda\mu^{rc}\nu^{rc}}\right)^{-1}
+q^{K_{3}(\lambda,\mu,\nu)}\cdot W^{n}_{\lambda\mu^r\nu^c}\cdot W^{n}_{\lambda\mu^c\nu^r}\cdot\left(W^{n}_{\lambda\mu^{rc}\nu^{rc}}\right)^{-1}.
\eea

\subsection{Weights for  orbifold PT theory}
For simplifications, we introduce  the following notations
\ben
&&q^{\vartheta_{1}(\eta,m,k,l)}:=\prod_{i=1}^{m-\ell(\eta)-l}\prod_{j=1}^{i}q_{-m+i-j+1}^{i+k},\\
&&q^{\vartheta_{2}(\eta,m,k)}:=\prod_{i:1\leq i\leq\eta_{i}}\prod_{j=1}^{m-\eta_{i}}q_{-i-j+k+1}^{m+\eta_{i}-i-k}\cdot\prod_{i:\eta_{i}< i\leq\ell(\eta)}\prod_{j=1}^{m-i}q_{-i-j+k+1}^{m+\eta_{i}-i-k},\\
&&q^{\vartheta_{3}(\eta,m,k,l)}:=\prod_{i=1}^{m-\ell(\eta)-1}\prod_{j=\ell(\eta)+1}^{m-i}q_{i+j-l}^{m+i-k},\\
&&q^{\vartheta_{4}(\eta,m,k)}:=\prod_{i:1\leq i\leq\eta_{i}}\prod_{j=1}^{m-\eta_{i}-k}q_{i+j+k-1}^{m+\eta_{i}+j+k-1}\cdot\prod_{i:\eta_{i}<i\leq\ell(\eta)}\prod_{j=1}^{m-i-k}q_{i+j+k-1}^{m+\eta_{i}+j+k-1},\\
&&q^{\vartheta_{5}(\eta,m,k)}:=\left\{
\begin{aligned}
	& 1 ,\; \;\;\mbox{if $\eta=\emptyset$}, \\
	& \prod_{i:1\leq i\leq\eta_{i}+1}\prod_{j=1}^{m-\eta_{i}-1}q_{-i-j+k+1}^{m+\eta_{i}-i-k+1}\cdot\prod_{i:\eta_{i}+1< i\leq\ell(\eta)}\prod_{j=1}^{m-i}q_{-i-j+k+1}^{m+\eta_{i}-i-k+1}, \;\;\;\mbox{if $\eta\neq\emptyset$},
\end{aligned}
\right.\\
&&q^{\vartheta_{6}(\eta,m)}:=\left\{
\begin{aligned}
	& 1 ,\; \;\;\mbox{if $\eta=\emptyset$}, \\
	& \prod_{i:1\leq i\leq\eta_{i}+1}\prod_{j=1}^{m-\eta_{i}-1}q_{i+j-1}^{m+\eta_{i}+j}\cdot\prod_{i:\eta_{i}+1<i\leq\ell(\eta)}\prod_{j=1}^{m-i}q_{i+j-1}^{m+\eta_{i}+j}, \;\;\;\mbox{if $\eta\neq\emptyset$},
\end{aligned}
\right.\\
&&q^{\vartheta_{7}(\eta,m,k)}:=\prod_{i:1\leq i<\eta_{i}-1}\prod_{j=1}^{m-\eta_{i}+k}q_{-i-j+k}^{m+\eta_{i}-i-1}\cdot\prod_{i:\eta_{i}-1\leq i\leq\ell(\eta)}\prod_{j=1}^{m-i+k-1}q_{-i-j+k}^{m+\eta_{i}-i-1},\\
&&q^{\vartheta_{8}(\eta,m,k,l)}:=\prod_{i:1\leq i<\eta_{i}-k}\prod_{j=1}^{m-\eta_{i}}q_{i+j+k-l-1}^{m+\eta_{i}+j-2l-1}\cdot\prod_{i:\eta_{i}-k\leq i\leq\ell(\eta)}\prod_{j=1}^{m-i-1}q_{i+j+k-l-1}^{m+\eta_{i}+j-2l-1},
\een
where $m, l, k$ are all integers and $\eta$ is a partition.

\subsubsection{Weights for the graphical condensation recurrence with $(\mathbf{H}(N),\mathbf{N}_{\lambda^{rc},\mu^{rc},\nu}(N))$}
We will  compute the edge-weights of several double-dimer configurations as  in [\cite{JWY}, Section 5.3] with the given weighted rule in Definition \ref{weight-rule}.
Since the base double-dimer configuration $\mathbf{D}_{(\mathrm{III},\mathrm{II}\cup\mathrm{III})}(N)$ has the edge-weight $q^{\omega_{base}(\lambda,\mu,\nu;N)}$, which is obtained from the  weights of all horizontal dimers of $\mathbf{D}_{A}(N)$ and the  weights of all horizontal dimers of  $\mathbf{D}_{B}(N)$ where $(A,B)=(\mathrm{III},\mathrm{II}\cup\mathrm{III})$. For the horizontal dimers of $\mathbf{D}_{A}(N)$ in sector 1, one can divide them into $\ell(\lambda^\prime)+1$ groups  with the  group $i$ along the $i$-th part of $\lambda^\prime$ if $1\leq i\leq\ell(\lambda^\prime)$ and the others as a group if $i>\ell(\lambda^\prime)$. For $i>\ell(\lambda^\prime)$, the product of weights of all the edges in this group is $\prod\limits_{i=1}^{N-\ell(\lambda^\prime)-1}\prod\limits_{j=1}^iq_{-N+i-j+1}^i$. If $\lambda^\prime_{i}\geq i$, there are still $(N-\lambda^\prime_{i})$ horizontal dimers in $\mathbf{H}(N)$ with weights: $q_{-i}^{N+\lambda^\prime_{i}-i}, \cdots, q_{-i-N+\lambda^\prime_{i}+1}^{N+\lambda^\prime_{i}-i}$. If $\lambda^\prime_{i}<i$, there are  still $(N-i)$ horizontal dimers in $\mathbf{H}(N)$ with weights: $q_{-i}^{N+\lambda^\prime_{i}-i}, \cdots, q_{1-N}^{N+\lambda^\prime_{i}-i}$. Similarly, one can divide all horizontal dimers of $\mathbf{D}_{A}(N)$ in sector 2 into $\ell(\mu)+1$ groups with the group $i$ along the $i$-th part of $\mu$ if $1\leq i\leq\ell(\mu)$ and the others as a group if $i>\ell(\mu)$. For $i>\ell(\mu)$, the product of weights of all the edges in this group
is $\prod\limits_{i=1}^{N-\ell(\mu)-1}\prod\limits_{j=\ell(\mu)+1}^{N-i}q_{i+j-1}^{N+i-1}$. If $\mu_{i}\geq i$, there are still $(N-\mu_{i})$ horizontal dimers in $\mathbf{H}(N)$ with weights: $q_{i}^{N+\mu_{i}}, q_{i+1}^{N+\mu_{i}+1}, \cdots, q_{i+N-\mu_{i}-1}^{2N-1}$. If $\mu_{i}<i$, there are still $(N-i)$ horizontal dimers in $\mathbf{H}(N)$ with weights: $q_{i}^{N+\mu_{i}}, q_{i+1}^{N+\mu_{i}+1}, \cdots, q_{N-1}^{2N+\mu_{i}-i-1}$. The total weight from the horizontal dimers of $\mathbf{D}_{A}(N)$ in sector 3 is $\prod\limits_{i=1}^{\ell(\nu)}\prod\limits_{j=0}^{\nu_{i}-1}q_{j-i+1}^{N-i}$. Since the horizontal dimers of  $\mathbf{D}_{B}(N)$ appear only in sector 3, the total weight from $\mathbf{D}_{B}(N)$ is $\prod\limits_{i=0}^{N-1}\prod\limits_{j=0}^{N-1}q_{j-i}^{N-i-1}$.
Therefore, we have

\begin{lemma}\label{recurrence4-weight1}
The edge-weight of the base double-dimer configuration is 
\ben
&&q^{\omega_{base}(\lambda,\mu,\nu;N)}\\
&=&\prod_{i=1}^{N-\ell(\lambda^\prime)-1}\prod_{j=1}^iq_{-N+i-j+1}^i\cdot\prod_{i:1\leq i\leq\lambda^\prime_{i}}\prod_{j=1}^{N-\lambda^\prime_{i}}q_{-i-j+1}^{N+\lambda^\prime_{i}-i}\cdot\prod_{i:\lambda^\prime_{i}< i\leq\ell(\lambda^\prime)}\prod_{j=1}^{N-i}q_{-i-j+1}^{N+\lambda^\prime_{i}-i}\\
&&\times\prod_{i=1}^{N-\ell(\mu)-1}\prod_{j=\ell(\mu)+1}^{N-i}q_{i+j-1}^{N+i-1}\cdot\prod_{i:1\leq i\leq\mu_{i}}\prod_{j=1}^{N-\mu_{i}}q_{i+j-1}^{N+\mu_{i}+j-1}\cdot\prod_{i:\mu_{i}<i\leq\ell(\mu)}\prod_{j=1}^{N-i}q_{i+j-1}^{N+\mu_{i}+j-1}\\
&&\times\prod_{i=1}^{\ell(\nu)}\prod_{j=0}^{\nu_{i}-1}q_{j-i+1}^{N-i}\cdot\prod_{i=0}^{N-1}\prod_{j=0}^{N-1}q_{j-i}^{N-i-1}\\
&=&q^{\vartheta_{1}(\lambda^\prime,N,0,1)}\cdot q^{\vartheta_{2}(\lambda^\prime,N,0)}\cdot q^{\vartheta_{3}(\mu,N,1,1)}\cdot q^{\vartheta_{4}(\mu,N,0)}\cdot \left(q^{\varpi_{3}(\nu,N,1)}\right)^{-1}\cdot q^{\varpi_{1}(N-1,0,0)}.
\een
\end{lemma}
Similarly, we have
\begin{lemma}\label{recurrence4-weight2}
The edge-weight of the base$_{up}$ double-dimer configuration is 	
\ben
&&q^{\omega_{base}^{U}(\lambda^r,\mu^c,\nu;n)}\\
&=&\prod_{i=1}^{N-\ell((\lambda^r)^\prime)-1}\prod_{j=1}^iq_{-N+i-j+1}^{i+1}\cdot\prod_{i:1\leq i\leq(\lambda^r)^\prime_{i}+1}\prod_{j=1}^{N-(\lambda^r)^\prime_{i}-1}q_{-i-j+1}^{N+(\lambda^r)^\prime_{i}-i+1}\cdot\prod_{i:(\lambda^r)^\prime_{i}+1< i\leq\ell((\lambda^r)^\prime)}\prod_{j=1}^{N-i}q_{-i-j+1}^{N+(\lambda^r)^\prime_{i}-i+1}\\
&&\times\prod_{i=1}^{N-\ell(\mu^c)-1}\prod_{j=\ell(\mu^c)+1}^{N-i}q_{i+j-1}^{N+i}\cdot\prod_{i:1\leq i\leq\mu^c_{i}+1}\prod_{j=1}^{N-\mu^c_{i}-1}q_{i+j-1}^{N+\mu^c_{i}+j}\cdot\prod_{i:\mu^c_{i}+1<i\leq\ell(\mu^c)}\prod_{j=1}^{N-i}q_{i+j-1}^{N+\mu^c_{i}+j}\\
&&\times q_{N}^{-N}\cdot\prod_{i=1}^{\ell(\nu)}\prod_{j=0}^{\nu_{i}-1}q_{j-i+1}^{N-i+1}\cdot \prod_{i=0}^{N}\prod_{j=0}^{N}q_{j-i}^{N-i}\\
&=&q^{\vartheta_{1}((\lambda^\prime)^c,N,1,1)}\cdot q^{\vartheta_{5}((\lambda^\prime)^c,N,0)}\cdot q^{\vartheta_{3}(\mu^c,N,0,1)}\cdot q^{\vartheta_{6}(\mu^c,N)}\cdot
\left(q^{\varpi_{3}(\nu,N+1,1)}\right)^{-1}\cdot q^{\varpi_{1}(N,0,0)}\cdot q_{N}^{-N}.
\een
\end{lemma}

\begin{lemma}\label{recurrence4-weight3}
The edge-weight of the base$_{down}$ double-dimer configuration is 	
\ben
&&q^{\omega_{base}^{D}(\lambda^c,\mu^r,\nu;n)}\\
&=&\prod_{i=1}^{N-\ell((\lambda^c)^\prime)-1}\prod_{j=1}^iq_{-N+i-j+1}^{i-1}\cdot\prod_{i:1\leq i<(\lambda^c)^\prime_{i}-1}\prod_{j=1}^{N-(\lambda^c)^\prime_{i}+1}q_{-i-j+1}^{N+(\lambda^c)^\prime_{i}-i-1}\cdot\prod_{i:(\lambda^c)^\prime_{i}-1\leq i\leq\ell((\lambda^c)^\prime)}\prod_{j=1}^{N-i}q_{-i-j+1}^{N+(\lambda^c)^\prime_{i}-i-1}\\
&&\times\prod_{i=1}^{N-\ell(\mu^r)-1}\prod_{j=\ell(\mu^r)+1}^{N-i}q_{i+j-1}^{N+i-2}\cdot\prod_{i:1\leq i<\mu^r_{i}-1}\prod_{j=1}^{N-\mu^r_{i}+1}q_{i+j-1}^{N+\mu^r_{i}+j-2}\cdot\prod_{i:\mu^r_{i}-1\leq i\leq\ell(\mu^r)}\prod_{j=1}^{N-i}q_{i+j-1}^{N+\mu^r_{i}+j-2}\\
&&\times\prod_{i=1}^{\ell(\nu)}\prod_{j=0}^{\nu_{i}-1}q_{j-i+1}^{N-i-1}\cdot\prod_{i=0}^{N-2}\prod_{j=0}^{N-2}q_{j-i}^{N-i-2}\\
&=&q^{\vartheta_{1}((\lambda^\prime)^r,N,-1,1)}\cdot q^{\vartheta_{7}((\lambda^\prime)^r,N,1)}\cdot q^{\vartheta_{3}(\mu^r,N,2,1)}\cdot q^{\vartheta_{8}(\mu^r,N+1,1,1)}\cdot \left(q^{\varpi_{3}(\nu,N-1,1)}\right)^{-1}\cdot q^{\varpi_{1}(N-2,0,0)}.
\een
\end{lemma}
Now it follows from Lemma \ref{recurrence4-weight1} that
\ben
\frac{q^{\omega_{base}(\lambda,\mu^{rc},\nu;N)}\cdot q^{\omega_{base}(\lambda^{rc},\mu,\nu;N)} }{q^{\omega_{base}(\lambda^{rc},\mu^{rc},\nu;N)}\cdot q^{\omega_{base}(\lambda,\mu,\nu;N)}}=1.
\een
And we still need the following lemmas to derive the graphical condensation recurrence \eqref{PT-recurrence1}.
\begin{lemma}\label{recurrence4-weight8}
	For any partition $\lambda, \mu\neq\emptyset$ and $\nu$,
	\ben
	\dfrac{q^{\omega_{base}^{U}(\lambda^r,\mu^c,\nu;N)}\cdot q^{\omega_{base}^{D}(\lambda^{c},\mu^{r},\nu;N)}}{q^{\omega_{base}(\lambda,\mu,\nu;N)}\cdot q^{\omega_{base}(\lambda^{rc},\mu^{rc},\nu;N)}}=q^{-K_{1}(\lambda,\mu,\nu)}
	\een
where 	$q^{-K_{1}(\lambda,\mu,\nu)}=\left(q^{K_{1}(\lambda,\mu,\nu)}\right)^{-1}$ is independent of $N$.
\end{lemma}

\begin{proof}
It follows from Lemmas	\ref{recurrence4-weight1}, \ref{recurrence4-weight2}, \ref{recurrence4-weight3}, \ref{recurrence4-weight4}, \ref{recurrence4-weight5}, 
\ref{recurrence4-weight6}, \ref{recurrence4-weight7}, \ref{recurrence1-weight6}, \ref{recurrence1-weight4}.
\end{proof}

\begin{remark}\label{weight-generalization2}
	When $n=1$,
	$q^{-K_{1}(\lambda,\mu,\nu)}$ is equal to $q^{K}$ in [\cite{JWY}, Section 4.5] as in Remark \ref{weight-generalization1}.
\end{remark}

\begin{lemma}\label{recurrence4-weight4}
For any partition $\eta\neq\emptyset$, we have	
\ben
\frac{q^{\vartheta_{1}(\eta^c,N,1,1)}\cdot q^{\vartheta_{1}(\eta^r,N,-1,1)}}{q^{\vartheta_{1}(\eta,N,0,1)}\cdot q^{\vartheta_{1}(\eta^{rc},N,0,1)}}=\left\{
\begin{aligned}
	&q_{-\ell(\eta)}^{N-\ell(\eta)-1}\cdot\prod\limits_{i=1-N}^{-\ell(\eta)-1}q_{i}^{-1} ,\; \;\;\;\; \;\;\; \;\;\;\;\;\;\;\;\;\;\;\;\;\;\;\;\;\;\;\mbox{if $d(\eta)>1$}; \\
	&\prod\limits_{i=1-N}^{-\ell(\eta)}q_{i}^{-1}\cdot\prod\limits_{i=1-\ell(\eta)}^{-1}q_{i}^{-N-i},   \;\;\;\;\;\;\;\;\;\;\;\;\;\;\;\;\;\;\;\mbox{if $d(\eta)=1$ and $\eta_{1}>1$};\\
	&\prod\limits_{i=1-N}^{-\ell(\eta)}q_{i}^{N-1}\cdot\prod\limits_{i=1-\ell(\eta)}^{-1}q_{i}^{-i},     \;\;\;\;\;\;\;\; \;\;\;\;\;\;\;\;\;\;\;\;\;\mbox{if $d(\eta)=1$ and $\eta_{1}=1$}.
\end{aligned}
\right.
\een
\end{lemma}
\begin{proof}
If $d(\eta)>1$, by Lemma \ref{length-relation},  we have 	
\ben
\frac{q^{\vartheta_{1}(\eta^c,N,1,1)}\cdot q^{\vartheta_{1}(\eta^r,N,-1,1)}}{q^{\vartheta_{1}(\eta,N,0,1)}\cdot q^{\vartheta_{1}(\eta^{rc},N,0,1)}}
=\frac{\prod\limits_{i=1}^{N-\ell(\eta)-2}\prod\limits_{j=1}^{i}q_{-N+i-j+1}^{i+1}\cdot\prod\limits_{i=1}^{N-\ell(\eta)}\prod\limits_{j=1}^{i}q_{-N+i-j+1}^{i-1}}{\prod\limits_{i=1}^{N-\ell(\eta)-1}\prod\limits_{j=1}^{i}q_{-N+i-j+1}^{2i}}
=\frac{q_{-\ell(\eta)}^{N-\ell(\eta)-1}}{\prod\limits_{j=1}^{N-\ell(\eta)-1}q_{-\ell(\eta)-j}}.
\een	
If $d(\eta)=1$ and $\eta_{1}>1$,  then by Lemma \ref{length-relation} and  Remark  \ref{special-value}	
\ben
\frac{q^{\vartheta_{1}(\eta^c,N,1,1)}\cdot q^{\vartheta_{1}(\eta^r,N,-1,1)}}{q^{\vartheta_{1}(\eta,N,0,1)}\cdot q^{\vartheta_{1}(\eta^{rc},N,0,1)}}
&=&\frac{\prod\limits_{i=N-\ell(\eta)}^{N-1}\prod\limits_{j=1}^{i}q_{-N+i-j+1}^{i}\cdot\prod\limits_{i=1}^{N-1}\prod\limits_{j=1}^{i}q_{-N+i-j+1}^{i+1}\cdot\prod\limits_{i=1}^{N-1}\prod\limits_{j=1}^{i}q_{-N+i-j+1}^{i-1}}{\prod\limits_{j=1}^{N-1}q_{-j}^{N}\cdot\prod\limits_{i=N-\ell(\eta)+1}^{N-1}\prod\limits_{j=1}^{i}q_{-N+i-j+1}^{i-1}\cdot\prod\limits_{i=1}^{N-1}\prod\limits_{j=1}^{i}q_{-N+i-j+1}^{i}\cdot\prod\limits_{i=1}^{N-1}\prod\limits_{j=1}^{i}q_{-N+i-j+1}^{i}}\\
&=&\frac{1}{\prod\limits_{j=1-N}^{-\ell(\eta)}q_{j}\cdot\prod\limits_{j=1-\ell(\eta)}^{-1}q_{j}^{N+j}}.
\een
If $d(\eta)=1$ and $\eta_{1}=1$,  then by Lemma \ref{length-relation} and  Remark  \ref{special-value}	
\ben
\frac{q^{\vartheta_{1}(\eta^c,N,1,1)}\cdot q^{\vartheta_{1}(\eta^r,N,-1,1)}}{q^{\vartheta_{1}(\eta,N,0,1)}\cdot q^{\vartheta_{1}(\eta^{rc},N,0,1)}}
&=&\frac{\prod\limits_{i=N-\ell(\eta)}^{N-1}\prod\limits_{j=1}^{i}q_{-N+i-j+1}^{i}\cdot\prod\limits_{i=1}^{N-1}\prod\limits_{j=1}^{i}q_{-N+i-j+1}^{i+1}\cdot\prod\limits_{i=1}^{N-1}\prod\limits_{j=1}^{i}q_{-N+i-j+1}^{i-1}}{\prod\limits_{i=N-\ell(\eta)+1}^{N-1}\prod\limits_{j=1}^{i}q_{-N+i-j+1}^{i-1}\cdot\prod\limits_{i=1}^{N-1}\prod\limits_{j=1}^{i}q_{-N+i-j+1}^{i}\cdot\prod\limits_{i=1}^{N-1}\prod\limits_{j=1}^{i}q_{-N+i-j+1}^{i}}\\
&=&\prod\limits_{j=1-N}^{-\ell(\eta)}q_{j}^{N-1}\cdot\prod\limits_{j=1-\ell(\eta)}^{-1}q_{j}^{-j}.
\een
\end{proof}

\begin{lemma}\label{recurrence4-weight5}
For any partition $\eta\neq\emptyset$, we have		
\ben
\frac{q^{\vartheta_{5}(\eta^c,N,0)}\cdot q^{\vartheta_{7}(\eta^r,N,1)}}{q^{\vartheta_{2}(\eta,N,0)}\cdot q^{\vartheta_{2}(\eta^{rc},N,0)}}=\left\{
\begin{aligned}
	&q_{-\widetilde{d}(\eta)}^{-d(\eta)}\cdot q_{-\ell(\eta)}^{-N+\ell(\eta)}\cdot\prod\limits_{i=\widetilde{d}(\eta)+1}^{\ell(\eta)-1}q_{-i}^{-1}\cdot\prod\limits_{i=\widetilde{d}(\eta)+1}^{\ell(\eta)}\left(\frac{q_{-i+1}}{q_{-i}}\right)^{\eta_{i}} ,\;   \;\;\;\;\mbox{if $d(\eta)>1$}; \\
	&\prod\limits_{i=1}^{\ell(\eta)-1}q_{-i}^{N-i},  \;\;\;\;\;\;\;\;\;\; \;\;\;\;\;\;\;\;\;\;\;\;\;\;\;\;\;\;\;\;\;\;\;\;\;\;\;\;\;\;\;\;\;\;\;\;\;\;\;\mbox{if $d(\eta)=1$ and $\eta_{1}>1$};\\
	&\prod\limits_{i=1}^{\ell(\eta)-1}q_{-i}^{-i}\cdot\prod\limits_{i=\ell(\eta)}^{N-1}q_{-i}^{-N},      \;\; \;\;\;\;\;\;\;\;\;  \;\;\;\;\;\;\;\;\;\;\;\;\;\;\;\;\;\;\;\;\; \;\;\;\mbox{if $d(\eta)=1$ and $\eta_{1}=1$}.
\end{aligned}
\right.
\een
\end{lemma}
\begin{proof}
If $d(\eta)>1$, then $\eta^{c}\neq\emptyset$ by Remark \ref{special-value}. Then by Lemmas \ref{modified-partition}, \ref{size-comparision}, \ref{value-set}, we have 
\ben
\frac{q^{\vartheta_{5}(\eta^c,N,0)}}{q^{\vartheta_{2}(\eta,N,0)}}
&=&\frac{\prod\limits_{i=1}^{d(\eta)}\prod\limits_{j=1}^{N-\eta_{i}}q_{-i-j+1}^{N+\eta_{i}-i}\cdot\prod\limits_{i=d(\eta)+1}^{\widetilde{d}(\eta)}\prod\limits_{j=1}^{N-i}q_{-i-j+1}^{N+d(\eta)-i}\cdot\prod\limits_{j=1}^{N-\widetilde{d}(\eta)-1}q_{-\widetilde{d}(\eta)-j}^{N+d(\eta)-\widetilde{d}(\eta)-1}\cdot\prod\limits_{i=\widetilde{d}(\eta)+2}^{\ell(\eta)+1}\prod\limits_{j=1}^{N-i}q_{-i-j+1}^{N+\eta_{i-1}-i+1}}{\prod\limits_{i=1}^{d(\eta)}\prod\limits_{j=1}^{N-\eta_{i}}q_{-i-j+1}^{N+\eta_{i}-i}\cdot\prod\limits_{i=d(\eta)+1}^{\widetilde{d}(\eta)}\prod\limits_{j=1}^{N-i}q_{-i-j+1}^{N+d(\eta)-i}\cdot\prod\limits_{i=\widetilde{d}(\eta)+1}^{\ell(\eta)}\prod\limits_{j=1}^{N-i}q_{-i-j+1}^{N+\eta_{i}-i}}\\
&=&\frac{\prod\limits_{j=1}^{N-\widetilde{d}(\eta)-1}q_{-\widetilde{d}(\eta)-j}^{N+d(\eta)-\widetilde{d}(\eta)-1}}{\prod\limits_{i=\widetilde{d}(\eta)+1}^{\ell(\eta)}q_{-i}^{N+\eta_{i}-i}}
\een
and
\ben
\frac{q^{\vartheta_{7}(\eta^r,N,1)}}{q^{\vartheta_{2}(\eta^{rc},N,0)}}
&=&\frac{\prod\limits_{i=1}^{d(\eta)-1}\prod\limits_{j=1}^{N-\eta_{i}}q_{-i-j+1}^{N+\eta_{i}-i}\cdot\prod\limits_{i=d(\eta)}^{\ell(\eta)-1}\prod\limits_{j=1}^{N-i}q_{-i-j+1}^{N+\eta_{i+1}-i-1}}{\prod\limits_{i=1}^{d(\eta)-1}\prod\limits_{j=1}^{N-\eta_{i}}q_{-i-j+1}^{N+\eta_{i}-i}\cdot\prod\limits_{i=d(\eta)}^{\widetilde{d}(\eta)}\prod\limits_{j=1}^{N-i}q_{-i-j+1}^{N+d(\eta)-i-1}\cdot\prod\limits_{i=\widetilde{d}(\eta)+1}^{\ell(\eta)}\prod\limits_{j=1}^{N-i}q_{-i-j+1}^{N+\eta_{i}-i}}\\
&=&\frac{\prod\limits_{i=d(\eta)+1}^{\widetilde{d}(\eta)}q_{-i+1}^{N+d(\eta)-i}\cdot\prod\limits_{i=d(\eta)+1}^{\widetilde{d}(\eta)}\prod\limits_{j=1}^{N-i}q_{-i-j+1}\cdot\prod\limits_{i=\widetilde{d}(\eta)+1}^{\ell(\eta)}q_{-i+1}^{N+\eta_{i}-i}}{\prod\limits_{j=1}^{N-d(\eta)}q_{-d(\eta)-j+1}^{N-1}}\\
&=&\prod\limits_{j=0}^{N-\widetilde{d}(\eta)-1}q_{-\widetilde{d}(\eta)-j}^{-N-d(\eta)+\widetilde{d}(\eta)+1}\cdot\prod\limits_{i=\widetilde{d}(\eta)+1}^{\ell(\eta)}q_{-i+1}^{N+\eta_{i}-i}.
\een
Then 
\ben
\frac{q^{\vartheta_{5}(\eta^c,N,0)}\cdot q^{\vartheta_{7}(\eta^r,N,1)}}{q^{\vartheta_{2}(\eta,N,0)}\cdot q^{\vartheta_{2}(\eta^{rc},N,0)}}
=\frac{q_{-\widetilde{d}(\eta)}^{-d(\eta)}\cdot q_{-\ell(\eta)}^{-N+\ell(\eta)}}{\prod\limits_{i=\widetilde{d}(\eta)+1}^{\ell(\eta)-1}q_{-i}}\prod\limits_{i=\widetilde{d}(\eta)+1}^{\ell(\eta)}\left(\frac{q_{-i+1}}{q_{-i}}\right)^{\eta_{i}}.
\een
If $d(\eta)=1$ and $\eta_{1}>1$, then $\eta^c\neq\emptyset$ by Remark \ref{special-value}. Now we have
\ben
\frac{q^{\vartheta_{5}(\eta^c,N,0)}\cdot q^{\vartheta_{7}(\eta^r,N,1)}}{q^{\vartheta_{2}(\eta,N,0)}\cdot q^{\vartheta_{2}(\eta^{rc},N,0)}}&=&\frac{\prod\limits_{j=1}^{N-\eta_{1}}q_{-j}^{N+\eta_{1}-1}\cdot\prod\limits_{i=2}^{\ell(\eta)}\prod\limits_{j=0}^{N-i}q_{-i-j+1}^{N+1-i}}{\prod\limits_{j=1}^{N-\eta_{1}}q_{-j}^{N+\eta_{1}-1}\cdot\prod\limits_{i=2}^{\ell(\eta)}\prod\limits_{j=1}^{N-i}q_{-i-j+1}^{N+1-i}}=\prod\limits_{i=1}^{\ell(\eta)-1}q_{-i}^{N-i}.
\een
If $d(\eta)=1$ and $\eta_{1}=1$, then $\eta^c=\emptyset$ by Remark \ref{special-value}. Then  we have
\ben
\frac{q^{\vartheta_{5}(\eta^c,N,0)}\cdot q^{\vartheta_{7}(\eta^r,N,1)}}{q^{\vartheta_{2}(\eta,N,0)}\cdot q^{\vartheta_{2}(\eta^{rc},N,0)}}=\frac{\prod\limits_{i=2}^{\ell(\eta)}\prod\limits_{j=0}^{N-i}q_{-i-j+1}^{N+1-i}}{\prod\limits_{j=1}^{N-1}q_{-j}^{N}\cdot\prod\limits_{i=2}^{\ell(\eta)}\prod\limits_{j=1}^{N-i}q_{-i-j+1}^{N+1-i}}=\frac{\prod\limits_{i=1}^{\ell(\eta)-1}q_{-i}^{-i}}{\prod\limits_{j=\ell(\eta)}^{N-1}q_{-j}^{N}}.
\een
\end{proof}

\begin{lemma}\label{recurrence4-weight6}
For any partition $\eta\neq\emptyset$, we have		
\ben
\frac{q^{\vartheta_{3}(\eta^c,N,0,1)}\cdot q^{\vartheta_{3}(\eta^r,N,2,1)}}{q^{\vartheta_{3}(\eta,N,1,1)}\cdot q^{\vartheta_{3}(\eta^{rc},N,1,1)}}
=\left\{
\begin{aligned}
	&q_{\ell(\eta)}^{N-1}\cdot\prod\limits_{i=\ell(\eta)+1}^{N-1}q_{i}^{-1} ,\;  \;\;\;\;\;\;\;\;\;   \;\;\;\;\;\;\;\;\;\;\;\;\;\;\; \;\;\;\;\mbox{if $d(\eta)>1$}; \\
	&\prod\limits_{i=\ell(\eta)}^{N-1}q_{i}^{-1}\cdot \prod\limits_{i=1}^{\ell(\eta)-1}q_{i}^{-N},   \;\;\;\;\;\;\;\;\;\;\;\;\;\;\;\;\;\;\;\;\;\;\;\mbox{if $d(\eta)=1$ and $\eta_{1}>1$};\\
	&\prod\limits_{i=1}^{\ell(\mu)-1}q_{i}^i\cdot\prod\limits_{i=\ell(\mu)}^{N-1}q_{i}^{N+i-1},      \;\;\;\;\;\;\;\;\;\;\;\;\;\;\;\;\;\;\;\;\; \mbox{if $d(\eta)=1$ and $\eta_{1}=1$}.
\end{aligned}
\right.
\een
\end{lemma}
\begin{proof}
It follows from the similar argument as in the proof of Lemma \ref{recurrence4-weight4}.
\end{proof}

\begin{lemma}\label{recurrence4-weight7}
For any partition $\eta\neq\emptyset$, we have	
\ben
\frac{q^{\vartheta_{6}(\eta^c,N)}\cdot q^{\vartheta_{8}(\eta^r,N+1,1,1)}}{q^{\vartheta_{4}(\eta,N,0)}\cdot q^{\vartheta_{4}(\eta^{rc},N,0)}}=\left\{
\begin{aligned}
	&q_{\widetilde{d}(\eta)}^{N}\cdot q_{\ell(\eta)}^{-N}\cdot\prod\limits_{i=\widetilde{d}(\eta)+1}^{\ell(\eta)}\left(\frac{q_{i-1}}{q_{i}}\right)^{\eta_{i}}\cdot\frac{\prod\limits_{i=d(\eta)}^{\widetilde{d}(\eta)-1}q_{i}}{q_{\widetilde{d}(\eta)}^{N+d(\eta)-1}\cdot\prod\limits_{i=\widetilde{d}(\eta)}^{\ell(\eta)-1}q_{i}} ,\;  \;\;\;\;\;\; \mbox{if $d(\eta)>1$}; \\
	&\prod\limits_{i=1}^{\ell(\eta)-1}q_{i}^{N},   \;\;\;\;\;\;\;\;\;\;\;\;\;\;\;\;\;\;\;\;\;\;\;\;\;\;\;\;\;\;\;\;\;\;\;\;\;\;\;\;\;\;\;\;\;\;\;\;\;\;\;\;\;\;\;\;\mbox{if $d(\eta)=1$ and $\eta_{1}>1$};\\
	&\prod\limits_{i=1}^{\ell(\eta)-1}q_{i}^{-i}\cdot\prod\limits_{i=\ell(\eta)}^{N-1}q_{i}^{-N-i},   \;\;\; \;  \;\;\;\;\;\;\;\;\;\;\;\;\;\;\;\;\;\;\;\;\;\;\;\;\;\; \;\;\;\mbox{if $d(\eta)=1$ and $\eta_{1}=1$}.
\end{aligned}
\right.
\een
\end{lemma}
\begin{proof}
It follows from the similar argument as in the proof  of Lemma \ref{recurrence4-weight5}.	
\end{proof}

\subsubsection{Weights for the graphical condensation recurrence with $(\mathbf{G}=\mathbf{H}(N),\mathbf{N}_{\lambda^{rc},\mu,\nu^{rc}}(N))$}

By the similar argument in Section 5.4.1, we have

\begin{lemma}\label{recurrence5-weight1}
The edge-weight of the base$_{left-up}$ double-dimer configuration is
\ben
&&q^{\omega_{base}^{LU}(\lambda^r,\mu,\nu^c;n)}\\
&=&\prod_{i=1}^{N-\ell((\lambda^r)^\prime)}\prod_{j=1}^iq_{-N+i-j+1}^{i}\cdot\prod_{i:1\leq i\leq(\lambda^r)^\prime_{i}+1}\prod_{j=1}^{N-(\lambda^r)^\prime_{i}}q_{-i-j+2}^{N+(\lambda^r)^\prime_{i}-i+1}\cdot\prod_{i:(\lambda^r)^\prime_{i}+1< i\leq\ell((\lambda^r)^\prime)}\prod_{j=1}^{N+1-i}q_{-i-j+2}^{N+(\lambda^r)^\prime_{i}-i+1}\\
&&\times\prod_{i=1}^{N-\ell(\mu)-2}\prod_{j=\ell(\mu)+1}^{N-i-1}q_{i+j}^{N+i}\cdot\prod_{i:1\leq i\leq\mu_{i}}\prod_{j=1}^{N-\mu_{i}-1}q_{i+j}^{N+\mu_{i}+j}\cdot\prod_{i:\mu_{i}<i\leq\ell(\mu)}\prod_{j=1}^{N-i-1}q_{i+j}^{N+\mu_{i}+j}\\
&&\times q_{N}^{-N}\cdot\prod_{i=1}^{\ell(\nu^c)}\prod_{j=0}^{\nu^c_{i}-1}q_{j-i+2}^{N-i+1}\cdot \prod_{i=0}^{N}\prod_{j=0}^{N-1}q_{j-i+1}^{N-i}\\
&=&q^{\vartheta_{1}((\lambda^\prime)^c,N,0,0)}\cdot q^{\vartheta_{5}((\lambda^\prime)^c,N+1,1)}\cdot q^{\vartheta_{3}(\mu,N-1,-1,0)}\cdot q^{\vartheta_{4}(\mu,N,1)}\cdot
\left(q^{\varpi_{3}(\nu^c,N+1,2)}\right)^{-1}\cdot q^{\varpi_{1}(N,1,1)}\cdot q_{N}^{-N}.
\een

\end{lemma}

\begin{lemma}\label{recurrence5-weight2}
The edge-weight of the base$_{right-down}$ double-dimer configuration is
\ben
&&q^{\omega_{base}^{RD}(\lambda^c,\mu,\nu^r;n)}\\
&=&\prod_{i=1}^{N-\ell((\lambda^c)^\prime)-2}\prod_{j=1}^iq_{-N+i-j+1}^{i}\cdot\prod_{i:1\leq i<(\lambda^c)^\prime_{i}-1}\prod_{j=1}^{N-(\lambda^c)^\prime_{i}}q_{-i-j}^{N+(\lambda^c)^\prime_{i}-i-1}\cdot\prod_{i:(\lambda^c)^\prime_{i}-1\leq i\leq\ell((\lambda^c)^\prime)}\prod_{j=1}^{N-i-1}q_{-i-j}^{N+(\lambda^c)^\prime_{i}-i-1}\\
&&\times\prod_{i=1}^{N-\ell(\mu)}\prod_{j=\ell(\mu)+1}^{N-i+1}q_{i+j-2}^{N+i-2}\cdot\prod_{i:1\leq i\leq\mu_{i}}\prod_{j=1}^{N-\mu_{i}+1}q_{i+j-2}^{N+\mu_{i}+j-2}\cdot\prod_{i:\mu_{i}< i\leq\ell(\mu)}\prod_{j=1}^{N-i+1}q_{i+j-2}^{N+\mu_{i}+j-2}\\
&&\times\prod_{i=1}^{\ell(\nu^r)}\prod_{j=0}^{\nu^r_{i}-1}q_{j-i}^{N-i-1}\cdot\prod_{i=0}^{N-2}\prod_{j=0}^{N-1}q_{j-i-1}^{N-i-2}\\
&=&q^{\vartheta_{1}((\lambda^\prime)^r,N,0,2)}\cdot q^{\vartheta_{7}((\lambda^\prime)^r,N,0)}\cdot q^{\vartheta_{3}(\mu,N+1,3,2)}\cdot q^{\vartheta_{4}(\mu,N,-1)}\cdot \left(q^{\varpi_{3}(\nu^r,N-1,0)}\right)^{-1}\cdot q^{\varpi_{1}(N-2,-1,-1)}.
\een

\end{lemma}

As in Section 5.4.1, we have 
\ben
\dfrac{q^{\omega_{base}(\lambda^{rc},\mu,\nu;N)}\cdot q^{\omega_{base}(\lambda,\mu,\nu^{rc};N)}}{q^{\omega_{base}(\lambda,\mu,\nu;N)}\cdot q^{\omega_{base}(\lambda^{rc},\mu,\nu^{rc};N)}}=1
\een
and still demand the following lemmas to obtain the recurrence \eqref{PT-recurrence2}.
\begin{lemma}\label{recurrence5-weight7}
	For any partitions $\lambda, \nu\neq\emptyset$ and $\mu$, we have 
	\ben
	\dfrac{q^{\omega_{base}^{LU}(\lambda^r,\mu,\nu^c;n)}\cdot q^{\omega_{base}^{RD}(\lambda^{c},\mu,\nu^{r};n)}}{q^{\omega_{base}(\lambda,\mu,\nu;n)}\cdot q^{\omega_{base}(\lambda^{rc},\mu,\nu^{rc};n)}}=q^{-K_{2}(\lambda,\mu,\nu)}
	\een
	where $q^{-K_{2}(\lambda,\mu,\nu)}=\left(q^{K_{2}(\lambda,\mu,\nu)}\right)^{-1}$ is independent of $N$.
\end{lemma}
\begin{proof}
	It follows from Lemmas \ref{recurrence4-weight1}, \ref{recurrence5-weight1}, \ref{recurrence5-weight2}, \ref{recurrence5-weight3}, \ref{recurrence5-weight4}, \ref{recurrence5-weight5}, \ref{recurrence5-weight6}, 
	 \ref{recurrence2-weight6}, \ref{recurrence2-weight3}.
\end{proof}

The following lemmas follow from the similar argument in Section 5.4.1.
\begin{lemma}\label{recurrence5-weight3}
	For any partition $\eta\neq\emptyset$, we have	
	\ben
	\frac{q^{\vartheta_{1}(\eta^c,N,0,0)}\cdot q^{\vartheta_{1}(\eta^r,N,0,2)}}{q^{\vartheta_{1}(\eta,N,0,1)}\cdot q^{\vartheta_{1}(\eta^{rc},N,0,1)}}
    =\left\{
	\begin{aligned}
		&1 ,\;  \;\;\;\;\; \;\;\;\;\; \;\;\; \;\;\;\;\;\;\;\; \;\;\;\;\;\;\; \mbox{if $d(\eta)>1$ or $d(\eta)=1$ and $\eta_{1}>1$}; \\
		&\prod\limits_{i=0}^{N-1}q_{-i}^{N},      \;\;\;\;\;\;\;\;\;\;\;\;\;\;\;\;\;\;\mbox{if $d(\eta)=1$ and $\eta_{1}=1$}.
	\end{aligned}
	\right.
	\een
\end{lemma}

\begin{lemma}\label{recurrence5-weight4}
For any partition $\eta\neq\emptyset$, we have		
	\ben
	\frac{q^{\vartheta_{5}(\eta^c,N+1,1)}\cdot q^{\vartheta_{7}(\eta^r,N,0)}}{q^{\vartheta_{2}(\eta,N,0)}\cdot q^{\vartheta_{2}(\eta^{rc},N,0)}}=\left\{
	\begin{aligned}
		&\frac{q_{0}^{N-1}\cdot q_{1-d(\eta)}^{\eta_{d(\eta)}}}{\prod\limits_{i=1}^{d(\eta)-1}q_{-i}}\cdot\prod\limits_{i=1}^{d(\eta)-1}\left(\frac{q_{-i+1}}{q_{-i}}\right)^{\eta_{i}} ,\;  \;\;\;\;\;\;\;\;\; \mbox{if $d(\eta)>1$}; \\
		&q_{0}^{N+\eta_{1}-1},   \;\;\;\;\;\;\;\;\;\;\;\;\;\;\;\;\;\;\;\;\;\;\;\;\;\;\;\;\;\;\;\;\;\;\;\;\;\;\;\;\;\;\;\;\;\;\;\;\;\mbox{if $d(\eta)=1$ and $\eta_{1}>1$};\\
		&\prod\limits_{i=1}^{N-1}q_{-i}^{-N},      \;\;\;\;\;\;\;\;\;\;\;\;\;\;\;\;\;\;\;\;\; \;\;\;\;\;\;\;\; \;\;\;\;\;\;\;\; \;\;\;\;\;\;\;\; \;\;\; \mbox{if $d(\eta)=1$ and $\eta_{1}=1$}.
	\end{aligned}
	\right.
	\een
\end{lemma}

\begin{lemma}\label{recurrence5-weight5}
	For any partition $\eta$, we have		
	\ben
	\frac{q^{\vartheta_{3}(\eta,N-1,-1,0)}\cdot q^{\vartheta_{3}(\eta,N+1,3,2)}}{\left(q^{\vartheta_{3}(\eta,N,1,1)}\right)^2}=\frac{q_{\ell(\eta)}^{N-1}}{\prod\limits_{i=\ell(\eta)+1}^{N-1}q_{i}}.
	\een
\end{lemma}

\begin{lemma}\label{recurrence5-weight6}
For any partition $\eta$, we have	
	\ben
	\frac{q^{\vartheta_{4}(\eta,N,1)}\cdot q^{\vartheta_{4}(\eta,N,-1)}}{\left(q^{\vartheta_{4}(\eta,N,0)}\right)^2}=\left\{
	\begin{aligned}
		&1 ,\;  \;\;\;\;\; \;\;\;\;\; \;\;\; \;\;\;\;\;\;\;\; \;\;\;\;\;\;\;\;\;\;\;\;\;\;\; \;\;\;\;\;\;\;\;\;\;\;\;\;\;\; \;\;\;\;\;\;\; \mbox{if $\eta=\emptyset$}; \\
		&\frac{q_{0}^{N-1}\cdot q_{\ell(\eta)}^{-N}}{\prod\limits_{i=1}^{\ell(\eta)-1}q_{i}}\cdot\prod\limits_{i=1}^{\ell(\eta)}\left(\frac{q_{i-1}}{q_{i}}\right)^{\eta_{i}},      \;\;\;\;\;\;\;\;\;\;\;\;\;\;\;\;\;\;\mbox{if $\eta\neq\emptyset$}.
	\end{aligned}
	\right.
	\een
\end{lemma}

\subsubsection{Weights for the graphical condensation recurrence with $(\mathbf{G}=\mathbf{H}(N),\mathbf{N}_{\lambda,\mu^{rc},\nu^{rc}}(N))$}
The similar argument in Section 5.4.1 shows that
\begin{lemma}\label{recurrence6-weight1}
The edge-weight of the base$_{left-down}$ double-dimer configuration is 
\ben
&&q^{\omega_{base}^{LD}(\lambda,\mu^r,\nu^c;n)}\\
&=&\prod_{i=1}^{N-\ell(\lambda^\prime)}\prod_{j=1}^iq_{-N+i-j+1}^{i-1}\cdot\prod_{i:1\leq i\leq\lambda^\prime_{i}}\prod_{j=1}^{N-\lambda^\prime_{i}+1}q_{-i-j+2}^{N+\lambda^\prime_{i}-i}\cdot\prod_{i:\lambda^\prime_{i}< i\leq\ell(\lambda^\prime)}\prod_{j=1}^{N+1-i}q_{-i-j+2}^{N+\lambda^\prime_{i}-i}\\
&&\times\prod_{i=1}^{N-\ell(\mu^r)-2}\prod_{j=\ell(\mu^r)+1}^{N-i-1}q_{i+j}^{N+i-1}\cdot\prod_{i:1\leq i<\mu^r_{i}-1}\prod_{j=1}^{N-\mu^r_{i}}q_{i+j}^{N+\mu^r_{i}+j-1}\cdot\prod_{i:\mu^r_{i}-1\leq i\leq\ell(\mu^r)}\prod_{j=1}^{N-i-1}q_{i+j}^{N+\mu^r_{i}+j-1}\\
&&\times \prod_{i=1}^{\ell(\nu^c)}\prod_{j=0}^{\nu^c_{i}-1}q_{j-i+2}^{N-i}\cdot \prod_{i=0}^{N-1}\prod_{j=0}^{N-2}q_{j-i+1}^{N-i-1}\\
&=&q^{\vartheta_{1}(\lambda^\prime,N,-1,0)}\cdot q^{\vartheta_{2}(\lambda^\prime,N+1,1)}\cdot q^{\vartheta_{3}(\mu^r,N-1,0,0)}\cdot q^{\vartheta_{8}(\mu^r,N,1,0)}\cdot
\left(q^{\varpi_{3}(\nu^c,N,2)}\right)^{-1}\cdot q^{\varpi_{1}(N-1,1,1)}.
\een

\end{lemma}

\begin{lemma}\label{recurrence6-weight2}
The edge-weight of the base$_{right-up}$ double-dimer configuration is 
\ben
&&q^{\omega_{base}^{RU}(\lambda,\mu^c,\nu^r;n)}\\
&=&\prod_{i=1}^{N-\ell(\lambda^\prime)-2}\prod_{j=1}^iq_{-N+i-j+1}^{i+1}\cdot\prod_{i:1\leq i\leq\lambda^\prime_{i}}\prod_{j=1}^{N-\lambda^\prime_{i}-1}q_{-i-j}^{N+\lambda^\prime_{i}-i}\cdot\prod_{i:\lambda^\prime_{i}< i\leq\ell(\lambda^\prime)}\prod_{j=1}^{N-i-1}q_{-i-j}^{N+\lambda^\prime_{i}-i}\\
&&\times\prod_{i=1}^{N-\ell(\mu^c)}\prod_{j=\ell(\mu^c)+1}^{N-i+1}q_{i+j-2}^{N+i-1}\cdot\prod_{i:1\leq i<\mu^c_{i}+1}\prod_{j=1}^{N-\mu^c_{i}}q_{i+j-2}^{N+\mu^c_{i}+j-1}\cdot\prod_{i:\mu^c_{i}+1\leq i\leq\ell(\mu^c)}\prod_{j=1}^{N-i+1}q_{i+j-2}^{N+\mu^c_{i}+j-1}\\
&&\times\prod_{i=1}^{\ell(\nu^r)}\prod_{j=0}^{\nu^r_{i}-1}q_{j-i}^{N-i}\cdot\prod_{i=0}^{N-1}\prod_{j=0}^{N}q_{j-i-1}^{N-i-1}\\
&=&q^{\vartheta_{1}(\lambda^\prime,N,1,2)}\cdot q^{\vartheta_{2}(\lambda^\prime,N-1,-1)}\cdot q^{\vartheta_{3}(\mu^c,N+1,2,2)}\cdot q^{\vartheta_{8}(\mu^c,N,-1,0)}\cdot \left(q^{\varpi_{3}(\nu^r,N,0)}\right)^{-1}\cdot q^{\varpi_{1}(N-1,-1,-1)}.
\een
\end{lemma}
In this case, we have
\ben
\dfrac{q^{\omega_{base}(\lambda,\mu^{rc},\nu;N)}\cdot q^{\omega_{base}(\lambda,\mu,\nu^{rc};nN)}}{q^{\omega_{base}(\lambda,\mu,\nu;N)}\cdot q^{\omega_{base}(\lambda,\mu^{rc},\nu^{rc};N)}}=1
\een
and still require the following lemmas to achieve Equation \eqref{PT-recurrence3}.
\begin{lemma}\label{recurrence6-weight7}
	For any partitions $\mu, \nu\neq\emptyset$ and $\lambda$, we have 
	\ben
	\dfrac{q^{\omega_{base}^{LD}(\lambda,\mu^r,\nu^c;N)}\cdot q^{\omega_{base}^{RU}(\lambda,\mu^{c},\nu^{r};N)}}{q^{\omega_{base}(\lambda,\mu,\nu;N)}\cdot q^{\omega_{base}(\lambda,\mu^{rc},\nu^{rc};N)}}=q^{-K_{3}(\lambda,\mu,\nu)}
	\een
where $q^{-K_{3}(\lambda,\mu,\nu)}=\left(q^{K_{3}(\lambda,\mu,\nu)}\right)^{-1}$ is independent of $N$.	
\end{lemma}
\begin{proof}
It follows from Lemmas \ref{recurrence4-weight1},	\ref{recurrence6-weight1}, \ref{recurrence6-weight2}, 
\ref{recurrence6-weight3}, 	\ref{recurrence6-weight4}, 
\ref{recurrence6-weight5},  \ref{recurrence6-weight6},  \ref{recurrence3-weight4}, \ref{recurrence3-weight3}.
\end{proof}

As in Section 5.4.1, we have
\begin{lemma}\label{recurrence6-weight3}
For any partition $\eta$, we have	
	\ben
	\frac{q^{\vartheta_{1}(\eta,N,-1,0)}\cdot q^{\vartheta_{1}(\eta,N,1,2)}}{\left(q^{\vartheta_{1}(\eta,N,0,1)}\right)^2}=\frac{q_{-\ell(\eta)}^{N-\ell(\eta)-1}}{\prod\limits_{i=1-N}^{-\ell(\eta)-1}q_{i}}.
	\een
\end{lemma}

\begin{lemma}\label{recurrence6-weight4}
	For any partition $\eta$, we have		
	\ben
	\frac{q^{\vartheta_{2}(\eta,N+1,1)}\cdot q^{\vartheta_{2}(\eta,N-1,-1)}}{\left(q^{\vartheta_{2}(\eta,N,0)}\right)^2}=\left\{
	\begin{aligned}
		&1 ,\;  \;\;\;\;\;\;\; \;\;\;\;\;\;\;\;\;\; \;\;\; \;\;\;\;\;\;\;\; \;\;\;\;\;\;\;\;\;\;\;\;\;\;\; \;\;\;\;\;\;\;\;\;\;\;\;\;\;\; \;\;\;\;\;\;\; \mbox{if $\eta=\emptyset$}; \\
		&\frac{q_{0}^{N-1}\cdot q_{-\ell(\eta)}^{\ell(\eta)-N}}{\prod\limits_{i=1}^{\ell(\eta)-1}q_{-i}}\cdot\prod\limits_{i=1}^{\ell(\eta)}\left(\frac{q_{-i+1}}{q_{-i}}\right)^{\eta_{i}},      \;\;\;\;\;\;\;\;\;\;\;\;\;\;\;\;\;\;\mbox{if $\eta\neq\emptyset$}.
	\end{aligned}
	\right.
	\een
\end{lemma}

\begin{lemma}\label{recurrence6-weight5}
	For any partition $\eta\neq\emptyset$, we have		
	\ben
	\frac{q^{\vartheta_{3}(\eta^c,N+1,2,2)}\cdot q^{\vartheta_{3}(\eta^r,N-1,0,0)}}{q^{\vartheta_{3}(\eta,N,1,1)}\cdot q^{\vartheta_{3}(\eta^{rc},N,1,1)}}
	=\left\{
	\begin{aligned}
		&1,    \;\;\; \;\;\;\;\;\;\;  \;\;\;\;\;\;\; \;\;\; \;\;\;\;\;\;\; \;\;\;  \;  \;\;\;\;\;\;\;\;\; \mbox{if $d(\eta)>1$ or $d(\eta)=1$ and $\eta_{1}>1$}; \\
		&q_{N-1}^{2N-1}\cdot\prod\limits_{i=0}^{N-2}q_{i}^{N+i},      \;\;\;\;\;\;\;\;\;\;\;\;\;\;\;\; \mbox{if $d(\eta)=1$ and $\eta_{1}=1$}.
	\end{aligned}
	\right.
	\een
\end{lemma}

\begin{lemma}\label{recurrence6-weight6}
	For any partition $\eta\neq\emptyset$, we have	
	\ben
	\frac{q^{\vartheta_{8}(\eta^c,N,-1,0)}\cdot q^{\vartheta_{8}(\eta^r,N,1,0)}}{q^{\vartheta_{4}(\eta,N,0)}\cdot q^{\vartheta_{4}(\eta^{rc},N,0)}}=\left\{
	\begin{aligned}
		&\frac{q_{0}^{N-1}\cdot q_{d(\eta)-1}^{\eta_{d(\eta)}}}{\prod\limits_{i=1}^{d(\eta)-1}q_{i}}\cdot\prod\limits_{i=1}^{d(\eta)-1}\left(\frac{q_{i-1}}{q_{i}}\right)^{\eta_{i}},           \;\;\; \;\;\;\;\;\;\;\;\; \mbox{if $d(\eta)>1$}; \\
		&q_{0}^{N+\eta_{1}-1},   \;\;\;\;\;\;\;\;\;\;\;\;\;\;\;\;\;\;\;\;\;\;\;\;\;\;\;\;\;\;\;\;\;\;\;\;\;\;\;\;\;\;\;\;\;\;\;\;\mbox{if $d(\eta)=1$ and $\eta_{1}>1$};\\
		&\prod\limits_{i=1}^{N-1}q_{i}^{-N-i},      \;\;\;\;\;\;\;\;\;\;\;\;\;\;\;\;\;\;\;\;\; \;\;\;\;\;\;\;\;\; \;\;\;\;\;\;\;\;\; \;\;\;\; \mbox{if $d(\eta)=1$ and $\eta_{1}=1$}.
	\end{aligned}
	\right.
	\een
\end{lemma}

\section{Orbifold DT/PT  correspondence}

\subsection{A proof of orbifold DT/PT vertex  correspondence}
In this subsection, we will give a proof of  Conjecture \ref{orbifold-DT/PT-correspondence} and show some applications of orbifold DT/PT vertex correspondence. In fact, the 1-leg correspondence has been proved by explicitly expressing the 1-leg   PT $\mathbb{Z}_{n}$-vertex formula in terms of Schur functions in [\cite{Zhang}, Theorem 5.22] via  establishing the connection between  1-leg colored labelled box configurations and  certain colored reverse plane partitions. Recall that a partition $\eta$ is called multi-regular if $|\eta|_{l}=\frac{|\eta|}{n}$ for all $l\in\{0, 1, \cdots, n-1\}$. 
\begin{theorem}([\cite{Zhang}, Theorem 5.22])\label{1-leg}
If $\nu$ is multi-regular, then
\ben
\dfrac{V^n_{\lambda\emptyset\emptyset}}{V_{\emptyset\emptyset\emptyset}^n}=W_{\lambda\emptyset\emptyset}^n;\;\;\;\;\;\;\;\;\dfrac{V^n_{\emptyset\mu\emptyset}}{V_{\emptyset\emptyset\emptyset}^n}=W_{\emptyset\mu\emptyset}^n;\;\;\;\;\;\;\;\;\dfrac{V^n_{\emptyset\emptyset\nu}}{V_{\emptyset\emptyset\emptyset}^n}=W_{\emptyset\emptyset\nu}^n.
\een
\end{theorem}
With the above theorem as the base case, we have
\begin{theorem}\label{3-leg correspondence}
If $\nu$ is multi-regular, then
\ben
V^n_{\lambda\mu\nu}=V_{\emptyset\emptyset\emptyset}^nW_{\lambda\mu\nu}^n.
\een
\end{theorem}
\begin{proof}
It follows from the recurrences \eqref{DT-recurrence1} and \eqref{PT-recurrence1} with $\nu=\emptyset$, Remark \ref{size-comparision},  and Theorem \ref{1-leg} that
\ben
\dfrac{V^n_{\lambda\mu\emptyset}}{V_{\emptyset\emptyset\emptyset}^n}=W_{\lambda\mu\emptyset}^n.
\een
It follows from the recurrences \eqref{DT-recurrence2} and \eqref{PT-recurrence2} with $\mu=\emptyset$, Remark \ref{size-comparision},  and Theorem \ref{1-leg} that
\ben
\dfrac{V^n_{\lambda\emptyset\nu}}{V_{\emptyset\emptyset\emptyset}^n}=W_{\lambda\emptyset\nu}^n.
\een
It follows from the recurrences \eqref{DT-recurrence3} and \eqref{PT-recurrence3} with $\lambda=\emptyset$, Remark \ref{size-comparision},  and Theorem \ref{1-leg} that
\ben
\dfrac{V^n_{\emptyset\mu\nu}}{V_{\emptyset\emptyset\emptyset}^n}=W_{\emptyset\mu\nu}^n.
\een
Therefore, the proof is completed by the recurrences \eqref{DT-recurrence1} and \eqref{PT-recurrence1}, Remark \ref{size-comparision}, and the 2-leg case just proved above.
\end{proof}
	
Now we recall some notation in [\cite{BCY}, Section 3.4] as follows. Let $\mathbf{q}_{0}=1$ and  define $\mathbf{q}_{t}$ for $t\in\mathbb{Z}$ by the recursive relations $\mathbf{q}_{t}=q_{t}\cdot \mathbf{q}_{t-1}$. Then we  have
\ben
\{\cdots, \mathbf{q}_{2}, \mathbf{q}_{1}, \mathbf{q}_{0}, \mathbf{q}_{-1}, \mathbf{q}_{-2}, \cdots\}=\{\cdots, q_{1}q_{2}, q_{1}, 1, q_{0}^{-1}, q_{0}^{-1}q_{-1}^{-1}, \cdots\}.
\een
For a partition $\eta=(\eta_{0}\geq\eta_{1}\geq\eta_{2}\geq\cdots)$, set
\ben
\mathbf{q}_{\bullet-\eta}=\{\mathbf{q}_{-\eta_{0}}, \mathbf{q}_{1-\eta_{1}}, \mathbf{q}_{2-\eta_{2}}, \mathbf{q}_{3-\eta_{3}}, \cdots\}.
\een
Let 
\ben
&&q^{-A_{\lambda}}:=\prod_{k=0}^{n-1}q_{k}^{-A_{\lambda}(k,n)},\\
&&H_{\nu}:=\prod_{(i,j)\in\nu}\frac{1}{1-\prod\limits_{s=0}^{n-1}q_{s}^{h_{\nu}^s(i,j)}},
\een
where
\ben
h_{\nu}^s(i,j):=\mbox{the number of boxes of color s in the $(i,j)$-hook of $\nu$}.
\een

It follows from Theorem \ref{3-leg correspondence} and [\cite{BCY}, Theorem 12] that

\begin{corollary}\label{application1}
The PT $\mathbb{Z}_{n}$-vertex $W_{\lambda\mu\nu}^n(q_{0},q_{1},\cdots,q_{n-1})$ has the following explicit formula:
\ben
W_{\lambda\mu\nu}^n=q^{-A_{\lambda}}\cdot \overline{q^{-A_{\mu^\prime}}}\cdot H_{\nu}\cdot \sum_{\eta}q_{0}^{-|\eta|}\cdot \overline{s_{\lambda^\prime/\eta}(\mathbf{q}_{\bullet-\nu^\prime})}\cdot s_{\mu/\eta}(\mathbf{q}_{\bullet-\nu})
\een
where $\nu$ is multi-regular, $s_{\xi/\eta}$ is the skew Schur function associated to partitions $\eta\subset\xi$ ($s_{\xi/\eta}=0$  if $\eta\not\subset\xi$), and the overline denotes the exchange of variables $q_{k}\leftrightarrow q_{-k}$ with subscripts in $\mathbb{Z}_{n}$.
\end{corollary}

\begin{remark}\label{DT-PT-correspondence1}
By [\cite{BCY}, Theorem 12],  it is actually proved in [\cite{Zhang}, Theorem 5.22] that for any partition $\nu\neq\emptyset$,
\ben
\dfrac{V^n_{\emptyset\emptyset\nu}}{ O_{\nu}\cdot V_{\emptyset\emptyset\emptyset}^n}=W_{\emptyset\emptyset\nu}^n
\een
where
\ben
O_{\nu}=\prod_{k=0}^{n-1}V_{\emptyset\emptyset\emptyset}^n(q_{k},q_{k+1},\cdots,q_{n+k-1})^{-2|\nu|_{k}+|\nu|_{k+1}+|\nu|_{k-1}}.
\een
However, this equality can not be extended to the full 3-leg case in  general when  $\nu$ is not multi-regular and $n>1$. Otherwise both $\dfrac{V^n_{\lambda\mu\nu}}{V_{\emptyset\emptyset\emptyset}^n}$  and $\dfrac{V^n_{\lambda\mu\nu}}{O_{\nu}\cdot V_{\emptyset\emptyset\emptyset}^n}$  ($\nu\neq\emptyset$) will satisfy Equation \eqref{recurrence2}  by recurrences \eqref{DT-recurrence2} and \eqref{PT-recurrence2}, which implies the equality $O_{\nu}\cdot O_{\nu^{rc}}= O_{\nu^r}\cdot O_{\nu^c}$. But for any $n\geq 2$, this equality does not hold for the non multi-regular partition $\nu=(1)$.
\end{remark}

\begin{remark}\label{loop-Schur-function}
It is shown in [\cite{RZ2}, Lemma A.1] that $H_{\nu}$ is related to the  loop Schur function $\mathfrak{s}_{\nu}$ by
\ben
\mathfrak{s}_{\nu}(q_{0},q_{1},\cdots,q_{n-1})=\left(\prod_{(i,j)\in\nu}q_{i-j}^j\right)\cdot H_{\nu}
\een  
where 
\ben
\mathfrak{s}_{\nu}(q_{0},q_{1},\cdots,q_{n-1}):=\left(\sum_{T\in SSYT(\nu,n)}\prod_{l=0}^{n-1}\prod_{\substack{(i,j)\in\nu\\i-j= l \;mod\; n}}q_{l,\omega((i,j),T)}\right)_{q_{i,j}=q_{i}^j}.
\een
Here, the partition $\nu\subset\mathbb{Z}_{\geq0}^2$ is viewed as a Young diagram in French notation which is different from  English notation used for  Young diagrams in [\cite{RZ2}]. For each  $\nu$ and $n\in\mathbb{N}$, the set  $SSYT(\nu,n)$ contains semi-standard Young tableaus of the colored Young diagram $\nu$ (the box $(i,j)\in\nu$ is colored by $i-j\mbox{ mod } n$) whose boxes are numbered such that numbers are weakly increasing from bottom to top and strictly increasing from left to right. For any $T\in SSYT(\nu,n)$, the weight $\omega((i,j),T)$ is defined to be the number appearing in the box $(i,j)$. The subscript of the parenthesis above means the substitution $q_{i,j}=q_{i}^j$ where $i\in\mathbb{Z}_{n}$ and $j\in\mathbb{N}$.
One may also refer to [\cite{Ross}] for loop Schur functions generalizing the classical Schur functions.
\end{remark}	

Next, we provide a proof of the multi-regular orbifold DT/PT correspondence in [\cite{Zhang}, Conjecture 4.7] for  toric CY 3-orbifolds with transverse $A_{n-1}$ singularities as follows.
\begin{corollary}\label{application2}
Let $\mathcal{X}$ be a  toric CY 3-orbifold with transverse $A_{n-1}$ singularities. Then we have 
\ben
DT^\prime_{mr}(\mathcal{X})=PT_{mr}(\mathcal{X}).
\een	
\end{corollary}	
\begin{proof}	
It follows from Theorem \ref{PT-partition function}, Theorem \ref{3-leg correspondence}, [\cite{BCY}, Theorem 10] and the degree 0 DT partition 
\ben
DT_{0}(\mathcal{X})=\prod_{v\in V_{\mathcal{X}}}V^{n_{e_{3,v}}}_{\emptyset\emptyset\emptyset}(\mathbf{q}_{v})\bigg|_{q_{e_{3,v},0}\to -q_{e_{3,v},0}}
\een
which is obtained by the similar argument in the proof of [\cite{BCY}, Theorem 10].
\end{proof}	

\begin{remark}
It is observed as in [\cite{RZ2,Ross3}] that if we choose for each edge $e$
\ben
&&\mathcal{O}_{\mathcal{C}_{e}}(\mathcal{D}_{e})=\mathcal{O}_{\mathcal{C}_{e}}(m^\prime_{e} p_{e}-\delta_{0,e}^\prime p_{0,e}-\delta_{\infty,e}^\prime p_{\infty,e}),\\
&&\mathcal{O}_{\mathcal{C}_{e}}(\mathcal{D}_{e}^\prime)=\mathcal{O}_{\mathcal{C}_{e}}(m_{e} p_{e}-\delta_{0,e}p_{0,e}-\delta_{\infty,e}p_{\infty,e})
\een
 which are different from the convention in Section 2.1 when $n_{e}=1$, then  we have 
\ben
\underline{PT}_{mr}(\mathcal{X})=\sum_{\{\lambda_{e}\}\in\Lambda_{\mathcal{X}}}\prod_{e\in E_{\mathcal{X}}}\widetilde{E}_{\lambda_{e}}^e\prod_{v\in V_{\mathcal{X}}}\widetilde{W}_{\lambda_{1,v}\lambda_{2,v}\lambda_{3,v}}^{n_{e_{3,v}}}(\mathbf{q}_{v})
\een
and 
\ben
PT_{mr}(\mathcal{X})=\underline{PT}_{mr}(\mathcal{X})\bigg|_{q_{e,0}\to -q_{e,0},\; q\to-q}
\een 
where 
\ben
&&\widetilde{E}_{\lambda_{e}}^e=(-1)^{(m_{e}+\delta_{0,e}+\delta_{\infty,e})|\lambda_{e}|}\cdot\left(\prod_{k=0}^{n_{e}-1}v_{e,k}^{\frac{|\lambda_{e}|}{n_{e}}}\right)\cdot q_{e}^{C_{m_{e}^\prime,m_{e}}^{\lambda_{e}}},\\
&&\widetilde{W}_{\lambda\mu\nu}^n= \left(\prod_{(i,j)\in\nu}q_{i-j}^{-j}\right)\cdot \mathfrak{s}_{\nu}\cdot \sum_{\eta}q_{0}^{-|\eta|}\cdot \overline{s_{\lambda^\prime/\eta}(\mathbf{q}_{\bullet-\nu^\prime})}\cdot s_{\mu/\eta}(\mathbf{q}_{\bullet-\nu}).
\een
\end{remark}

\subsection{Remarks on orbifold GW/DT correspondence and  DT crepant resolution conjecture}
In [\cite{RZ1,RZ2,Ross2}], the authors establish the orbifold GW/DT correspondence for toric CY 3-orbifolds with transverse $A_{n-1}$ singularities by proving orbifold GW/DT vertex correspondence which is compatible with their gluing formulas. As the reduced DT $\mathbb{Z}_{n}$-vertex used in [\cite{RZ1,RZ2,Ross2}] is the PT $\mathbb{Z}_{n}$-vertex by Theorem \ref{3-leg correspondence}, the orbifold GW/DT correspondence there is actually the orbifold GW/PT correspondence. Similarly, the proof of DT crepant resolution conjecture for toric CY 3-folds with transverse $A_{n-1}$ singularities in [\cite{Ross3}] can be viewed as the validity  of the corresponding PT crepant resolution conjecture. Since the orbifold PT theory can be interpreted as a geometric approach to the reduced orbifold DT theory as in the smooth case, we will investigate the orbifold GW/PT correpsondence further as in [\cite{PP3,PP5,PP1}].


\end{document}